\documentclass[a4paper,10pt]{article} 

\usepackage[utf8]{inputenc} % set input encoding (not needed with XeLaTeX)

\usepackage[margin=1in]{geometry} % to change the page dimensions
\usepackage[french]{babel}
\usepackage{csquotes}
\usepackage{graphicx} % support the \includegraphics command and options

\usepackage{amssymb}
\usepackage{stmaryrd}
\usepackage{xcolor} %couleur
\usepackage[bookmarks=false]{hyperref} %référence avec lien, \ref

\usepackage[symbol]{footmisc}

\usepackage{booktabs} % for much better looking tables
\usepackage{array} % for better arrays (eg matrices) in maths
\usepackage{paralist} % very flexible & customisable lists (eg. enumerate/itemize, etc.)
\usepackage{verbatim} % adds environment for commenting out blocks of text & for better verbatim
\usepackage{subfig} % make it possible to include more than one captioned figure/table in a single float
\usepackage{fancyhdr} % This should be set AFTER setting up the page geometry
\usepackage{sectsty}
\allsectionsfont{\sffamily\mdseries\upshape} % (See the fntguide.pdf for font help)
\usepackage[nottoc,notlof,notlot]{tocbibind} % Put the bibliography in the ToC
\usepackage[titles,subfigure]{tocloft} % Alter the style of the Table of Contents

\usepackage{times}
\usepackage{amsmath, amsfonts, bbm, dsfont, mathrsfs}

\usepackage[
backend=biber,
style=alphabetic
]{biblatex}

\bibliography{biblioEF.bib}

\usepackage{amsthm}
\usepackage{mathtools}
\usepackage{tikz-cd} % commutative diagrams
\usepackage{xfrac} % bien écrire des groupes quotients
\usepackage{enumitem}

\newtheorem{theo}{Théorème}[section]
\newtheorem{prop}[theo]{Proposition}
\newtheorem{lemma}[theo]{Lemme}
\newtheorem{coro}[theo]{Corollaire}

\theoremstyle{definition}
\newtheorem{defi}[theo]{Définition}
\newtheorem{rem}[theo]{Remarque}
\newtheorem{ex}[theo]{Exemple}
\newtheorem{defiprop}[theo]{Définition/Proposition}

\newcommand{\Z}{\mathbb{Z}}
\newcommand{\N}{\mathbb{N}}
\newcommand{\R}{\mathbb{R}}
\newcommand{\Q}{\mathbb{Q}}

\newcommand{\sep}{\mathrm{sep}}
\newcommand{\ext}{\mathrm{ext}}
\newcommand{\gp}{\mathrm{gp}}
\newcommand{\gaal}{\mathrm{gal}}
\newcommand{\gl}[2]{\mathrm{GL}_{#1}(#2)}

\newcommand{\coindu}[3]{\mathrm{Coind}_{#1}^{#2}(#3)}
\newcommand{\bigcoindu}[3]{\mathrm{Coind}_{#1}^{#2}\left( #3\right)}
\newcommand{\qp}[0]{\mathbb{Q}_p}
\newcommand{\zp}{\mathbb{Z}_p}

\newcommand{\hcal}[1]{\mathcal{#1}}
\newcommand{\hrm}[1]{\mathrm{#1}}
\newcommand{\und}[1]{\underline{#1}}

\newcommand{\cg}{\mathcal{G}}
\newcommand{\ch}{\mathcal{H}}
\newcommand{\rf}{\mathrm{f}}

\newcommand{\Oe}{\mathcal{O}_{\mathcal{E}}}

\newcommand{\Oedhat}{\mathcal{O}_{\widehat{\mathcal{E}_{\Delta}^{\mathrm{nr}}}}}
\newcommand{\Oehat}{\mathcal{O}_{\widehat{\mathcal{E}^{\mathrm{nr}}}}}

\newcommand{\Oekand}[1]{\mathcal{O}_{\mathcal{E}_{K,#1}}}
\newcommand{\Oehatkand}[1]{\mathcal{O}_{\widehat{\mathcal{E}_{K,#1}^{\mathrm{nr}}}}}

\newcommand{\detale}[2]{\hrm{Mod}^{\hrm{\acute{e}t}}\left(#1,#2\right)}
\newcommand{\detaleproj}[2]{\hrm{Mod}_{\hrm{prj}}^{\hrm{\acute{e}t}}\left(#1,#2\right)}
\newcommand{\detaledvproj}[3]{\hrm{Mod}_{#3 \hrm{\text{-}prjdv}}^{\hrm{\acute{e}t}}(#1,#2)}
\newcommand{\dmod}[2]{\hrm{Mod}\left(#1,#2\right)}

\newcommand{\cdetale}[2]{\mathscr{M}\hrm{od}^{\hrm{\acute{e}t}}(#1,#2)}
\newcommand{\cdetaleproj}[2]{\mathscr{M}\hrm{od}_{\hrm{prj}}^{\hrm{\acute{e}t}}\left(#1,#2\right)}
\newcommand{\cdetaledvproj}[3]{\mathscr{M}\hrm{od}_{#3 \hrm{\text{-}prjdv}}^{\hrm{\acute{e}t}}(#1,#2)}

\newcommand{\quot}[2]{\raisebox{.15em}{$#1$}/\raisebox{-.15em}{$#2$}}

\newcommand{\modulo}[1]{\,\,\, \mathrm{mod} \,\, #1}

\newcommand{\lt}{\mathrm{LT}}
\newcommand{\spa}[2]{\mathrm{Spa}\left(#1,#2\right)}
\newcommand{\spd}[2]{\mathrm{Spd}(#1,#2)}

\newcommand{\fet}[1]{\mathrm{F\acute{E}t}\left(#1\right)}
\newcommand{\piun}[2]{\pi_1^{\mathrm{f\acute{e}t}}\left(#1,\overline{#2}\right)}
\newcommand{\colim}[1]{\mathop{\mathrm{colim}}}
\newcommand{\spe}{\mathrm{sp\acute{e}}}
\newcommand{\plec}{\mathrm{plec}}
\newcommand{\tgplec}{\mathrm{T}\mathcal{G}_{K,\plec}}
\newcommand{\tplec}{\mathrm{T}_{K,\plec}}
\newcommand{\glec}{\mathrm{glec}}
\newcommand{\tgglec}{\mathrm{T}\mathcal{G}_{K,\glec}}
\newcommand{\tglec}{\mathrm{T}_{K,\glec}}
\newcommand{\tggloc}{\mathrm{T}\mathcal{G}'_{K,\glec}}
\newcommand{\tgloc}{\mathrm{T}'_{K,\glec}}
\newcommand{\gal}[2]{\mathrm{Gal}\left(#1|#2\right)}
\newcommand{\tsim}{\text{\textasciitilde}}

\graphicspath{./Visuels} % le chemin pour aller chercher les images

\title{\textsc{Multivariable Lubin-Tate or plectic Fontaine's equivalences for a $p$-adic local field}}
\author{Nataniel Marquis}
\usepackage[T1]{fontenc}

\begin{document}

	\begin{titlepage}
		\parindent=0pt
		\vspace{50ex}
		
		%\vspace*{\stretch{1}}
		%\begin{center}
		%\includegraphics[scale=.8]{brouwerdegree.png}
		%\end{center}
		%\vspace*{\stretch{1}}
		
		\hrulefill
		\begin{center}\bfseries \Huge \textbf{\'Equivalences de Fontaine multivariables Lubin-Tate et plectiques pour un corps local $p$-adique}
		\end{center}
		%
		%\begin{center}
		%\huge 
		%\end{center}
		\hrulefill

		\vspace*{0.5cm}
		
		\renewcommand*{\thefootnote}{\fnsymbol{footnote}}
		
		\begin{center}\Large
			par Nataniel Marquis\footnote[2]{\underline{n.marquis@uni-muenster.de}. Universität Münster, Mathematisches Institut, Orléans-Ring 12, 48149 Münster, Germany..}
		\end{center}
		
		\vspace{1.5 cm}
		\textit{\textbf{Abstract.}} \textemdash \,\,  Let $\Delta$ be a finite set. We adapt the techniques of Carter-Kedlaya-Z\'abr\'adi to obtain a multivariable Fontaine equivalence which relates continuous finite dimensional $\mathbb{F}_q$-representations of $\prod_{\alpha \in \Delta} \cg_{\mathbb{F}_q(\!(X)\!)}$ to multivariable $\varphi$-modules over a $\mathbb{F}_q$-algebra which is a domain. Building on this, we construct a multivariable Lubin-Tate period ring and deduce a multivariable Lubin-Tate Fontaine equivalence for continuous finite type \linebreak$\mathcal{O}_K$-representations of $\prod_{\alpha \in \Delta} \cg_K$, where $K|\qp$ is a finite extension. We also obtain a plectic Fontaine equivalence and two equivalences for the subgroup $\cg_{K,\glec}$ of the plectic Galois group.

\vspace{0.75 cm}
\textit{\textbf{Résumé.}} \textemdash \,\, Soit $\Delta$ un ensemble fini. Nous adaptons les méthodes de Carter-Kedlaya-Z\'abr\'adi pour obtenir une équivalence de Fontaine multivariable pour les représentations continues de $\prod_{\alpha \in \Delta} \cg_{\mathbb{F}_q(\!(X)\!)}$ de dimension finie sur $\mathbb{F}_q$ qui les fait correspondre à des $\varphi$-modules multivariables sur une $\mathbb{F}_q$-algèbre intègre. Nous construisons un anneau de périodes multivariable Lubin-Tate. À partir du résultat précédent, nous en déduisons une équivalence de Fontaine multivariable Lubin-Tate pour les représentations continues de $\prod_{\alpha \in \Delta} \cg_K$ de type fini sur $\mathcal{O}_K$, où $K|\qp$ est une extension finie. Nous en déduisons également une équivalence de Fontaine plectique et deux équivalences pour le sous-groupe $\cg_{K,\glec}$ du groupe de Galois plectique.
		
		%\begin{center}\bfseries\Large
		%Encadrant : Pierre COLMEZ (CNRS, Sorbonne Université)
		%\end{center}
		
		%\vspace*{0.8cm}
		%\begin{center}
		%$^1$\'Ecole normale supÃ©rieure, Sorbonne UniversitÃ©
		%\end{center}
		\vspace*{\stretch{3}}
		
		\begin{flushright}
			%Le \today            
		\end{flushright}

	\end{titlepage}
	
	\renewcommand*{\thefootnote}{\arabic{footnote}}
	
	\tableofcontents
	
	\hspace{5 cm}

	\section{Introduction}
	
	Soit $p$ un nombre premier et $K$ une extension finie de $\qp$ de corps résiduel fini $k_K$.	L'archétype d'une équivalence de Fontaine est le théorème de J.M.-Fontaine \cite[Th. 3.4.3]{Fontaine_equiv}. Il établit une équivalence de catégories explicite $$\mathbb{D}\, : \, \mathrm{Rep}_{\zp} \cg_{K} \rightleftarrows \cdetale{\varphi^{\N}\times \Gamma_K}{\Oe}\, :\, \mathbb{V}.$$ Ici, la source est la catégorie des $\zp$-représentations de type fini continues de $\cg_{K}:=\gal{\overline{\qp}}{K}$. La catégorie de droite est celle des $(\varphi,\Gamma_K)$-modules étales cyclotomiques (topologiques), des objets qui se sont avérés cruciaux en théorie de Hodge $p$-adique (voir \cite{herr}, \cite{colmez_correspondance} et \cite{emerton2022moduli}).
	
	Ces équivalences ont été généralisées dans deux directions. Bien que l'équivalence de Fontaine originelle considère l'extension cyclotomique $K(\mu_{p^{\infty}})$ pour se ramener à un corps de caractéristique $p$, il est possible de lui substituer extension de Lubin-Tate $K_{\lt,\pi}$. On obtient dans ce cas une équivalence de catégories explicite\footnote{Voir \cite[\S 1.4.1]{phd_fourquaux}, \cite{kr} ou \cite{schneider_tate_mod} pour une exposition détaillée.} $$\mathbb{D}_{\lt} \, : \, \mathrm{Rep}_{\mathcal{O}_K} \cg_K \rightleftarrows \cdetale{\varphi_q^{\N}\times \Gamma_{K,\lt}}{\mathcal{O}_{\mathcal{E}_K}}\, : \, \mathbb{V}_{\lt}$$ entre les $\mathcal{O}_K$-représentations de $\cg_K$ de type fini et continues et les $(\varphi_q,\mathcal{O}_K^{\times})$-modules étales (topologiques). Cette équivalence rend transparente la théorie du corps de classes locale dans la description des $(\varphi_q,\mathcal{O}_K^{\times})$-modules de dimension $1$. De plus, remarquer que $\Gamma_{K,\lt}\cong \begin{psmallmatrix}\mathcal{O}_K^{\times} & 0 \\ 0 & 1 \end{psmallmatrix}\subset\gl{2}{K}$ donne des espoirs pour adapter certaines idées de \cite{colmez_correspondance}. 
	
	Simultanément, de récents travaux sur un foncteur de Colmez pour des groupes réductifs $\qp$-déployés \linebreak(\cite{zabradi_gln}) et sur une compatibilité locale globale pour $\gl{2}{\mathbb{Q}_{p^f}}$ (\cite{breuilandco_phi}) utilisent des catégories de \linebreak$(\varphi,\Gamma)$-modules sur des anneaux multivariables. Parallèlement, \cite{zabradi_equiv} et \cite{zabradi_kedlaya_carter} établissent des équivalences de catégories entre les $\zp$-représentations libres et continues du produit finis $\cg_{K,\Delta}=\prod_{\alpha\in \Delta} \cg_K$ et des catégories de $(\varphi,\Gamma)$-modules cyclotomiques sur des anneaux multivariables.

	Cet article propose dans un premier temps de résoudre une question de l'introduction de \cite{zabradi_kedlaya_carter} : comment combiner les deux directions ci-dessus et obtenir une équivalence de Fontaine multivariable Lubin-Tate ? 
	
	\vspace{0.25 cm}	
		
	Pour énoncer notre équivalence, nous avons besoin des constructions suivantes. Fixons une uniformisante $\pi$ de $\mathcal{O}_K$, un polynôme de Lubin-Tate $\rf$ associé à $\pi$ et une famille $\pi^{\flat}=(\pi_n)_{n\geq 1}$ telle que $\rf(\pi_1)= 0$, $\pi_1\neq 0$ et $\forall n, \,\,\rf(\pi_{n+1})=\pi_n$. Nous utilisons les séries formelles $[\gamma]_{\lt,\rf}\in \mathcal{O}_K \llbracket T \rrbracket$ pour $\gamma \in \mathcal{O}_K$ qui définissent le \linebreak$\mathcal{O}_K$-module formel de la théorie de Lubin-Tate. Soit $\Delta$ un ensemble fini. Nous définissons le monoïde topologique discret $\Phi_{\Delta,q}:=\prod_{\alpha \in \Delta} \varphi_{\alpha,q}^{\N}$ et le groupe topologique $\Gamma_{K,\lt,\Delta}:=\prod_{\alpha \in \Delta} \mathcal{O}_K^{\times}$. Soient $$\Oekand{\Delta}:= \left(\mathcal{O}_K \llbracket X_{\alpha} \, |\, \alpha \in \Delta \rrbracket [X_{\Delta}^{-1}]\right)^{\wedge p}, \,\,\, \text{où} \,\,X_{\Delta}=\prod X_{\alpha}.$$ On peut le munir de trois topologies d'anneaux dont l'une, la topologie adique faible, a pour base de voisinages de $0$ les $$\left(\pi^m\Oekand{\Delta} + X_{\Delta}^n \mathcal{O}_K\llbracket X_{\alpha}\, |\, \alpha \in \Delta \rrbracket\right)_{n,m\geq 0}.$$ Il est également muni d'une action $\mathcal{O}_K$-linéaire du monoïde $(\Phi_{\Delta,q}\times \Gamma_{K,\lt,\Delta})$, continue pour n'importe laquelle des topologies sur $\Oekand{\Delta}$. L'action de $\Phi_{\Delta,q}$ vérifie $$ \forall \varphi=(\varphi_{\alpha,q}^{n_{\alpha}})_{\alpha\in \Delta} \in \Phi_{\Delta,q}, \,\, \forall \beta\in \Delta,\,\,\,\varphi(X_{\beta})=\rf^{\circ n_{\beta}}(X_{\beta}).$$ L'action de $\Gamma_{K,\lt,\Delta}$ vérifie
		
		$$\forall \gamma=(\gamma_{\alpha})_{\alpha \in \Delta} \in \Gamma_{K,\lt,\Delta}, \, \forall \beta \in \Delta, \,\,\, \gamma(X_{\beta})=[\gamma_{\beta}]_{\lt,\rf}(X_{\beta}).$$

	Avec ces notations le premier aboutissement de cet article est le théorème suivant.

	\begin{theo}[Voir Théorème \ref{equiv_lt}]\label{intro_equiv_lt}
		Choisissons l'une des trois topologies. Il existe une équivalence explicite de catégories symétriques monoïdales fermées $$\mathbb{D}_{\Delta,\lt} \, : \, \mathrm{Rep}_{\mathcal{O}_K} \cg_{K,\Delta} \rightleftarrows \cdetaledvproj{\Phi_{\Delta,q}\times \Gamma_{K,\lt,\Delta}}{\Oekand{\Delta}}{\pi}\, : \, \mathbb{V}_{\Delta,\lt},$$ où la catégorie au but est celle\footnote{Pour une définition plus précise, se référer à \cite[Déf. 5.24]{formalisme_marquis}.} des $\Oekand{\Delta}$-modules de type fini $D$ dont le dévissage $\sfrac{\pi^n D}{\pi^{n+1} D}$ est fini projectif sur $\sfrac{\Oekand{\Delta}}{\pi \Oekand{\Delta}}$, munis d'une action semi-linéaire continue de $(\Phi_{\Delta,q}\times \cg_{K,\Delta})$ pour laquelle chacune des images par l'action d'un $\varphi_{\alpha,q}$ engendre encore $D$.
	\end{theo}
	
	J'ai été averti après la première version de cet article que la thèse non-publiée \cite{pupazan} démontre une équivalence similaire. L’auteur y utilise les méthodes de \cite{zabradi_equiv} et ne démontre donc pas les résultats intermédiaires pour les corps de caractéristique $p$ (voir Théorèmes \ref{equiv_perf_mod_p}, \ref{equiv_non_perf_car_p} et \ref{equiv_car_zero}) que cette introduction n'explicite pas. Nous retrouvons son résultat et notre formalisme permet d’obtenir un choix de topologies
	plus varié\footnote{La topologie utilisée dans \cite{pupazan} est l'une des deux autres topologies.}.
	
	\vspace{0.25 cm}
	La preuve commence, comme dans \cite{Fontaine_equiv}, par établir une équivalence de Fontaine pour des représentations modulo $p$ de $\cg_{\widetilde{E},\Delta}$. Ici, le corps perfectoïde $\widetilde{E}$ choisi est $\mathbb{F}_q(\!(X^{p^{-\infty}})\!) \cong \widehat{K_{\lt,\pi}}^{\flat}$ où $q=|k_K|$. La démonstration de ce résultat intermédiaire suit une idée crucial dans \cite{zabradi_kedlaya_carter} : utiliser le lemme de Drinfeld pour les diamants pour construire géométriquement\footnote{Il n'est pas explicite dans \cite{zabradi_kedlaya_carter} que ce soit le même foncteur, mais une preuve se trouve dans la thèse de l'auteur.} le foncteur $\widetilde{\mathbb{V}}_{\Delta}$ afin d'analyser ses propriétés. Nous adaptons leur stratégie en nous assurant de conserver tous nos anneaux de coefficients de $\varphi$-modules et tous nos diamants sur $\mathbb{F}_q$. Grâce au lemme de Drinfeld pour les diamants sur $\mathbb{F}_q$, il suffit de construire à partir d'un $\Phi_{\Delta,q}$-module multivariable un objet en $\mathbb{F}_p$-espaces vectoriels dans une certaine catégorie de diamants $\fet{\prod_{\alpha \in \Delta, \, \mathrm{Spd}(\mathbb{F}_q)} \spd{\widetilde{E}}{\widetilde{E}^{\circ}} \, ||\, \Phi_{\Delta,q}}$ pour en déduire une représentation de $\cg_{\widetilde{E},\Delta}$. Nous obtenons une équivalence de Fontaine multivariable pour $\widetilde{E}$ dont l'anneau de coefficients $\widetilde{E}_{\Delta}$ des $\varphi$-modules correspondants est intègre : il s'agit d'une $\mathbb{F}_q$-algèbre et non d'une $\left(\otimes_{\alpha\in \Delta, \, \mathbb{F}_p} \mathbb{F}_q\right)$-algèbre. Cette intégrité est un premier ajout aux résultats de \cite{zabradi_kedlaya_carter}. À ce stade de l'article, elle permet de contourner certains des arguments les plus délicats dans la preuve de l'essentielle surjectivité de $\widetilde{\mathbb{D}}_{\Delta}$. Plus tard, avoir une structure canonique de $\mathbb{F}_q$-algèbre sera crucial pour l'équivalence multivariable Lubin-Tate. 
	
	Dans un second temps, un choix de topologie adéquat sur $\widetilde{E}_{\Delta}$ permet de déperfectoïdiser l'équivalence précédente pour que nos $\varphi$-modules soient à coefficients dans $E_{\Delta}:= \mathbb{F}_q\llbracket X_{\alpha} \, |\, \alpha \in \Delta\rrbracket [X_{\Delta}^{-1}]$. Dans un troisième temps, l'utilisation du $(\pi,\mu)$-dévissage introduit dans \cite[\S 4]{formalisme_marquis} permet d'obtenir une équivalence entre des $\mathcal{O}_K$-représentations de type fini continues de $\cg_{\widetilde{E},\Delta}$ et une catégorie $\detaledvproj{\Phi_{\Delta,q}}{\Oekand{\Delta}}{\pi}$ de $\Oekand{\Delta}$-modules. Notons également que notre équivalence capture les représentations sur $\zp$ qui ne sont pas nécessairement libres\footnote{Nous renvoyons à \cite{nataniel_calcul_phig} pour des exemples d'extensions de représentations qui ne se décomposent pas comme somme directe de représentations libres sur chaque $\sfrac{\Z}{p^n \Z}$ et sur $\zp$. Ce sont ces représentations dont l'image par une équivalence de Fontaine est caractérisée en utilisant \cite{formalisme_marquis}.} pour des corps de caractéristiques $p$. L'article \cite{zabradi_equiv} obtenait une telle équivalence en utilisant l'action de $\Gamma_{\Delta}$ et \cite{zabradi_kedlaya_carter} ne traite que les représentations libres\footnote{Voir les énoncés de \cite[Th. 4.5, 4.30, 4.31 et 6.15]{zabradi_kedlaya_carter}. Dans la preuve de \cite[Th. 4.39]{zabradi_kedlaya_carter}, il est même explicitement dit que les $(\varphi_{\Delta},\Gamma_{K,\Delta})$-modules considérés sont finis projectifs.}. Nous retrouvons également par coinduction les équivalences de \cite{zabradi_kedlaya_carter}.

	Pour terminer la preuve du théorème comme dans le cas univariable, nous regardons plus attentivement l'anneau de comparaison $\Oedhat$, caché jusque là. Pour obtenir l'équivalence pour le corps $\widetilde{E}$, nous avons utilisé une action de $(\Phi_{\Delta}\times \cg_{\widetilde{E},\Delta})$. Pour obtenir une équivalence Lubin-Tate il reste à étendre l'action précédente en une action de $(\Phi_{\Delta,q}\times \cg_{K,\Delta})$, où l'on a identifié $\cg_{\widetilde{E},\Delta}$ à $\prod_{\alpha \in \Delta} \cg_{K_{\lt,\pi}} \subset \cg_{K,\Delta}$, et que cette action induisent sur $\Oekand{\Delta}$ l'action déjà décrite.
	
	\vspace{0.5 cm}
	
	Ce texte démontre également des équivalences de Fontaine plectique et glectique. Le groupe de Galois plectique, introduit par J. Nekov\'a\v{r} et T. Scholl  dans \cite{nekoscholl_intro} et \cite{nekoschollI}, est défini par $$\cg_{K,\plec}:= \mathrm{Aut}_K\left(K\otimes_{\qp}\overline{\qp}\right).$$ Pour des données de Shimura globales associées à la restriction $\mathrm{Res}_{\Q}^F(H)$ d'un groupe algébrique $H$ sur une extension finie $F|\Q$, ils remarquent que la cohomologie étale est munie d'une action ad hoc de $\cg_{F,\plec}$ (et non seulement d'une action de $\cg_{\Q}$) et conjecturent que cette action peut-être construite fonctoriellement. En choisissant $(\pi,\rf,\pi^{\flat})$ comme dans le cas Lubin-Tate-multivariable, nous obtenons une équivalence de Fontaine pour ce groupe de Galois dans le cas local $p$-adique.
	 
	 \begin{theo}[Voir Théorème \ref{equiv_plec}]\label{intro_theo_plec}
	 	Il existe une équivalence explicite de catégories symétriques monoïdales fermées
	 		$$\mathbb{D}_{\plec,\lt}  \, : \, \mathrm{Rep}_{\mathcal{O}_K} \cg_{K,\plec} \rightleftarrows \cdetaledvproj{T_{K,\plec}}{\Oekand{\plec}}{\pi}\, : \, \mathbb{V}_{\plec,\lt}.$$
	\end{theo}
	
	Dans le théorème qui précède, le monoïde topologique $T_{K,\plec}$ s'écrit $$T_{K,\plec} \cong \left(\Phi_{\mathcal{P},q} \times \Gamma_{K,\lt,\mathcal{P}}\right) \rtimes_{\plec} \mathfrak{S}_{\mathcal{P}},$$ où $\mathcal{P}:=\{\tau\, : \,K\rightarrow \overline{\qp}\}$ et où $\mathfrak{S}_{\mathcal{P}}$ agit sur le facteur de gauche en permutant les copies de $\varphi_q^{\N}$ et $\Gamma_{K,\lt}$. En écrivant comme un produit en couronne $$\cg_{K,\plec} \cong \prod_{\tau\in \mathcal{P}} \cg_K \, \rtimes_{\plec} \mathfrak{S}_{\mathcal{P}},$$ l'anneau $\Oekand{\plec}$ se comprend comme l'anneau multivariable Lubin-Tate à $\mathcal{P}$ variables où l'on a ajouté une action de $\mathfrak{S}_{\mathcal{P}}$ qui permute les variables.
	
	Je me suis également intéressée à un sous-groupe du groupe de Galois plectique, le groupe de Galois glectique, et nous obtenons également dans cet article plusieurs équivalences de Fontaine pour ce dernier. \'Enonçons nos résultats dans le cas où $K|\qp$ est galoisienne. Le groupe glectique $\cg_{K,\glec}$ est le sous-groupe de $\cg_{K,\plec}$ contenant $\cg_{K,\mathcal{P}}$ qui correspond dont l'image dans $\mathfrak{S}_{\mathcal{P}}$ est égale à l'image de $\gal{K}{\qp}$ par l'action de composition à droite sur les plongements. Nous établissons deux équivalences glectiques.
	
	\begin{theo}[Voir Théorèmes \ref{equiv_sglec} et \ref{equiv_glec}]\label{intro_glec}
		Il existe deux équivalences explicites de catégories symétriques monoïdales fermées
		
			$$\mathbb{D}_{s\glec,\lt}  \, : \, s\mathrm{Rep}_{\mathcal{O}_K} \cg_{K,\glec} \rightleftarrows \cdetaledvproj{\tglec}{\Oekand{s\glec}}{\pi}\, : \, \mathbb{V}_{s\glec,\lt}.$$
			
			$$\mathbb{D}_{\glec,\lt}  \, : \, \mathrm{Rep}_{\mathcal{O}_K} \cg_{K,\glec} \rightleftarrows \cdetaledvproj{\tglec}{\Oekand{\glec}}{\pi}\, : \, \mathbb{V}_{\glec,\lt}.$$
	\end{theo}
	
	Ici, la catégorie $s\mathrm{Rep}_{\mathcal{O}_K} \cg_{K,\glec}$ est celle des $\mathcal{O}_K$-représentations semi-linéaires continues de type fini où $\cg_{K,\glec}$ agit sur $\mathcal{O}_K$ via son quotient $\gal{K}{\qp}$. Le monoïde $\tglec$ s'écrit 
	
	$$\tglec \cong \quot{\left(\left(\Phi_{\mathcal{P},q}\times \Gamma_{K,\lt,\mathcal{P}}\right) \rtimes_{\glec} \sfrac{\mathrm{W}_{\qp}^+}{\mathrm{I}_{K}}\right)}{\tsim \sfrac{\mathrm{W}_{K}^+}{\mathrm{I}_{K}}\tsim}$$ où $\mathrm{W}_{\qp}^+$ agit en permutant les copies de $\varphi_q^{\N}$ et $\Gamma_{K,\lt}$ et où le quotient identifie $(\varphi_{\tau,q})_{\tau\in \mathcal{P}}\in\Phi_{\mathcal{P},q}$ au générateur de $\mathrm{W}_{K}^+/\mathrm{I}_K$. Les anneaux $\Oekand{s\glec}$ et $\Oekand{\glec}$ correspondent quant à eux à ajouter sur l'anneau multivariable Lubin-Tate $\Oekand{\mathcal{P}}$ une action de $\mathrm{W}_{\qp}^+/\mathrm{I}_{K}$ qui permute les variables et se souvient des degrés des plongements. Pour $K=\mathbb{Q}_{p^f}$ par exemple, les variables sont indexées par $\llbracket 0, f-1\rrbracket$ et le générateur $\mathrm{Frob} \in \mathrm{W}_{\qp}^+/\mathrm{I}_{\mathbb{Q}_{p^f}}$ agit par $\mathrm{Frob}(X_i)=X_{i+1}$ et $\mathrm{Frob}(X_{f-1})=\rf(X_0)$. Cela entrebaille une minscule porte : obtenir une représentation glectique à partir d'une représentation de $\cg_{K,\mathcal{P}}$ consiste à préciser l'action d'un sous-groupe canonique $\cg_{\qp}\subset\cg_{K,\glec}$, or ce genre d'action abondent dans les constructions provenant de la géométrie.
	
	Pour prouver ces résultats à partir de l'équivalence multivariable Lubin-Tate, nous utilisons \cite{formalisme_marquis} donne une liste de conditions à vérifier pour obtenir une équivalence de Fontaine. Dans les cas plectique ou glectique, elles reviennent à construire le bon anneau de comparaison et le bon monoïde à faire agir, le reste étant déjà démontrées par le cas Lubin-Tate multivariable.
	\vspace{0.5 cm}
	
	La \textit{section 1} propose une introduction aux lois de Lubin-Tate et aux anneaux de l'équivalence de Fontaine Lubin-Tate. Il s'agit principalement de fixer les notations pour la suite. La \textit{section 2} établit l'équivalence de Fontaine multivariable pour des $\mathcal{O}_K$-représentations et groupes de Galois de corps perfectoïdes de caractéristique $p$. Nous commençons d'abord par établir la version perfectoïde en caractéristique $p$ grâce au lemme de Drinfeld, puis à déperfectoïdiser et dévisser. La \textit{section 3} construit les anneaux multivariables Lubin-Tate et établit le Théorème \ref{intro_equiv_lt}. La \textit{section 4} démontre les Théorème \ref{intro_theo_plec} et \ref{intro_glec}. L'\textit{annexe A} effectue une étude détaillée et technique des anneaux $\widetilde{E}_{\Delta}$, qui aurait coupé le rythme de la section 2. L'\textit{annexe B} liste des constructions et résultats sur les monoïdes topologiques dont les équivalences plectique et glectiques ont besoin.

	\vspace{1 cm}
	
	\textbf{\underline{Remerciements :}} les résultats de cet article ont été obtenu durant ma thèse, sous la direction de Pierre Colmez et Antoine Ducros. Je les remercie tous les deux chaleureusement. J'étais alors simultanément membre de l'IMJ-PRG à Sorbonne Université et du DMA à l'École Normale Supérieure-PSL. J'aimerais remercier tout particulièrement Pierre Colmez pour les nombreuses discussions, à la fois sur les grandes directions à prendre, et sur certaines difficultés techniques. Je suis reconnaissant à Benjamin Schraen et Gergely Zábr\'adi d'avoir accepté de rapporter ma thèse. Je remercie enfin toustes celleux dont les discussions informelles sur un coin de tableau, de mail ou d'après-midi ont été de sacrés coups de pouce : Gaëtan Chenevier, Arthur-César Le Bras, Ariane Mézard, Arnaud Vanhaecke, Paul Wang et  Guillaume Pignon--Ywanne.

	\clearpage
	
	\section*{Notations et conventions}
	
	Dans tout cet article, nous fixons un premier $p$ et une clôture algébrique $\overline{\qp}$ de $\qp$. Nous notons $\mathbb{C}_p$ la complétion $p$-adique de $\overline{\qp}$ avec son anneau d'entiers $\mathcal{O}_{\mathbb{C}_p}$ et son idéal maximal $\mathfrak{m}_{\mathbb{C}_p}$. La lettre $q$ désigne toujours un entier de la forme $p^f$. Pour une telle puissance $q$ et un anneau $A$ dans lequel $p=0$, nous appelons $q$-Frobenius l'endomorphisme d'anneaux donné par la puissance $q$-ième. Lorsque $q=p$, nous l'appelons parfois Frobenius absolu. Nous fixons également $\overline{\mathbb{F}_p}$ une clôture algébrique de $\mathbb{F}_p$ et nous appelons $\mathbb{F}_q$ les points fixes du $q$-Frobenius.
	
	Par extension de corps, écrit $l|k$, nous désigner un plongement $k\hookrightarrow l$. Pour un corps $k$ et une clôture algébrique $\overline{k}$ fixée, nous notons $\cg_k:=\gal{\overline{k}}{k}$ le groupe de Galois absolu muni de sa topologie profinie. 
	
Dire que l'action d'un groupe (ou d'un monoïde) topologique $G$ sur un espace topologique $X$  est continue signifie dans ce texte que l'application déduite $G\times X \rightarrow X$ est continue. Ainsi, l'action de $\cg_k$ sur $\overline{k}$ avec la topologie discrète est continue. Pour un monoïde (topologique) $\mathcal{S}$ et un anneau (topologique) $R$ muni d'une action (continue) de $\mathcal{S}$ par morphismes d'anneaux, nous avons défini et étudié dans \cite{formalisme_marquis} la catégorie $\dmod{\mathcal{S}}{R}$ des $R$-modules munis d'une action semi-linéaire de $\mathcal{S}$. Nous utilisons sans rappel la sous-catégorie pleine des modules étales finis projectifs de rang constant $\detaleproj{\mathcal{S}}{R}$. Lorsque l'on spécifie un élément $r$ pour lequel $R$ est $r$-adiquement séparé complet et sans torsion, nous utilisons la variante $\detaledvproj{\mathcal{S}}{R}{r}$ des modules à $r$-dévissage projectif, i.e. tels que tous les $r^n D/ r^{n+1} D$ sont projectifs de rang constant sur $R/r$. Lorsque nos anneaux et monoïdes sont munis de topologies, nous utilisons les variantes $\cdetale{\mathcal{S}}{R}$, $\cdetaleproj{\mathcal{S}}{R}$ et $\cdetaledvproj{\mathcal{S}}{R}{r}$.

	\vspace{1.5cm}
	
	\section{L'équivalence de Fontaine Lubin-Tate}\label{section_intro_lt}
	
	\vspace{0.75cm}
	\subsection{Lois de Lubin-Tate}
	
	Soit $\overline{\qp}|K|\qp$ une extension finie. Nous donnons ici les résultats essentiels de la construction des lois de groupes formels de Lubin-Tate et de l'extension abélienne associée. Pour une exposition plus détaillée culminant par théorie du corps de classes locale, les lecteurs et lectrices pourront se référer à \cite{yoshida}. Nous fixons pour cette section une uniformisante $\pi$ de $K$ et nous notons $q$ le cardinal de son corps résiduel.
	
	\begin{defi}
		Un polynôme de Lubin-Tate est un polynôme (de degré $q$) unitaire $\rf\in \mathcal{O}_K[T]$ tel que $$\mathrm{f}(T) \equiv T^q + \pi T \modulo{(\pi T^2)}.$$ 
	\end{defi}
	
	\begin{ex}
		Pour $K=\qp$ et l'uniformisante $p$, il existe un polynôme de Lubin-Tate bien pratique $(1+T)^p-1$.
		
		Pour $K$ général, une choix naturell consiste à prendre $\rf=T^q+\pi T$.
	\end{ex}
	\vspace{0.25cm}
	
	Nous fixons pour la suite un polynôme de Lubin-Tate $\rf$.
	
	\begin{prop}
		Il existe une unique $T_1 +_{\lt,\rf} T_2 \in \mathcal{O}_K \llbracket T_1, T_2\rrbracket$ telle que $$T_1 +_{\lt}  T_2 \equiv T_1 + T_2 \modulo{(T_1,T_2)^2} \,\,\, \mathrm{et} \,\,\, \rf(T_1 +_{\lt,\rf} T_2)=\rf(T_1) +_{\lt,\rf} \rf(T_2).$$ Elle est associative et commutative au sens où $$(T_1 +_{\lt,\rf} T_2) +_{\lt,\rf} T_3 = T_1 +_{\lt,\rf} (T_2 +_{\lt,\rf} T_3) \,\,\, \text{et}\,\,\, T_1 +_{\lt,\rf} T_2 = T_2 +_{\lt,\rf} T_1.$$ De plus, il existe une unique $i\in \mathcal{O}_K\llbracket T \rrbracket$ telle que $T+_{\lt,\rf} i(T)=0$.
		
		Pour tout $a\in \mathcal{O}_K$, il existe une unique $[a]_{\lt,\rf} \in \mathcal{O}_K\llbracket T \rrbracket$ telle que $$[a]_{\lt,\rf}\equiv aT \modulo{T^2} \,\,\, \text{et} \,\,\, [a]_{\lt,\rf}\circ \rf = \rf \circ [a]_{\lt,\rf}.$$
	\end{prop}
	
	Nous avons composé des séries formelles dans $T \mathcal{O}_K\llbracket T \rrbracket$ et/ou $T_1 \mathcal{O}_K\llbracket T_1,T_2 \rrbracket + T_2 \mathcal{O}_K\llbracket T_1,T_2 \rrbracket$. Remarquons dès à présent qu'une telle série s'évalue sur $\mathfrak{m}_{\mathbb{C}_p}$.
	
	\begin{coro}\label{coro_module_formel}
		Les séries construites précédemment vérifient $$[0]_{\lt,\rf}=0,\,\,\, [1]_{\lt,\rf}=T, \,\,\, 		i=[-1]_{\lt,\rf}\,\,\, \text{et}\,\,\, \forall n\geq 1, \,\,\, [\pi^m]_{\lt,\rf}=\rf^{\circ m}$$
		$$\forall a,b \in \mathcal{O}_K, \,\,\, [a+b]_{\lt,\rf}=[a]_{\lt,\rf} +_{\lt,\rf} [b]_{\lt,\rf} $$
		 $$\forall a,b\in \mathcal{O}_K,\,\,\, [ab]_{\lt,\rf}=[a]_{\lt,\rf} \circ [b]_{\lt,\rf}.$$
	\end{coro}
	
	\begin{ex}
		Pour le cas particulier de $K=\mathbb{Q}_p$ et $\rf=(1+T)^p-1$, nous obtenons $$ T_1 +_{\lt,\rf} T_2 = T_1 + T_2 + T_1 T_2, \,\,\, i(T)=\sum_{n\geq 1} (-1)^n T^n $$		$$\text{et} \,\,\, [a]_{\lt,\rf}=(1+T)^a-1:= \sum_{n\geq 1} \binom{a}{n} T^n.$$
	\end{ex}
	
	\begin{defi}
		Pour $m\geq 1$, nous définissons l'extension finie galoisienne $\overline{\qp}|K_{\lt,\rf,m}|K$ comme le corps de décomposition de $\rf^{\circ m}$ sur $K$. Définissons $$K_{\lt,\rf}:=\bigcup_{m\geq 1} K_{\lt,\rf,m}.$$
		
		Définissons $\mu_{\rf,m}$ l'ensemble des racines de $\rf^{\circ m}$ dans $K_{\lt,\rf,m}$ et $\mu_{\rf,m}^{\times}=\mu_{\rf,m} \backslash \mu_{\rf,m-1}$.
	\end{defi}
	
	\begin{theo}
		\begin{enumerate}[itemsep=0mm]
			\item Pour tout $x\in \mu_{\rf,m}^{\times}$, l'extension $K_{\lt,\rf,m}$ est engendrée par $x$.
			
			\item Pour tout $x \in \mu_{\rf,m}^{\times}$ les éléments de $\mu_{\rf,m}^{\times}$ sont précisément les $[a]_{\lt,\rf}(x)$ pour $a\in \mathcal{O}_K^{\times}$.
			
			\item Au regard des points précédents, l'application $$\iota_{\pi} \, : \,\mathcal{O}_K^{\times} \rightarrow \gal{K_{\lt,\rf}}{K}, \,\,\, a\mapsto \left[x\in \bigcup_{m\geq 1}  \mu_{f,m} \mapsto [a]_{\lt,\rf}(x)\right]$$ est correctement définie. C'est un isomorphisme de groupes topologiques.
		\end{enumerate}
	\end{theo}
	
	\begin{defi}
		Nous utilisons le caractère $$\chi_{\lt,\pi} \, : \,\cg_K \rightarrow \mathcal{O}_{K}^{\times}, \,\,\, \sigma \mapsto \iota_{\pi}^{-1}\left(\sigma_{|K_{\lt,\rf}}\right).$$
	\end{defi}

	\vspace{0.75cm}
	\subsection{Corps de normes imparfait dans $\widetilde{A}_K$}
	
		Nous définissons le corps complet de valuation discrète $E:= \mathbb{F}_q(\!(X)\!)$ et fixons une clôture séparable $E^{\sep}$. L'équivalence de Fontaine pour les corps de caractéristique $p$, dont on trouve l'énoncé original dans \cite[Prop. 1.2.6]{Fontaine_equiv}, établit en premier lieu une équivalence pour les représentations sur $\mathcal{O}_K$ de $\cg_E$. Récapitulons comment en déduire l'équivalence de Fontaine version Lubin-Tate.

	\begin{defi}
		Un système de Lubin-Tate associé à $\rf$ est un générateur du $\mathcal{O}_K$-module libre de rang $1$ donné par $$\lim \limits_{\substack{\longleftarrow\\ m\geq 0 \\ a\mapsto \rf(a)}} \mu_{\rf,m}.$$ Concrètement, il s'agit d'une famille $\pi^{\flat}=(\pi_{m})$ telle que $\pi_0=0$, $\pi_1 \in\mu_{\rf,1}^{\times}$ et $\forall m, \,\, \rf(\pi_{m+1})=\pi_m$. Elle s'identifie à un élément de $\mathbb{C}_p^{\flat}$.
		\end{defi}
		
		\begin{prop}\label{identif_theo_galois}
			Le corps $\widehat{K_{\lt,\rf}}$ est perfectoïde. L'élément $\pi^{\flat}$ dans son basculé est une pseudo-uniformisante et le morphisme de corps topologiques $$j\, :\, \mathbb{F}_q(\!(X^{p^{-\infty}})\!) \rightarrow \widehat{K_{\lt,\rf}}^{\flat}, \,\,\, X\mapsto \pi^{\flat}$$ est un isomorphisme.
		\end{prop}
		
		Puisque la théorie de Galois d'un corps ne change pas en le complétant, en prenant des extensions radicielles et en basculant des corps perfectoïdes, les corps $E$ et $K_{\lt}$ ont la même théorie de Galois. Plus précisément, pour une clôture séparable de $\mathbb{F}_q(\!(X^{q^{-\infty}})\!)$ fixées et un choix d'extension à ladite clôture de $j$ à valeurs dans $\mathbb{C}_p^{\flat}$, l'application $E\mapsto K_{\lt,\rf}^{\mathrm{alg}} \cap j(E \mathbb{F}_q(\!(X^{q^{-\infty}})\!))^{\sharp}$ est une bijection entre extensions finies, respecte le caractère galoisien et donne des isomorphismes de groupe de Galois. Nous définissons $\ch_{K,\lt,\rf}:= \cg_{K_{\lt,\rf}} \triangleleft \cg_K$, isomorphe à $\cg_E$, canoniquement si l'on fixe des clôtures séparables et une extension de $j$. Il devient plausible de promouvoir l'équivalence de Fontaine pour $\cg_E$ quitte à définir correctement nos anneaux.

	\begin{defi}\label{def_oe}
		Définissons $$\mathcal{O}_{\mathcal{E}}^+:=\mathcal{O}_K\llbracket X\rrbracket \,\,\, \text{et}\,\,\, \mathcal{O}_{\mathcal{E}}:=\mathcal{O}_{\mathcal{E}}^+\left[X^{-1}\right]^{\wedge p}.$$ 
	\end{defi}
	
	\begin{prop}\label{prop_defi_oe}\label{defi_oe}
		\begin{enumerate}[itemsep=0mm]
			\item Il existe une $\mathcal{O}_{\mathcal{E}}$-algèbre  $\mathcal{O}_{\widehat{\mathcal{E}^{\mathrm{nr}}}}$ qui est $\pi$-adiquement séparée et complète, telle que $\Oehat/\pi\cong E^{\sep}$. Elle est unique à isomorphisme près. 
		\end{enumerate}
		\vspace{0.05cm}
	Nous supposons désormais une telle algèbre fixée.
		
		\begin{enumerate}[itemsep=0mm]
			\setcounter{enumi}{1}
			\item Pour tout $h \in \mathrm{End}_{\mathrm{Ann}}(E^{\sep})$ tel que $h(E)\subseteq E$ et tout $f \in \mathrm{End}_{\mathcal{O}_K}\left(\Oe\right)$ tel que $(f\modulo{\pi})=h_{|E}$, il existe une unique $f_1\in \mathrm{End}_{\mathcal{O}_K} \mathcal{O}_{\widehat{\mathcal{E}^{\mathrm{ur}}}}$ telle que $(f_1)_{|\Oe}=f$ et $(f_1 \modulo{\pi})=h$.
					
			\item Pour toute extension finie $E^{\sep}|F|E$, il existe une unique sous-$\mathcal{O}_{\mathcal{E}}$-algèbre $\mathcal{O}_{\mathcal{F}}$ de $\mathcal{O}_{\widehat{\mathcal{E}^{\mathrm{ur}}}}$, $\pi$-adiquement complète et séparée, de corps résiduel $F$. Nous notons $\mathcal{O}_{\mathcal{F}}^+$ la clôture intégrale de $\mathcal{O}_{\mathcal{E}}^+$ dans $\mathcal{O}_{\mathcal{F}}$.
			
			\item Dans le cadre du deuxième point, si $h$ stabilise $F$ (resp. $F^+$), alors l'extension de $f$ stabilise $\mathcal{O}_{\mathcal{F}}$ (resp. $\mathcal{O}_{\mathcal{F}}^+$).
		\end{enumerate}
	\end{prop}
	\begin{proof}
	Conséquences de \cite[\href{https://stacks.math.columbia. edu/tag/04GK}{Tag 04GK}]{stacks-project} et \cite[\href{https://stacks.math.columbia. edu/tag/08HQ}{Tag 08HQ}]{stacks-project}.
	\end{proof}
		
Nous avons donc choisi un $\cg_E$-anneau topologique $\Oehat$ avec la topologie $\pi$-adique tel que $\Oe=\Oehat^{\cg_E}$.

Soit $\widetilde{A}_K:= \mathrm{W}_{\mathcal{O}_K}(\mathbb{C}_p^{\flat})$ que l'on munit de la topologie faible, i.e. la topologie produit via l'identification $(\mathbb{C}_p^{\flat})^{\N} \cong \mathrm{W}_{\mathcal{O}_K}(\mathbb{C}_p^{\flat}), \,(x_n)\mapsto \sum_{n\geq 0} \pi^n [x_n]$. Il est également muni d'une structure de $\left(\varphi_q^{\N}\times \cg_K\right)$-anneau topologique en relevant $\mathcal{O}_K$-linéairement les actions sur $\mathbb{C}_p^{\flat}$
	
	\begin{prop}\label{construct_oek}
		Il existe un élément\footnote{Notons que l'on peut retrouver $\rf$ à partir de $\pi^{\flat}$ : nous nous épargnons donc des indices inutiles.} $\{\pi^{\flat}\}_{\lt}$ dans $\widetilde{A}_K$, topologiquement nilpotent, tel que :
		
		\begin{enumerate}[label=\alph*),itemsep=0mm]
			\item On a $\{\pi^{\flat}\}_{\lt} \modulo{\pi}=\pi^{\flat}$.
			
			\item L'application $$\mathcal{O}_{\mathcal{E}} \rightarrow \widetilde{A}_K, \,\,\, X\mapsto \{\pi^{\flat}\}_{\lt}$$ est injective, d'image contenue dans $\mathrm{W}_{\mathcal{O}_K}\big(\widehat{K_{\lt,\rf}}^{\flat}\big)$ et stable par $(\varphi_q^{\N}\times \cg_K)$.
			
			\item L'image est invariante par $\ch_{K,\lt,\rf}$ et l'action $\mathcal{O}_K$-linéaire de $(\varphi_q^{\N}\times \mathcal{O}_K^{\times})$ déduite sur $\mathcal{O}_{\mathcal{E}}$ vérifie $$\varphi_q(X)=\rf(X) \,\,\, \text{et}\,\,\, \forall a\in \mathcal{O}_K^{\times}, \,\, a\cdot X=[a]_{\lt,\rf}(X).$$
			
			\item La topologie d'anneau induite sur $\mathcal{O}_{\mathcal{E}}$ est la topologie faible qui a pour base de voisinage de zéro les $\left(p^n \mathcal{O}_{\mathcal{E}} + X^m \mathcal{O}_{\mathcal{E}}^+\right)_{n,m\geq 0}$.
			
			\item L'application précédente s'étend en une injection de $\mathcal{O}_{\widehat{\mathcal{E}^{\mathrm{nr}}}}$ dans $\widetilde{A}_K$. Son image est canonique, stable par $(\varphi_q^{\N}\times \cg_K)$. L'action de $\cg_{K_{\lt}}$ s'identifie,via l'identification des théories de Galois précédemment construite, à l'action de $\cg_E$ sur la hensélisation.
			
			La topologie induite est encore la topologie faible.
		\end{enumerate}
		
	On appelle $\mathcal{O}_{\mathcal{E}_K}$ et $\mathcal{O}_{\widehat{\mathcal{E}_K^{\mathrm{nr}}}}$ les $(\varphi_q^{\N}\times \cg_K)$-anneaux topologiques obtenus ci-dessus.
	\end{prop}
	
	\begin{rem}
		Cette exposition fonctionne également dans le cas cyclotomique, ou plus généralement celui d'une extension galoisienne $K_{\infty}|K$ de groupe de Galois localement isomorphe à $\zp$. Ici on pourrait également utiliser la théorie du corps des normes imparfait dans \cite{wintenberger_corps_normes} pour identifier les théories de Galois de $K_{\infty}$ et de $\mathbb{F}_q(\!(X)\!)$.%Cependant, obtenir une théorie des $(\varphi,\Gamma)$-modules requiert une application $\Oe \rightarrow \widetilde{A}_K$ comme à la Proposition précédente dont l'image est stable par $\cg_K$. Fontaine-Wintenberger savait le faire pour l'extension cyclotomique mais l'action de $\gal{K_{\infty}}{K}$ n'a alors pas d'expression aisée. L'extension Lubin-Tate permet également cette construction et la proposition précédent illustre que l'action de $\gal{K_{\lt,\rf}}{K}$ sur $\Oe$ possède une expression explicite. L. Poyeton a donné une condition nécessaire et suffisante pour pouvoir trouver de tels relevés \cite[Th. A]{poyeton}.
	\end{rem}

	\vspace{1.5cm}
	
	\section{\'Equivalence de Fontaine multivariable  pour certains corps perfectoïdes de caractéristique $p$}

	Dans cette section, nous commençons par établir une équivalence de Fontaine pour les représentations \linebreak$p$-adiques de produits de groupes de Galois absolus de corps perfectoïdes de caractéristique $p$. Suivant l'idée de \cite{zabradi_kedlaya_carter}, nous commençons par utiliser le lemme de Drinfeld pour les diamants pour passer de modules sur un anneaux perfectoïdes multivariables $\widetilde{E}_{\Delta}$ à des $\widetilde{E}_{\Delta}$-algèbres finies étales puis à des ensembles avec action du produit de groupes souhaité. Nous nous attachons à conserver l'intégrité de nos anneaux. Le cas qui nous intéresse pour les corps locaux $p$-adiques est celui des corps perfectoïdes $\mathbb{F}_q(\!(X^{q^{-\infty}})\!)$ ; nous avons déjà dit qu'ils ont la même théorie de Galois que des extensions de Lie de nos corps locaux $p$-adiques. Nous établirons en un deuxième temps une équivalence de Fontaine imparfaite pour ces corps, d'abord pour des représentations de caractéristiques $p$ puis en dévissant grâce à \cite{formalisme_marquis} pour capturer toutes les représentations $p$-adiques de type fini.
	
	\vspace{0.75cm}
	\subsection{Construction du foncteur $\widetilde{\mathbb{D}}_{\Delta}$ modulo $p$}\label{section_wddelta}
	
	Le formalisme développé dans \cite{formalisme_marquis} souligne que les points délicats d'une équivalence de Fontaine sont de deux natures : la définition correcte de l'anneau de comparaison avec sa topologie et son action de monoïde d'un côté, l'obtention des isomorphismes de comparaison de l'autre. Nous commençons ainsi par définir anneaux adaptés à une équivalence de Fontaine pour des corps perfectoïdes de caractéristique $p$. En caractéristique $p$, les conditions topologiques pour appliquer le formalisme dans \cite{formalisme_marquis} sont aisées : pour les anneaux de la section \ref{section_intro_lt}, si l'action de $\cg_K$ sur le corps de normes imparfait est uniquement continue pour la topologie $X$-adique, celle de $\cg_E$ est continue pour la topologie discrète sur $E^{\sep}$. Malheureusement, nous voulons utiliser ici la théorie des perfectoïdes ce qui nous force à considérer tout de même des topologies plus malines sur nos anneaux.
	
	Fixons pour cette sous-section un corps perfectoïde $\widetilde{E}$ de caractéristique $p$ tel que $\widetilde{E}\cap \mathbb{F}_p^{\mathrm{sep}}$ est fini de cardinal $q$. Nous fixons une structure de $\mathbb{F}_q$-algèbre sur $\widetilde{E}$, une clôture séparable $\widetilde{E}^{\sep}$ et un plongement de $\overline{\mathbb{F}}_p$ dans ladite clôture séparable. Le groupe de Galois absolu $\cg_{\widetilde{E}}$ sera toujours muni de sa topologie profinie. Fixons également $\varpi$ une pseudo-uniformisante de $\widetilde{E}$, et notons $\widetilde{E}^+:= \widetilde{E}^{\circ}$ son anneau d'entiers. Pour toute extension finie $\widetilde{F}|\widetilde{E}$, l'élément $\varpi$ est encore une pseudo-uniformisante et la clôture algébrique de $\mathbb{F}_p$ dans $\widetilde{F}$ est encore finie. Nous notons $q^f$ son cardinal.
	
	L'adjectif \textit{multivariable} implique de se fixer un ensemble fini $\Delta$. Pour chaque $\alpha \in \Delta$, nous considérons un corps perfectoïde $\widetilde{E}_{\alpha}$ muni d'un isomorphisme avec $\widetilde{E}$ et d'une extension de cet isomorphisme à leurs clôtures séparables. Ainsi, nous fixons pour toute extension finie $\widetilde{F}|\widetilde{E}$ une extension isomorphe $\widetilde{E}_{\alpha}|\widetilde{E}_{\alpha}$ et une structure de $\mathbb{F}_{q^f}$-algèbre. De manière générale, pour chaque objet obtenu à partir de $\widetilde{E}$, nous notons avec un indice $\alpha$ l'objet obtenu pour $\widetilde{E}_{\alpha}$ à partir d'un objet choisi pour $\widetilde{E}$ et du choix d'isomorphismes précédents. Nous définissons $$\cg_{\widetilde{E},\Delta} := \prod_{\alpha \in \Delta} \cg_{\widetilde{E}_{\alpha}}.$$
	
	Nous voulons établir une équivalence pour des représentations $\mathbb{F}_r$-linéaires de $\cg_{\widetilde{E},\Delta}$ pour $\mathbb{F}_r\subseteq \mathbb{F}_q$.
	\begin{defi}
		Définissons le monoïde\footnote{Il est effectivement isomorphe à $\N^{\Delta}$ mais nous préférons nommer une base de manière suggestive.}  $\Phi_{\Delta,p}=\prod_{\alpha\in \Delta} \varphi_{\alpha,p}^{\N}$. Il sera toujours muni de la topologie discrète. Pour $b \geq 1$, nous notons $$\varphi_{\alpha,p^b}:=\varphi_{\alpha}^b \,\,\, \text{et} \,\,\, \varphi_{\Delta,p^b}:=( \varphi_{\alpha,p^b})_{\alpha \in \Delta}.$$
		
		Pour $p^a|p^b$, nous définissons $$\Phi_{\Delta,p^b,p^a}=\langle \varphi_{\Delta,p^a}, \, \varphi_{\alpha,p^b} \, |\, \alpha \in \Delta \rangle < \Phi_{\Delta,p}.$$ Pour $a=b$, nous simplifions cette notations en $\Phi_{\Delta,p^b}$.
		
	Nous pouvons plonger les monoïdes simplifiables $\Phi_{\Delta,p^b,p^a}$ dans leurs symétrisés $\Phi_{\Delta,p^b,p^a}^{\gp}$.
	\end{defi} 
	
	Commençons par définir une $\mathbb{F}_q$-algèbre analogue du $E$ de Fontaine dans le cas multivariable perfectoïde. Ici, nous faisons agir le groupe $\Phi_{\Delta,q,r}^{\gp}$ : des $q$-Frobenius sur chaque copie qui encodent les actions des différents $\cg_{\widetilde{E}_{\alpha}}$ et un Frobenius $r$-Frobenius global qui permet de redescendre à des $\mathbb{F}_r$ représentations. La construction entre dans le cadre de \cite[\S 4.1]{zabradi_kedlaya_carter}.
	
	\begin{defi}\label{def_constr_prod}
		Soit $\widetilde{E}^{\sep}|\widetilde{F}|\widetilde{E}$ une extension finie. Dans le produit tensoriel $$\bigotimes_{\alpha\in \Delta,\mathbb{F}_q} \widetilde{F}_{\alpha}^+,$$ appelons encore $\varpi_{\alpha}$ l'image de $\varpi_{\alpha} \in \widetilde{F}_{\alpha}^+$, notons $(\underline{\varpi})=(\varpi_{\alpha} \, |\, \alpha \in \Delta)$ et $\varpi_{\Delta}=\prod_{\alpha \in \Delta} \varpi_{\alpha}$. 
		
		Définissons $$\widetilde{F}_{\Delta,q}^+:=\left(\bigotimes_{\alpha\in \Delta,\mathbb{F}_q} \widetilde{F}_{\alpha}^+ \right)^{\wedge (\underline{\varpi})} \,\,\, \text{et} \,\,\, \widetilde{F}_{\Delta,q}:= \widetilde{F}_{\Delta,q}^+\left[\frac{1}{\varpi_{\Delta}}\right].$$
		
		Trois topologies seront utilisées pour ces deux anneaux\footnote{Les trois topologies sont utiles respectivement pour les considérer comme coefficients de catégories de $\varphi$-modules, comme espace adique dans \S \ref{section_equiv_car_p_parf} ou pour un raisonnement fin dans \S \ref{section_equiv_car_p_imparf}.}. Sur $\widetilde{F}_{\Delta,q}^+$ ces trois topologies sont la topologie discrète, la topologie $\varpi_{\Delta}$-adique et la topologie $(\underline{\varpi})$-adique. Sur $\widetilde{F}_{\Delta,q}$, ce sont la topologie discrète, la topologie d'anneau\footnote{Voir \cite[\S 6.3]{bourbaki_topo} pour les axiomes que doivent vérifier une telle base de voisinages.} ayant pour base de voisinages de $0$ la famille\footnote{Il n'est pas si clair que à ce stade que $\widetilde{F}_{\Delta,q}^+$ s'injecte dans $\widetilde{F}_{\Delta,q}$. L'abus de notations sera vite réparé (voir Proposition \ref{integrite_edplus}).} $\left(\varpi_{\Delta}^n\widetilde{F}_{\Delta,q}^+\right)_{n\geq 0}$ que nous appelons \textit{topologie adique} et la topologie colimite des topologies $(\underline{\varpi})$-adiques via l'écriture $$\widetilde{F}_{\Delta,q}=\colim \limits_{n\geq 0} \frac{1}{\varpi_{\Delta}^n} \widetilde{F}_{\Delta,q}^+$$ que nous appelons \textit{topologie colimite}.
		
		Le produit tensoriel des $\widetilde{F}_{\alpha}^+$ est muni d'une structure de $\Phi_{\Delta,q,r}^{\gp}$-anneau topologique pour la topologie \linebreak$(\underline{\varpi})$-adique, l'élément $\varphi_{\alpha,q}$ agissant par le $q$-Frobenius sur $\widetilde{F}_{\alpha}^+$ et l'identité sur les $\widetilde{F}_{\beta}^+$ et l'élément $\varphi_{\Delta,r}$ agissant par le $r$-Frobenius. En complétant, on obtient une structure de $\Phi_{\Delta,q,r}^{\gp}$-anneau topologique sur $\widetilde{F}_{\Delta,q}^+$ muni de la topologie $(\underline{\varpi})$-adique. Cette action du monoïde est également une structure de $\Phi_{\Delta,q,r}^{\gp}$-anneau topologique sur $\widetilde{F}_{\Delta,q}^+$ muni de la topologie $\varpi_{\Delta}$-adique ou de topologie discrète. Après localisation, l'action du monoïde fournit donc une structure $\Phi_{\Delta,q,r}^{\gp}$-anneau topologique sur $\widetilde{F}_{\Delta,q}$ pour chacune des trois topologies ci-dessus\footnote{C'est plus subtile pour la topologie colimite car les Frobenius ne stabilisent par les termes de la colimite.}.
	\end{defi}
	
	\vspace{0.25 cm}
	Soient $\widetilde{E}^{\sep}|\widetilde{F}|\widetilde{E}$ une extension finie et $q^f$ le cardinal de la clôture séparable de $\mathbb{F}_q$ dans $\widetilde{F}$. Nous utilisons aussi les $\Phi_{\Delta,q^f,r}^{\gp}$-anneaux topologiques $\widetilde{F}_{\Delta}^+$ et $\widetilde{F}_{\Delta}$ obtenus comme à la définition précédente en considérant $\widetilde{F}$ comme extension finie de lui-même, i.e. en faisant les produits tensoriels sur $\mathbb{F}_{q^f}$.	
	
	Puisque nos produits tensoriels ne sont pas sur $\mathbb{F}_p$, nous ne pouvons décomposer le Frobenius absolu et obtenir un $\Phi_{\Delta,p}^{\gp}$-anneau contrairement aux anneaux dans \cite{zabradi_kedlaya_carter}. En revanche, nous prouvons que l'anneau $\widetilde{E}_{\Delta}$ est intègre.
	
	\begin{prop}\label{integrite_edplus}
 		Pour une extension finie $\widetilde{E}^{\sep}|\widetilde{F}|\widetilde{E}$, l'anneau $\widetilde{F}_{\Delta,q}^+$ est parfait, réduit et sans $\widetilde{E}_{\Delta}^+$-torsion. En particulier l'anneau $\widetilde{E}_{\Delta}^+$ est intègre.
 		
 		L'application $\widetilde{F}_{\Delta,q}^+ \rightarrow \widetilde{F}_{\Delta,q}$ est injective.  L'anneau $\widetilde{F}_{\Delta,q}$ est parfait, réduit et sans $\widetilde{E}_{\Delta}$-torsion. En particulier, l'anneau $\widetilde{E}_{\Delta}$ est intègre.
	\end{prop}
	\begin{proof}
		Reléguée en annexe au Corollaire \ref{sanstors_perf}.
	\end{proof}
	
	\begin{rem}\label{choix_topo_completion}
	Le balancier entre les idéaux $(\underline{\varpi})$ et $(\varpi_{\Delta})$ est une subtilité importante de ces anneaux. Puisque nous commençons par prendre la complétion $(\underline{\varpi})$-adique, il serait naturel d'essayer d'en conserver une trace topologique. Malheureusement, la famille d'idéaux de $\widetilde{E}_{\Delta}^+$ donnée par $(\underline{\varpi})^k$ ne définit par sur $\widetilde{E}_{\Delta}$ une structure d'anneau topologique\footnote{Par exemple, chaque idéal  $\sfrac{(\underline{\varpi})^k}{\varpi_{\Delta}}$ contient un élément $\sfrac{\varpi_{\alpha}^k}{\varpi_{\Delta}}$ ce qui démontre que le produit n'est pas continu au point $(\varpi_{\Delta}^{-1},0)$}. Pour en garder une trace, il faut considérer la topologie colimite, dont nous nous servirons d'ailleurs dans \S \ref{section_equiv_car_p_imparf}. Toutefois, nous voudrons utiliser nos anneaux dans contextes pour lesquels cette topologie n'est pas adéquate : dans \S \ref{section_equiv_car_p_parf}, nous voulons considérer $(\widetilde{E}_{\Delta},\widetilde{E}_{\Delta}^ +)$ comme une paire de Huber. Pour la topologie colimite, une base de voisinage de zéro s'écrit $$\left\{\bigcup_{n\geq 0} \frac{(\underline{\varpi})^{m_n}}{\varpi_{\Delta}^n} \, \bigg|\, (m_n) \in \N^{\N} \,\, \text{telle que}\,\, m_{n+1} \geq m_n + |\Delta|\right\}$$ dont aucun n'est contenu dans $\widetilde{E}_{\Delta}^+$. Nous n'aurions pas une paire de Huber. Le ticket d'entrée dans la théorie des espaces adiques est précisement d'utiliser la topologie adique.
	
	Dans cette sous-section cependant, nous n'aurons besoin que de la topologie discrète puisque les actions de Galois considérées pour notre équivalence de Fontaine sont toutes à stabilisateurs ouverts. En réalité, pour l'équivalence modulo $p$, nous pourrions appliquer les méthodes dans \cite{formalisme_marquis} avec n'importe laquelle des trois topologies ; les conditions de continuité sont vides dans chacun des cas pour les $\Phi_{\Delta,q,r}^{\gp}$-modules finis projectifs sur $\widetilde{E}_{\Delta}$. Cependant, le dévissage de \S \ref{section_equiv_car_zero_imparf} fonctionne bien mieux avec la topologie discrète.
	\end{rem}
	
\begin{rem}
	Nous spécifierons notre étude à des corps perfectoïdes plus agréables dans \S \ref{section_equiv_car_p_imparf} et nous décrirons complètement l'anneau $\widetilde{E}_{\Delta}$ dans ce cas. Les lecteurs et lectrices souhaitant avoir une prise plus concrète sur ces anneaux pourront lire dès à présent la Remarque \ref{ecriture_explicite_anneaux}.
\end{rem}
	
	\vspace{0.25 cm}

	Nous définissons à présent l'analogue de $E^{\sep}$ dans le cadre multivariable perfectoïde .

	\begin{defi}
		Soit $\widetilde{E}^{\sep}|\widetilde{F}|\widetilde{E}$ une extension finie galoisienne. L'action de $\cg_{\widetilde{E}}$ sur $\widetilde{F}^+$ est $\mathbb{F}_q$-linéaire, continue pour les topologies discrète et $\varpi$-adique, commutant au Frobenius. L'action de $\cg_{\widetilde{E},\Delta}$ sur le produit tensoriel des $\widetilde{F}^+_{\alpha}$ facteur par facteur se complète $(\underline{\varpi})$-adiquement en une action sur $\widetilde{F}^+_{\Delta,q}$ continue pour la topologie discrète, commutant à l'action de $\Phi_{\Delta,q,r}^{\gp}$.
		
		Nous obtenons donc une structure de $\left(\Phi_{\Delta,q,r}^{\gp} \times \cg_{\widetilde{E},\Delta}\right)$-anneau topologique sur $\widetilde{F}_{\Delta,q}^+$. En localisant, nous obtenons une structure de $\left(\Phi_{\Delta,q,r}^{\gp} \times \cg_{\widetilde{E},\Delta}\right)$-anneau topologique sur $\widetilde{F}_{\Delta,q}$.
	\end{defi}
	
		\begin{lemma}\label{description_construction_produit}
		Soit $\mathcal{G}\mathrm{al}_{\widetilde{E}}$ la catégorie des sous-extensions finies galoisiennes de $\widetilde{E}$ dans $\widetilde{E}^{\sep}$ avec les inclusions pour morphismes. La construction $$\widetilde{F}\mapsto \widetilde{F}_{\Delta,q}$$ où l'on met la topologie discrète est canoniquement un foncteur de $\mathcal{G}\mathrm{al}_{\widetilde{E}}$ vers les $(\Phi_{\Delta,q,r}^{\gp}\times \cg_{\widetilde{E},\Delta})$-anneaux topologiques. Tous les morphismes impliqués sont injectifs.
		
		La construction $$\widetilde{F}\mapsto \widetilde{F}_{\Delta}$$ est canoniquement un foncteur depuis la catégorie de $\mathcal{G}\mathrm{al}_{\widetilde{E}}$ vers la catégorie des anneaux topologiques. Pour toute tour d'extensions finies galoisiennes $\widetilde{F}'| \widetilde{F}| \widetilde{E}$, le morphisme associé $$\widetilde{F}_{\Delta} \rightarrow \widetilde{F}'_{\Delta}$$ est une injection $\Phi_{\Delta,q^{f'},r}^{\gp}$-équivariante.
	\end{lemma}
	\begin{proof}
		Pour la première construction, on prend le produit tensoriel, complète puis localise les injections $\varphi_{\alpha,p}$-équivariantes $\widetilde{F}_{\alpha}^+ \hookrightarrow \widetilde{F}'^+_{\alpha}$.

		Pour la deuxième construction et $\widetilde{E}\subset \widetilde{F}$, nous écrivons $\widetilde{E}'$ le corps de décomposition de $X^{q^f}-X$. Le morphisme à déconstruire se décompose via $\widetilde{E}'_{\Delta}$. Le morphisme $\widetilde{E}'_{\Delta} \rightarrow \widetilde{F}_{\Delta}$ est simplement donné par la première construction. Pour $\widetilde{E}'|\widetilde{E}$, on sait grâce à \cite[V \S 5, Proposition 9]{bourbaki_alg} que $\widetilde{E}|\mathbb{F}_q$ est régulière : l'anneau $\widetilde{E}\otimes_{\mathbb{F}_q}\mathbb{F}_{q^f}$ est donc un corps et une analyse des dimension montre qu'il est canoniquement isomorphe à $\widetilde{E}'$. On peut donc compléter et localiser l'injection $$\bigotimes_{\alpha \in \Delta,\,\, \mathbb{F}_q} \widetilde{E}^+_{\alpha} \hookrightarrow \left(\bigotimes_{\alpha \in \Delta,\,\, \mathbb{F}_q} \widetilde{E}^+_{\alpha}\right) \otimes_{\mathbb{F}_q} \mathbb{F}_{q^f} \cong \bigotimes_{\alpha \in \Delta,\,\, \mathbb{F}_{q^f}} \widetilde{E}'^+_{\alpha}.$$
	\end{proof}
	
	\begin{defi}
		Définissons le $(\Phi_{\Delta,q,r}^{\gp}\times \cg_{\widetilde{E},\Delta})$-anneau topologique discret $$\widetilde{E}_{\Delta}^{\sep}=\colim \limits_{\widetilde{F}\in \mathcal{G}\mathrm{al}_{\widetilde{E}}} \widetilde{F}_{\Delta,q}.$$
	\end{defi}
	
	\begin{prop}\label{edseptorsion}
		L'anneau $\widetilde{E}_{\Delta}^{\sep}$ est parfait, réduit et sans $\widetilde{E}_{\Delta}$-torsion.
	\end{prop}
	\begin{proof}
		Rassembler les résultats de la Proposition \ref{integrite_edplus} et du Lemme \ref{description_construction_produit}.
	\end{proof}
	
	\begin{rem}
		La topologie est toujours discrète ici. Nous ne complétons pas $\varpi_{\Delta}$-adiquement : d'un côté, l'action de $\cg_{\widetilde{E},\Delta}$ n'est pas continue pour la topologie discrète sur le complété, de l'autre, certains groupes de cohomologie continue ne s'annulent pas pour la topologie $\varpi_{\Delta}$-adique.
	\end{rem}
		
	\begin{rem} Nos $\varphi$-modules multivariables vivront sur $\widetilde{E}_{\Delta}$. Si cet anneau de coefficients est intègre, ce n'est pas le cas de l'anneau de comparaison $\widetilde{E}_{\Delta}^{\sep}$. En effet, il existe une injection $$\bigotimes_{\alpha \in \Delta,\,\mathbb{F}_q} \overline{\mathbb{F}_q} \hookrightarrow \widetilde{E}_{\Delta}^{\sep}.$$
	\end{rem}
	
	\vspace{0.25cm}
	
	Nous imitons la construction du foncteur de Fontaine et obtenons un foncteur $$\widetilde{\mathbb{D}}_{\Delta} \, : \, \mathrm{Rep}_{\mathbb{F}_r} \cg_{\widetilde{E},\Delta} \rightarrow \cdetaleproj{\Phi_{\Delta,q,r}^{\gp}}{\widetilde{E}_{\Delta}}.$$ Pour cela, il reste essentiellement à démontrer que la descente galoisienne fonctionne pour des anneaux multivariables.

	\begin{lemma}\label{module_induit_multi}
		Le morphisme canonique
		
		$$\widetilde{E}_{\alpha}^{\sep}\otimes_{\widetilde{F}_{\alpha}}\left(\widetilde{E}_{\beta}^{\sep} \otimes_{\widetilde{F}_{\beta}} \cdots \left(\widetilde{E}_{\delta}^{\sep}\otimes_{\widetilde{F}_{\delta}} \widetilde{F}_{\Delta,q}\right)\right)\rightarrow \widetilde{E}^{\sep}_{\Delta}$$ est un isomorphisme de $\cg_{\widetilde{E},\Delta}$-équivariant.
	\end{lemma}
	\begin{proof}
		Puisque les injections $\widetilde{F}_{\alpha}^{\sep} \hookrightarrow \widetilde{E}_{\Delta}^{\sep}$ sont $\cg_{\widetilde{E}_{\alpha}}$-équivariantes, l'équivariance est automatique. De plus, quitte à passer à la colimite, on se restreint à prouver une identité similiaire entre $\widetilde{F}'_{\Delta,q}$ et $\widetilde{F}_{\Delta,q}$ pour une extension finie $\widetilde{F}'|\widetilde{F}$.
		
		Nous utilisons, pour toute paire $(A,\varpi)$ formée d'un anneau et d'un élément admettant des racines \linebreak$p^n$-ièmes pour tout $n$, la catégorie $(A,\varpi)\text{-}\mathrm{Mod}$ des presque-$A$-modules par rapport à l'idéal $\mathrm{Rad}(\varpi)$. De la même manière, nous appelons presque-$(A,\varpi)$-isomorphisme un morphisme de $A$-modules qui devient un isomorphisme dans $(A,\varpi)\text{-}\mathrm{Mod}$.
		
		Grâce à \cite[Prop. 5.23]{scholze_perf}, nous savons que $\widetilde{F}'^+$ est un $(\widetilde{F}^+,\varpi)$-module uniformément libre de type fini. Choisissons une famille $(x_k)_{1\leq  k \leq d}$ dans $\widetilde{F}'^+$ telle que $$ \bigoplus_{k=1}^d \widetilde{F}^+ x_k \rightarrow \widetilde{F}'^+$$ est un presque-$(\widetilde{F}^+,\varpi)$-isomorphisme. Le morphisme $$\bigoplus_{(i_{\alpha}) \in \llbracket 1, d\rrbracket^{\Delta}} \left(\bigotimes_{\alpha\in \Delta, \, \mathbb{F}_q} \widetilde{F}_{\alpha}^+\right) \prod_{\alpha\in \Delta} x_{i_\alpha,\alpha} \rightarrow \left(\bigotimes_{\alpha\in \Delta, \, \mathbb{F}_q} \widetilde{F}_{\alpha}'^+\right)$$ est un presque-$\left(\left(\otimes_{\mathbb{F}_q} \widetilde{F}_{\alpha}^+\right),\varpi_{\Delta}\right)$-isomorphisme en le décomposant comme suite de changement de base du presque-isomorphisme précédent. Puisque la structure monoïdale sur $(A,\varpi)\text{-}\mathrm{Mod}$ vient du produit tensoriel sur les $A$-modules (voir \cite[\S 2.2.5]{gabber2002ring}), le morphisme ci-dessus est encore un presque-isomorphisme après quotient par $(\underline{\varpi})^n$. En passant à la limite\footnote{La catégorie $(A,\varpi)\text{-}\mathrm{Mod}$ admet toutes les limite puisque la localisation $A\text{-}\mathrm{Mod}\rightarrow (A,\varpi)\text{-}\mathrm{Mod}$ admet un adjoint à gauche (voir \cite[Coro. 2.2.15]{gabber2002ring}).}, nous en déduisons que
		
		$$ \bigoplus_{(i_{\alpha}) \in \llbracket 1, d\rrbracket^{\Delta}} \widetilde{F}_{\Delta,q}^+ \prod_{\alpha \in \Delta} x_{i_{\alpha,\alpha}} \rightarrow \widetilde{F}'^+_{\Delta,q}$$ est un presque-$(\widetilde{F}^+_{\Delta,q},\varpi_{\Delta})$-isomorphisme. En décomposant à nouveau, le terme de gauche est presque-$(\widetilde{F}^+_{\Delta,q},\varpi_{\Delta})$-isomorphe à $$\widetilde{F}'^+_{\alpha} \otimes_{\widetilde{F}^+_{\alpha}} \left( \widetilde{F}'^+_{\beta} \otimes_{\widetilde{F}^+_{\beta}} \cdots \left(\widetilde{F}'^+_{\delta} \otimes_{\widetilde{F}^+_{\delta}} \widetilde{F}^+_{\Delta,q}\right)\right).$$ En inversant $\varpi_{\Delta}$, nous obtenons l'isomorphisme escompté.
	\end{proof}

	\begin{coro}\label{coro_descente_1}		
		\begin{enumerate}[itemsep=0mm]
			\item Pour toute extension finie $\widetilde{E}^{\sep}|\widetilde{F}|\widetilde{E}$, l'inclusion $\widetilde{F}_{\Delta,q} \subseteq \left(\widetilde{E}_{\Delta}^{\sep}\right)^{\cg_{\widetilde{F},\Delta}}$ est une égalité.
			
			\item 	Pour tout objet $D$ de $\cdetaleproj{\Phi_{\Delta,q,r}^{\gp}\times\cg_{\widetilde{E},\Delta}}{\widetilde{E}_{\Delta}^{\sep}}$, le morphisme de comparaison
			
			$$\widetilde{E}_{\Delta}^{\sep}\otimes_{\widetilde{E}_{\Delta}} \mathrm{Inv}(D) \rightarrow D$$ est un isomorphisme.
		\end{enumerate}
	\end{coro}
	\begin{proof}
		Grâce à la description ci-dessus, nous pouvons appliquer \cite[Th. 3.6]{note_descente}. La descente fidèlement plate des modules projectifs pour $\widetilde{F}_{\Delta,q} \rightarrow \widetilde{E}_{\Delta}^{\sep}$ (resp. pour $\widetilde{E}_{\Delta} \rightarrow \widetilde{E}_{\Delta}^{\sep}$) établit une équivalence de catégorie entre les $\widetilde{F}_{\Delta,q}$-modules finis projectifs et les $\widetilde{E}_{\Delta}^{\sep}$-modules finis projectifs munis d'une action lisse de $\cg_{\widetilde{F},\Delta}$. Cela implique les deux énoncés.
	\end{proof}
	
	\begin{defiprop}\label{defipropwDdelta}
		Le foncteur $$\widetilde{\mathbb{D}}_{\Delta} \, : \, \mathrm{Rep}_{\mathbb{F}_r} \cg_{\widetilde{E},\Delta} \rightarrow \dmod{\Phi_{\Delta,q,r}^{\gp}}{\widetilde{E}_{\Delta}}, \,\,\, V \mapsto \left(\widetilde{E}^{\sep}_{\Delta} \otimes_{\mathbb{F}_r} V \right)^{\cg_{\widetilde{E},\Delta}}$$ est correctement défini et son image essentielle est incluse dans $\detaleproj{\Phi_{\Delta,q,r}^{\gp}}{\widetilde{E}_{\Delta}}$. Cette dernière catégorie est une sous-catégorie pleine monoïdale fermée et $\widetilde{\mathbb{D}}_{\Delta}$ commute naturellement au produit tensoriel et au $\mathrm{Hom}$ interne.
	\end{defiprop}
	\begin{proof}
		Nous identifions la catégorie de représentations à $\cdetaleproj{\cg_{\widetilde{E},\Delta}}{\mathbb{F}_r}$ et nous décomposons le foncteur $\widetilde{\mathbb{D}}_{\Delta}$ comme suit.
		
		\begin{center}
			\begin{tikzcd}
				\widetilde{\mathbb{D}}_{\Delta} \,: \, &\cdetaleproj{\cg_{\widetilde{E}_{\Delta}}}{\mathbb{F}_r} \ar[r,"\mathrm{triv}"] & \cdetaleproj{\Phi_{\Delta,q,r}^{\gp}\times \cg_{\widetilde{E},\Delta}}{\mathbb{F}_r} \ar[d,"\mathrm{Ex}"] \\
				& \dmod{\Phi_{\Delta,q,r}^{\gp}}{\widetilde{E}_{\Delta}} & \dmod{\Phi_{\Delta,q,r}^{\gp}\times \cg_{\widetilde{E},\Delta}}{\widetilde{E}_{\Delta}^{\sep}} \ar[l,"\mathrm{Inv}"]
			\end{tikzcd}
		\end{center} Pour montrer que $\mathrm{Ex}$ et $\mathrm{Inv}$ préservent les sous-catégories, nous utilisons \cite[Prop. 5.17 et 5.18]{formalisme_marquis} respectivement pour le morphisme de $(\Phi_{\Delta,q,r}^{\gp}\times \cg_{\widetilde{E},\Delta})$-anneaux topologiques discrets $\mathbb{F}_r \rightarrow \widetilde{E}_{\Delta}^{\sep}$ et pour $\widetilde{E}_{\Delta}^{\sep}$ avec comme sous-monoïde $\cg_{\widetilde{E},\Delta}<(\Phi_{\Delta,q,r}^{\gp}\times \cg_{\widetilde{E},\Delta})$. Les conditions de \cite[Prop. 5.17]{formalisme_marquis} sont déjà démontrées lors la construction des anneaux. Nous listons et démontrons les conditions de \cite[Prop. 5.18]{formalisme_marquis}.
		
		\underline{Condition 1 :} le sous-monoïde $\cg_{\widetilde{E},\Delta}$ est distingué et le groupe topologique quotient s'identifie à $\Phi_{\Delta,q,r}^{\gp}$. Tout découle de ce que $(\Phi_{\Delta,q,r}^{\gp}\times \cg_{\widetilde{E},\Delta})$ est un produit direct.
		
		\underline{Condition 2 :} l'anneau topologique $\widetilde{E}_{\Delta}$ s'identifie à $\left(\widetilde{E}_{\Delta}^{\sep}\right)^{\cg_{\widetilde{E},\Delta}}$. L'identification ensembliste est l'objet du premier point du Corollaire \ref{coro_descente_1}. Puisque les anneaux sont discrets, la condition topologique est vide.
		
		\underline{Condition 3 :} l'inclusion $\widetilde{E}_{\Delta} \subset \widetilde{E}_{\Delta}^{\sep}$ est fidèlement plate. Le Lemme \ref{module_induit_multi} démontre que c'est une suite de changement de bases le long de morphismes de corps, a fortiori un morphisme fidèlement plat.
		
		\underline{Condition 4 :} les morphismes de comparaison sont des isomorphismes. C'est l'objet du deuxième point du Corollaire \ref{coro_descente_1}.
		
			Puisque $\Phi_{\Delta,q,r}^{\gp}$ est discret les catégories $\cdetaleproj{\Phi_{\Delta,q,r}^{\gp}}{\widetilde{E}_{\Delta}}$ et $\detaleproj{\Phi_{\Delta,q,r}^{\gp}}{\widetilde{E}_{\Delta}}$ coïncident (voir \cite[Ex. 5.6]{formalisme_marquis}) pour n'importe quelle topologie. 
		%Toutefois, garder plutôt la topologie discrète en tête sera utile pour l'étape de dévissage de la section \ref{section_equiv_car_zero_imparf}.
	\end{proof}

	\vspace{0.75cm}
	
	\subsection{\'Equivalence de Fontaine multivariable perfectoïde modulo $p$ pour certains corps perfectoïdes de caractéristique $p$}\label{section_equiv_car_p_parf}
	
	Pour démontrer une équivalence de Fontaine perfectoïde modulo $p$, il nous reste à montrer que le candidat naturel à être quasi-inverse de $\widetilde{\mathbb{D}}_{\Delta}$ est correctement défini. Pour ce faire, nous aurons besoin de démontrer un autre isomorphisme de comparaison, beaucoup plus délicat. Nous suivons la stratégie dans \cite{zabradi_kedlaya_carter} qui consiste à utiliser le lemme de Drinfeld pour les diamants afin d'obtenir d'une autre manière une $\mathbb{F}_r$-représentation de $\cg_{\widetilde{E},\Delta}$ à partir d'un $\Phi_{\Delta,q,r}^{\gp}$-module sur $\widetilde{E}_{\Delta}$. Dans la preuve de \cite[Prop. 4.20]{zabradi_kedlaya_carter}, les trois auteurs utilisent la commutation de leur foncteur au dual, peu évidente à ce stade de démonstration. Nous préférons utiliser activement l'intégrité des anneaux dans notre version de l'équivalence pour contourner une partie des arguments.
	
	\begin{lemma}\label{inv_frob_multi_perf}
		L'inclusion $\mathbb{F}_r \subset \left(\widetilde{E}_{\Delta}^{\sep}\right)^{\Phi_{\Delta,q,r}^{\gp}}$ est une égalité d'anneaux topologiques.
	\end{lemma}
	\begin{proof}
		Puisque les topologies sont discrètes, c'est un énoncé algébrique relégué au Corollaire \ref{inv_phi_carp_perf}.
	\end{proof}
	
	En admettant le Théorème \ref{premier_écrit_theo_coeur} qui occupera la majorité de cette sous-section, nous pouvons démontrer que $\widetilde{\mathbb{D}}_{\Delta}$ est une équivalence de catégories.
	
\begin{defiprop}\label{defipropwVDelta}
	Le foncteur $$\widetilde{\mathbb{V}}_{\Delta} \, : \,   \detaleproj{\Phi_{\Delta,q,r}^{\gp}}{\widetilde{E}_{\Delta}} \rightarrow \dmod{\cg_{\widetilde{E},\Delta}}{\mathbb{F}_r}, \,\,\, D \mapsto \left(\widetilde{E}^{\sep}_{\Delta} \otimes_{\widetilde{E}_{\Delta}} D \right)^{\Phi_{\Delta,q,r}^{\gp}}$$ est correctement défini et son image essentielle est incluse dans $\mathrm{Rep}_{\mathbb{F}_r} \cg_{\widetilde{E},\Delta}$. De plus, le foncteur $\widetilde{\mathbb{V}}_{\Delta}$ commute naturellement au produit tensoriel et au $\mathrm{Hom}$ interne.
\end{defiprop}
\begin{proof}
	Nous décomposons le foncteur $\widetilde{\mathbb{V}}_{\Delta}$ comme suit en se souvenant que la condition de continuité à la source est automatique :
	
	\begin{center}
		\begin{tikzcd}
			\widetilde{\mathbb{V}}_{\Delta} \,: \, &\cdetaleproj{\Phi_{\Delta,q,r}^{\gp}}{\widetilde{E}_{\Delta}} \ar[r,"\mathrm{triv}"] & \cdetaleproj{\Phi_{\Delta,q,r}^{\gp}\times \cg_{\widetilde{E},\Delta}}{\widetilde{E}_{\Delta}} \ar[d,"\mathrm{Ex}"] \\
			& \dmod{\Phi_{\Delta,q,r}^{\gp}}{\mathbb{F}_r} & \dmod{\Phi_{\Delta,q,r}^{\gp}\times \cg_{\widetilde{E},\Delta}}{\widetilde{E}_{\Delta}^{\sep}} \ar[l,"\mathrm{Inv}"]
		\end{tikzcd}
	\end{center} Puisque $\mathrm{Rep}_{\mathbb{F}_r} \cg_{\widetilde{E},\Delta}$, s'identifie à $\cdetaleproj{\cg_{\widetilde{E},\Delta}}{\mathbb{F}_r}$, il faut prouver que $\mathrm{Ex}$ et $\mathrm{Inv}$ préservent la sous-catégorie considérée. Nous utilisons \cite[Prop. 5.17 et 5.18]{formalisme_marquis} respectivement pour le morphisme de  \linebreak$(\Phi_{\Delta,q,r}^{\gp}\times \cg_{\widetilde{E},\Delta})$-anneaux discrets $\widetilde{E}_{\Delta} \rightarrow \widetilde{E}_{\Delta}^{\sep}$ et pour $\widetilde{E}_{\Delta}^{\sep}$ avec $\Phi_{\Delta,q,r}^{\gp}<(\Phi_{\Delta,q,r}^{\gp}\times \cg_{\widetilde{E},\Delta})$. Les conditions pour $\mathrm{Ex}$ sont encore une fois démontrées au fil de la construction des anneaux. La liste des conditions à démontrer pour $\mathrm{Inv}$ est similaire à la Proposition \ref{defipropwDdelta}. La condition 1 est automatique ; la condition 2 est contenue dans le Lemme \ref{inv_frob_multi_perf} ; la condition 3 découle de ce que $\mathbb{F}_r$ est un corps ; la condition 4 est l'objet du Théorème \ref{premier_écrit_theo_coeur}. 
\end{proof}	

\begin{theo}\label{equiv_perf_mod_p}
	Les foncteurs $\widetilde{\mathbb{D}}_{\Delta}$ et $\widetilde{\mathbb{V}}_{\Delta}$ forment une paire de foncteurs quasi-inverses et établissent une équivalence de catégories symétriques monoidales fermées $$\widetilde{\mathbb{D}}_{\Delta} \, : \, \mathrm{Rep}_{\mathbb{F}_r} \cg_{\widetilde{E},\Delta} \rightleftarrows \detaleproj{\Phi_{\Delta,q,r}^{\gp}}{\widetilde{E}_{\Delta}} \, : \, \widetilde{\mathbb{V}}_{\Delta}.$$
\end{theo}
\begin{proof}
	Reste à prouver que les foncteurs sont quasi-inverses l'un de l'autre. Grâce aux isomorphismes de comparaison naturels que nous avons obtenus en étudiant les foncteurs, prouver que $\widetilde{\mathbb{D}}_{\Delta}\circ \widetilde{\mathbb{V}}_{\Delta}$ est isomorphe à l'identité revient à passer aux $\cg_{\widetilde{E},\Delta}$-invariants l'isomorphisme de comparaison naturel $$\widetilde{E}_{\Delta}^{\sep}\otimes_{\mathbb{F}_r} \mathbb{V}_{\Delta}(D) = \widetilde{E}_{\Delta}^{\sep}\otimes_{\mathbb{F}_r} \mathrm{Inv}\left( \widetilde{E}_{\Delta}^{\sep}\otimes_{\widetilde{E}_{\Delta}} D\right) \xrightarrow{\sim} \widetilde{E}_{\Delta}^{\sep}\otimes_{\widetilde{E}_{\Delta}} D$$ et à prouver que $$\forall D\in \detaleproj{\Phi_{\Delta,q,r}^{\gp}}{\widetilde{E}_{\Delta}}, \,\,\, \text{l'application} \,\, D \rightarrow \left(\widetilde{E}_{\Delta}^{\sep} \otimes_{\widetilde{E}_{\Delta}} D\right)^{\cg_{\widetilde{E},\Delta}} \,\, \text{est un isomorphisme}.$$ La Proposition \ref{edseptorsion} permet d'utiliser \cite[Prop. 3.10]{formalisme_marquis} pour conclure. L'autre composition se traite de même.
\end{proof}

\vspace{0.5 cm}

Il reste à démontrer que le morphisme de comparaison est un isomorphisme pour tous les objets de \linebreak$\cdetaleproj{\Phi_{\Delta,q,r}^{\gp}\times \cg_{\widetilde{E},\Delta}}{\widetilde{E}^{\sep}_{\Delta}}$. Grâce au Corollaire \ref{coro_descente_1}, la descente galoisienne établit une équivalence de catégories $$\mathrm{Inv} \, : \, \cdetaleproj{\Phi_{\Delta,q,r}^{\gp}\times \cg_{\widetilde{E},\Delta}}{\widetilde{E}^{\sep}_{\Delta}} \rightarrow \cdetaleproj{\Phi_{\Delta,q,r}^{\gp}}{\widetilde{E}_{\Delta}} \, : \, \mathrm{Ex}.$$ On se ramène donc à prouver le théorème suivant.

\begin{theo}\label{premier_écrit_theo_coeur}[voir Théorème \ref{morph_comparaison_coeur}]
	Pour tout objet $D$ de $\detaleproj{\Phi_{\Delta,q,r}^{\gp}}{\widetilde{E}_{\Delta}}$, le morphisme de comparaison $$\widetilde{E}_{\Delta}^{\mathrm{sep}}\otimes_{\mathbb{F}_r}\widetilde{\mathbb{V}}_{\Delta}(D) \rightarrow \widetilde{E}_{\Delta}^{\mathrm{sep}} \otimes_{\widetilde{E}_{\Delta}} D$$ est un isomorphisme.
\end{theo}

		Suivant l'idée de \cite{zabradi_kedlaya_carter} et adaptant les preuves à notre contexte, nous construisons par une méthode géométrique un foncteur $\mathrm{V}_{\widetilde{E}}$ puis nous l'analysons pour obtenir l'isomorphisme de comparaison. Nous aurons besoin d'utiliser tous les foncteurs $\mathrm{V}_{\widetilde{F}}$ avec $\widetilde{F}|\widetilde{E}$ finie pour analyser $\widetilde{\mathbb{V}}_{\Delta}$. Bien que ce qui suit est rédigé avec $\widetilde{E}$ et $q$, gardons en tête que nous démontrons les résulats également pour chaque $\widetilde{F}$ et $q'$.

	\vspace{0.25cm}
		
		Nous rappelons la notion de catégorie galoisienne que nous pouvons comprendre comme une liste de conditions pour être équivalente à une catégorie de représentations d'un groupe profini. Les lecteurs et lectrices souhaitant un exposé détaillé de la théorie pourront se réferer à \cite[V, \S 4]{SGA1}.
		
		\begin{defi}
			\begin{enumerate}[itemsep=0mm]
				\item Soit $\mathcal{C}$ une catégorie ayant toutes les limites et colimites. On dit qu'un objet $X$ est \linebreak\textit{$\mathcal{C}$-connexe} si pour tout monomorphisme $f\, : \,Y\rightarrow X$, l'objet $Y$ est initial ou le morphisme $f$ est un isomorphisme.
				
				\item Soit $F\ : \, \mathcal{C}\rightarrow \mathcal{D}$ un foncteur entre deux catégories ayant toutes les limites et colimites finies. On dit qu'il est exact s'il envoie les objets initiaux (et finaux) sur des objets initiaux (et finaux), et que pour tout diagrammes $(X\rightarrow Y \leftarrow Z)$ et $(X'\leftarrow Y'\rightarrow Z')$, les morphismes naturels $$F\left( X\times_Y Z\right) \rightarrow F(X)\times_{F(Y)} F(Z) \,\,\, \text{et}\,\,\, F(X')\sqcup_{F(Y')} F(Z') \rightarrow F\left(X'\sqcup_{Y'} Z'\right)$$ sont des isomorphismes.
				
				\item Toujours pour un tel foncteur $F$, on dit que $F$ reflète les isomorphismes si pour tout morphisme $f$ de $\mathcal{C}$, si $F(f)$ est un isomorphisme alors $f$ aussi.
			\end{enumerate}
		\end{defi}
		
		\begin{defi}
			Une \textit{catégorie galoisienne} est un couple $(\mathcal{C},F)$, où $\mathcal{C}$ est une catégorie essentiellement petite et $F \, :\, \mathcal{C}\rightarrow \mathrm{Ens}$ est un foncteur, tel que :
			
			\begin{itemize}[itemsep=0mm]
				\item La catégorie $\mathcal{C}$ a toutes les limites et colimites finies.
				
				\item Tout objet de $\mathcal{C}$ est une union disjointe finie d'objets $\mathcal{C}$-connexes.
				
				\item Le foncteur $F$ est à valeurs dans les ensembles finis.
				
				\item Le foncteur $F$ est exact.
				
				\item Le foncteur $F$ reflète les isomorphismes.
			\end{itemize}
		\end{defi}
		
		\begin{theo}
			Soit $(\mathcal{C},F)$ une catégorie galoisienne. Définissons le groupe fondamental $\pi_1(\mathcal{C},F)$ comme le groupe d'automorphismes du foncteur $F$. Il s'identifie à un sous-groupe du groupe topologique profini $$\prod_{X \in \mathcal{D}} \mathfrak{S}_{F(X)}$$ où $\mathcal{D}$ est une petite sous-catégorie de $\mathcal{C}$ qui lui est équivalente. Pour tout objet $X$, l'action de $\pi_1(\mathcal{C},F)$ sur $F(X)$ est donc à stabilisateurs ouverts et $F$ se promeut en une équivalence de catégories $$F\, : \, \mathcal{C} \rightarrow \pi_1(\mathcal{C},F)\text{-}\mathrm{EnsFinis}$$ où cette dernière catégorie est celles des $\pi_1(\mathcal{C},F)$-ensembles finis à stabilisateurs ouverts.
		\end{theo}

		La géométrie fournit quantité de catégories galoisiennes intéressantes parmi le groupes fondamentaux desquelles nous pouvons trouver $\cg_{\widetilde{E},\Delta}$. Pour cela, nous utilisons la théorie des diamants dont nous donnons ici une exposition adaptée à nos besoins, avec pour référence \cite{scholze_berkeley} et \cite{scholze_diamonds}. La première référence fournit une approche plus accessible et plus vaste aux perfectoïdes, diamants et objets qui en découlent. La deuxième est plus compacte et technique mais nous servira pour avoir des énoncés millimétrés. Les lecteurs et lectrices souhaitant d'abord de familiariser avec la géométrie adique et les perfectoïde pourront se référer à \cite{morrow_adic}.
		
		\begin{defi}
			Une \textit{paire de Tate perfectoïde} est une paire $(A,A^+)$ où $A$ est un anneau topologique, $A^+$ un sous-anneau ouvert borné et intégralement clos et possédant un élément $\pi$ vérifiant que :
			\begin{itemize}[itemsep=0mm]
				\item La topologie sur $A^+$ est la topologie $\pi$-adique et $A^+$ est $\pi$-adiquement complet.
				
				\item Nous avons $p\in \pi^p A^+$.
				
				\item Le morphisme d'anneaux déduit $\sfrac{A^+}{\pi A^+} \rightarrow \sfrac{A^+}{\pi^p A^+}, \,\, x\mapsto x^p$ est un isomorphisme.
			 \end{itemize}
			Un tel élément $\pi$ s'appelle une pseudo-uniformisante. La définition de paire de Tate perfectoïde requiert l'existence d'une pseudo-uniformisante et non d'en choisir une.
			
			À toute paire de Tate perfectoïde, nous associons un espace topologique $\spa{A}{A^+}$ muni de deux faisceux d'anneaux que nous appelons espace affinoïde. L'ensemble sous-jacent sera $$\quot{\left\{\text{Semi-norme multiplicative}\,\, |\cdot|\, : \, A \rightarrow \Gamma\cup\{0\}\,\,\, \substack{\text{à valeurs dans un groupe abélien totalement ordonné,} \\\text{continue pour la topologie de l'ordre sur }\Gamma,  \\ \text{telle que } |A^+|\subseteq \Gamma_{\leq 1}.}\right\}}{\sim}.$$ La topologie a pour base d'ouverts les ouverts rationnels $$U\left(\frac{f_1,\ldots,f_r}{g}\right):=\left\{ |\cdot| \, \big. \text{ telle que }\, \forall i, \,\,\, |f_i|\leq |g|\neq 0\right\} \,\,\, \text{pour}\,\, (f_1,\ldots,f_r)=A.$$ Cet espace est spectral. Les faisceaux d'anneaux $\mathcal{O}_{\spa{A}{A+}}$ et $\mathcal{O}_{\spa{A}{A+}}^+$ sont caractérisés par leurs valeurs sur les ouverts rationnels :
			
			$$\mathcal{O}_{\spa{A}{A+}}\left(U\left(\frac{f1,\ldots,f_r}{g}\right)\right)=\left(A^+\left[\frac{f_1}{g}, \ldots, \frac{f_r}{g}\right]\right)^{\wedge \pi} \left[\frac{1}{\pi}\right], $$
			
			\begin{align*}\mathcal{O}^+_{\spa{A}{A+}}\left(U\left(\frac{f1,\ldots,f_r}{g}\right)\right)&= \text{clôture intégrale de }\, \left(A^+\left[\frac{f_1}{g}, \ldots, \frac{f_r}{g}\right]\right)^{\wedge \pi} \\ & \phantom{=} \text{ dans } \left(A^+\left[\frac{f_1}{g}, \ldots, \frac{f_r}{g}\right]\right)^{\wedge \pi} \left[\frac{1}{\pi}\right].\end{align*}
		\end{defi}

		\begin{defi}
			Un \textit{espace perfectoïde} est un triplet $(X,\mathcal{O}_X, \mathcal{O}_X^+)$ formé d'un espace topologique et de deux faisceaux d'anneaux, localement isomorphe à un espace affinoïde $(\spa{A}{A^+}, \mathcal{O}_{\spa{A}{A+}}, \mathcal{O}_{\spa{A}{A+}}^+)$. On note souvent $X$ un tel triplet par négligence. Les morphismes d'espaces perfectoïdes sont les morphismes d'espaces annelés $f\, : \, (X,\mathcal{O}_X)\rightarrow(Y, \mathcal{O}_Y)$ induit localement par un morphisme continu d'anneaux tel que l'image de $f^{-1}(\mathcal{O}_Y^+)$ soit contenue dans $\mathcal{O}_X^+$.
			
			On dit qu'il est de caractéristique $p$ (resp. de caractéristique mixte) si $\mathcal{O}_X$ a pour valeurs des anneaux de caractéristique $p$ (resp. de caractéristique mixte).
		\end{defi}
		
		\begin{prop}
			La catégorie des paires de Tate perfectoïdes a pour morphismes $(A,A^+)\rightarrow (B,B^+)$ les morphismes d'anneaux continus $f\, : \, A\rightarrow B$ tel que $f(A^+)\subset B^+$. L'association $(A,A^+)\mapsto \spa{A}{A^+}$ fournit un foncteur pleinement fidèle de la catégorie opposée des paires de Tate perfectoïdes vers celle des espaces perfectoïdes.
		\end{prop}

		\begin{rem}
			Pour toute paire de Tate perfectoïde $(K,K^+)$ telle que $K$ est un corps, l'ensemble des éléments bornés $\mathcal{O}_K$ vérifie que $(K,\mathcal{O}_K)$ est une paire de Tate perfectoïde. L'espace $\spa{K}{\mathcal{O}_K}$ n'a qu'un seul point donné par la classe d'équivalence des valuations $\pi$-adiques\footnote{C'est une valuation de rang $1$ non discrète définie par $|x|_{\pi}=\lim_n p^{-\sfrac{\min\{k \in \mathbb{Z} \, |\, x^n\in \pi^k A^+ \}}{n}}$.} pour toute pseudo-uniformisante $\pi$ et $\mathcal{O}_K$ est l'anneau de valuation associé. L'espace $\spa{K}{K^+}$ peut avoir plus de points, mais $\spa{K}{\mathcal{O}_K}$ y est dense.
		\end{rem}

		\begin{defi}[voir Définitions 6.2 et 7.8 dans \cite{scholze_diamonds}]
			Soit $f\,: \, Y\rightarrow X$ un morphisme d'espaces perfectoïdes.
			
			1) Le morphisme $f$ est dit \textit{fini étale} si pour tout ouvert affinoïde $\spa{R}{R^+}=U\subset X$, son image réciproque est affinoïde disons $f^{-1}(U)=\spa{A}{A^+}$ et le morphisme $R\rightarrow A$ est fini étale, identifiant $A^+$ à la clôture intégrale de $R^+$ dans $A$. Il suffit de le vérifier sur un recouvrement ouvert du but (voir \cite[\S 7]{scholze_perf}).
			
			2) Le morphisme $f$ est dit \textit{étale} s'il se factorise localement sur $Y$ comme la composée d'une immersion ouverte et d'un morphisme fini étale.
			
			3) Le morphisme $f$ est \textit{pro-étale} s'il existe un recouvrement de $Y$ en ouverts affinoïdes $V\!=\spa{A}{A^+}\! \subset Y$, tels qu'il existe $f(V)\!\subset U\!=\spa{R}{R^+}$ affinoïde, un petit système cofiltré d'affinoïdes $\big(\spa{A_i}{A_i^+}\big)_I$ et un morphismes de diagrammes formé de morphismes étales $\big(\spa{A_i}{A_i^+}\big)_I\rightarrow \spa{R}{R^+}$ tel que $f_{|V}$ s'identifie à la limite de ce diagramme.
		\end{defi}
		
		\begin{ex}\label{fet_corps}
			Soit $F$ un corps perfectoïde. Les espaces perfectoïdes finis étales sur $\spa{F}{\mathcal{O}_F}$ sont isomorphes aux $$\bigsqcup_{i\in I} \,\spa{F_i}{\mathcal{O}_{F_i}}$$ où $I$ est fini et où les $F_i|F$ sont des extensions finies séparables.
			
			En particulier si $F$ est algébriquement clos, les seuls espaces finis étales sur $\spa{F}{\mathcal{O}_F}$ sont des unions de copies de ce point.
		\end{ex}
		
		\begin{defi}[voir Définition 8.1 dans \cite{scholze_diamonds}]
			Nous définissons la sous-catégorie pleine $\mathrm{Perf}$ des espaces perfectoïdes dont les objets sont les espaces perfectoïdes de caractéristique $p$. Nous la munissons canoniquement de la toplogie pro-étale : les recouvrement de $X$ sont les familles $(f_i \, : \, X_i \rightarrow X)_I$ de morphismes pro-étales tels que pour tout ouvert quasi-compact $U$ de $X$, il existe un sous-ensemble fini $J\subset I$ et des ouverts quasi-compacts $V_j\subset X_j$ vérifiant que $U\subseteq \cup_J\, f_j(V_j)$. La catégorie $\mathrm{Perf}$ munie de cette topologie sera appelée le site pro-étale $\mathrm{Perf}$.
		\end{defi}
		
		Grâce à \cite[Coro. 8.6]{scholze_diamonds}, il se trouve que la topologie sur le site pro-étale $\mathrm{Perf}$ est sous-canonique, au sens où pour tout espace perfectoïde $X$ de caractéristique $p$, le préfaisceau $Z\mapsto \mathrm{Hom}_{\mathrm{Perf}}(Z,X)$ est un faisceau sur le site pro-étale $\mathrm{Perf}$ (voir \cite[\href{https://stacks.math.columbia.edu/tag/00WP}{Tag 00WP}]{stacks-project} and \cite[\href{https://stacks.math.columbia.edu/tag/00WQ}{Tag 00WQ}]{stacks-project}). Nous n'opérons aucune différence de notation entre l'espace perfectoïde et le faisceau associé, d'autant plus que cette identification est pleinement fidèle. Les propriétés générales des topos, en particulier lorsqu'une topologie sous-jacente est sous-canonique, et plus généralement la notion de morphismes représentables de faisceaux peuvent se trouver dans \cite[\href{https://stacks.math.columbia.edu/tag/00UZ}{Tag 00UZ}]{stacks-project} et \cite[\href{https://stacks.math.columbia.edu/tag/0021}{Tag 0021}]{stacks-project}.

		\begin{defi}
			Un \textit{diamant} est un faisceau $X$ sur le site pro-étale $\mathrm{Perf}$ de la forme $\sfrac{X_0}{R}$. Ici $X_0$ est un espace perfectoïde de caractéristique $p$ et $R$ est une relation d'équivalence\footnote{Une relation d'équivalence sur un faisceau $X_0$ est sous-faisceau $R$ de $X_0\times X_0$ qui induit une relation d'équivalence au niveau des sections sur chaque objet de $\mathrm{Perf}$.} sur $X_0$ telle que $R$ est un faisceau représentable et que les compositions $R\rightarrow X_0\times X_0 \rightrightarrows X_0$ sont des morphismes pro-étales d'espaces perfectoïdes.
			
			Les especes perfectoïdes de caractéristiques $p$, vus comme faisceau représentables, forment une sous-catégorie pleine de la catégorie des diamants.
			
			Tout espace perfectoïde de caractéristique $p$ étant muni d'une action du Frobenius absolu, c'est encore le cas de tout diamant. Tout morphisme de diamant est Frobenius-équivariant.
		\end{defi}
		
		\begin{defi}
			Pour tout corps discret $k$ de caractéristique $p$, nous notons abusivement $\mathrm{Spd}(k)$ le faisceau sur le site pro-étale $\mathrm{Perf}$ qui à un affinoïde $\spa{A}{A^+}$ envoie l'ensemble des morphisme d'anneaux $k\rightarrow A$. Ce n'est pas un diamant.
						
			Pour notre puissance $q$ de $p$ et pour tout diamant $X$ sur $\mathrm{Spd}(\mathbb{F}_q)$, le $q$-Frobenius est un automorphisme de diamant sur $\mathrm{Spd}(\mathbb{F}_q)$.
		\end{defi}

		\begin{prop}[Similaire à la Proposition 11.4 dans \cite{scholze_diamonds}]\label{prod_diamants}
		 Soient $X,Y$ deux diamants sur $\mathrm{Spd}(\mathbb{F}_q)$. Le faisceau pro-étale $X\times_{\mathrm{Spd}(\mathbb{F}_q)} Y$ est un diamant.
		\end{prop}
		
		Pour une introduction plus complète que les définitions qui suivront, nous renvoyons les lecteurs et lectrices à \cite[\S 11]{scholze_diamonds}.
		
		\begin{defi}[voir Définitions 10.1 et 10.7 dans \cite{scholze_diamonds}]
			Un morphisme de diamants $X'\rightarrow X$ est appelé \textit{fini étale} (resp. \textit{immersion ouverte}) lorsque, pour tout espace perfectoïde $X$ et tout morphisme de faisceaux $Y\rightarrow X$, le produit fibré $Y\times_X X'$ est représentable et le morphisme d'espaces perfectoïdes $Y\times_X X' \rightarrow Y$ est fini étale (resp. est une immersion ouverte).
			
			Dans le cas d'une immersion ouverte, nous disons par abus de langage que $X'$ est un ouvert de $X$.
		\end{defi}
		
		\begin{prop}[voir Proposition 10.4 dans \cite{scholze_diamonds}]
			Soient $f\, : \, Y \rightarrow Y_1$, $g\, : \, Y_1\rightarrow Y_2$ et $h\, : \, X\rightarrow Y$ des morphismes de diamants.
			
			1) Si $f$ et $g$ sont finis étales alors $g\circ f$ aussi.
			
			2) Si $g\circ f$ et $g$ sont finis étales, alors $f$ aussi.
			
			3) Si $f$ est fini étale, alors le morphisme $Y \times_{Y_1} X\rightarrow X$ est fini étale.
		\end{prop}
		
		\begin{defiprop}
			\begin{enumerate}[itemsep=0mm]
			\item Un \textit{point géométrique} d'un diamant $X$ est un morphisme de faisceaux sur $\mathrm{Perf}$ $$\overline{x}\,:\, \spa{K}{\mathcal{O}_K} \rightarrow X$$ où $K$ est un corps perfectoïde algébriquement clos.
			
			\item Un point d'un diamant $X$ est un morphisme de faisceaux sur $\mathrm{Perf}$ $$x\, : \, \spa{K}{K^+} \rightarrow X$$ où $K$ est un corps perfectoïde et $K^+$ est un sous-anneau ouvert et borné. Deux points $\overline{x}_1$ et $\overline{x}_2$ sont dits équivalents s'il existe un corps perfectoïde $K$ et un diagramme commutatif
			
			\begin{center}
				\begin{tikzcd}
					& \spa{K_1}{\mathcal{O}_{K_1}} \ar[dr,"x_1"]  &\\
					\spa{K}{\mathcal{O}_K} \ar[ur] \ar[dr] & & X \\
					& \spa{K_2}{\mathcal{O}_{K_2}} \ar[ur,"x_2"']
				\end{tikzcd}
			\end{center} où les flèches de gauche sont surjectives\footnote{Sans cette condition, tout point serait dans la classe d'équivalence d'un point géométrique.}

			\item Soit $X$ un diamant. L'ensemble des points à isomorphisme près est noté $|X|$. Si $X$ est un espace perfectoïde, l'ensemble $|X|$ s'identifie à l'espace topologique sous-jacent et le munir de cette topologie. Si $X=\sfrac{X_0}{R}$ est un écriture comme quotient d'un espace perfectoïde, la topologie quotient sur $|X|$ induite par $|X_0|$ ne dépend pas de l'écriture. Nous appelons espace sous-jacent au diamant, encore noté $X$, l'ensemble $|X|$ muni de cette topologie. 
			
			\item Les ouverts de $X$ correspondent exactement aux ouverts de l'espace topologique $|X|$.
			
			\item Un diamant $X$ est dit connexe s'il n'admet pas d'écriture comme union disjointe de deux ouverts non vides. Cela équivaut à ce que $|X|$ soit connexe.
			
			\end{enumerate}
		\end{defiprop}

		\begin{defi}
			Soit $X$ un diamant. La catégorie $\fet{X}$ a pour objets les morphismes de diamants finis étales $Y\rightarrow X$ et pour morphismes les morphismes de $X$-diamants (qui sont automatiquement finis étales).
			
			Soit $\overline{x}$ un point géométrique de $X$. Nous définissons le foncteur fibre $$\mathrm{Fib}_{\overline{x}} \, : \, \fet{X} \rightarrow \mathrm{Ens}, \,\,\, [Y\rightarrow X] \mapsto \left| Y\times_X \overline{x}\right|.$$
		\end{defi}
		
		\begin{theo}\label{fet_diamant_galoisienne}
			Soit $X$ un diamant connexe et $\overline{x}$ un point géométrique. Le couple $(\fet{X},\mathrm{Fib}_{\overline{x}})$ est une catégorie galoisienne dont on notera $\piun{X}{x}$ le groupe fondamental. Les diamants $\fet{X}$-connexes sont les diamants connexes.
		\end{theo}
		
	\begin{ex}\label{piun_corps}
		Soit $F$ un corps perfectoïde. Soit $\overline{x}$ un point géométrique de $\spd{F}{\mathcal{O}_F}$ donné par le choix d'une clôture algébrique $F^{\mathrm{alg}}$ et le morphisme $$\overline{x}\, : \, \spd{\widehat{F^{\mathrm{alg}}}}{\mathcal{O}_{\widehat{F^{\mathrm{alg}}}}} \rightarrow \spd{F}{\mathcal{O}_F}.$$  
		
		L'exemple \ref{fet_corps} décrit les diamants $X$ dans $\fet{\spd{F}{\mathcal{O}_F}}$. Le foncteur $\mathrm{Fib}_{\overline{x}}$ est isomorphe au foncteur qui envoie $\spd{F_i}{\mathcal{O}_{F_i}}$, où $F_i|F$ est finie, sur l'ensemble de plongements de $F_i$ dans $\widehat{F^{\mathrm{alg}}}$. L'action du groupe de Galois $\cg_F$ par post-composition sur ces plongements fournit des endotransformations du foncteur fibre. Elle identifie $\piun{\spd{F}{\mathcal{O}_F}}{x}$ à $\cg_F$.
	\end{ex}
		
		Nous souhaitons ajouter l'action d'un groupe localement profini.
		
		\begin{defi}
			Soit $\Gamma$ un groupe localement profini.
			
			\begin{enumerate}[itemsep=0mm]
				\item Nous définissons $\underline{\Gamma}$ le faisceau en groupes sur $\mathrm{Perf}$ donné par $\underline{\Gamma}(X)=\mathscr{C}(|X|,\Gamma)$.
				
				\item Une \textit{action de $\Gamma$ sur un diamant $X$} est un morphisme de faisceau sur $\mathrm{Perf}$ $$\underline{\Gamma}\times X\rightarrow X$$ qui fait commuter les diagrammes d'actions de groupes.
				
				\item Un diamant $X$ avec action de $\Gamma$ est dit $\Gamma$-connexe s'il n'est pas l'union de deux ouverts non vides stables par $\Gamma$.
			\end{enumerate}
		\end{defi}
		
		\begin{rem}
			Une telle action de $\Gamma$ sur $X$ induit une action continue de $\Gamma$ sur $|X|$. Les ouverts $\Gamma$-stables de $X$ correspondent bijectivement aux ouverts $\Gamma$-stables de $|X|$.
		\end{rem}
		
		\begin{defi}
			Soit $\Gamma$ un groupe localement profini et $X$ un diamant muni d'une action de $\Gamma$. La catégorie $\fet{X|\Gamma}$ a pour objets les morphismes finis étales $\Gamma$-équivariants $Y\rightarrow X$ où $Y$ est un diamant avec action de $\Gamma$. Les morphismes sont les morphismes de $X$-diamants $\Gamma$-équivariants.
			
			Soit $\overline{x}$ un point géométrique de $X$, nous définissons le foncteur fibre $$\mathrm{Fib}_{\overline{x}} \, : \, \fet{X|\Gamma} \rightarrow \mathrm{Ens}, \,\,\, [Y\rightarrow X] \mapsto \left| Y\times_X \overline{x}\right|.$$
		\end{defi}
		
		\begin{theo}\label{existence_piun}
			Soit $X$ un diamant muni d'une action de $\Gamma$ qui le rend $\Gamma$-connexe. On suppose également que l'action est libre, i.e. que $\forall Z$, l'action de $\Gamma \subset \mathscr{C}(|Z|,\Gamma)$ sur $X(Z)$ est libre. Alors, le couple $(\fet{X|\Gamma},\mathrm{Fib}_{\overline{x}})$ est une catégorie galoisienne dont on note $\piun{X|\Gamma}{x}$ le groupe fondamental.
			
			Les objets $\fet{X|\Gamma}$-connexes sont les diamants $\Gamma$-connexes.
		\end{theo}
		\begin{proof}
			Le théorème se déduit du Théorème \ref{fet_diamant_galoisienne} en construisant le diamant $\sfrac{X}{\underline{\Gamma}}$ et en prouvant que le couple $(\fet{X|\Gamma},\mathrm{Fib}_{\overline{x}})$ est équivalent à la catégorie galoisienne $(\fet{\sfrac{X}{\underline{\Gamma}}},\mathrm{Fib}_{\overline{x}})$.
		\end{proof}

	\begin{prop}\label{changement_base_piun}
		Soit $\Gamma_1 \rightarrow \Gamma_2$ un morphisme continu de groupes localement profinis. Soit $Y$ un diamant avec action de $\Gamma_1$ et $X$ un diamant avec action de $\Gamma_2$, qui le munit en particulier $X$ d'une action de $\Gamma_1$. Supposons que les deux diamants vérifient les conditions du théorème \ref{existence_piun}. Soit $Y\rightarrow X$ un morphisme de diamants \linebreak$\Gamma_1$-équivariant. Soit $\overline{y}$ un point géométrique de $Y$. Le foncteur suivant est correctement défini. 
		
		\begin{center}
			\begin{tikzcd}
				\fet{X|\Gamma_2} \ar[r] & \fet{Y|\Gamma_1} \\
				Z \ar[r,maps to] & Z\times_X Y \\
				f\ar[r, maps to] & f\times_X \mathrm{Id}_Y
			\end{tikzcd}
		\end{center}
		
		Il existe un morphisme de groupes profinis
		
		$$\piun{Y|\Gamma_1}{y} \rightarrow \piun{X|\Gamma_2}{y}$$ qui rend le diagramme suivant commutatif
		
		\begin{center}
			\begin{tikzcd}
				\fet{X|\Gamma_2} \ar[r,"\sim"',"\mathrm{Fib}_{\overline{y}}"] \ar[d] & \piun{X|\Gamma_2}{y}\text{-}\mathrm{EnsFinis} \ar[d]  \\
				\fet{Y|\Gamma_1} \ar[r,"\sim","\mathrm{Fib}_{\overline{y}}"'] & \piun{Y|\Gamma_1}{y}\text{-}\mathrm{EnsFinis}    \end{tikzcd}
		\end{center} où le foncteur de droite est obtenu grâce au morphisme entre groupes fondamentaux ci-dessus.
	\end{prop}
	\begin{proof}
		La définition du foncteur ne pose aucun problème. Pour ce qui est du morphisme entre groupes, soit $T$ un automorphisme de $\mathrm{Fib}_{\overline{y}} \, : \,\fet{Y|\Gamma_1} \rightarrow \mathrm{EnsFinis}$. Pour tout objet $Z$ de $\fet{X|\Gamma_2}$, l'ensemble $Z\times_X \overline{y}$ peut s'identifie à $\big(Z \times_X Y\big) \times_Y \overline{y}$ sur lequel agit la bijection $T(Z\times_X Y)$. Cette famille de bijections étant naturelle en $Z$, elle fournit un automorphisme du foncteur $\fet{X|\Gamma_2} \rightarrow \mathrm{EnsFinis}$, soit un élément de $\piun{X|\Gamma_2}{y}$. Nous vérifions qu'il s'agit d'un morphisme de groupes topologiques et qu'il fait commuter le diagramme.
	\end{proof}

	Après cette introduction assez générale, nous allons nous replacer dans un contexte plus propice à notre équivalence. Ressaisissons-nous de notre puissance $q$ du premier $p$ et l'ensemble fini $\Delta$.
	
	\begin{prop}\label{constr_action_phi_prod}
		Soient $(X_{\alpha})_{\alpha\in \Delta}$ des diamants sur $\mathrm{Spd}(\mathbb{F}_q)$. Le diamant donné par $$\prod_{\alpha\in \Delta, \, \mathrm{Spd}(\mathbb{F}_q)} X_{\alpha} $$ possède une unique action du groupe $\Phi_{\Delta,q}^{\mathrm{gp}}$ telle que l'action de $\varphi_{\beta,q}$ est un morphisme de $\prod_{\alpha \neq \beta} X_{\alpha}$-diamants $\varphi_{\Delta,q}$ est celle du $q$-Frobenius.
	\end{prop}
	
	\begin{defi}
		Soit $X$ un diamant sur $\mathrm{Spd}(\mathbb{F}_q)$ muni d'une action de $\Phi_{\Delta,q}^{\gp}$ telle que l'action de $\varphi_{\Delta,q}$ est le $q$-Frobenius. Nous appelons $\fet{X \, ||\, \Phi_{\Delta,q}}$ la sous-catégorie pleine de $\fet{X|\Phi_{\Delta,q}^{\gp}}$ dont les objets sont les diamants tels que l'action de $\varphi_{\Delta,q}$ est le $q$-Frobenius.
	\end{defi}

	\begin{theo}[Lemme de Drinfeld]\label{lemme_drinfeld}
		Soient $(X_{\alpha})_{\alpha\in \Delta}$ des diamants connexes et localement spatiaux sur $\mathrm{Spd}(\mathbb{F}_q)$ pour lesquels l'action de $\varphi_q^{\Z}$ est libre. Soit également $\overline{x}$ un point géométrique du diamant $\prod_{\alpha\in \Delta, \, \mathrm{Spd}(\mathbb{F}_q)} X_{\alpha}$. La paire $$\left(\fet{\prod_{\alpha\in \Delta, \, \mathrm{Spd}(\mathbb{F}_q)} X_{\alpha} \, \bigg|\bigg|\,\Phi_{\Delta,q}},\mathrm{Fib}_{\overline{x}}\right)$$ est une catégorie galoisienne.
			
		De plus, la famille de morphismes $\varphi_{\alpha,q}$-équivariants $$ \prod_{\alpha\in \Delta, \, \mathrm{Spd}(\mathbb{F}_q)} X_{\alpha} \longrightarrow X_{\alpha}$$ construit un morphisme $$\piun{\prod_{\alpha\in \Delta, \, \mathrm{Spd}(\mathbb{F}_q)} X_{\alpha} \, \bigg|\bigg|\, \Phi_{\Delta,q}}{x} \longrightarrow \prod_{\alpha \in \Delta} \piun{X_{\alpha}}{x}$$ qui est un isomorphisme de groupes profinis.
		\end{theo}
		\begin{proof}
		Voir la thèse de l'auteur, l'hypothèse de liberté est probablement superflue.
		\end{proof}

		\begin{prop}\label{morph_grp_fond_drinfeld}
			Soit $\widetilde{E}$ un corps perfectoïde de caractéristique $p$, extension de $\mathbb{F}_q$. Soit $\widetilde{F}|\widetilde{E}$ une extension finie, extension de $\mathbb{F}_{q'}$ pour $\mathbb{F}_{q'}|\mathbb{F}_q$. Soit $\Delta$ un ensemble fini. Il existe un foncteur canonique $$\fet{\prod_{\alpha \in \Delta,\,\mathrm{Spd}(\mathbb{F}_{q'})} \spd{\widetilde{F}_{\alpha}}{\widetilde{F}_{\alpha}^+} \,\bigg|\bigg|\, \Phi_{\Delta,q'}} \rightarrow \fet{\prod_{\alpha \in \Delta,\,\mathrm{Spd}(\mathbb{F}_{q})} \spd{\widetilde{E}_{\alpha}}{\widetilde{E}_{\alpha}^+} \,\bigg|\bigg|\, \Phi_{\Delta,q}}.$$
			 Soit $\overline{x} \rightarrow \prod_{\alpha \in \Delta, \,\mathrm{Spd}(\mathbb{F}_{q'})} \spd{\widetilde{F}_{\alpha}}{\widetilde{F}_{\alpha}^+}$ un point géométrique. Le morphisme $\cg_{\widetilde{F},\Delta}\rightarrow \cg_{\widetilde{E},\Delta}$ déduit de $$ \piun{\prod_{\alpha \in \Delta,\mathrm{Spd}(\mathbb{F}_{q'})} \spd{\widetilde{F}_{\alpha}}{\widetilde{F}_{\alpha}^+} \,\bigg|\bigg|\, \Phi_{\Delta,q'}}{x} \rightarrow \piun{\prod_{\alpha \in \Delta,\mathrm{Spd}(\mathbb{F}_{q})} \spd{\widetilde{E}_{\alpha}}{\widetilde{E}_{\alpha}^+} \,\bigg|\bigg|\, \Phi_{\Delta,q}}{x}$$ grâce au lemme de Drinfeld correspond à l'inclusion.
		\end{prop}
		\begin{proof}
			Par propriétés universelles, nous construisons un morphisme $\Phi_{\Delta,q'}^{\gp}$-équivariant $$ \prod_{\alpha \in \Delta,\mathrm{Spd}(\mathbb{F}_{q'})} \spd{\widetilde{F}_{\alpha}}{\widetilde{F}_{\alpha}^+} \longrightarrow \prod_{\alpha \in \Delta,\mathrm{Spd}(\mathbb{F}_{q})} \spd{\widetilde{E}_{\alpha}}{\widetilde{E}_{\alpha}^+}.$$ Le tiré en arrière par ce morphisme fournit le foncteur annoncé entre catégories galoisiennes.
			
			Pour tout $\beta\in \Delta$, la construction de notre morphisme canonique fournit un diagramme commutatif :
			\begin{center}
				\begin{tikzcd}
					\prod_{\alpha \in \Delta,\mathrm{Spd}(\mathbb{F}_{q'})} \spd{\widetilde{F}_{\alpha}}{\widetilde{F}_{\alpha}^+} \ar[r] \ar[d] & \spd{\widetilde{F}_{\beta}}{\widetilde{F}_{\beta}^+} \ar[d] \\
					\prod_{\alpha \in \Delta,\mathrm{Spd}(\mathbb{F}_{q})} \spd{\widetilde{E}_{\alpha}}{\widetilde{E}_{\alpha}^+} \ar[r] &  \spd{\widetilde{E}_{\beta}}{\widetilde{E}_{\beta}^+}
				\end{tikzcd}
			\end{center} Dans le diagramme obtenu en passant aux groupes fondamentaux, l'énoncé précis du Théorème \ref{lemme_drinfeld} affirme que les morphismes horizontaux s'identifient aux projections $$\cg_{\widetilde{F},\Delta}\rightarrow \cg_{\widetilde{F}_{\beta}} \,\,\, \,\mathrm{et}\,\,\,\, \cg_{\widetilde{E},\Delta}\rightarrow \cg_{\widetilde{E}_{\beta}}.$$ Or, en se rappelant comment on identifie $\cg_{\widetilde{E}_{\beta}}$ aux automorphismes du foncteurs fibres sur $\fet{\spd{\widetilde{E}_{\beta}}{\widetilde{E}_{\beta}^+}}$, les morphismes verticaux de droite correspondent aux inclusions $\cg_{\widetilde{F}_{\beta}}\subset \cg_{\widetilde{E}_{\beta}}$ ce qui conclut.
		\end{proof}

		\begin{lemma}\label{faisceau_constants}
			Soit $\Gamma$ un groupe localement profini, $X$ un diamant muni d'une action libre de $\underline{\Gamma}$ qui le rend \linebreak$\underline{\Gamma}$-connexe et $\overline{x}$ un point géométrique de $X$. Notons $\mathbf{T}$ l'équivalence $\fet{X|\underline{\Gamma}} \cong \piun{X|\underline{\Gamma}}{x}\text{-}\mathrm{Ens}$.
			
			\begin{enumerate}[itemsep=0mm]
				\item Il y a un isomorphisme naturel entre foncteurs de $\piun{X|\underline{\Gamma}}{x}\text{-}\mathrm{Ens}$ dans $\mathrm{Ens}$ $$H^0_{\Gamma} \xRightarrow{\sim} \mathrm{Orb} \circ \mathbf{T}$$ où $H^0_{\Gamma}$ est le foncteur des composantes $\Gamma$-connexes et $\mathrm{Orb}$ est le foncteur des orbites.
				
				\item Soit $k$ un corps fini et $V$ un $k$-espace vectoriel. Soit $Z$ un objet en $k$-espaces vectoriels de $\fet{X|\underline{\Gamma}}$ qui représente le \textit{faisceau constant} associé à $V$ $$\fet{X|\underline{\Gamma}} \rightarrow k\text{-}\mathrm{EspacesVect}, \,\,\, Y \mapsto V^{H^0_{\Gamma}(Y)}.$$ Alors $\mathbf{T}(Z)$ isomorphe au $k$-espace vectoriel $V$ avec action triviale de $\piun{X|\underline{\Gamma}}{x}$.
			\end{enumerate}
		\end{lemma}
		\begin{proof}
			\begin{enumerate}[itemsep=0mm]
				\item Conséquence de la théorie générale des catégories galoisiennes et de la description des objets connexes dans $\fet{X|\underline{\Gamma}}$.
				
				\item Le $k$-espace vectoriel $V$ avec action triviale de $\piun{X|\underline{\Gamma}}{x}$ représente le foncteur $$\piun{X|\underline{\Gamma}}{x} \text{-}\mathrm{Ens} \rightarrow k\text{-}\mathrm{EspacesVect}, \,\,\, A \mapsto \mathrm{Hom}_{\piun{X|\underline{\Gamma}}{x} \text{-}\mathrm{Ens}}\left(A, V\right)=V^{\mathrm{Orb}(A)}.$$ Or, $\mathbf{T}(Z)$ représente par hypothèse le foncteur $$\piun{X|\underline{\Gamma}}{x} \text{-}\mathrm{Ens} \rightarrow k\text{-}\mathrm{EspacesVect}, \,\,\, A \mapsto \mathrm{Hom}_{\fet{X|\underline{\Gamma}}}\left(\mathbf{T}^{-1}(A), Z\right)=V^{H^0_{\Gamma}\circ \mathbf{T}^{-1}(A)}.$$ D'après le premier point, les deux foncteurs sont isomorphes ce qui conclut.
			\end{enumerate}
		\end{proof}

	\vspace{0.25 cm}
	
	Nous sommes prêts à construire le foncteur $\mathrm{V}_{\widetilde{E}}$ qui nous permettra de mieux comprendre appréhender l'isomorphisme de comparaison.
	
		\begin{prop}
		Soit $D$ un objet de $\detaleproj{\Phi_{\Delta,q,r}}{\widetilde{E}_{\Delta}}$, le foncteur $$\mathfrak{I}_D \, : \,\widetilde{E}_{\Delta}\text{-}\mathrm{Alg} \rightarrow \mathrm{Ens}, \,\,\, T \mapsto \big( T\otimes_{\widetilde{E}_{\Delta}} D\big)^{\varphi_{\Delta,r}=\mathrm{Id}},$$ où $\varphi_{\Delta,r}$ agit sur $T$ comme de $r$-Frobenius, est représentable par une $\widetilde{E}_{\Delta}$-algèbre finie étale que l'on nomme $S_D$.
	\end{prop}
	\begin{proof}
		\underline{Étape 1 :} soit $R$ une $\mathbb{F}_r$-algèbre et $\varphi_r$ le $r$-Frobenius. Soit $D$ un objet de $\detale{\varphi_r^{\mathbb{N}}}{R}$ qui est libre de rang fini comme $R$-module. Alors le foncteur $$R\text{-}\mathrm{Alg} \rightarrow \mathrm{Ens}, \,\,\, T \mapsto \big(T\otimes_R D\big)^{\varphi_r=\mathrm{Id}}$$ est représentable par une $R$-algèbre finie étale. Fixons $\mathcal{B}=(e_i)_{1\leq i\leq d}$ une base de $D$ et $A$ la matrice de $\varphi_{r,D}$ dans la base $\mathcal{B}$, autrement dit $$\forall (x_i)\in R^d,\,\, \varphi_{r,D}\left(\sum_{1\leq i \leq d} x_i e_i\right)=\sum_{1\leq i,j\leq d} x_i^r a_{i,j} e_j.$$ Pour toute $R$-algèbre $T$, le produit tensoriel $T\otimes_R D$ est un $T$-module libre de base $\mathcal{B}$ et $\sum_i t_i e_i \in \left(T\otimes_R D\right)^{\varphi_r=\mathrm{Id}}$ si et seulement si
		$$\sum_j t_j e_j  = \sum_j \left(\sum_i t_i^r a_{i,j}\right) e_j.$$ Ceci équivaut à $$\forall j, \,\,\, t_j=\sum_i t_i^r a_{i,j}.$$ Comme $D$ est étale, la matrice $A$ est inversible d'inverse $B$ et on obtient $$\sum_j t_j e_j \in\left( T\otimes_R D\right)^{\varphi_r=\mathrm{Id}} \,\, \Leftrightarrow \,\, \forall j, \,\,\, t_j^r=\sum_i t_i b_{i,j}.$$ Le foncteur est donc représenté par la $R$-algèbre $\sfrac{R[T_i]}{(T_j^r-\sum_i b_{i,j} T_i)_{1\leq j \leq d}}$. La jacobienne de la famille de polynômes définissant le quotient vaut $B$ en tout point d'annulation ce qui prouve que la $R$-algèbre est finie étale. Par Yoneda, cette $R$-algèbre ne dépend pas de $\mathcal{B}$ à isomorphisme près.

		\underline{Étape 2 :} en général, puisque $D$ est fini projectif, le Lemme \cite[\href{https://stacks.math.columbia.edu/tag/00NX}{Tag 00NX}]{stacks-project} affirme qu'il est localement libre de rang fini. Sur tout ouvert affine $\mathrm{V}(z)$ tel que $D[z^{-1}]$ est libre, nous construisons $S_z$ avec la première étape. Yoneda fournit une immersion ouverte $f_{y,z} \, : \, \mathrm{Spec}(S_y) \rightarrow \mathrm{Spec}(S_z)$ pour toute inclusion $\mathrm{V}(y) \subseteq \mathrm{V}(z)$. D'après \cite[\S 5.2.2.2]{poly_ducros}, on peut recoller les $\mathrm{Spec}(S_z)$ le long des $f_{y,z}$. On vérifie que le schéma sur $\mathrm{Spec}(\widetilde{E}_{\Delta})$ obtenu est le spectre de l'algèbre finie étale escomptée.
	\end{proof}	
	
	Nous démontrons au Corollaire \ref{cestperfectoid} que la paire $(\widetilde{E}_{\Delta},\widetilde{E}_{\Delta}^+)$ est une paire de Huber perfectoïde. Comme l'anneau $S_D$ est une $\widetilde{E}_{\Delta}$-algèbre finie étale, nous pouvons le voir comme un diamant muni d'un morphisme fini étale $$\spd{S_D}{S_D^+}\rightarrow \spd{\widetilde{E}_{\Delta}}{\widetilde{E}_{\Delta}^+}.$$
	
	\begin{prop}\label{propstructurefetphiq}
		Le diamant $\spd{S_D}{S_D^+}$ est naturellement un objet de $\fet{\spd{\widetilde{E}_{\Delta}}{\widetilde{E}_{\Delta}^+}\, ||\,\Phi_{\Delta,q}}$.
	\end{prop}
	\begin{proof}
		Le sous-anneau $S_D^+$ est par définition la clôture intégrale de $\widetilde{E}_{\Delta}^+$. Puisque $\widetilde{E}_{\Delta}^+$ est un sous-$\Phi_{\Delta,q,r}^{\gp}$-\linebreak-anneau de $\widetilde{E}_{\Delta}$, toute structure de $\Phi_{\Delta,q,r}^{\gp}$-anneau sur $S_D$ compatible à la structure de $\widetilde{E}_{\Delta}$-algèbre se restreint-\linebreak-corestreint en une structure sur $S_D^+$ et fournit une action sur le diamant. Pour obtenir un objet de \linebreak$\fet{\spd{\widetilde{E}_{\Delta}}{\widetilde{E}_{\Delta}^+} \, \big|\big|\, \Phi_{\Delta,q}}$, il reste donc à construire une action telle que $\varphi_{\Delta,q}$ coïncide avec le $q$-Frobenius.
		
		Construisons l'action de $\Phi_{\Delta,q}$. Soit $\alpha \in \Delta$ et $T$ une $\widetilde{E}_{\Delta}$-algèbre. Le morphisme $$ T\otimes_{\widetilde{E}_{\Delta}} D \rightarrow \varphi_{\alpha,q}^*T\otimes_{\widetilde{E}_{\Delta}} D, \,\,\, t\otimes d\mapsto (1\otimes t)\otimes \varphi_{\alpha,q,D}(d)$$ est $\varphi_{\Delta,r}$-équivariant et naturel en $T$, ce qui fournit une application naturelle en $T$ $$\big(T\otimes_{\widetilde{E}_{\Delta}} D\big)^{\varphi_{\Delta,r}=\mathrm{Id}}\rightarrow \big(\varphi_{\alpha,q}^*T \otimes_{\widetilde{E}_{\Delta}} D\big)^{\varphi_{\Delta,r}=\mathrm{Id}}$$ qui se réinteprète en une transformation naturelle $$\Psi_{\alpha} \, : \, \mathrm{Hom}_{\widetilde{E}_{\Delta}\text{-}\mathrm{Alg}}(S_D,-)\Rightarrow \mathrm{Hom}_{\widetilde{E}_{\Delta}\text{-}\mathrm{Alg}}(S_D,\varphi_{\alpha,q}^*-).$$		
		Définissons $\psi_{\alpha,D}:=\Psi_{\alpha}(S_D)(\mathrm{Id}_{S_D})$. 
		
		Pour n'importe quel $\varphi_{1}\in \Phi_{\Delta}$, nous construisons de même $\Psi_1$ et $\psi_{1,D}\, : \, S_D\rightarrow \varphi_1^* S_D$ à partir de l'action de $\varphi_1$. Commençons par prouver que les $\psi_{-,D}$ se comportent bien vis-à-vis de la composition. Pour tout morphisme de $\widetilde{E}_{\Delta}$-algèbres $g \, : \, S_D \rightarrow T$, le diagramme suivant commute :
		
		\begin{center}
			\begin{tikzcd}
				\mathrm{Hom}_{\widetilde{E}_{\Delta}\text{-}\mathrm{Alg}}(S_D,S_D) \ar[r,"\Psi_{\alpha}(S_D)"] \ar[d,"-\,\circ \,g"] & \mathrm{Hom}_{\widetilde{E}_{\Delta}\text{-}\mathrm{Alg}}(S_D,\varphi_{\alpha,q}^*S_D) \ar[d,"-\,\circ\, \varphi_{\alpha,q}^*g"] \\
				\mathrm{Hom}_{\widetilde{E}_{\Delta}\text{-}\mathrm{Alg}}(S_D,T) \ar[r,"\Psi_{\alpha}(T)"] & \mathrm{Hom}_{\widetilde{E}_{\Delta}\text{-}\mathrm{Alg}}(S_D,\varphi_{\alpha,q}^*T)
			\end{tikzcd}
		\end{center} et en suivant l'image de $\mathrm{Id}_{S_D}$, on trouve que $\Psi_{\alpha}(T)(g)=\varphi_{\alpha,q}^*g \circ \psi_{\alpha,D}$. En particulier, \begin{equation*}\label{eqn:eq1}\varphi_{\alpha,q}^*\psi_{\beta,D} \circ \psi_{\alpha,D}=\Psi_{\alpha}(\varphi_{\beta,q}^* S)(\psi_{\beta,D})=\big[\Psi_{\alpha}(\varphi^*_{\beta,q'}S_D)\circ \Psi_{\beta}(S_D)\big](\mathrm{Id}_{S_D}) \tag{*1}\end{equation*} Appelons $i_T$ l'isomorphisme naturel en $T$ de $(\varphi_{\alpha,q}\varphi_{\beta,q})^*T$ à $\varphi_{\alpha,q}^*(\varphi_{\beta,q}^*T)$. Grâce à un diagramme similaire, il fournit une égalité les foncteurs $(f\mapsto i_{-} \circ f)\circ \Psi_{\alpha\beta}$, où $\Psi_{\alpha\beta}$ est associé à $\varphi_{\alpha,q}\varphi_{\beta,q}$ et $\Psi_{\alpha}(\varphi_{\beta,q}^*-) \circ \Psi_{\beta}$. Appliqué à l'équation (\ref{eqn:eq1}), cela donne $$i_{S_D} \circ \psi_{\alpha\beta,D}=\varphi_{\alpha,q'}^*\psi_{\beta,D} \circ \psi_{\alpha,D}.$$ Nous pouvons généraliser ce résultat à $\varphi_1,\varphi_2\in \Phi_{\Delta,q}$. En particulier, nous avons $$i\circ\psi_{\Delta,D}=\varphi_{\alpha_1,q'}^*\psi_{\alpha_2 \ldots \alpha_{[\Delta|},D} \circ \cdots \circ \psi_{\alpha_{|\Delta|},D}$$ où $i \, :\, \varphi_{\Delta,q}^*S_D \xrightarrow{\sim} \varphi_{\alpha_1,q}^*\cdots \varphi_{\alpha_{|\Delta|},q}^* S_D$ et où $\psi_{\Delta,D}$ est construit à partir du Frobenius relatif $\varphi_{\Delta,q}$. Ainsi, démontrer que les $\psi_{\alpha,D}$ sont inversibles se réduit à démontrer que $\psi_{\Delta,D}$ est inversible.
		
		Soit $\Psi$ la transformation naturelle obtenu à partir du Frobenius relatif. Pour tout $\widetilde{E}_{\Delta}$-algèbre $T$, nous considérons le linéarisé du Frobenius relatif $$\varphi_{q,T}^* \, : \, \varphi_{q}^*T \rightarrow T, \,\,\, x\otimes t \mapsto x t^{q}.$$ Nous obtenons le diagramme commutatif suivant :
		
		\begin{center}
			\begin{tikzcd}
				\big(T\otimes_{\widetilde{E}_{\Delta}}D\big)^{\varphi_{\Delta,r}=\mathrm{Id}} \ar[rrr] \ar[d] & & & \big(\varphi_{q}^*T\otimes_{\widetilde{E}_{\Delta}}D\big)^{\varphi_{\Delta,r}=\mathrm{Id}} \ar[r,"\mathfrak{I}_D(\varphi_{q}^*)"] \ar[d] & \big(T\otimes_{\widetilde{E}_{\Delta}}D\big)^{\varphi_{\Delta,r}=\mathrm{Id}} \ar[d] \\
				\mathrm{Hom}_{\widetilde{E}_{\Delta}\text{-}\mathrm{Alg}} (S_D,T) \ar[rrr,"\Psi(T)=\varphi_{q}^*(-)\circ \psi_{\Delta,D}"] & & & \mathrm{Hom}_{\widetilde{E}_{\Delta}\text{-}\mathrm{Alg}} (S_D,\varphi_{q}^*T) \ar[r,"\varphi_{q}^* \circ \cdot"] & \mathrm{Hom}_{\widetilde{E}_{\Delta}\text{-}\mathrm{Alg}} (S_D,T)
			\end{tikzcd}
		\end{center} Si l'on considère $\sum t_i \otimes d_i$ en haut à gauche, son image en haut à droite est exactement $$\sum t_i^{q}\otimes \varphi_{\Delta,q,D}(d_i)=\varphi_{\Delta,q,T\otimes D} (\sum t_i \otimes d_i)=\sum t_i\otimes d_i.$$ Il en découle que la composée du bas est l'identité. En appliquant à $T=S_D$ et à $\mathrm{Id}_{S_D}$, il vient que \linebreak$\varphi_{q}^*\circ\psi_{\Delta,D} =\mathrm{Id}_{S_D}$, autrement dit que $\psi_{\Delta,D}$ est simplement l'inverse du Frobenius absolu sur $S_D$. 
		
		Nous venons de finir de prouver que les $\psi_{\alpha,D}$ sont inversibles. Le fait qu'ils soient $\mathcal{E}_{\Delta}$-linéaires, les relations qu'ils vérifient et l'expression de $\psi_{\Delta,D}$ prouvent précisément que leurs inverses fournissent une structure d'objet de $\fet{\spd{\widetilde{E}_{\Delta}}{\widetilde{E}_{\Delta}^+}\, ||\,\Phi_{\Delta,q}}$.
	\end{proof}
	
	Notre construction atterrit désormais dans une catégorie davantage adaptée au lemme de Drinfeld. En revanche, nous vivons toujours sur $\spd{\widetilde{E}_{\Delta}}{\widetilde{E}_{\Delta}^+}$ et non sur $X_{\widetilde{E}}=\prod_{\alpha \in \Delta,\,\mathrm{Spd}(\mathbb{F}_{q})} \spd{\widetilde{E}_{\alpha}}{\widetilde{E}_{\alpha}^+}$. Ce sont deux espaces distincts : philosophiquement, certaines valuations autorisent que tous les $\varpi_{\alpha}$ ne sont pas toplogiquement nilpotents sur le premier espace, là où le deuxième les force à être tous topologiquement nilpotents. Nous donnons un énoncé précis.
	
	\begin{lemma}\label{X_et_Y}
		Définissons $U_m:=\{|\cdot|\in \spa{\widetilde{E}_{\Delta}}{\widetilde{E}_{\Delta}^+} \,|\, \forall \alpha,\beta, \, |\varpi_{\alpha}|^m\leq |\varpi_{\beta}|\}$. C'est un ouvert rationnel\footnote{En effet, puisque tous les $\varpi_{\beta}$ sont inversibles dans $\widetilde{F}_{\Delta}$ les conditions $|\varpi_{\beta}|\neq 0$ peuvent être sous-entendues. De plus, les $\varpi_{\alpha}^m$ engendrent un idéal ouvert de $\widetilde{F}_{\Delta}^+$, de type fini.} de $\spa{\widetilde{E}_{\Delta}}{\widetilde{E}_{\Delta}^+}$.  
		
		Appelons $Y_{\widetilde{E}}$ l'espace perfectoïde réunion des ouverts $U_m$ de $\spa{\widetilde{E}_{\Delta}}{\widetilde{E}_{\Delta}^+}$, et le diamant associé. 
		
		Les diamants $Y_{\widetilde{E}}$ et $X_{\widetilde{E}}$ sont canoniquement isomorphes. L'action de $\Phi_{\Delta,q}^{\gp}$ sur $\spa{\widetilde{E}_{\Delta}}{\widetilde{E}_{\Delta}^+}$se restreint-\linebreak-corestreint sur $Y_{\widetilde{E}}$ et correspond à l'action sur $X_{\widetilde{E}}$ construite à la Proposition \ref{constr_action_phi_prod}.
	\end{lemma}
	\begin{proof}
		Prouvons d'abord que les points de $Y_{\widetilde{E}}$ sont précisément les valuations $|\cdot |\in \spa{\widetilde{E}_{\Delta}}{\widetilde{E}_{\Delta}^+}$ telles que $\forall \alpha, \, |\varpi_{\alpha}|^n\rightarrow 0$. Soit $|\cdot|\in U_m$. Puisque $\varpi_{\Delta}\in \widetilde{F}_{\Delta}^{\circ \circ}$, nous avons $|\varpi_{\Delta}|^n \rightarrow 0$. Puisque $|\cdot|\in U_m$, nous obtenons $$\forall \alpha,\,\,\, |\varpi_{\alpha}|^{m |\Delta|}\leq |\varpi_{\Delta}|$$ d'où $|\varpi_{\alpha}|^n \rightarrow 0$. Réciproquement, supposons que $\forall \alpha, \,\, |\varpi_{\alpha}|^n\rightarrow 0$. Pour tout $(\alpha,\beta)$, on a $|\varpi_{\alpha}|^m\leq |\varpi_{\beta}|$ pour $m\gg 0$. Pour $m$ assez grand, elles sont toutes vérifiées et $|\cdot|\in U_m$.
		
		Pour prouver que les deux diamants coïncident, commençons par trouver une famille de morphismes \linebreak$Y_{\widetilde{E}} \rightarrow \spd{\widetilde{E}_{\alpha}}{\widetilde{E}_{\alpha}^+}$. Posons jusqu'à la fin de la sous-section $(\widetilde{E}_{\Delta,m},\widetilde{E}_{\Delta,m}^+)$ la paire de Huber correspondant à l'ouvert rationnel $U_m$. Soit $m\geq 1$ et $\alpha \in \Delta$. Nous venons de prouver que pour tout $|\cdot|\in U_m$, nous avons $|\varpi_{\alpha}|^n\rightarrow 0$, ce qui implique que $\varpi_{\alpha}\in \widetilde{E}_{\Delta,m}^{\circ\circ}$ puis que l'application canonique $$\widetilde{E}_{\alpha} \rightarrow \widetilde{E}_{\Delta} \rightarrow \widetilde{E}_{\Delta,m}$$ est continue\footnote{N'oublions pas que $\varpi_{\alpha}$ est une pseudo-uniformisante pour $\widetilde{E}_{\alpha}$. Attention au fait que la première application de cette composée n'est pas continue puisque $\varpi_{\alpha}$ n'est pas topologiquement nilpotente dans $\widetilde{E}_{\Delta}$.}. L'image de $\widetilde{E}_{\alpha}^+$ est contenue dans $\widetilde{E}_{\Delta}^+$ puis dans $\widetilde{E}_{\Delta,m}^+$. Nous avons construit un morphisme $U_m \rightarrow \spa{\widetilde{E}_{\alpha}}{\widetilde{E}_{\alpha}^+}$. En passant à la colimite, nous obtenons un morphisme $g_{\alpha} \, : \, Y_{\widetilde{E}} \rightarrow \spd{\widetilde{E}_{\alpha}}{\widetilde{E}_{\alpha}^+}$.

		Nous prouvons l'isomorphisme en construisant une naturelle au niveau des $\spa{A}{A^+}$-points des faisceaux sur $\mathrm{Perf}$ $X_{\widetilde{E}}$ et $Y_{\widetilde{E}}$. À $f\in \mathrm{Hom}_{\mathrm{Perf}}(\spa{A}{A^+},Y_{\widetilde{E}})$, on associe la famille des $(g_{\alpha} \circ f)_{\alpha} \in X_{\widetilde{E}}(\spa{A}{A^+})$. Réciproquement, pour une famille $(f_{\alpha})_{\alpha}$, l'image de chaque $\varpi_{\alpha}$ appartient à $A^{\circ \circ}$ ce qui permet de compléter $(\underline{\varpi})$-adiquement et localiser $\otimes f_{\alpha}$ en un morphisme $f \, : \,\spa{A}{A^+} \rightarrow \spa{\widetilde{E}_{\Delta}}{\widetilde{E}_{\Delta}^+}$. Or,
		
		$$\forall |\cdot| \in \spa{A}{A^+}, \,\alpha \in \Delta, \,\,\, f(|\cdot|)\big(\varpi_{\alpha}\big)^n= |f_{\alpha}(\varpi_{\alpha})|^n \rightarrow 0$$ puisque $f_{\alpha}(\varpi_{\alpha})\in A^{\circ\circ}$. Ainsi, $f$ se factorise par $Y_{\widetilde{E}}$. Nous laissons les lecteurs et lectrices vérifier que les deux applications tout juste définies sont inverses l'une de l'autre.
		\medskip
		
		L'action sur $\spa{\widetilde{E}_{\Delta}}{\widetilde{E}_{\Delta}^+}$ se retreint-corestreint à $Y_{\widetilde{E}}$ puisque $(\varphi_{\alpha,q}^{n_{\alpha}})_{\alpha}(U_m) \subset U_{m+ q \max |n_{\alpha}-n_{\beta}|}$. L'action de $\varphi_{\alpha,q}$ s'identifie à celle sur $X_{\widetilde{E}}$ : cette dernière est caractérisée par la commutation des diagrammes
		\begin{center}
			\begin{tikzcd}
				\spd{\widetilde{E}_{\beta}}{\widetilde{E}_{\beta}^+} \ar[rrr,"\varphi_{q} \,\,  \mathrm{si} \,\, \beta=\alpha \,\,\mathrm{et} \,\, \mathrm{Id} \,\, \mathrm{sinon}"] & & & \spd{\widetilde{E}_{\beta}}{\widetilde{E}_{\beta}^+} \\
				X_{\widetilde{E}} \ar[rrr,"\varphi_{\alpha,q}"] \ar[u] & & & X_{\widetilde{E}} \ar[u]
			\end{tikzcd}
		\end{center}
	\end{proof}
	
	\begin{rem}
		Le lemme précédent retrouve, dans le cas de $X_{\widetilde{E}}$, qu'un produit d'affinoïdes est un diamant union strictement croissante d'affinoïdes perfectoïdes.
	\end{rem}
	
		\begin{lemma}\label{lemme_géométrique}
		Considérons les morphismes d'anneaux
		
		$$\widetilde{E}_{\Delta}\rightarrow H^0\bigg(Y_{\widetilde{E}}, \mathcal{O}_{\spa{\widetilde{E}_{\Delta}}{\widetilde{E}_{\Delta}^+}}\bigg)$$ et $$\widetilde{E}_{\Delta}^+\rightarrow H^0\bigg(Y_{\widetilde{E}}, \mathcal{O}^+_{\spa{\widetilde{E}_{\Delta}}{\widetilde{E}_{\Delta}^+}}\bigg).$$
		
		\noindent Le deuxième est injectif et $\varpi_{\Delta}$-presque surjectif. En particulier, le premier morphisme est injectif.
	\end{lemma}
	\begin{proof}
		Voir \cite[Prop. 4.27]{zabradi_kedlaya_carter}.
	\end{proof}

	\begin{rem}
		Les deux derniers lemmes terminent de justifier notre définition de $\widetilde{E}_{\Delta}$ et sa topologie.  Le Lemme \ref{X_et_Y} fonctionnerait pour la complétion $\varpi_{\Delta}$-adique du produit tensoriel à la place de $\widetilde{E}_{\Delta}^+$. Cela signifie que notre espace $X_{\widetilde{E}}$ possède deux épaissements affinoïdes, le plus petit étant $\mathrm{Spa}(\widetilde{E}_{\Delta},\widetilde{E}_{\Delta}^+)$. Les points supplémentaires de l'autre épaississement correspondent par exemple aux valuations $\varpi_{\alpha}$-adiques. Il semblerait raisonnable de les considérer mais ce second épaississement est un peu trop gros et le Lemme \ref{lemme_géométrique} ne fonctionnerait plus. Il est donc important de compléter $(\underline{\varpi})$-adiquement dans la définition de $\widetilde{E}_{\Delta}$.
		
		En revanche, comme nous l'avions déjà évoqué à la Remarque \ref{choix_topo_completion}, équiper $\widetilde{E}_{\Delta}^+$ de sa topologie $\varpi_{\Delta}$-adique est primordial pour lui associer un espace perfectoïde.
	\end{rem}
	
	\begin{defi}
		Grâce à la Proposition \ref{propstructurefetphiq}, le produit fibré $$Z_D:=Y_{\widetilde{E}} \times_{\spd{\widetilde{E}_{\Delta}}{\widetilde{E}_{\Delta}^+}} \spd{S_D}{S_D^+}$$ est un objet de $\fet{Y_{\widetilde{E}}\,||\,\Phi_{\Delta,q}}$. Soit $\overline{y}$ un point géométrique de $Y_{\widetilde{E}}$. Le lemme de Drinfeld équipe l'ensemble fini $\mathrm{V}_{\widetilde{E}}(D) := \big| Z_D \times_{Y_{\widetilde{E}}} \overline{y}\big|$ d'une action de $\cg_{\widetilde{E},\Delta}$.
	\end{defi}
	
	\begin{lemma}
		Le diamant $Z_D$ possède une structure canonique d'objet en $\mathbb{F}_r$-espaces vectoriels dans \linebreak$\fet{X_{\widetilde{E}}\,||\,\Phi_{\Delta,q'}}$. L'ensemble $\mathrm{V}_{\widetilde{E}}(D)$ appartient canoniquement à $\mathrm{Rep}_{\mathbb{F}_r} \cg_{\widetilde{E},\Delta}$. 
	\end{lemma}
	\begin{proof}
		Chaque $\big(T\otimes_{\widetilde{F}_{\Delta}} D\big)^{\varphi_{\Delta,r}=\mathrm{Id}}$ est muni d'une structure de $\mathbb{F}_r$-espace vectoriel naturelle en $T$.  Cela munit $S_D$ d'une structure naturelle de $\widetilde{E}_{\Delta}$-algèbre en $\mathbb{F}_r$-espaces vectoriels. Puisque $\mathbb{F}_r \subset \widetilde{E}_{\Delta}^{\Phi_{\Delta,q}}$, l'action de $\Phi_{\Delta,q}$ est $\mathbb{F}_r$-linéaire  ce qui fait de $\spd{S_D,S_D^+}$ un objet de $\fet{\spd{\widetilde{E}_{\Delta}}{\widetilde{E}_{\Delta}^+}\,||\,\Phi_{\Delta,q}}$ en $\mathbb{F}_r$-espaces vectoriels. Nous en déduisons\footnote{Considérer plutôt qu'une structure d'objet en espaces vectoriels est un famille de morphismes avec conditions pour le déduire.} la structure sur $Z_D$ par restriction à $Y_{\widetilde{E}}$.
		
		L'équivalence de catégorie entre $\fet{X_{\widetilde{E}}\,||\,\Phi_{\Delta,q}}$ et les $\cg_{\widetilde{E},\Delta}$-ensembles finis conclut pour $\mathrm{V}_{\widetilde{E}}(D)$.
	\end{proof}

	\begin{prop}\label{coeur_preuve_perf}
		Supposons que $\mathrm{V}_{\widetilde{E}}(D)$ est isomorphe à la représentation triviale $\mathbb{F}_r^d$. Alors $D$ est isomorphe à $\widetilde{E}_{\Delta}^d$ dans $\cdetaleproj{\Phi_{\Delta,q,r}^{\gp}}{\widetilde{E}_{\Delta}}$.
	\end{prop}
	\begin{proof}
		\underline{Étape 1 :} soit $S_{\mathrm{base}}$ la $\widetilde{E}_{\Delta}$-algèbre finie étale donnée par $\sfrac{\widetilde{E}_{\Delta}[X]}{(X^r-X)}$ avec l'action semi-\linebreak-linéaire de $\Phi_{\Delta,q}$ qui laisse fixe la classe de $X$. On le voit ainsi comme objet de $\fet{\spd{\widetilde{E}_{\Delta}}{\widetilde{E}_{\Delta}^+}\,||\,\Phi_{\Delta,q}}$ sur $\spd{S_{\mathrm{base}}}{S_{\mathrm{base}}^+}$. Nous notons $Z_{\mathrm{base}}$ sa restriction à $Y_{\widetilde{E}}$ qui est encore un objet de $\fet{Y_{\widetilde{E}}\,||\,\Phi_{\Delta,q}}$. Remarquons que $S_{\mathrm{base}}$ est l'algèbre finie étale associée à $\widetilde{E}_{\Delta}\in \cdetaleproj{\Phi_{\Delta,q,r}^{\gp}}{\widetilde{E}_{\Delta}}$. 
		
		Montrons que $\spd{S_{\mathrm{base}}}{S_{\mathrm{base}}^+}$, vu comme faisceau en $\mathbb{F}_r$-espaces vectoriels sur $\fet{\spd{\widetilde{E}_{\Delta}}{\widetilde{E}_{\Delta}^+}\,||\,\Phi_{\Delta,q}}$ est le faisceau constant associé à $\mathbb{F}_r$. Soit $Z$ un objet connexe de $\fet{\spd{\widetilde{E}_{\Delta}}{\widetilde{E}_{\Delta}^+}\,||\,\Phi_{\Delta,q}}$. Un morphisme vers $\spd{S_{\mathrm{base}}}{S_{\mathrm{base}}^+}$ correspond à une section globale $\Phi_{\Delta,q}$-invariante $z$ de $Z$ telle que $\prod_{a \in \mathbb{F}_r} (z-a)=0$. Les sections $(z-a)$ engendrent deux à deux le faisceau d'idéaux structural $\mathcal{O}_Z$ ; leur lieux d'annulation sont donc disjoints. De plus, leurs lieux d'annulations sont des fermés $\Phi_{\Delta,q}$-stables. Par connexité, on en déduit que $z$ appartient à $\mathbb{F}_r$. La décomposition des objets en union disjointe de composantes $\underline{\Gamma}$-connexes montre qu'un morphisme vers $\spd{S_{\mathrm{base}}}{S_{\mathrm{base}}^+}$ correspond à une fonction à valeurs dans $\mathbb{F}_r$ constante sur les composantes $\underline{\Gamma}$-connexes. Par tiré en arrière,  $Z_{\mathrm{base}}$ est encore le faisceau constant associé à $\mathbb{F}_r$ et le Lemme \ref{faisceau_constants} affirme alors que la représentation de $\cg_{\widetilde{E},\Delta}$ obtenue grâce au lemme de Drinfeld est la représentation triviale sur $\mathbb{F}_r$.
		
		\medskip
		
		\underline{Étape 2 :} dans cette étape, nous cherchons à décrire $S_D$. 
		
		L'étape 1 montre que l'objet $\sqcup_d \, Z_{\mathrm{base}}$ dans la catégorie galoisienne $\fet{Y_{\widetilde{E}}\,||\,\Phi_{\Delta,q}}$ à la représentation triviale $\mathbb{F}_r^d$. Ainsi, nous obtenons $Z_D\cong \sqcup_d \, Z_{\mathrm{base}}$ dans cette catégorie galoisienne. Bien que le terme de droite coïncide avec $Z_{\widetilde{F}_{\Delta}^d}$, tout notre travail consiste à remonter cette identification à $S_D$ puis à $D$. Nous venons de calculer $\spd{S_D}{S_D^+}_{| Y_{\widetilde{E}}}$. Nous allons en déduire un calcul sur l'épaississement $\spa{\widetilde{E}_{\Delta}}{\widetilde{E}_{\Delta}^+}$. Il s'agit de la partie la plus technique de la preuve.
		
		\medskip
		Puisque $S_D$ est finie étale, elle est plate et de présentation finie d'où l'existence d'une présentation comme module fini projectif : on se fixe une telle présentation sous la forme d'un isomorphisme de $\widetilde{E}_{\Delta}$-module 
$$S_D \oplus M = \bigoplus_{j\leq N} \widetilde{E}_{\Delta} e_j.$$  En localisant et complétant\footnote{Ne pas oublier que la topologie sur $S_D$ est la topologie initiale sur un module fini projectif.}, on obtient pour tout ouvert affine $V$ de $\spa{\widetilde{E}_{\Delta}}{\widetilde{E}_{\Delta}^+}$ un diagramme commutatif aux lignes exactes

\begin{center}
	\begin{tikzcd}
		0 \ar[r] & S_D \ar[r] \ar[d] & \bigoplus_{j\leq N}  \widetilde{E}_{\Delta} e_j \ar[d] \ar[r] & M \ar[r] \ar[d] &0  \\
		0 \ar[r] & H^0\left(V,f_*\mathcal{O}_{\spa{S_D}{S_D^+}}\right)\ar[r] & \bigoplus_{j \leq N} H^0\left(V,f_*\mathcal{O}_{\spa{\widetilde{E}_{\Delta}}{\widetilde{E}_{\Delta}^+}}\right) e_j \ar[r]&  H^0(V,\widetilde{M}) 
	\end{tikzcd}
\end{center} où $f$ désignera toujours le morphisme structural vers $\spa{\widetilde{E}_{\Delta}}{\widetilde{E}_{\Delta}^+}$. En passant à la limite ce diagramme sur la famille des ouverts $U_m$, on obtient

\begin{center}
	\begin{tikzcd}
		0 \ar[r] & S_D \ar[r] \ar[d] & \bigoplus_{j\leq N}  \widetilde{E}_{\Delta} e_j \ar[d] \ar[r] & M \ar[r] \ar[d] &0  \\
		0 \ar[r] & H^0(Y_{\widetilde{E}},f_*\mathcal{O}_{Z_D}) \ar[r] & \bigoplus_{j \leq N} H^0\left(Y_{\widetilde{E}},\mathcal{O}_{\spa{\widetilde{E}_{\Delta}}{\widetilde{E}_{\Delta}^+}}\right) e_j \ar[r] & H^0(Y_{\widetilde{E}},\widetilde{M})
	\end{tikzcd}
\end{center} Le Lemme \ref{lemme_géométrique} affirmant que le morphisme vertical central est injectif donc celui de gauche aussi. Dans la suite de la preuve, nous identifions les quatre modules du carré de gauche à des sous-modules les uns des autres. En inversant les rôles de $S_D$ et $M$, le morphisme vertical de droite est également injectif. Le lemme du serpent affirme alors qu'un élément de $\lim \limits_{\longleftarrow} S_{D,m}$ appartient à $S_D$ si et seulement si ses coordonnées dans la base $e_j$ appartiennent à $\widetilde{E}_{\Delta}$.

		\medskip
		Soit $x$ une section $\Phi_{\Delta,q}$-invariante de $H^0(Y_{\widetilde{E}},f_*\mathcal{O}_{Z_D})$ et prouvons qu'elle appartient à $S_D$.

Soit $\alpha\in \Delta$. Par définition de $Y_{\widetilde{E}}$, pour tout $|\cdot|\in Y_{\widetilde{E}}$, $\exists n, \,\, \forall \beta, \,\, |\varpi_{\beta}|^{q^n}\leq |\varpi_{\alpha}|$. Autrement dit, $\prod_{\beta\neq\alpha} \varphi_{\beta,q'}^n$ envoie $|\cdot|$ dans l'ouvert rationnel $U=\spa{\widetilde{F}_{\Delta}}{\widetilde{F}_{\Delta}^+}\big(\frac{\varpi_{\alpha}^2,\, \varpi_{\beta} \, |\, \beta\neq \Delta}{\varpi_{\alpha}}\big)$. Ce dernier est quasi-compact et contient donc un domaine fondamental. Sur le domaine $U$, l'écriture $$\varpi_{\alpha}^{|\Delta|}=\varpi_{\Delta} \times \prod_{\beta\neq \alpha} \quot{\varpi_{\beta}}{\varpi_{\alpha}}$$ illustre que $\varpi_{\alpha}$ est une pseudo-uniformisante. Nous fixons alors un entier $k$ tel que $$x_{|U} \in \varpi_{\alpha}^{-k} \sum H^0\left(U,f_* \mathcal{O}^+_{\spa{S_D}{S_D^+}}\right).$$ L'action du Frobenius absolu est triviale sur les points de $Y_{\widetilde{E}}$ ; l'espace $Y_{\widetilde{E}}$ est recouvert par les $\prod_{\beta\neq \alpha} \varphi_{\beta,q}^{n_{\beta}}(U)$. Prenons $\psi \in \Phi_{\Delta \backslash \{\alpha\},q}$. L'invariance de $x$ donne $$x_{|\psi(U)}=(\psi(x))_{|\psi(U)}=\psi(x_{|U})$$ ce qui entraîne que 
		\begin{align*}
	x_{\psi(U)} &\in \psi\bigg( \varpi_{\alpha}^{-k} H^0(U,f_*\mathcal{O}^+_{\spa{S_D}{S_D^+}})\bigg) \\
	& \subseteq \varpi_{\alpha}^{-k} \sum H^0\left(\psi(U),f_*\mathcal{O}^+_{\spa{S_D}{S_D^+}}\right) \\
\end{align*} où le passage à la deuxième ligne utilise le fait que $\varphi_{\alpha,q}$ n'apparaît pas dans $\psi$ et se souvient que $\Phi_{\Delta,q,r}$ agit sur $\spa{S_D}{S_D^+}$ dans la catégorie des $\spa{\widetilde{E}_{\Delta}}{\widetilde{E}_{\Delta}^+}$-espaces perfectoïdes (voir la discussion au début de la Proposition \ref{propstructurefetphiq}). Comme la topologie sur $S_D$ est la topologie initiale, il se trouve que $S_D \cap (\sum_j \widetilde{E}_{\Delta}^+ e_j)$ est ouvert et nous savons également que $S_D^+$ est borné ; nous pouvons fixer $m$ tel que $$S_D^+ \subset \sum_j \varpi_{\Delta}^{-m} \widetilde{E}_{\Delta}^+ e_j.$$  Pour tout ouvert affine $V$ de $\spa{\widetilde{E}_{\Delta}}{\widetilde{E}_{\Delta}^+}$, nous avons encore $$H^0\left(V,f_*\mathcal{O}^+_{\spa{S_D}{S_D^+}}\right) \subset \sum_j \varpi_{\Delta}^{-m} H^0\left(V,\mathcal{O}^+_{\spa{\widetilde{E}_{\Delta}}{\widetilde{E}_{\Delta}^+}}\right) e_j.$$ Ainsi, les coordonnées de $x$ dans la base $(e_j)$ appartiennent à $\varpi_{\alpha}^{-k} \varpi_{\Delta}^{-m} H^0\left(\psi(U),\mathcal{O}^+_{\spa{\widetilde{E}_{\Delta}}{\widetilde{E}_{\Delta}^+}}\right)$ pour tout $\psi(U)$. En recollant, les coordonnées de $x$ appartiennent donc à $$\varpi_{\Delta}^{-(k+m)}H^0\left(Y_{\widetilde{E}},\mathcal{O}^+_{\spa{\widetilde{E}_{\Delta}}{\widetilde{E}_{\Delta}^+}}\right)\subset \varpi_{\Delta}^{-(k+m+1)} \widetilde{E}_{\Delta}^+$$ grâce au lemme \ref{lemme_géométrique}. Les coordonnées sont donc dans $\widetilde{E}_{\Delta}$, ce qui conclut.
		
	Nous savons que $$Z_D \cong \bigsqcup_d \, Z_{\mathrm{base}}=Y_{\widetilde{E}} \times_{\spa{\widetilde{E}_{\Delta}}{\widetilde{E}_{\Delta}^+}} \spa{\prod_{1\leq i \leq d} \sfrac{\widetilde{E}_{\Delta}[X_i]}{(X_i^r-X_i)}}{\prod_{1\leq i \leq d} \sfrac{\widetilde{E}_{\Delta}[X_i]}{(X_i^r-X_i)}^+}.$$ Les $X_i$ s'interprètent comme des sections globales $\Phi_{\Delta,q,r}^{\gp}$-invariantes de $Z_D$. Tous les monômes en $X_i$ sont des sections globales et $\Phi_{\Delta,q,r}^{\gp}$-invariantes de $Z_D$  et appartiennent donc à $S_D$. Elles vérifient les mêmes relations ce qui permet de déduire un isomorphisme de $\widetilde{E}_{\Delta}$-algèbres

	$$S_D \cong \prod_{1\leq i \leq d} S_{\mathrm{base}}.$$
		
		\underline{Étape 3 :} construisons une bijection de $\mathrm{V}_{\widetilde{E}}(D)$ dans $D^{\Phi_{\Delta,q,r}}$.
		
		Grâce à l'isomorphisme de l'étape 2, il est possible d'écrire la suite de bijections suivantes :
		\begin{align*}
			\mathrm{V}_{\widetilde{E}}(D)& \cong \mathbb{F}_r^d \\
			& \cong \mathrm{Hom}_{\widetilde{E}_{\Delta}\text{-}\mathrm{Alg}} \bigg(\prod_{1\leq i \leq d} S_{\mathrm{base}},\widetilde{E}_{\Delta}\bigg) \\
			& \cong \mathrm{Hom}_{\widetilde{E}_{\Delta}\text{-}\mathrm{Alg}}(S_D,\widetilde{E}_{\Delta}) \\
			&\cong D^{\varphi_{\Delta,r}=\mathrm{Id}}
		\end{align*} Nous construisons le diagramme commutatif de $\mathbb{F}_r$-espaces vectoriels suivant :
		\begin{center}
			\begin{tikzcd}[column sep=small]
				D^{\varphi_{\Delta,r}=\mathrm{Id}} \ar[r,"\mathbf{1}\otimes \mathrm{Id}_{\widetilde{E}_{\Delta}}\otimes\varphi_{\alpha,q,D}"] \ar[d] &  \left(\varphi_{\alpha,q}^*\widetilde{E}_{\Delta} \otimes_{\widetilde{E}_{\Delta}} D\right)^{\varphi_{\Delta,r}=\mathrm{Id}} \ar[r]\ar[d] & D^{\varphi_{\Delta,r}=\mathrm{Id}} \ar[d] \\
				\mathrm{Hom}_{\widetilde{E}_{\Delta}\text{-}\mathrm{Alg}} \Big(\underset{1\leq i\leq d}{\prod} S_{\mathrm{base}},\widetilde{E}_{\Delta}\Big) \arrow{r}[yshift=-1.5ex,swap]{\varphi_{\alpha,q}^*(\cdot)\circ \psi_{\alpha,q,\mathrm{base}}} & \mathrm{Hom}_{\widetilde{E}_{\Delta}\text{-}\mathrm{Alg}} \left(\underset{1\leq i\leq d}{\prod} S_{\mathrm{base}},\varphi_{\alpha,'}^*\widetilde{E}_{\Delta}\right) \ar[r] & \mathrm{Hom}_{\widetilde{E}_{\Delta}\text{-}\mathrm{Alg}} \left(\underset{1\leq i\leq d}{\prod} S_{\mathrm{base}},\widetilde{E}_{\Delta}\right)
			\end{tikzcd}
		\end{center} Le carré de gauche est issu des définitions à la Proposition \ref{propstructurefetphiq}. Le deuxième carré est donné par Yoneda. L'image $d$ par la ligne supérieure devient $\varphi_{\alpha,q,D}(d)$. L'élément $d$ est envoyé en bas à gauche sur un morphisme $g$ tel que $g(X_i)\in \mathbb{F}_r$. Une application de cette forme est laissé fixe par la ligne inférieure. Ainsi, pour notre module, nous avons $D^{\Phi_{\Delta,q,r}}=D^{\varphi_{\Delta,r}=\mathrm{Id}}$.
		
		\medskip
		
		\underline{Étape 4 :} nous avons construit un isomorphisme $\mathbb{F}_r$-linéaire $\mathbb{F}_r^d \cong D^{\Phi_{\Delta,q,r}}$. En tensorisant par $\widetilde{E}_{\Delta}$, nous obtenons un isomorphisme $\widetilde{E}_{\Delta}$-linéaire $$\widetilde{E}_{\Delta}^d \cong \widetilde{E}_{\Delta}\otimes_{\mathbb{F}_r}D^{\Phi_{\Delta,q,r}}$$ Sa post-composition par $\widetilde{E}_{\Delta}\otimes_{\mathbb{F}_r} D^{\Phi_{\Delta,q,r}} \rightarrow D$ est un morphisme dans $\dmod{\Phi_{\Delta,q,r}}{\widetilde{E}_{\Delta}}$ dont nous voulons prouver qu'il s'agit d'un isomorphisme.
		
		Prouvons que $\widetilde{E}_{\Delta}\otimes_{\mathbb{F}_r} D^{\Phi_{\Delta,q,r}} \rightarrow D$ est injectif. Puisque $\widetilde{E}_{\Delta}$ est intègre, cette injectivité se vérifie après tensorisation par $\mathrm{Frac}(\widetilde{E}_{\Delta})$. Soit alors $\sum_{1\leq i \leq k} x_i\otimes d_i$ dans le noyau de ce dernier morphisme, dont on suppose le nombre de termes non nuls minimal. Quitte à diviser par $x_1$, nous pouvons supposer que $x_1=1$. L'identité $\sum x_i d_i=0$ se transforme par application de $\varphi_{\Delta,r}$ en $d_1+\sum_{2\leq i \leq k} x_i^r d_i=0$. En soustrayant, ceci implique que $\sum_{2\leq i \leq k} (x_i-x_i^r))\otimes d_i$ est également dans le noyau. Par hypothèse de minimalité, tous les $x_i$ valent $x_i^r$. Puisque $\mathrm{Frac}(\widetilde{E}_{\Delta})$ est un corps, ceci implique que tous les $x_i$ sont des éléments de $\mathbb{F}_r$. La relation $\sum_{1\leq i\leq k} x_id_i=0$ originelle s'avère être une relation dans $D^{\Phi_{\Delta,q,r}}$, soit $\sum_{1\leq i \leq k} x_i \otimes d_i=0$.
		
		Le résultat du paragraphe précédent est un morphisme injectif dans $\detaleproj{\Phi_{\Delta,q',r}}{\widetilde{E}_{\Delta}}$ que nous notons $i \, : \, \widetilde{E}_{\Delta}^d \hookrightarrow D$. L'identification de $S_D$ à l'étape $2$ affirme de plus que $d$ est la dimension locale de $D$. Soit $z\in \widetilde{E}_{\Delta}$ tel que $D[z^{-1}]$ est libre de rang $d$. L'action de $\varphi_{\Delta,r}$ s'étend à $\widetilde{E}_{\Delta}[z^{-1}]$. Par extension des scalaires comme en \cite[Prop. 3.3]{formalisme_marquis}, nous obtenons un morphisme injectif dans $\detaleproj{\varphi_{\Delta,r}^{\mathbb{N}}}{\widetilde{E}_{\Delta}[z^{-1}]}$ de la forme $$i_z \, : \, \left(\widetilde{E}_{\Delta}[z^{-1}]\right)^d \hookrightarrow D[z^{-1}].$$ Posons $\mathcal{B}$ une base de $D[z^{-1}]$ et $A\in \mathrm{M}_d(\widetilde{E}_{\Delta}[z^{-1}])$ la matrice de l'image de la base canonique par $i_z$ dans $\mathcal{B}$. Soit également $B$ la matrice de $\varphi_{\Delta,r,D[z^{-1}]}$ dans la base $\mathcal{B}$. Puisque $i_z$ est $\varphi_{\Delta,r}$-équivariante, son image est stable par $\varphi_{\Delta,r,D[z^{-1}]}$ et l'image de la base canonique est formé d'éléments invariants. Matriciellement, cela se traduit par $B\varphi(A)=A$ dans $\mathrm{M}_d(\widetilde{E}_{\Delta})$. En termes de déterminant, cela implique que $\det(B) \det(A)^r=\det(A)$. Puisque $i_z$ est injective, son déterminant est non nul. Par intégrité de $\widetilde{E}_{\Delta}$, on en déduit que $\det(A)$ est inversible dans $\widetilde{F}_{\Delta}$, d'inverse $\det(B) \det(A)^{r-2}$. On en déduit que l'image de $i_z$ engendre $D[z^{-1}]$. Autrement dit que $i_z$ est un isomorphisme. Par conséquent, le morphisme $i$ est un isomorphisme ce qui conclut.
	\end{proof}
	
	Nous démontrons à présent le théorème au centre de toutes les équivalences de cet article. Pour ce faire, nous utilisons la discussion qui précède pour toutes les extensions finies de $\widetilde{E}$.
	
	\begin{theo}\label{morph_comparaison_coeur}
		Pour tout objet $D$ de $\detaleproj{\Phi_{\Delta,q,r}^{\gp}}{\widetilde{E}_{\Delta}}$, le morphisme de comparaison $$\widetilde{E}_{\Delta}^{\mathrm{sep}}\otimes_{\mathbb{F}_r}\widetilde{\mathbb{V}}_{\Delta}(D) \rightarrow \widetilde{E}_{\Delta}^{\mathrm{sep}} \otimes_{\widetilde{E}_{\Delta}} D$$ est un isomorphisme.
	\end{theo}
	
	\begin{proof}
		Pour cette preuve, nous utilisons encore la notation $\mathrm{V}_{\widetilde{E}}$ puisqu'il nous faut absolument garder en mémoire le corps de base sur lequel nous travaillons. Pour toute extension finie $\widetilde{F}|\widetilde{E}$ et $q'$ associé, nous appelons dans cette preuve $\mathrm{Ex}_{\widetilde{F}}$ le foncteur défini en \cite[Déf. 3.2]{formalisme_marquis} pour le groupe $\Phi_{\Delta,q',r}^{\gp}$ et le morphisme d'anneaux $\widetilde{E}_{\Delta}\rightarrow \widetilde{F}_{\Delta}$ du Lemme \ref{description_construction_produit}. Nous appelons également dans cette preuve $\mathrm{Ex}_{\widetilde{F},q}$ l'analogue pour le groupe $\Phi_{\Delta,q,r}^{\gp}$ et le morphisme d'anneaux $\widetilde{E}_{\Delta}\rightarrow \widetilde{F}_{\Delta,q}$.
		
		\underline{Étape 1 :} soit $\widetilde{F}$ une extension finie de $\widetilde{E}$. Nous commençons par montrer que l'extension des scalaires $\mathrm{Ex}_{\widetilde{F}}$ correspond à un foncteur d'oubli. Plus précisément, soit $D$ un objet de $\detaleproj{\Phi_{\Delta,q,r}^{\gp}}{\widetilde{E}_{\Delta}}$. Montrons que $$\mathrm{V}_{\widetilde{F}}\big(\mathrm{Ex}_{\widetilde{F}}(D)\big) \cong \mathrm{V}_{\widetilde{E}}(D)_{|\cg_{\widetilde{F},\Delta}}.$$
		
		Pour toute $\widetilde{F}_{\Delta}$-algèbre $T$, nous avons la suite de bijections naturelles suivante :
		\begin{align*}
			\big( T\otimes_{\widetilde{F}_{\Delta}}\mathrm{Ex}_{\widetilde{F}}(D)\big)^{\varphi_{\Delta,r}=\mathrm{Id}} &\cong \big(T\otimes_{\widetilde{E}_{\Delta}} D\big)^{\varphi_{\Delta,r}=\mathrm{Id}} \\
			&\cong \mathrm{Hom}_{\widetilde{E}_{\Delta}\text{-}\mathrm{Alg}}(S_D,T) \\
			&\cong \mathrm{Hom}_{\widetilde{F}_{\Delta}-\mathrm{Alg}}\big( \widetilde{F}_{\Delta} \otimes_{\widetilde{E}_{\Delta}} S_D,T\big)
		\end{align*} Ainsi, la $\widetilde{F}_{\Delta}$-algèbre finie étale $S_{\mathrm{Ex}_{\widetilde{F}}(D)}$ s'identifie $\widetilde{F}_{\Delta}\otimes_{\widetilde{E}_{\Delta}} S_D$. Cette identification est compatible à l'action de $\Phi_{\Delta,q'}^{\gp}$ et à la structure de $\mathbb{F}_r$-espace vectoriel. En restreignant à $Y_{\widetilde{F}}$, on obtient \begin{align*} Z_{\mathrm{Ex}_{\widetilde{F}}(D)} & = Y_{\widetilde{F}} \times_{Y_{\widetilde{E}}} \bigg(Y_{\widetilde{E}}\times_{\spa{\widetilde{E}_{\Delta}}{\widetilde{E}_{\Delta}^+}} \spa{S_D}{S_D^+}\bigg) \\
		&= Y_{\widetilde{F}} \times_{Y_{\widetilde{E}}} Z_D
		\end{align*} en tant qu'objet en $\mathbb{F}_r$-espaces vectoriels de $\fet{Y_{\widetilde{F}}\,|\, \Phi_{\Delta,q'}}$. La Proposition \ref{changement_base_piun} conclut que $$\mathrm{V}_{\widetilde{F}}(\mathrm{Ex}_{\widetilde{F}}(D))\cong \mathrm{V}_{\widetilde{E}}(D)_{| \cg_{\widetilde{F},\Delta}}.$$
		
		\underline{Étape 2 :} prenons toujours $D$ un objet de $\detaleproj{\Phi_{\Delta,q,r}^{\gp}}{\widetilde{E}_{\Delta}}$ et fixons $d$ son rang. Considérons une extension finie $\widetilde{F}|\widetilde{E}$ telle que l'action de $\cg_{\widetilde{F},\Delta}$ de la $\cg_{\widetilde{E},\Delta}$-représentation lisse $\mathrm{V}_{\widetilde{E}}(D)$ est triviale. D'après la première étape, cela signifie que la représentation $\mathrm{V}_{\widetilde{F}}(\mathrm{Ex}_{\widetilde{F}}(D))$ est triviale. Puisque $\mathrm{Ex}_{\widetilde{F}}$ préserve le rang, la première étape fournit un isomorphisme dans $\cdetaleproj{\Phi_{\Delta,q',r}^{\gp}}{\widetilde{F}_{\Delta}}$  entre $\widetilde{F}_{\Delta}^d$ et $\mathrm{Ex}_{\widetilde{F}}(D)$. 
		
		La Proposition \ref{identif_coindu_perf} démontre que 	$$\widetilde{F}_{\Delta,q} \cong \coindu{\Phi_{\Delta,q',r}^{\mathrm{gp}}}{\Phi_{\Delta,q,r}^{\mathrm{gp}}}{\widetilde{F}_{\Delta}}.$$ En appliquant \cite[Lem. 3.13]{formalisme_marquis} pour $R=\widetilde{F}_{\Delta,q}$, $T=\widetilde{F}_{\Delta}$, $\Phi_{\Delta,q',r}^{\gp}<\Phi_{\Delta,q,r}^{\gp}$ d'indice fini et $i$ la projection sur la coordonnée du neutre, nous obtenons un isomorphisme dans $\dmod{\Phi_{\Delta,q,r}^{\gp}}{\widetilde{F}_{\Delta,q}}$ 
		
		$$\mathrm{Ex}_{\widetilde{F},q}(D) \cong \coindu{\Phi_{\Delta,q',r}^{\gp}}{\Phi_{\Delta,q,r}^{\gp}}{\mathrm{Ex}_{\widetilde{F}}(D)},$$ puis une suite d'isomorphismes dans $\dmod{\Phi_{\Delta,q,r}^{\gp}}{\widetilde{F}_{\Delta,q}}$
		
		$$\mathrm{Ex}_{\widetilde{F},q}(D) \cong \coindu{\Phi_{\Delta,q',r}}{\Phi_{\Delta,q,r}}{\mathrm{Ex}_{\widetilde{F}}(D)} \cong \coindu{\Phi_{\Delta,q',r}}{\Phi_{\Delta,q,r}}{\widetilde{F}_{\Delta}^d} \cong \widetilde{F}_{\Delta,q}^d.$$

	En particulier, en étendant des scalaires à $\widetilde{E}_{\Delta}^{\sep}$, nous obtenons un isomorphisme dans $\dmod{\Phi_{\Delta,q,r}^{\gp}}{\widetilde{E}_{\Delta}^{\sep}}$ entre $\left(\widetilde{E}_{\Delta}^{\sep}\otimes_{\widetilde{E}_{\Delta}} D\right)$ et $(\widetilde{E}_{\Delta}^{\sep})^d$. Le caractère isomorphique du morphisme de comparaison découle\footnote{On démontre un isomorphisme uniquement algébriquement, puisque l'on ne s'est pas fatigués à s'occuper de l'action de $\cg_{\widetilde{E},\Delta}$. Comme le morphisme de comparaison est automatiquement équivariant et ne dépend pas ensemblistement de l'action galoisienne, cela suffit.} du Corollaire \ref{coro_descente_1}.
	\end{proof}

	\vspace{0.75cm}
	\subsection{Construction du foncteur pleinement fidèle $\mathbb{D}_{\Delta}$ modulo $p$}
	
	Dans cette sous-section et la suivante, nous cherchons à déperfectoïdiser notre équivalence. Nous fixons $E$ un corps de caractéristique $p$, de valuation discrète et complet pour cette valuation, tel que la clôture algébrique de $\mathbb{F}_p$ dans son corps résiduel $k$ est de cardinal fini $q$. En fixant $X$ une uniformisante, nous savons que $E$ est isomorphe à $k(\!(X)\!)$. Définissons également $\widetilde{E}$ le complété de la clôture radicielle de $E$. C'est un corps perfectoïde et nous pouvons le décrire par $$\widetilde{E}= \left(\bigcup_{n\geq 0} k\llbracket X^{q^{-n}} \rrbracket \right)^{\wedge X} \left[\frac{1}{X}\right]$$ que l'on nomme traditionnellement $k(\!(X^{q^{-\infty}})\!)$. Les théories de Galois de $E$ et de $\widetilde{E}$ étant identiques, nous cherchons à obtenir une équivalence de Fontaine pour $\mathrm{Rep}_{\mathbb{F}_r} \cg_{E,\Delta}$ à partir de l'équivalence perfectoïde précédente, tout en obtenant des anneaux de coefficients imparfaits plus aisément manipulables du côté des $\varphi$-modules.
	
	Nous nous plaçons dans le contexte de \S \ref{section_wddelta} pour le corps $\widetilde{E}$, avec $X$ comme choix de pseudo-uniformisante. En fixant $\widetilde{E}^{\sep}$, nous fixons aussi une clôture séparable $E^{\sep}$ de $E$. Nous imitons dans la suite les définitions de \S \ref{section_wddelta}, ce qui nous épargne  des preuves identiques.
	
	\begin{defi}\label{def_constr_prod_imparf}
	Soit $E^{\sep}|F|E$ une extension finie. Nous définissons $$F_{\Delta,q}^+:=\left(\bigotimes_{\alpha\in \Delta,\mathbb{F}_q} F_{\alpha}^+ \right)^{\wedge (\underline{X})}$$
		
		$$F_{\Delta,q}:= F_{\Delta,q}^+\left[\frac{1}{X_{\Delta}}\right].$$
		
		Trois topologies seront utilisées pour ces deux anneaux respectivement pour le considérer comme coefficients de catégories de $\varphi$-modules, comme coefficients d'une catégories de $(\varphi,\Gamma)$-modules et pour un raisonnement fin à la sous-section suivante. Sur $F_{\Delta,q}^+$ ces trois topologies sont respectivement la topologie discrète, la topologie \linebreak$X_{\Delta}$-adique et la topologie $(\underline{X})$-adique. Sur $F_{\Delta,q}$, ce sont la topologie discrète, la topologie d'anneau ayant pour base de voisinages de $0$ la famille $\left(X_{\Delta}^n F_{\Delta,q}^+\right)_{n\geq 0}$ que nous appelons \textit{topologie adique} et la topologie colimite des topologies $(\underline{X})$-adiques que nous appelons \textit{topologie colimite}.
		
		Le produit tensoriel des $F_{\alpha}^+$ est muni d'une structure de $\Phi_{\Delta,q,r}$-anneau topologique pour la topologie \linebreak$(\underline{X})$-adique, l'élément $\varphi_{\alpha,q}$ agissant par le $q$-Frobenius arithmétique sur $F_{\alpha}^+$ et l'identité sur les $F_{\beta}^+$ et l'élément $\varphi_{\Delta,r}$ agissant par le $r$-Frobenius arithmétique. En complétant, on obtient une structure de $\Phi_{\Delta,q,r}$-anneau topologique sur $F_{\Delta,q}^+$ pour chacune des trois topologies.  Après localisation, elle fournit une structure $\Phi_{\Delta,q,r}$-anneau topologique sur $F_{\Delta,q}$ pour chacune des trois topologies ci-dessus.
		
		Lorsque $E^{\sep}|F|E$ est finie galoisienne, l'action de $\cg_{E}$ sur $F^+$ est $\mathbb{F}_q$-linéaire, continue pour les topologies discrète et $X$-adique, et commute au Frobenius. L'action de $\cg_{E,\Delta}$ facteur par facteur sur le produit tensoriel des $F^+_{\alpha}$ se complète en une action sur $F^+_{\Delta,q}$ continue en particulier pour les topologies discrètes et adique, et commutant à l'action de $\Phi_{\Delta,q,r}$. Nous obtenons donc deux structure de $\left(\Phi_{\Delta,q,r} \times \cg_{E,\Delta}\right)$-anneau topologique sur $F_{\Delta,q}^+$. En localisant, nous obtenons deux structures de $\left(\Phi_{\Delta,q,r} \times \cg_{E,\Delta}\right)$-anneau topologique sur $F_{\Delta,q}$.
	\end{defi}
	
		\begin{rem}\label{ecriture_explicite_anneaux}
		Nous avions annoncé que dans le cadre de cette sous-section, nous pouvions décrire explicitement les anneaux perfectoïdes et les anneaux imparfaits. Commençons par les anneaux perfectoïdes. Nous décrivons $\widetilde{E}^+$ comme 		
		$$\widetilde{E}^+=\left\{ \sum_{d\in \mathbb{N}[q^{-1}]} a_d X^d \,\, \Bigg|\,\, \substack{(a_d)\in \left(\bigotimes_{\alpha \in \Delta,\,\mathbb{F}_q} k\right)^{\mathbb{N}[q^{-1}]} \,\, \text{telle que} \\ \forall r\in \mathbb{N}[q^{-1}], \,\, \left\{d\leq r\, |\, a_d \neq 0\right\} \,\, \text{est fini.}}\right\}.$$
		
		Nous obtenons que
		
		$$\widetilde{E}_{\Delta}^+=\left\{ \sum_{\underline{d}\in (\mathbb{N}[q^{-1}])^{\Delta}} a_{\underline{d}} \underline{X}^{\underline{d}} \,\, \Bigg|\, \,  \substack{(a_{\underline{d}})\in \left(\bigotimes_{\alpha \in \Delta,\,\mathbb{F}_q} k\right)^{\mathbb{N}[q^{-1}]^{\Delta}} \,\,\text{tel que} \\ \forall r\in \mathbb{N}[q^{-1}], \,\, \left\{\underline{d} \, |\, \sum d_{\alpha}\leq r, \,\, \mathrm{et}\,\, a_{\underline{d}} \neq 0\right\}\,\, \text{est fini}}\right\}.$$ Pour cet anneau perfectoïde, notons que les complétions $(\underline{X})$-adiques et $(X_{\Delta})$-adiques du produit tensoriel ne coïncident même pas algébriquement : à deux variables par exemple, l'élément $\sum_n X_{\alpha}^{q^{-n}} X_{\beta}^{q^n}$ appartient à la complétion $(X_{\alpha},X_{\beta})$-adique mais pas à la complétion $X_{\{\alpha,\beta\}}$-adique.
		
		Pour comprendre de manière graphique les complétions pour deux variables, nous représentons une série par une famille de points dans une grille labellisés par les coefficients correspondants. Par exemple, la représentation du polynôme $$Q=X_{\alpha}^{\sfrac{1}{q}} X_{\beta}-X_{\alpha} X_{\beta}^{2+\sfrac{1}{q}+\sfrac{1}{q^2}}-X_{\alpha}^{1+\sfrac{1}{q^2}} X_{\beta}^{1+\sfrac{1}{q^2}}+X_{\alpha}^{1+\sfrac{1}{q}+\sfrac{1}{q^2}} X_{\beta}^{\sfrac{1}{q^2}}+X_{\alpha}^{2+\sfrac{1}{q}} X_{\beta}^{3+\sfrac{1}{q}}$$ est illustré comme suit.

		\begin{figure}[h!]
			\centering
			
			\includegraphics[scale=0.3]{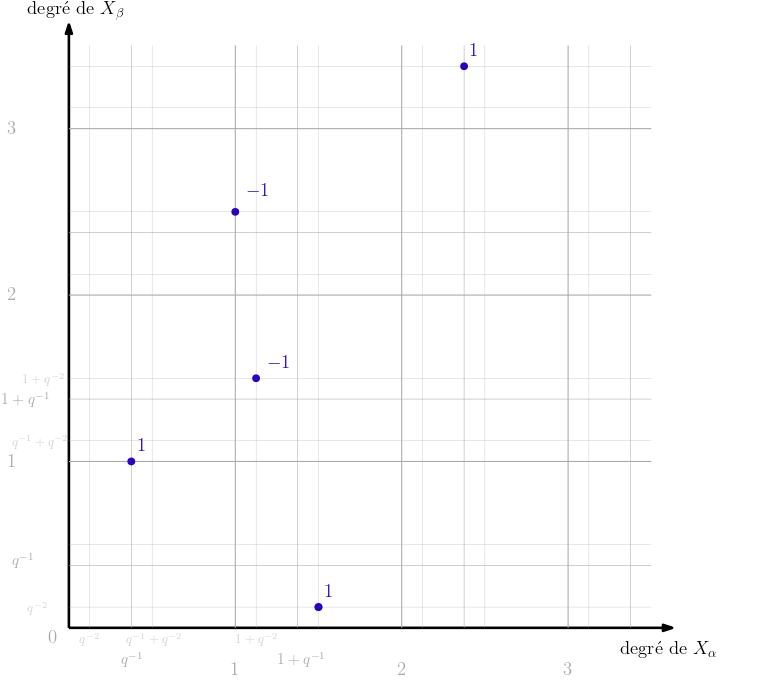} \caption{Représentation du polynôme $Q$ sur la grille pour $q=3$}
		\end{figure}
		
		%    \includegraphics{representation.png}
		%    \caption{Représentation du polynôme $X_{\alpha}^{\sfrac{1}{q}} X_{\beta}-X_{\alpha} X_{\beta}^{2+\sfrac{1}{q}+\sfrac{1}{q^2}}-X_{\alpha}^{1+\sfrac{1}{q^2}} X_{\beta}^{1+\sfrac{1}{q^2}}+X_{\alpha}^{1+\sfrac{1}{q}+\sfrac{1}{q^2}} X_{\beta}^{\sfrac{1}{q^2}}+X_{\alpha}^{2+\sfrac{1}{q}} X_{\beta}^{3+\sfrac{1}{q}}$}
		%    \end{subfigure}
	Un élément de la complétion $(X_{\alpha},X_{\beta})$-adique possède une représentation telle que son intersection avec toutes les zones triangulaires\footnote{Les zones triangulaires correspondent à des quotients par $I_{n,\Delta}$. De manière équivalente, nous pourrions considérer les quotients par $(\underline{X})^{n}$ et obtenir des zones carrées} coïncide avec la représentation d'un élément de $\widetilde{E}_{\alpha}^+\otimes_{\mathbb{F}_q} \widetilde{E}_{\beta}^+$.

	\begin{figure}[h!]
		\centering
		\includegraphics[scale=0.5]{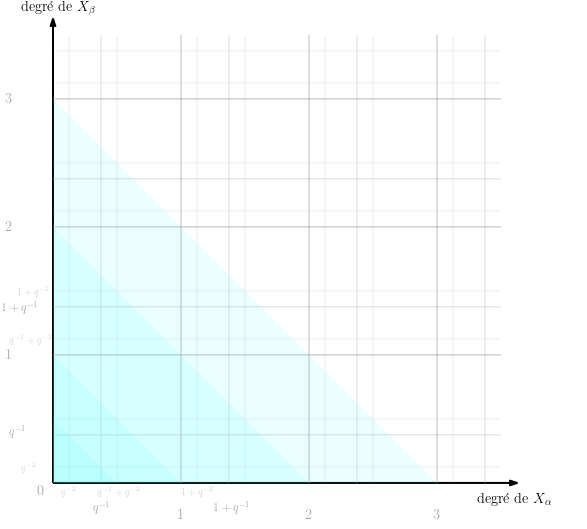}
	\end{figure} \noindent Cela signifie exactement que l'intersection avec toute zone triangulaire ne contient qu'un nombre fini de points.
	
	\clearpage
	
	De manière analogue, un élément de la complétion $X_{\alpha}X_{\beta}$-adique possède une représentation telle que son intersection avec chaque zone en L provient de la représentation d'un élément de $\widetilde{E}_{\alpha}^+\otimes_{\mathbb{F}_q} \widetilde{E}_{\beta}^+$.

	\begin{figure}[h!]
		\centering
		\includegraphics[scale=0.4]{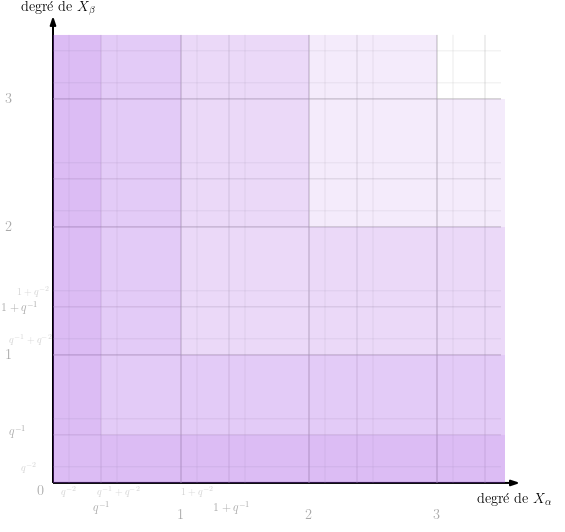}
	\end{figure} \noindent Cela revient à choisir une représentation telle que l'intersection avec toute partie rectangulaire verticale d'une zone en L ne contient qu'un nombre fini d'abscisses, que sur les droites correspondant à ces abscisses la représentation soit d'intersection finie avec tout compact, et la condition analogue pour les parties rectangulaires horizontales et les ordonnées.
	
	Ceci nous permet de comprendre pourquoi $\sum_{n\geq 0} X_{\alpha}^{q^{-n}} X_{\beta}^{q^n}$ appartient à une complétion mais pas à l'autre via sa représentation :

	\begin{figure}[h]
		\centering
		\includegraphics[scale=0.3]{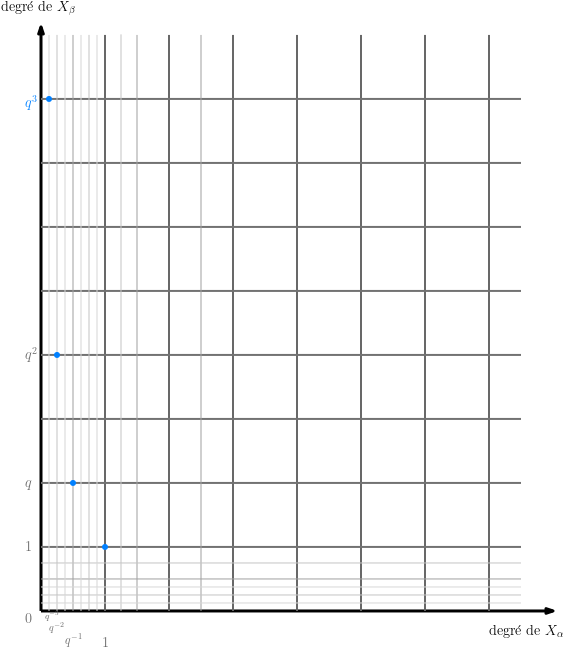}
		\caption{Représentation de $\sum_{n\geq 0} X_{\alpha}^{q^{-n}} X_{\beta}^{q^n}$ pour $q=2$}
	\end{figure}
	
	En inversant $X_{\Delta}$, nous pouvons écrire
	
	$$\widetilde{E}_{\Delta}=\left\{ \sum_{\underline{d}\in (\mathbb{Z}[q^{-1}])^{\Delta}} a_{\underline{d}} \underline{X}^{\underline{d}} \,\, \Bigg|\,\,\substack{(a_{\underline{d}})\in \left(\bigotimes_{\alpha \in \Delta,\,\mathbb{F}_q} k\right)^{\mathbb{Z}[q^{-1}]^{\Delta}} \,\, \text{tel que} \\ \forall r\in \mathbb{N}[q^{-1}], \,\, \left\{\underline{d} \, |\, \sum d_{\alpha}\leq r, \,\, \mathrm{et}\,\, a_{\underline{d}} \neq 0\right\} \,\,\text{est fini et que}\\ \exists N>0, \,\,\, \forall \underline{d},\, \forall \alpha, \,\,\, a_{\underline{d}}\neq 0 \implies  d_{\alpha}\geq -N}\right\}.$$
	
\noindent	De manière similaire, nous obtenons $$E_{\Delta}= \left\{ \sum_{\underline{d}\in \mathbb{Z}^{\Delta}} a_{\underline{d}} \underline{X}^{\underline{d}} \,\, \Bigg|\,\,\substack{(a_{\underline{d}})\in \left(\bigotimes_{\alpha \in \Delta,\,\mathbb{F}_q} k\right)^{\mathbb{Z}^{\Delta}} \,\, \text{tel que} \\  \exists N>0, \,\,\, \forall \underline{d},\, \forall \alpha, \,\,\, a_{\underline{d}}\neq 0 \implies  d_{\alpha}\geq -N}\right\}.$$ Remarquons que dans le cas imparfait, les complétions $(\underline{X})$- et $X_{\Delta}$-adiques coïncident algébriquement puisque la finitude de l'intersection avec chaque zone triangulaire implique la condition sur les zones en L.
	
	Toute extension finie $E^{\sep}|F|E$ s'écrit $l(\!(Y)\!)$ pour une uniformisante $Y$ et une extension finie $l|k$ ce qui fournit une description semblable de $\widetilde{F}_{\Delta,q}$ et $F_{\Delta,q}$ en remplaçant $\underline{X}$ par $\underline{Y}$ et $\left(\bigotimes_{\alpha \in \Delta,\,\mathbb{F}_q} k\right)$ par $\left(\bigotimes_{\alpha\in \Delta,\mathbb{F}_q} l\right)$.
	
	Cette description explicite nous fournit également sans argument technique une famille d'injections \linebreak$F_{\Delta',q} \hookrightarrow F_{\Delta,q}$ pour $\Delta'\subseteq \Delta$.
	\end{rem}
	
	Nous pouvons faire le lien entre les anneaux perfectoïdes et imparfaits.
	
	\begin{lemma}\label{lien_anneaux_e}
		Pour toute extension finie galoisienne $E^{\sep}|F|E$ nous appelons $\widetilde{F}=\widetilde{E} F$ l'extension finie galoisienne de $\widetilde{E}$ associée. 
		
		Il existe une injection de $(\Phi_{\Delta,q,r}\times \cg_{E,\Delta})$-anneaux  $F_{\Delta,q}^+\hookrightarrow \widetilde{F}_{\Delta,q}^+$ naturelle en $F$.
		
		Pour toute famille finie de multi-indices $(\underline{d_i})_{1\leq i \leq n}  \in (\N^{\Delta})^n$, l'image réciproque de l'idéal $(X^{\underline{d_i}})$ est engendrée par les mêmes éléments.
	\end{lemma}
	\begin{proof}
	Construire l'application et calculer les images réciproques des idéaux s'obtient à partir de la description de nos anneaux à la Remarque \ref{ecriture_explicite_anneaux}. Nous vérifions que l'application construite coïncide avec le produit tensoriel des injections $F_{\alpha}^+\hookrightarrow \widetilde{F}_{\alpha}^+$, complété et localisé. L'équivariance et la naturalité découlent de cette seconde construction.
	\end{proof}
	
	\begin{coro}\label{integrite_imparf}
		\begin{enumerate}[itemsep=0mm]
			\item Les anneaux $F_{\Delta,q}^+$ sont réduits et sans $E_{\Delta}^+$-torsion. Il en découle que $F_{\Delta,q}^+ \rightarrow F_{\Delta,q}$ est injective, que $F_{\Delta,q}$ est réduit et sans $E_{\Delta}$-torsion. En particulier, l'anneau $E_{\Delta}$ est intègre.
			
			\item Soit $\mathcal{G}\mathrm{al}_{E}$ la catégorie des extensions finies galoisiennes de $E$ dans $E^{\sep}$ avec les inclusions pour morphismes. La construction $$F\mapsto F_{\Delta,q}$$ où l'on met la topologie discrète est canoniquement un foncteur de $\mathcal{G}\mathrm{al}_{E}$ vers les $(\Phi_{\Delta,q,r}\times \cg_{E,\Delta})$-anneaux topologiques. Tous les morphismes déduits sont injectifs. De plus, localiser les injections du Lemme \ref{lien_anneaux_e} fournit une transformation naturelle vers le foncteur du Lemme \ref{description_construction_produit}.
		\end{enumerate}
	\end{coro}

	\begin{defi}\label{defi_anneau_sep_imparf}
		Définissons les $(\Phi_{\Delta,q,r}\times \cg_{E,\Delta})$-anneaux $$E_{\Delta}^{\sep,+}=\colim \limits_{F\in \mathcal{G}\mathrm{al}_{E}} F^+_{\Delta,q} \,\,\,\, \text{et}\,\,\,\, E_{\Delta}^{\sep}=\colim \limits_{F\in \mathcal{G}\mathrm{al}_{E}} F_{\Delta,q}.$$ Nous pouvons munir ce dernier de la topologie discrète ou de la topologie adique, ayant pour base de voisinages de $0$ les $(X_{\Delta}^n E_{\Delta}^{\sep,+})_{n\geq 0}$. Les deux fournissent une structure de $(\Phi_{\Delta,q,r}\times \cg_{E,\Delta})$-anneau topologique. Nous n'utilisons que la topologie discrète dans cette sous-section.
	\end{defi}
	
	\begin{lemma}\label{module_induit_multi_imparf}
	Pour toute extension finie $E^{\sep}|F|E$, le morphisme canonique
	
	$$E_{\alpha}^{\sep}\otimes_{F_{\alpha}}\left(E_{\beta}^{\sep} \otimes_{F_{\beta}} \cdots \left(E_{\delta}^{\sep}\otimes_{F_{\delta}} F_{\Delta,q}\right)\right)\rightarrow E^{\sep}_{\Delta}$$ est un isomorphisme de $\cg_{E,\Delta}$-équivariant.
\end{lemma}
\begin{proof}
	Identique au Lemme \ref{module_induit_multi}, les presque mathématiques en moins. En effet, le $F^+$-module $F'^+$ est libre de rang fini pour toute extension finie $F'|F$ ce qui donne même un isomorphisme $$F'^+_{\alpha} \otimes_{F^+_{\alpha}}\left(F'^+_{\beta} \otimes_{F^+_{\beta}} \cdots \left(F'^+_{\delta}\otimes_{F^+_{\delta}} F^+_{\Delta,q}\right)\right)\rightarrow F'^+_{\Delta,q}.$$
\end{proof}

	\begin{coro}\label{coro_descente_1_imparf}
	Nous obtenons que :
	
	\begin{enumerate}[itemsep=0mm]
		\item L'anneau $E_{\Delta}^{\sep}$ est réduit et sans $E_{\Delta}$-torsion.
		
		\item L'inclusion $E_{\Delta} \subseteq \left(E_{\Delta}^{\sep}\right)^{\cg_{E,\Delta}}$ est une égalité.
		
		\item 	Pour tout objet $D$ de $\cdetaleproj{\Phi_{\Delta,q,r}\times\cg_{E,\Delta}}{E_{\Delta}^{\sep}}$, le morphisme de comparaison
		
		$$E_{\Delta}^{\sep}\otimes_{E_{\Delta}} \mathrm{Inv}(D) \rightarrow D$$ est un isomorphisme.
	\end{enumerate}
\end{coro}
\begin{proof}
	Identique à la Proposition \ref{edseptorsion} et au Corollaire \ref{coro_descente_1}.
\end{proof}

	\begin{defiprop}\label{defipropDdelta}
	Le foncteur $$\mathbb{D}_{\Delta} \, : \, \mathrm{Rep}_{\mathbb{F}_r} \cg_{E,\Delta} \rightarrow \dmod{\Phi_{\Delta,q,r}}{E_{\Delta}}, \,\,\, V \mapsto \left(E^{\sep}_{\Delta} \otimes_{\mathbb{F}_r} V \right)^{\cg_{E,\Delta}}$$ est correctement défini, pleinement fidèle et son image essentielle est incluse dans $\detaleproj{\Phi_{\Delta,q,r}}{E_{\Delta}}$. Cette dernière catégorie est une sous-catégorie pleine monoïdale fermée. Le foncteur $\mathbb{D}_{\Delta}$ commute naturellement au produit tensoriel et au $\mathrm{Hom}$ interne.
\end{defiprop}
\begin{proof}
	Identique à la Proposition \ref{defipropwDdelta}, mise à part la pleine fidélité que nous n'avions alors pas traitée. Soit $V_1,V_2 \in \mathrm{Rep}_{\mathbb{F}_r} \cg_{E,\Delta}$. En utilisant la commuation naturelle de $\mathbb{D}_{\Delta}$ au $\mathrm{Hom}$ interne puis l'isomorphisme de comparaison pour $\underline{\mathrm{Hom}}(V_1,V_2)$, nous obtenons des isomorphismes dans $\dmod{\Phi_{\Delta,q,r}\times \cg_{E,\Delta}}{E_{\Delta}^{\sep}}$ 
	$$E_{\Delta}^{\sep} \otimes_{E_{\Delta}} \underline{\mathrm{Hom}}(\mathbb{D}_{\Delta}(V_1),\mathbb{D}_{\Delta}(V_2)) \xleftarrow{\sim} E_{\Delta}^{\sep} \otimes_{E_{\Delta}}\mathbb{D}_{\Delta}\left( \underline{\mathrm{Hom}}(V_1,V_2)\right) \xrightarrow{\sim} E_{\Delta}^{\sep}\otimes_{\mathbb{F}_r} \underline{\mathrm{Hom}}(V_1,V_2).$$ Calculons les invariants par $(\Phi_{\Delta,q,r}\times \cg_{E,\Delta})$. À gauche en utilisant \cite[Prop. 3.10]{formalisme_marquis} pour $\cg_{E,\Delta}$, la $E_{\Delta}$-algèbre sans torsion\footnote{Nous utilisons les deux premiers points du Corollaire \ref{coro_descente_1_imparf}.} $E_{\Delta}^{\sep}$ et le $E_{\Delta}$-module fini projectif $\underline{\mathrm{Hom}}(\mathbb{D}_{\Delta}(V_1),\mathbb{D}_{\Delta}(V_2))$. À droite, le Corollaire \ref{inv_frob_multi_perf} couplé à l'injection $E_{\Delta}^{\sep}\hookrightarrow \widetilde{E}_{\Delta}^{\sep}$ donne $\mathbb{F}_r=\left(E_{\Delta}^{\sep}\right)^{\Phi_{\Delta,q,r}}$ ; on applique ensuite \cite[Prop. 3.10]{formalisme_marquis} à l'inclusion $\mathbb{F}_r\subset E_{\Delta}^{\sep}$. On obtient
		
		$$\mathrm{Hom}_{\mathrm{Rep}_{\mathbb{F}_r} \cg_{E,\Delta}}(V_1,V_2) \xrightarrow{\sim}\mathrm{Hom}_{\dmod{\Phi_{\Delta,q,r}}{E_{\Delta}}}\left(\mathbb{D}_{\Delta}(V_1),\mathbb{D}_{\Delta}(V_2)\right)$$ et l'on vérifie qu'il s'agit de l'application donnée par fonctorialité\footnote{Cette identification de l'application est déjà contenu dans \cite[Prop. 3.3]{formalisme_marquis}.}.
\end{proof}

	\vspace{1cm}
	\subsection{\'Equivalence de Fontaine multivariable modulo $p$ pour certains corps de caractéristique $p$}\label{section_equiv_car_p_imparf}

	Il reste à démontrer que $\mathbb{D}_{\Delta}$ est essentiellement surjectif et à expliciter son quasi-inverse. À dessein, nous avons commencé par l'équivalence de Fontaine perfectoïde pour pouvoir en déduire des propriétés de $\mathbb{D}_{\Delta}$. 
	
	\begin{defi}
		Appelons $\widetilde{\mathrm{Ex}}\, : \, \detaleproj{\Phi_{\Delta,q,r}}{E_{\Delta}}\rightarrow \detaleproj{\Phi_{\Delta,q,r}}{\widetilde{E}_{\Delta}}$ le foncteur obtenu à partir de l'inclusion de $\Phi_{\Delta,q,r}$-anneaux du Lemme \ref{lien_anneaux_e} en suivant \cite[Prop. 3.3]{formalisme_marquis}.
	\end{defi}

	\smallskip
	
	La stratégie consiste à démontrer que $\widetilde{\mathbb{D}}_{\Delta} \cong \widetilde{\mathrm{Ex}}\circ \mathbb{D}_{\Delta}$ puis que $\widetilde{\mathrm{Ex}}$ est pleinement fidèle. Grâce à la commutation au $\mathrm{Hom}$ interne \cite[Prop. 3.3]{formalisme_marquis}, il suffit de démontrer que $$\forall D\in \detaleproj{\Phi_{\Delta,q,r}}{E_{\Delta}}, \,\,\, D^{\Phi_{\Delta,q,r}}\cong \widetilde{\mathrm{Ex}}(D)^{\Phi_{\Delta,q,r}}.$$
	
	Rappelons brièvement la stratégie de preuve de ce énoncé dans le cas univariable. Soit $d\in (\widetilde{E}\otimes_E D)^{\varphi=\mathrm{Id}}$ que l'on écrit $d=\sum x_i e_i$ pour une $E$-base $(e_i)$ de $D$. Nous munissons $\left(\widetilde{E}\otimes_E D\right)$ de son unique structure de \linebreak$\widetilde{E}$-espace vectoriel topologique pour la topologie adique sur $\widetilde{E}$. Quitte à choisir correctement la base, nous pouvons supposer que $$\left(\bigoplus X\widetilde{E}^+ e_i\right) \subset (\widetilde{E}\otimes_E D)^{++}:=\{d \in \widetilde{E}\otimes_E D \,  |\, \varphi_D^n(d)\xrightarrow[n\rightarrow + \infty]{} 0\}$$ et que $\left(\oplus \, E^+ e_i\right)$ est stable par $\varphi_D$. Pour $N$ assez grand, on peut écrire $\varphi^N(x_i)\in E + X\widetilde{E}^+$. Nous en déduisons que $$d=\varphi^N(d)=\sum \varphi^N(x_i) \varphi^N(e_i)\in D + (\widetilde{E}\otimes_E D)^{++}.$$ En appliquant $\varphi_D^{n}$ pour $n$ arbitrairement grand on en déduit que $d$ est dans l'adhérence de $D$, i.e. dans $D$.
	
	Dans le cas multivariable, la topologie adique sur $\widetilde{E}_{\Delta}$ ne permet pas de décalquer la preuve. En effet, si l'on considère $x=\sum_{n\geq 0} X_{\alpha}^{q^n+q^{-n}}$, aucun $\varphi_{\Delta,q}^N(x)$ n'appartient à $E_{\Delta}+X_{\Delta} \widetilde{E}_{\Delta}^+$. Il faut considérer la topologie colimite, malgré son apparence tarabiscotée.
	
	\begin{lemma}\label{lemme_d_ferme_1}
		Le sous-anneau $E_{\Delta}$ est fermé dans $\widetilde{E}_{\Delta}$ muni de la topologie colimite.
	\end{lemma}
	\begin{proof}
		Au niveau entier, utilisons la description de nos anneaux à la Remarque \ref{ecriture_explicite_anneaux}. Soit $x\in \widetilde{E}_{\Delta}^+ \backslash E_{\Delta}^+$. Il existe un monôme $a X^{\underline{d}}$ apparaissant dans $x$ avec $d\in \N[q^{-1}]^{\Delta}\backslash \N^{\Delta}$. Ce terma apparaît encore $(\underline{X})$-adique localement, d'où l'on tire que $E_{\Delta}^+$ est fermé dans $\widetilde{E}_{\Delta}^+$.Avec l'égalité $E_{\Delta} \cap X_{\Delta}^{-n}\widetilde{E}_{\Delta}^+=X_{\Delta}^{-n}E_{\Delta}^+$ on fait le même raisonnement sur chaque terme de la colimite.
	\end{proof}
	
	\begin{lemma}\label{lemm_ferme_mod_top}
	Soient $R\subseteq T$ un sous-anneau fermé d'un anneau topologique $T$ et $D$ un $R$-module fini projectif. Puisque $D$ est plat sur $R$, il est possible de voir $D$ comme un sous-espace du $T$-module fini projectif $\left(T\otimes_R D\right)$. Le sous-espace $D$ est fermé pour la topologie initiale.
	\end{lemma}
	\begin{proof}
		Fixons une présentation $D\oplus D'=R^k$ de $D$. L'écriture $(T\otimes_R D)\oplus (T\otimes_R D') = T^k$ fournit une présentation de $\left(T\otimes_R D\right)$. Puisque $R\subseteq T$ est fermé, c'est encore le cas de $R^k \subseteq T^k$ avec les topologies produits. La topologie initiale sur $T\otimes_R D$ étant la topologie induite depuis $T^k$, on en déduit que $D=(T\otimes_R D)\cap R^k$ est fermé dans $T\otimes_R D$.
	\end{proof}

	\begin{prop}\label{ext_inv_1}
		Soit $D$ un objet de $\detaleproj{\Phi_{\Delta,q,r}}{E_{\Delta}}$. L'inclusion $$D^{\varphi_{\Delta,r}=\mathrm{Id}}\subseteq \widetilde{\mathrm{Ex}}(D)^{\varphi_{\Delta,r}=\mathrm{Id}}$$ est une égalité.
	\end{prop}
	\begin{proof}
		Prenons $(e_i)_{1\leq i \leq d}$ une famille génératrice de $D$ et posons $A=(a_{i,j})\in \mathrm{M}_d(E_{\Delta})$ une matrice de $\varphi_{\Delta,r,D}$ dans la famille $(e_i)$. Quitte à considérer la famille génératrice $(X_{\Delta}^n e_i)$, on  $A$ est à coefficients dans $E_{\Delta}^+$. Nous avons alors par récurrence immédiate  $$\forall N\geq 0, \,\,\,\varphi_{\Delta,r}^N\left( \sum_i (\underline{X}) \widetilde{E}_{\Delta}^+ e_i\right) \subseteq \sum_i (\underline{X})^{r^N}\widetilde{E}_{\Delta}^+ e_i$$ En particulier $$\sum_i (\underline{X}) \widetilde{E}_{\Delta}^+ e_i \subseteq (\widetilde{E}_{\Delta}\otimes_{E_{\Delta}} D)^{++}:=\left\{ d\in\left( \widetilde{E}_{\Delta}\otimes_{E_{\Delta}} D\right)\, \bigg|\, \substack{\varphi_{\Delta,r,\widetilde{\mathrm{Ex}}(D)}^n(d) \xrightarrow[n\rightarrow +\infty]{} 0 \\ \text{pour la topologie initiale déduite de} \\ \text{la topologie colimite sur } \widetilde{E}_{\Delta}}\right\}.$$
		
		Prenons à présent une écriture $d=\sum x_i e_i$ de $d\in (\widetilde{E}_{\Delta}\otimes_{E_{\Delta}} D)^{\varphi_{\Delta,r}=\mathrm{Id}}$. Grâce à la description de la Remarque \ref{ecriture_explicite_anneaux}, nous savons que les monômes des $x_i$ de degré total inférieur à $1$ sont en nombre fini. En particulier, leurs dénominateurs sont des puissances de $p$ divisant un certain $r^N$. Nous en déduisons  $$d=\varphi_{\Delta,r,\widetilde{\mathrm{Ex}}(D)}^N(d)=\sum_i \varphi_{\Delta,r}^N(x_i) \varphi_{\Delta,r,D}^N(e_i) \in \left[D+\sum_i (\underline{X})\widetilde{E}_{\Delta}^+ e_j \right]\subseteq \left[D+ (\widetilde{E}_{\Delta}\otimes_{E_{\Delta}} D)^{++}\right].$$ Écrivons une telle décomposition $d=d_0+d^{++}$. Alors,
		
	$$d	=\lim \limits_{n\rightarrow + \infty} \varphi_{\Delta,r,D}^n(d_0)$$ pour la topologie initiale associée à la topologie colimite sur $\widetilde{E}_{\Delta}$. Les Lemmes \ref{lemme_d_ferme_1} et \ref{lemm_ferme_mod_top} combinés démontrent que $D$ est fermé dans $\left(\widetilde{E}_{\Delta}\otimes_{E_{\Delta}} D\right)$ pour la topologie initiale, d'où $d\in D$.
	\end{proof}
	
	Nous en tirons le corollaire qui nous servira vraiment en prenant les invariants dans la Proposition \ref{ext_inv_1}.
	
	\begin{coro}\label{ext_inv_2}
		Soit $D$ un objet de $\detaleproj{\Phi_{\Delta,q,r}}{E_{\Delta}}$. L'inclusion suivante est une égalité $$D^{\Phi_{\Delta,q,r}}\subseteq \widetilde{\mathrm{Ex}}(D)^{\Phi_{\Delta,q,r}}.$$
	\end{coro}
	
	\medskip

	 Nous identifions les deux catégories $\detaleproj{\Phi_{\Delta,q,r}^{\gp}}{\widetilde{E}_{\Delta}}$ et $\detaleproj{\Phi_{\Delta,q,r}}{\widetilde{E}_{\Delta}}$ dans la suite..
	
	\begin{prop}\label{tildDetD}
		1) Les foncteurs $\widetilde{\mathbb{D}}_{\Delta}$ et $\widetilde{\mathrm{Ex}}\circ \mathbb{D}_{\Delta}$ sont isomorphes.
		
		2) Le foncteur $\widetilde{\mathrm{Ex}}$ est pleinement fidèle.
	\end{prop}
	
	\begin{proof}
		1) Soit $V$ un objet de $\mathrm{Rep}_{\mathbb{F}_r} \cg_{E,\Delta}$. Considérons le morphisme d'anneaux \linebreak$(\Phi_{\Delta,q,r}\times \cg_{E,\Delta})$-équivariant $E^{\sep}_{\Delta} \hookrightarrow \widetilde{E}^{\sep}_{\Delta}$ obtenu à partir de la Propositon \ref{lien_anneaux_e}. En tensorisant par $V$ et passant aux invariants par $\cg_{E,\Delta}$, il fournit une application $E_{\Delta}$-linéaire et $\Phi_{\Delta,q,r}$-équivariante $\mathbb{D}_{\Delta}(V)\rightarrow \widetilde{\mathbb{D}}_{\Delta}(V)$, puis comme $\widetilde{\mathbb{D}}_{\Delta}(V)$ est un $\widetilde{E}_{\Delta}$-module, une transformation naturelle $S\, : \, \widetilde{\mathrm{Ex}}\circ \mathbb{D}_{\Delta} \Rightarrow \widetilde{\mathbb{D}}_{\Delta}$.
		
		Puisque $\widetilde{E}_{\Delta}\rightarrow \widetilde{E}^{\mathrm{sep}}_{\Delta}$ est fidèlement plat, nous vérifions qu'il s'agit d'un isomorphisme de foncteur après changement de base. La transformation naturelle entre foncteur au Corollaire \ref{integrite_imparf} fournit l'isomorphisme de gauche du diagramme suivant
		
		\begin{center}
			\begin{tikzcd}
				\widetilde{E}^{\sep}_{\Delta} \otimes_{\widetilde{E}_{\Delta}}\bigg(\widetilde{E}_{\Delta} \otimes_{E_{\Delta}} \mathbb{D}_{\Delta}(V)\bigg) \arrow{d}[anchor=center, rotate=90, yshift=1ex]{\sim} \ar[rr,"\mathrm{Id}\otimes S(V)"] & & \widetilde{E}^{\sep}_{\Delta} \otimes_{\widetilde{E}_{\Delta}} \widetilde{\mathbb{D}}_{\Delta}(V) \ar[d] \\
				\widetilde{E}^{\sep}_{\Delta}\otimes_{E^{\sep}_{\Delta}} \bigg(E^{\sep}_{\Delta}\otimes_{E_{\Delta}} \mathbb{D}_{\Delta}(V)\bigg)  \ar[r] & \widetilde{E}^{\sep}_{\Delta} \otimes_{E^{\sep}_{\Delta}}\bigg(E^{\sep}_{\Delta} \otimes_{\mathbb{F}_r} V\bigg) \ar[r,"\sim"] & \widetilde{E}^{\sep}_{\Delta}\otimes_{\mathbb{F}_r} V 
			\end{tikzcd}
		\end{center} où les flèches non labellisées sont les isomorphismes de comparaison déjà étudiés. Pour que $S(V)$ soit un isomorphisme, il suffit de vérifier la commutation du diagramme .
		\medskip
		
		2) Soient $D_1,D_2$ deux objets de $\detaleproj{\Phi_{\Delta,q,r}}{E_{\Delta}}$. L'application donnée par $\widetilde{\mathrm{Ex}}$ au niveau des morphismes s'obtient  selon \cite[Prop. 3.5]{formalisme_marquis} en prenant les $\Phi_{\Delta,q,r}$-invariants de la composée $$\mathrm{Hom}_{E_{\Delta}}(D_1,D_2)\rightarrow \mathbb{E}\mathrm{x}(\mathrm{Hom}_{E_{\Delta}}(D_1,D_2)) \xrightarrow[\sim]{\iota_{D_1,D_2}} \mathrm{Hom}_{\widetilde{E}_{\Delta}}(\mathbb{E}\mathrm{x}(D_1),\mathbb{E}\mathrm{x}(D_2)),$$ où $\iota_{D_1,D_2}$ est construit dans \cite[Lem. 2.16]{formalisme_marquis}. C'est un isomorphisme puisque $D_1$ est fini projectif. En passant aux invariants par $\Phi_{\Delta,q,r}$, le premier morphisme devient un isomorphisme d'après le Corollaire \ref{ext_inv_2}.	
	\end{proof}
	
	\begin{lemma}\label{inv_frob_imparf}
		L'inclusion $\mathbb{F}_r\subseteq \left(E_{\Delta}^{\sep}\right)^{\Phi_{\Delta,q,r}}$ est une égalité.
	\end{lemma}
	\begin{proof}
		Combiner le résultat sur $\widetilde{E}_{\Delta}^{\sep}$ au Lemme \ref{inv_frob_multi_perf} et l'injection $E_{\Delta}^{\sep} \hookrightarrow \widetilde{E}_{\Delta}^{\sep}$ obtenue en passant à la limite dans le Lemme \ref{lien_anneaux_e}.
	\end{proof}
	
	\begin{theo}\label{equiv_non_perf_car_p}
		Le foncteur $$\mathbb{D}_{\Delta}\, : \,\mathrm{Rep}_{\mathbb{F}_r} \cg_{E,\Delta} \rightarrow \detaleproj{\Phi_{\Delta,q,r}}{E_{\Delta}}$$ est une équivalence de catégories monoïdales fermées. Un quasi-inverse est donné par $$\mathbb{V}_{\Delta}\, : \, \detaleproj{\Phi_{\Delta,q,r}}{E_{\Delta}} \rightarrow \mathrm{Rep}_{\mathbb{F}_r} \cg_{E,\Delta}, \,\,\, D \mapsto \left(E_{\Delta}^{\sep}\otimes_{E_{\Delta}} D\right)^{\Phi_{\Delta,q,r}}.$$
	\end{theo}	
	\begin{proof}
	 Nous avons vu que $\widetilde{\mathbb{D}}_{\Delta}\cong \widetilde{\mathrm{Ex}}\circ \mathbb{D}_{\Delta}$ en Proposition \ref{tildDetD}, que $\widetilde{\mathbb{D}}_{\Delta}$ est une équivalence au Théorème \ref{equiv_perf_mod_p} et que $\mathbb{D}_{\Delta}$ et $\widetilde{\mathrm{Ex}}$ sont pleinement fidèles aux Propositions \ref{defipropDdelta} et \ref{tildDetD}. Cela suffit à prouver que $\mathbb{D}_{\Delta}$ est une équivalence de catégories.
		
	Soit $V$ une représentation. Nous avons prouvé en Proposition \ref{defipropDdelta} un isomorphisme de comparaison naturel $$E_{\Delta}^{\sep} \otimes_{E_{\Delta}} \mathbb{D}_{\Delta}(V) \xrightarrow{\sim} E_{\Delta}^{\sep} \otimes_{\mathbb{F}_r} V.$$ En décomposant dans une $\mathbb{F}_r$-base de $V$, les $\Phi_{\Delta,q,r}$-invariants du terme de droite redonnent $V$. Passer aux invariants fournit donc un isomorphisme naturel $$\mathbb{V}_{\Delta} \circ \mathbb{D}_{\Delta} \xRightarrow{\sim} \mathrm{Id}.$$\qedhere
	\end{proof}

	\vspace{0.75cm}
	\subsection{Dévissage vers une équivalence de Fontaine multivariable pour certains corps de caractéristique $p$}\label{section_equiv_car_zero_imparf}
	
	Dans cette sous-section nous cherchons à démontrer une équivalence de Fontaine multivariable imparfaite pour des coefficients de caractéristique mixte. Nous fixons encore $E$ un corps de caractéristique $p$, de valuation discrète, complet pour cette valuation, tel que la clôture algébrique de $\mathbb{F}_p$ dans son corps résiduel $k$ est de cardinal fini $q$. Nous fixons également $K$ une extension finie de $\qp$ de corps résiduel de cardinal $r$ tel que $q=r^f$. Soit $L=K\mathbb{Q}_q$. La lettre $\pi$ dénote toujours une uniformisante de $K$, qui est également une uniformisante de $L$. 
	
	Nous fixons un ensemble fini $\Delta$. Puisque $\mathbb{Q}_r=K\cap \mathbb{Q}_p^{\mathrm{nr}}$, il existe un $r$-Frobenius canonique sur $L$. L'action de $\Phi_{\Delta,q,r}$ sur $\mathcal{O}_L$ par son quotient $\sfrac{\Phi_{\Delta,q,r}}{\Phi_{\Delta,q}}\cong\varphi_{\Delta,r}^{\sfrac{\N}{f\N}}$ fournit une structure de $\Phi_{\Delta,q,r}$-anneau topologique pour la topologie $\pi$-adique.
	
	Nous fixons une famille $\mathcal{O}_{\mathcal{E}_{\alpha}}^+$ de $\mathcal{O}_L$-algèbres $\pi$-adiquement séparées et complète, d'anneau résiduel en l'idéal $\pi$ égal à $E_{\alpha}^+$ et munies d'un relèvement semi-linéaire $\phi_{\alpha,r}$ du $r$-Frobenius sur $E_{\alpha}^+$. On notera toujours $X_{\alpha}$ une uniformisante de $E_{\alpha}$ ou un relevé dans $\mathcal{O}_{\mathcal{E}_{\alpha}}$. Exactement comme dans la Proposition \ref{prop_defi_oe} et la Définition \ref{defi_oe}, on choisit des $\cg_{E_{\alpha}}$-anneaux $\mathcal{O}_{\widehat{\mathcal{E}_{\alpha}^{\mathrm{ur}}}}$ ainsi que leurs sous-anneaux $\mathcal{O}_{\mathcal{F}_{\alpha}}$ et $\mathcal{O}_{\mathcal{F}_{\alpha}}^+$ pour chaque extension finie $E^{\sep}|F|E$. En appliquant \cite[\href{https://stacks.math.columbia. edu/tag/08HQ}{Tag 08HQ}]{stacks-project} pour le $r$-Frobenius sur les corps résiduels, on peut étendre l'endomorphisme $\phi_{\alpha,r}$ en un endomorphisme $\mathcal{O}_L$-semi-linéaire de $\mathcal{O}_{\widehat{\mathcal{E}_{\alpha}^{\mathrm{ur}}}}$ qui stabilise chaque $\mathcal{O}_{\mathcal{F}_{\alpha}}$ et commute à l'action galoisienne.

	\begin{defi}\label{defi_oed}
		On note alors $$\mathcal{O}_{\mathcal{F}_{\Delta}}^+ = \left(\bigotimes_{\alpha \in \Delta, \, \mathcal{O}_L} \mathcal{O}_{\mathcal{F}_{\alpha}}^+\right)^{\wedge (\pi, \underline{X})} \,\,\, \text{et} \,\,\, \mathcal{O}_{\mathcal{F}_{\Delta}}=\left(\mathcal{O}_{\mathcal{F}_{\Delta}}^+ \left[\frac{1}{X_{\Delta}}\right]\right)^{\wedge \pi}.$$ Ces anneaux ne dépendent pas du choix de $\pi$ ni de celui des $X_{\alpha}$. Ce sont des $\mathcal{O}_L$-algèbres $\pi$-adiquement séparées et complètes d'anneaux résiduels en $\pi$ $F_{\Delta,q}^+$ et $F_{\Delta,q}$. Nous les munissons de la topologie $\pi$-adique.
		
		%Parallèlement, nous munissons $\mathcal{O}_{\mathcal{F}_{\Delta}}^+$ de la topologie $(\pi,X_{\Delta})$-adique et $\mathcal{O}_{\mathcal{F}_{\Delta}}$ de la topologie ayant pour base de voisinages de $0$ les sous-groupes $\left(\pi^n \mathcal{O}_{\mathcal{F}_{\Delta}} + X_{\Delta}^m \mathcal{O}_{\mathcal{F}_{\Delta}}^+\right)_{n,m\geq 0}$. Cette deuxième topologie sera appelée \textit{topologie faible}.
		
		Le produit tensoriel des $\phi_{\alpha,r}$ tous $\mathcal{O}_L$-semi-linéaires par rapport au $r$-Frobenius, les $\phi_{\alpha,r}^{\circ f}$ et l'action de $\cg_{E,\Delta}$ terme à terme fournissent après complétion et/ou localisation une structure de $\left(\Phi_{\Delta,q,r}\times \cg_{E,\Delta}\right)$-anneaux topologiques sur $\mathcal{O}_{\mathcal{F}_{\Delta}}^+$ et $\mathcal{O}_{\mathcal{F}_{\Delta}}$. Les injections de $\mathcal{O}_L$ sont des morphismes de $\Phi_{\Delta,q,r}$-anneaux topologiques.
	\end{defi}
	
	\begin{defi}\label{defi_oeddeltasep}
		Nous définissons comme au Corollaire \ref{integrite_imparf} deux diagrammes de $(\Phi_{\Delta,q,r}\times \cg_{E,\Delta})$-anneaux dont nous définissons $$\mathcal{O}_{\widehat{\mathcal{E}^{\mathrm{nr}}_{\Delta}}}^+:=\left(\colim \limits_{F\in \mathcal{G}\mathrm{al}_{E}} \mathcal{O}_{\mathcal{F}_{\Delta}}^+ \right)^{\wedge \pi} \,\, \text{et} \,\, \Oedhat:=\mathcal{O}_{\widehat{\mathcal{E}^{\mathrm{nr}}_{\Delta}}}^+\left[\frac{1}{X_{\Delta}}\right].$$ Nous munissons ce dernier anneau de la topologie $\pi$-adique pour cette sous-section.
		
		%deux topologies d'anneau : la topologie $\pi$-adique et la topologie dont une base de voisinages de $0$ est donnée par $(\pi^n \mathcal{O}_{\widehat{\mathcal{E}^{\mathrm{nr}}_{\Delta}}}^+ +X_{\Delta}^m \mathcal{O}_{\widehat{\mathcal{E}^{\mathrm{nr}}_{\Delta}}})_{n,m\geq 0}$ que nous nommons topologie faible. Les deux fournissent une structure de $(\Phi_{\Delta,q}\times \cg_{E,\Delta})$-anneau topologique. Dans cette section, nous ne nous servons que de la topologie $\pi$-adique.
	\end{defi}

	\begin{lemma}\label{coh_phideltaq}
	Nous avons $H^1(\Phi_{\Delta,q,r},E_{\Delta}^{\sep})=\{0\}$.
	\end{lemma}
	\begin{proof}
		La suite restriction-inflation et le résultat du Leme \ref{inv_frob_multi_perf} pour $r=q$ donnent $$0\rightarrow H^1\left(\quot{\varphi_{\Delta,r}^{\N}}{\varphi_{\Delta,q}^{\N}},\mathbb{F}_q\right) \rightarrow H^1\left(\Phi_{\Delta,q,r},E_{\Delta}^{\sep}\right) \rightarrow H^1\left(\Phi_{\Delta,q},E_{\Delta}^{\sep}\right).$$ La base normale de Hilbert affirme que $\mathbb{F}_q$ est induit comme $\mathbb{F}_r$-module avec action de $\gal{\mathbb{F}_q}{\mathbb{F}_r}=\sfrac{\varphi_{\Delta,r}^{\N}}{\varphi_{\Delta,q}^{\N}}$. Il est acyclique.
		
		Pour analyser le terme de droite, nous avons besoin des deux résultats qui suivent.
		
		\underline{Premier résultat :} soit $\beta \in \Delta$. Il y a une égalité $E_{\Delta\backslash \{\beta\}}^{\sep} = \left(E_{\Delta}^{\sep}\right)^{\varphi_{\beta,q}=\mathrm{Id}}$. Soit $F|E$ une extension finie que nous écrivons $l(\!(Y)\!)$ pour une certaine uniformisante $Y$ et $l|k$. Selon la Remarque \ref{ecriture_explicite_anneaux} nous pouvons décrire $$F_{\Delta,q}= \left\{ \sum_{\underline{d}\in \mathbb{Z}^{\Delta}} a_{\underline{d}} \underline{Y}^{\underline{d}} \,\, \Bigg|\,\,\substack{(a_{\underline{d}})\in \left(\bigotimes_{\alpha \in \Delta,\,\mathbb{F}_q} l \right)^{\mathbb{Z}^{\Delta}} \,\, \text{tel que} \\\exists N>0, \,\,\, \forall \underline{d},\, \forall \alpha, \,\,\, a_{\underline{d}}\neq 0 \implies  d_{\alpha}\geq -N}\right\}.$$ De cette description se déduit l'inclusion $F_{\Delta\backslash\{\beta\},q}\subseteq F_{\Delta,q}^{\varphi_{\beta,q}=\mathrm{Id}}$. Soit $x=\sum a_{\underline{d}} Y^{\underline{d}}$ non nul dans les invariants. Munissons $\Z^{\Delta}$ de l'ordre lexicographique par rapport à un ordre sur $\Delta$ de minimum $\beta$ et considérons \linebreak$\underline{d_{\min}}= \min \{\underline{d} \, |\, a_{\underline{d}}\neq 0\}$. Le coefficient de $(qd_{\min,\beta}, d_{\alpha},\cdots)$ dans $\varphi_{\beta,q}(x)$ vaut $\varphi_{\beta,q}(a_{\underline{d_{\min}}})$. Puisque $\varphi_{\beta,q}$ est injectif sur $\left(\otimes_{\mathbb{F}_q} l\right)$, c'est donc le degré minimal de $\varphi_{\beta,q}(x)$. Il en découle que $d_{\min, \beta}=0$. Appelons $$y=\sum_{\underline{d}\in \Z^{\Delta\backslash\{\beta\}}} a_{(0,\underline{d})} Y^{(0,\underline{d})} \in F_{\Delta\backslash \{\beta\},q}.$$ Il se trouve que $x-y$ est encore invariant mais n'a aucune terme de degré en $\beta$ nul. Nous obtenons donc $x-y=0$ soit $x\in  F_{\Delta\backslash \{beta\},q}$.
					
		\underline{Deuxième résultat :} soit $\beta\in \Delta$. L'endomorphisme de $E_{\Delta \backslash \{\beta\}}$-espaces vectoriels $\left(\varphi_{\alpha,q}-\mathrm{Id}\right)$ est surjectif sur $E_{\Delta}^{\sep}$. Chaque polynôme $T^q-T-x$ étant scindé sur $E_{\beta}^{\sep}$, nous savons que l'endomorphisme est inversible sur $E_{\beta}^{\sep}$. Le Lemme \ref{module_induit_multi_imparf} affirme en particulier que $$E_{\beta}^{\sep}\otimes_{E_{\beta}}\left(E_{\Delta \backslash \{\beta\}}^{\sep} \otimes_{E_{\Delta \backslash\{\beta\}}} E_{\Delta}\right) \cong E_{\Delta}^{\sep}$$ ce qui permet de conclure.
			
		\underline{Utilisation des résultats :} nous choisissons $\beta\in \Delta$ et utilisons la restriction-inflation pour $\varphi_{\beta,q}^{\N}<\Phi_{\Delta,q}$ :
		
		$$0\rightarrow H^1\left(\Phi_{\Delta \backslash \{\beta\},q}, (E_{\Delta}^{\sep})^{\varphi_{\beta,q}=\mathrm{Id}}\right)\rightarrow  H^1\left(\Phi_{\Delta,q},E_{\Delta}^{\sep}\right) \rightarrow H^1\left(\varphi_{\beta,q}^{\N},E_{\Delta}^{\sep}\right).$$ Le premier résultat identifie le terme de gauche à la cohomology de $E_{\Delta\backslash\{\beta\}}^{\sep}$. Par récurrence sur $|\Delta|$, on se ramène donc à démontrer que le terme de droite s'annule. Or, soit $f$ un $1$-cocycle. Les relations de cocycles imposent que $$\forall n\geq 0, \,\,\, f(\varphi_{\beta,q}^n)=\sum_{0\leq i <n} \varphi_{\beta,q}^{\circ \, i}(f(\varphi_{\beta,q})).$$ Ainsi, pour $(\varphi_{\beta,q}-\mathrm{Id})(y)=f(\varphi_{\beta,q})$, qui existe grâce au deuxième résultat, le cocycle $f$ est le cobord associé à $y$.
	\end{proof}
	
	\begin{lemma}\label{Ktheory_ed}
		La $\mathrm{K}$-théorie de $E_{\Delta}$ vérifie $\mathrm{K}_0(E_{\Delta})=\Z$.
	\end{lemma}
	\begin{proof}
		Voir \cite[Lem. 2.3]{zabradi_equiv}.
	\end{proof}

\begin{theo}\label{equiv_car_zero}
	Les foncteurs $$\mathbb{D}_{\Delta}\, : \,\mathrm{Rep}_{\mathcal{O}_K} \cg_{E,\Delta} \rightarrow \dmod{\Phi_{\Delta,q,r}}{\mathcal{O}_{\mathcal{E}_{\Delta}}}, \,\,\, V\mapsto \left(\mathcal{O}_{\widehat{\mathcal{E}_{\Delta}^{\mathrm{ur}}}} \otimes_{\mathcal{O}_K} V\right)^{\cg_{E,\Delta}}$$
	
	$$\mathbb{V}_{\Delta}\, : \, \cdetaledvproj{\Phi_{\Delta,q,r}}{\mathcal{O}_{\mathcal{E}_{\Delta}}}{\pi} \rightarrow \dmod{\cg_{E,\Delta}}{\mathcal{O}_K}, \,\,\, D\mapsto \left( \mathcal{O}_{\widehat{\mathcal{E}_{\Delta}^{\mathrm{ur}}}} \otimes_{\mathcal{O}_{\mathcal{E}_{\Delta}}} D\right)^{\Phi_{\Delta,q,r}},$$ sont correctement définis et sont lax monoïdaux fermés. Leurs images essentielles sont contenues respectivement dans $\cdetaledvproj{\Phi_{\Delta,q,r}}{\mathcal{O}_{\mathcal{E}_{\Delta}}}{\pi}$ et $\mathrm{Rep}_{\mathcal{O}_K} \cg_{E,\Delta}$. Leurs corestrictions forment une paire de foncteurs quasi-inverses.
\end{theo}
\begin{proof}
	Démontrer la définition correcte et les propriétés sur leurs images essentielles se fait en décomposant les deux foncteurs comme aux Propositions \ref{defipropwDdelta} et \ref{defipropwVDelta}. Nous appliquons pour chaque foncteur \cite[Prop. 5.22]{formalisme_marquis} respectivement aux inclusions $\mathcal{O}_K \subset \mathcal{O}_{\widehat{\mathcal{E}_{\Delta}^{\mathrm{ur}}}}$ et $\mathcal{O}_{\mathcal{E}_{\Delta}} \subset \mathcal{O}_{\widehat{\mathcal{E}_{\Delta}^{\mathrm{ur}}}}$, puis \cite[Prop. 5.26]{formalisme_marquis} au \linebreak$(\Phi_{\Delta,q,r}\times \cg_{E,\Delta})$-anneau $\mathcal{O}_{\widehat{\mathcal{E}_{\Delta}^{\mathrm{ur}}}}$ muni de la topologie $\pi$-adique pour seule topologie respectivement pour les sous-monoïdes $\cg_{E,\Delta}$ et $\Phi_{\Delta,q,r}$. Il nous faut pour cela vérifier plusieurs conditions que nous listons. Les deux premières vérifient que nous nous plaçons dans un cadre cohérent avec \cite[\S 4]{formalisme_marquis}. La troisième justifie les utilisations de \cite[Prop. 5.22]{formalisme_marquis}. Les suivantes vérifient les conditions multiples de \cite[Prop. 5.26]{formalisme_marquis}.
	
	\underline{Condition 1 :} les couples $(\pi,\mathcal{O}_K)$, $(\pi,\mathcal{O}_{\mathcal{E}_{\Delta}})$ et $(\pi, \mathcal{O}_{\widehat{\mathcal{E}_{\Delta}^{\mathrm{ur}}}})$ sont des contextes de dévissages (voir \cite[Déf. 4.1]{formalisme_marquis}). 
	
	\noindent Les anneaux sont par construction $\pi$-adiquement séparés et complets. Reste à démontrer qu'ils sont sans $\pi$-torsion. C'est évident pour $\mathcal{O}_K$. Le produit tensoriel des $\mathcal{O}_{\mathcal{E}_{\alpha}}^+$ est sans $\pi$-torsion puisque chacun des termes l'est, a fortiori sont complété et séparé. De même, en passant à la colimite puis en complétant, $\mathcal{O}_{\widehat{\mathcal{E}_{\Delta}^{\mathrm{ur}}}}$ est sans $\pi$-torsion.
	
	\underline{Condition 2 :} pour chacun des trois anneaux, et chaque endomorphisme $f$ dans l'action des monoïdes, les idéaux engendrés par $f(\pi)$ et $\pi$ coïncident. Ici, $\pi$ est même invariant par chaque action.
	
	\underline{Condition 3 :} les inclusions  $\mathcal{O}_K \subset \mathcal{O}_{\widehat{\mathcal{E}_{\Delta}^{\mathrm{ur}}}}$ et $\mathcal{O}_{\mathcal{E}_{\Delta}} \subset \mathcal{O}_{\widehat{\mathcal{E}_{\Delta}^{\mathrm{ur}}}}$ soient $\pi$-adiquement continues et équivariantes. C'est le cas par construction.
	
	\underline{Condition 4 :} les anneaux topologiques $\mathcal{O}_{\widehat{\mathcal{E}_{\Delta}^{\mathrm{ur}}}}^{\cg_{E,\Delta}}$ et $\mathcal{O}_{\widehat{\mathcal{E}_{\Delta}^{\mathrm{ur}}}}^{\Phi_{\Delta,q,r}}$ coïncident respectivement avec avec $\mathcal{O}_{\mathcal{E}_{\Delta}}$ et $\mathcal{O}_K$.
	
	\noindent L'identification des topologies $\pi$-adiques (resp. l'égalité algébrique) s'obtient par dévissage (resp. et complétion) à partir de l'injectivité du morphisme modulo $\pi$ (resp. d'une identification algébrique modulo $\pi$). C'est alors exactement l'injectivité dans le Corollaire \ref{integrite_imparf} ou simplement la définition de la caractéristique (resp. les résultats du deuxième point du Corollaire \ref{coro_descente_1_imparf} et du Lemme \ref{inv_frob_multi_perf}).
		
	\underline{Condition 5 :} les inclusions $E_{\Delta} \subset E_{\Delta}^{\sep}$ et $\mathbb{F}_r \subset E_{\Delta}^{\sep}$ sont fidèlement plates. 
	
	\noindent La deuxième était déjà utilisée dans la Proposition \ref{defipropDdelta} et est conséquence du Lemme \ref{module_induit_multi_imparf}.
	
	\underline{Condition 6 :} l'élément $\pi$ est irréductible dans $\mathcal{O}_{\mathcal{E}_{\Delta}}$ et $\mathcal{O}_K$. 
	
	\noindent Cela équivaut à l'intégrité de $E_{\Delta}$ démontrée au Corollaire \ref{integrite_imparf} et à celle de $\mathbb{F}_r$.
	
	\underline{Condition 7 :} nous avons $\mathrm{K}_0(E_{\Delta})=\Z$ et $\mathrm{K}_0(\mathbb{F}_r)=\Z$. 
	
	\noindent La première est démontrée au Lemme \ref{Ktheory_ed}. De plus $\mathbb{F}_r$ est un corps.
	
	\underline{Condition 8 :} la topologie $\pi$-adique a de bonnes propriétés de $\pi$-dévissage. 
	
	\noindent C'est une remarque nous faisions dans \cite[Ex. 5.27]{formalisme_marquis}.
	
	\underline{Condition 9 :} les groupes de cohomologie de monoïdes $H^1_{\mathrm{cont}}(\cg_{E,\Delta}, E_{\Delta}^{\sep})$ et $H^1_{\mathrm{cont}}(\Phi_{\Delta,q,r},E_{\Delta}^{\sep})$ s'annulent. 
	
	\noindent Le deuxième équivaut à l'énoncé du Lemme \ref{coh_phideltaq} puisque tous les objets en jeu sont discrets. Pour le premier, nous utilisons le résultat du Lemme \ref{module_induit_multi_imparf} pour des extensions finies. Comme nous l'avions déjà évoqué pour démontrer la descente galoisienne, cela implique par base normale de Hilbert que chaque $F_{\Delta,q}$ est un $E_{\Delta}$-module induit pour l'action de $\prod \gal{F_{\alpha}}{E_{\alpha}}$, a fortiori acyclique. Nous en déduisons l'énoncé puisque $$H^1_{\mathrm{cont}}(\cg_{E,\Delta},E_{\Delta}^{\sep}):= \colim \lim_{F\in \mathcal{G}\mathrm{al}_E} H^1\left(\prod_{\alpha \in \Delta} \gal{F_{\alpha}}{E_{\alpha}}, \left(E_{\Delta}^{\sep}\right)^{\cg_{F,\Delta}}\right) =\colim \lim_{F\in \mathcal{G}\mathrm{al}_E} H^1\left(\prod_{\alpha \in \Delta} \gal{F_{\alpha}}{E_{\alpha}}, F_{\Delta,q}\right).$$
	
	\underline{Condition 10 :} les deux morphismes de comparaison sont des isomorphismes pour les objets de \linebreak$\cdetaleproj{\Phi_{\Delta,q,r}\times\cg_{E,\Delta}}{E_{\Delta}^{\sep}}$.
	
	\noindent Le troisième point du Corollaire \ref{coro_descente_1_imparf} prouve déjà l'isomorphisme de comparaison pour les invariants par $\cg_{E,\Delta}$. Comme cela implique que l'extension des scalaires et la prise de $\cg_{E,\Delta}$-invariants fournissent une équivalence de catégorie entre $\cdetaleproj{\Phi_{\Delta,q,r}}{E_{\Delta}}$ et $\cdetaleproj{\Phi_{\Delta,q,r}\times \cg_{E,\Delta}}{E_{\Delta}^{\sep}}$, le Théorème \ref{equiv_non_perf_car_p} et la décomposition des deux foncteurs implique que l'extension des scalaires et la prise de $\Phi_{\Delta,q,r}$-invariants fournissent une équivalence de catégorie entre $\mathrm{Rep}_{\mathcal{O}_K} \cg_{E,\Delta}$ et $\cdetaleproj{\Phi_{\Delta,q,r}\times \cg_{E,\Delta}}{E_{\Delta}^{\sep}}$. Ceci contient en particulier l'autre isomorphisme de comparaison.
		
	Une fois que nos deux foncteurs sont construits, démontrer qu'ils sont quasi-inverses l'un de l'autre se fait comme à la démonstration du Théorème \ref{equiv_perf_mod_p} en passant aux invariants l'isomorphisme de comparaison et en utilisant \cite[Prop. 4.17]{formalisme_marquis}. On utilise le Corollaire \ref{coro_descente_1_imparf} pour dire que $E_{\Delta}^{\sep}$ est sans $E_{\Delta}$-torsion.
\end{proof}
	
	\vspace{0.75cm}
	\subsection{Comment récupérer les équivalences de Carter-Kedlaya-Z\'abr\'adi ?}\label{retrouver_ckz}
	
	Deux versions des premières équivalences de Carter-Kedlaya-Z\'abr\'adi ont été démontrées. Précisément, le Théorème \ref{equiv_perf_mod_p} pour $r=p$ fournit un analogue de \cite[Th. 4.6]{zabradi_kedlaya_carter} où tous les corps sont identiques. 
	
	%De même, le Théorème \ref{equiv_lt} à venir est un analogue de \cite[Th. 4.39]{zabradi_kedlaya_carter} où nous considérerions uniquement des représentations à coefficients dans $\mathcal{O}_K$, et des lois de Lubin-Tate plutôt que les actions d'un sous-groupe de $\zptimes$ obtenues en choisissant $K(\mu_{p^{\infty}})$ comme dans l'article original \cite{Fontaine_equiv}. 
	
	Les versions de ce texte ont plusieurs avantages. Nous avons conservé l'obtention d'une équivalent pour les corps de caractéristique $p$ sans avoir besoin d'utiliser l'action supplémentaire de $\Gamma$, qui est également un avantage de \cite{zabradi_kedlaya_carter} par rapport à \cite{zabradi_equiv}. Enfin, ces équivalences, contrairement à celles de \cite{zabradi_kedlaya_carter}, considèrent des catégories de modules étales sur des anneaux \textit{intègres} ; les sous-sections précédents utilisent d'ailleurs ces questions d'intégrité pour faire fonctionner différemment les preuves\footnote{Voir les preuves de la Proposition \ref{coeur_preuve_perf} puis celle du Théorème \ref{morph_comparaison_coeur} qui s'en sert.}. Cette sous-section est consacrée à retrouver à partir de nos versions sur des anneaux intègres l'un des théorèmes de \cite{zabradi_kedlaya_carter}.
	
	Le lien entre les équivalences de ce texte et celles de \cite{zabradi_kedlaya_carter} repose sur le fait que les anneaux non intègres introduits par Carter-Kedaya-Z\'abr\'adi sont coinduits à partir des anneaux intègres considérés dans ce texte, idée déjà présentes pour se ramener au cas intègre dans le Théorème \ref{morph_comparaison_coeur}. De surcroît, les sous-sections précédentes illustrent que l'essentiel du travail consiste à obtenir une équivalence en modulo $p$ pour des corps perfectoïdes. Le reste consiste à dévisser, passer à la limite, ajouter des actions, en bref tout un tas de joyeuseries qui s'adaptent pour obtenir l'équivalence de votre choix. Pour ne pas surcharger ce texte, on se contente de retrouver une version imparfaite de \cite[Th. 4.6]{zabradi_kedlaya_carter} avec nos méthodes.
	
	Nous continuons de fixer un corps $E$ de caractéristique $p$, de valuation discrète, séparé complet pour cette valuation, et tel que la clôture algébrique de $\mathbb{F}_p$ dans le corps résiduel $k$ est finie. La complétion de sa perfection est notée $\widetilde{E}$.
	
	\begin{defi}
		En imitant les définitions \ref{def_constr_prod_imparf} et \ref{defi_anneau_sep_imparf} avec des produits tensoriels sur $\mathbb{F}_p$, nous définissons pour tout extension finie galoisienne $F|E$ les $\left(\Phi_{\Delta,p}\times \cg_{E,\Delta}\right)$-anneaux topologiques discrets $$F_{\Delta,p}^+:=\left(\bigotimes_{\alpha\in \Delta,\mathbb{F}_p} F_{\alpha}^+ \right)^{\wedge (\underline{X})}, \,\,\, F_{\Delta,p}:=F_{\Delta,p}^+\left[\frac{1}{X_{\Delta}}\right]$$
		 $$E_{\Delta,p}^{\sep}=\colim \limits_{F\in \mathcal{G}\mathrm{al}_{E}} F_{\Delta,p}.$$ 
	\end{defi}

	\begin{prop}\label{identif_coindu_imperf_sanspreuve}
		\begin{enumerate}[itemsep=0mm]
			\item Il existe un isomorphisme de $\Phi_{\Delta,p}$-anneaux topologiques discrets $$E_{\Delta,p}\cong \coindu{\Phi_{\Delta,q,p}}{\Phi_{\Delta,p}}{E_{\Delta}},$$ où la coinduite est munie de la topologie limite.
			
			\item Il existe un isomorphisme de $(\Phi_{\Delta,p} \times \cg_{E,\Delta})$-anneaux topologiques discrets $$E_{\Delta,p}^{\mathrm{sep}}\cong \coindu{\Phi_{\Delta,q,p}\times\cg_{E,\Delta}}{\Phi_{\Delta,p}\times \cg_{E,\Delta}}{E_{\Delta}^{\mathrm{sep}}},$$ où la coinduite est munie de la topologie limite.
		\end{enumerate}
	\end{prop}
	\begin{proof}
		Reléguée à la Proposition \ref{identif_coindu_imperf}, sous une écriture simplifiée.
	\end{proof}

	\begin{defiprop}\label{defipropDdeltackz}
		Le foncteur $$\mathbb{D}_{\Delta,\mathrm{CKZ}} \, : \, \mathrm{Rep}_{\mathbb{F}_p} \cg_{E,\Delta} \rightarrow \dmod{\Phi_{\Delta,p}}{E_{\Delta,p}}, \,\,\, V \mapsto \left(E^{\sep}_{\Delta,p} \otimes_{\mathbb{F}_p} V \right)^{\cg_{E,\Delta}}$$ est correctement défini, pleinement fidèle et son image essentielle est incluse dans $\detaleproj{\Phi_{\Delta,p}}{E_{\Delta,p}}$. Cette dernière catégorie est une sous-catégorie pleine monoïdale fermée et $\mathbb{D}_{\Delta,\mathrm{CKZ}}$ commute naturellement au produit tensoriel et au $\mathrm{Hom}$ interne.
	\end{defiprop}
	\begin{proof}
		Identique à la Propositon \ref{defipropDdelta} en démontrant des analogues des lemmes qui la précèdent pour $E_{\Delta,p}^{\sep}$.
	\end{proof}
	
	Une version imparfaite du Théorème 4.6 de \cite{zabradi_kedlaya_carter}, dans le cas où tous les corps perfectoïdes considérés sont isomorphes à $\widetilde{E}$, se formule comme suit :

	\begin{theo}\label{equiv_ckz}
		Le foncteur $$\mathbb{D}_{\Delta,\mathrm{CKZ}}\, : \,\mathrm{Rep}_{\mathbb{F}_p} \cg_{\widetilde{E},\Delta} \rightarrow \detaleproj{\Phi_{\Delta,p}}{E_{\Delta,p}}$$ est une équivalence de catégories monoïdales fermées. Un quasi-inverse est donné par $$\mathbb{V}_{\Delta,\mathrm{CKZ}}\, : \, \detaleproj{\Phi_{\Delta,p}}{E_{\Delta,p}} \rightarrow \mathrm{Rep}_{\mathbb{F}_p} \cg_{\widetilde{E},\Delta}, \,\,\, D \mapsto \left(E_{\Delta,p}^{\sep}\otimes_{E_{\Delta,p}} D\right)^{\Phi_{\Delta,p}}.$$
	\end{theo}
	\begin{proof}
	Considérons le diagramme suivant :
		
		\begin{center}
			\begin{tikzcd}
				\dmod{\Phi_{\Delta,q,p}}{E_{\Delta}} \ar[rr,"\mathrm{Coindu}_{\Phi_{\Delta,q,p}}^{\Phi_{\Delta,p}}"] & & \dmod{\Phi_{\Delta,p}}{E_{\Delta,p}} \\
				\detaleproj{\Phi_{\Delta,q,p}}{E_{\Delta}} \ar[rr,dotted] \ar[u,phantom,sloped,"\subset"] & & \detaleproj{\Phi_{\Delta,p}}{E_{\Delta,p}} \ar[u,phantom,sloped,"\subset"] \\
				\mathrm{Rep}_{\mathbb{F}_p} \cg_{\widetilde{E},\Delta} \arrow{u}[swap,anchor=center, rotate=90, yshift=-1.5ex]{\sim}{\mathbb{D}_{\Delta}} \ar[urr,hook,"\mathbb{D}_{\Delta,\mathrm{CKZ}}"']
			\end{tikzcd}
		\end{center}
		Le foncteur $\mathbb{D}_{\Delta}$ est une équivalence d'après le Théorème \ref{equiv_non_perf_car_p}. Le premier morphisme horizontal l'identification de $E_{\Delta,p}$ à $\coindu{\Phi_{\Delta,q,p}}{\Phi_{\Delta,p}}{E_{\Delta}}$ comme $\Phi_{\Delta,p}$-anneau à la Proposition \ref{identif_coindu_imperf} et par la construction dans \cite[Prop. 3.11]{formalisme_marquis}. La flèche en pointillé indique que $\detaleproj{\Phi_{\Delta,p}}{E_{\Delta,p}}$ est contenu dans l'image essentielle de la coinduction ; ceci découle du Lemme \ref{indice_subtil_fini_Phi} et de \cite[Prop. 3.19]{formalisme_marquis}.
		
		Prouvons que l'enveloppe de ce diagramme commute. Soit $V$ une représentation dans $\mathrm{Rep}_{\mathbb{F}_p} \cg_{\widetilde{E},\Delta}$, que nous voyons comme objet de $\detaleproj{\Phi_{\Delta,p}\times\cg_{\widetilde{E},\Delta}}{\mathbb{F}_p}$. La première étape de $\mathbb{D}_{\Delta}$ et $\mathbb{D}_{\Delta,\mathrm{CKZ}}$ consiste à étendre les scalaires. Nous appliquons \cite[Lem. 3.18]{formalisme_marquis} au sous-monoïde d'indice subtil fini $\Phi_{\Delta,q,p}<\Phi_{\Delta,p}$, au $\left(\Phi_{\Delta,p}\times \cg_{\widetilde{E},\Delta}\right)$-anneau $\mathbb{F}_p$, au $\left(\Phi_{\Delta,q,p}\times \cg_{\widetilde{E},\Delta}\right)$-anneau $E_{\Delta}^{\sep}$, à l'inclusion $\mathbb{F}_p\hookrightarrow E_{\Delta}^{\sep}$ et à la représentation $V$ ; on obtient un isomorphisme dans $\dmod{\Phi_{\Delta,p}\times \cg_{\widetilde{E},\Delta}}{E_{\Delta,p}^{\sep}}$ $$E_{\Delta,p}^{\mathrm{sep}} \otimes_{\mathbb{F}_p} V \cong \bigcoindu{\Phi_{\Delta,q,p}\times \cg_{\widetilde{E},\Delta}}{\Phi_{\Delta,p}\times \cg_{\widetilde{E},\Delta}}{E_{\Delta}^{\mathrm{sep}}} \otimes_{\mathbb{F}_p} V \cong \bigcoindu{\Phi_{\Delta,q,p}\times \cg_{\widetilde{E},\Delta}}{\Phi_{\Delta,p}\times \cg_{\widetilde{E},\Delta}}{E_{\Delta}^{\mathrm{sep}}\otimes_{\mathbb{F}_p} V}.$$ En passant aux invariants, on obtient un isomorphisme dans $\dmod{\Phi_{\Delta,p}}{E_{\Delta,p}}$ naturel en $V$ $$\mathbb{D}_{\Delta,\mathrm{CKZ}}(V):= \left(E_{\Delta,p}^{\mathrm{sep}} \otimes_{\mathbb{F}_p} V\right)^{\cg_{\widetilde{E},\Delta}} \cong \bigcoindu{\Phi_{\Delta,q,p}}{\Phi_{\Delta,p}}{\left(E_{\Delta}^{\mathrm{sep}}\otimes_{\mathbb{F}_p} V\right)^{\cg_{\widetilde{E},\Delta}}}=\bigcoindu{\Phi_{\Delta,q,p}}{\Phi_{\Delta,p}}{\mathbb{D}_{\Delta}(V)}.$$ C'est exactement la commutativité du diagramme que nous recherchions.
		
		La flèche pointillée démontre alors que $\mathbb{D}_{\Delta,\mathrm{CKZ}}$ est essentiellement surjective. On trouve l'expression d'un quasi-inverse comme pour le Théorème \ref{equiv_non_perf_car_p}.
	\end{proof}

	\begin{rem}
		Nous pourrions établir la version perfectoïde \cite[Th. 4.6]{zabradi_kedlaya_carter} de manière similaire. Malheureusement, l'anneau $\overline{R}$ dans l'article de Carter-Kedlaya-Z\'abr\'adi est un complété de la colimite $\widetilde{E}_{\Delta,p}^{\sep}$ que nos stratégies nous poussent à définir. Nous préférons ainsi éluder une telle preuve pour nous épargner un énoncé de décomplétion des invariants. Il est d'ailleurs à noter une subtilité dans \cite{zabradi_kedlaya_carter} puisque l'espace $\mathrm{Spa}(\overline{R},\overline{R}^+)$ obtenu par construction produit n'est pas pro-étale sur $\mathrm{Spa}(R,R^+)$. Pour appliquer la descente pro-étale, il faudrait remplacer $\overline{R}^+$ par la complétion $\varpi_{\Delta}$-adique de la colimite des constructions produits sur les extensions finies de $E$, et non sa complétion $(\underline{\varpi})$-adique.
	\end{rem}
	
	\vspace{1cm}
	
	\section{\'Equivalence de Fontaine Lubin-Tate multivariable }\label{partie_equiv_lt}
	
	Soit $\overline{\qp}|K|\qp$ une extension finie dont le corps résiduel est de cardinal $q$. Soit $\Delta$ un ensemble fini. Dans cette section, nous obtenons une équivalence de Fontaine Lubin-Tate multivariable laissée ouverte dans \cite{zabradi_kedlaya_carter}. Pour préparer l'équivalence glectique semi-linéaire, nous autorisons un choix de loi de Lubin-Tate différent pour chaque $\alpha \in \Delta$. Pour chaque $\alpha \in \Delta$, nous nous donnons une uniformisante $\pi_{\alpha}$ de $K$, un polynôme de Lubin-Tate $\rf_{\alpha}$ associé à $\pi_{\alpha}$ et un système de Lubin-Tate $\pi^{\flat}_{\alpha}$ associé à $\rf_{\alpha}$.

	\begin{defi}\label{defi_topo_oed}
		Considérons la situation de \S \ref{section_equiv_car_zero_imparf} pour $E=\mathbb{F}_q(\!(X)\!)$, pour l'extension $K=L$, l'ensemble fini $\Delta$, une famille $(E_{\alpha})$ de corps valués isomorphes à $E$ et la famille de $\mathcal{O}_K$-algèbres $\pi$-adiquement séparées et complètes d'anneau résiduel $E_{\alpha}^+$ en $\pi$ données comme à la Définition \ref{defi_oe} par $$\mathcal{O}_{\mathcal{E}_{\alpha}}=\left(\mathcal{O}_K \llbracket X_{\alpha}\rrbracket \left[\frac{1}{X_{\alpha}}\right]\right)^{\wedge \pi},$$ munies du relèvement $\mathcal{O}_K$-linéaire du $q$-Frobenius vérifiant $\phi_{\alpha,q}(X_{\alpha})=\rf_{\alpha}(X_{\alpha})$. Dans cette sous-section, nous équippons les anneaux des Définitions \ref{defi_oed} et \ref{defi_oeddeltasep} de topologies plus variées.
		
		Nous munissons $\hcal{O}_{\hcal{E}_{\Delta}}^+$ de trois topologies : la topologie $\pi$-adique, la topologie $(\pi,X_{\Delta})$-adique et la topologie $(\pi,\underline{X})$-adique.
		
		Nous munissons $\hcal{O}_{\hcal{E}_{\Delta}}$ de trois topologies d'anneau
		
		Premièrement, la topologie appelée \textit{topologie adique faible} ayant pour base de voisinages de $0$ les sous-groupes $\left(\pi^n \hcal{O}_{\hcal{E}_{\Delta}} + X_{\Delta}^m \hcal{O}_{\hcal{E}_{\Delta}}^+\right)_{n,m\geq 0}$.
		
		La deuxième topologie considérée est celle obtenue en écrivant $$\hcal{O}_{\hcal{E}_{\Delta}}= \lim \limits_{n\geq 0} \,\,\colim \limits_{m\geq 0}\, \quot{X_{\Delta}^{-m}\hcal{O}_{\hcal{E}_{\Delta}}^+}{\pi^n X_{\Delta}^{-m} \hcal{O}_{\hcal{E}_{\Delta}}^+}$$ et en munissant chaque $X_{\Delta}^{-m} \hcal{O}_{\hcal{E}_{\Delta}}^+$ de la topologie initiale venant de la topologie $(\pi,\und{X})$-adique sur $\hcal{O}_{\hcal{E}_{\Delta}}^+$.  Nous l'appelons la \textit{topologie colimite}.
		
		La dernière topologie, que nous nommons \textit{topologie faible}, a pour base de voisinages de zéro les \linebreak$\left(\pi^n \hcal{O}_{\hcal{E}_{\Delta}} +\sum_{\und{d} \in \Z^{\Delta},  \Sigma \,\und{d} \geq m}\,\, \und{X}^{\und{d}} \hcal{O}_{\hcal{E}_{\Delta}}^+\right)_{n,m\geq 0}$. Il s'agit de la topologie de \cite[Lem. 2.15]{pupazan}.
		
		Nous définissons les mêmes topologies pour les $\hcal{O}_{\hcal{F}_{\Delta}}$ et pour $\Oedhat$.
	\end{defi}
	
	\begin{defi}\label{defi_oedk}
		 Pour toute extension finie galoisienne $E^{\sep}|F|E$, considérons les $(\varphi_{\alpha,q}^{\N}\!\times\! \cg_{E_{\alpha}})$-anneaux topologiques $\mathcal{O}_{\mathcal{F}_{\alpha}}^+$ de la Définition \ref{defi_oed}. Les plongements dans $\widetilde{A}_K$ construits à la Proposition \ref{construct_oek} à partir\footnote{Et de choix de clôtures séparables de chaque $\mathbb{F}_q(\!(X_{\alpha}^{1/p^{\infty}})\!)$ et d'isomorphismes avec $\bigcup_{[K'\!:\!K_{\lt}]<\infty} (K' \widehat{K_{\lt}})^{\flat}$. Puisque nous avions déjà choisi une identification des clôtures séparables de $\widetilde{E}_{\alpha}$ et $\widetilde{E}$, cela revient à faire un unique choix pour $E$.} des $\pi_{\alpha}^{\flat}$ les munissent d'une structure de $(\varphi_{\alpha,q}^{\N}\!\times \!\cg_K)$-anneaux topologiques pour la topologie faible que l'on appelle $\mathcal{O}_{\mathcal{F}_{K,\alpha}}^+$. L'action de $(\Phi_{\Delta,q}\!\times \!\cg_{K,\Delta})$ terme à terme sur le produit tensoriel des $\mathcal{O}_{\mathcal{F}_{K,\alpha}}^+$ est $(\pi,\und{X})$-adiquement continue. Elle se complète et localise sur $\mathcal{O}_{\mathcal{F}_{\Delta}}$ en une structure de  $(\Phi_{\Delta,q}\!\times \!\cg_{K,\Delta})$-anneau topologique pour les topologies adique faible, colimite et faible que nous appelons $\mathcal{O}_{\mathcal{F}_{K,\Delta}}$.
		 
		 Les injections de la Définition \ref{defi_oeddeltasep} sont $\cg_{K,\Delta}$-équivariantes. Cela permet de passer à la colimite et compléter en une structure de $(\Phi_{\Delta,q}\!\times\! \cg_{K,\Delta})$-anneau sur $\mathcal{O}_{\widehat{\mathcal{E}_{\Delta}^{\mathrm{nr}}}}$ que l'on nomme $\Oehatkand{\Delta}$. C'est une structure de \linebreak$(\Phi_{\Delta,q}\times \cg_{K,\Delta})$-anneau topologique pour chacune des trois topologies ci-dessus. Comme dans le paragraphe qui précède la Définition \ref{defi_oe}, après choix des $\pi_{\alpha}^{\flat}$ et d'une extension de $j$, le groupe $\cg_{E,\Delta}$ s'identifie canoniquement à un sous-groupe de $\ch_{K,\lt,\Delta}$ de $\cg_{K,\Delta}$. Via cette identification, le $(\Phi_{\Delta,q}\!\times\! \cg_{E,\Delta})$-anneau sous-jacent à $\Oehatkand{\Delta}$ est $\mathcal{O}_{\widehat{\mathcal{E}_{\Delta}^{\mathrm{nr}}}}$. 
		 
		 Nous notons également $\Gamma_{K,\lt,\Delta}:= \prod_{\alpha \in \Delta} \mathcal{O}_K^{\times}$  et l'identifions canoniquement à $\prod_{\alpha\in \Delta} \gal{K_{\lt,\rf_{\alpha}}}{K}$ et $\cg_{K,\Delta}/\ch_{K,\lt,\Delta}$.
	\end{defi}

	\begin{rem}
		Nous pointons à nouveau que $(\pi,\und{X})^n\hcal{O}_{\hcal{E}_{\Delta}}^+$ ne fournissent pas une base de voisinages de zéro d'une structure d'anneau topologique sur $\hcal{O}_{\hcal{E}_{\Delta}}$. Les topologies colimites et faibles sont deux bonnes alternatives pour lesquelles chaque $X_{\alpha}$ est topologiquement nilpotent.
	\end{rem}

	\begin{rem}\label{descr_oekand}
		Nous pouvons décrire explicitement le $(\Phi_{\Delta,q}\times \Gamma_{K,\lt,\Delta})$-anneau obtenu. Nous avons $$\Oekand{\Delta}^+\cong \mathcal{O}_K \llbracket X_{\alpha} \, |\, \alpha \in \Delta\rrbracket \, \,\, \text{et}\,\,\,\Oekand{\Delta} \cong \left(\mathcal{O}_K \llbracket X_{\alpha} \, |\, \alpha \in \Delta\rrbracket \left[ \frac{1}{X_{\Delta}}\right] \right)^{\wedge \pi}.$$ L'action de $(\Phi_{\Delta,q}\times \Gamma_{K,\lt,\Delta})$ est donnée par la seule action de $\mathcal{O}_K$-algèbre topologique telle que $$\left( \varphi_{\alpha,q}^{n_{\alpha}}\right)_{\alpha}(X_{\beta})=\rf_{\beta}^{\circ n_{\beta}}(X_{\beta})$$
		
		$$\forall (x_{\alpha})_{\alpha} \in \Gamma_{K,\lt,\Delta}, \,\,\, (x_{\alpha})_{\alpha}(X_{\beta})=[x_{\beta}]_{\lt,\rf_{\beta}}(X_{\beta}).$$
	\end{rem}
	
	Pour obtenir une équivalence de Fontaine multivariable Lubin-Tate, il nous reste à analyser les topologies adique faible, colimite et faible, dont nous n'avions pas l'usage pour les corps de caractéristique $p$. Cela conduit à se replonger dans les topologies adiques sur nos anneaux de caractéristique $p$, également évincées des précédentes équivalences.
	
	\begin{lemma}\label{topo_faibles}
		Pour toute famille $\mathfrak{D}\!\subset\! \Z^{\Delta}$, nous avons les égalités suivantes.	
		\begin{enumerate}[itemsep=0mm]
			\item Dans $E_{\Delta}^{\sep}$, l'égalité $$\left(\sum_{\und{d}\in \mathfrak{D}} \und{X}^{\und{d}} E_{\Delta}^{\sep,+}\right)\cap E_{\Delta}=\sum_{\und{d}\in \mathfrak{D}} \und{X}^{\und{d}} E_{\Delta}^+.$$
			
			\item Pour $n\geq 0$, l'égalité dans $\mathcal{O}_{\widehat{\mathcal{E}_{\Delta}^{\mathrm{nr}}}}$ $$\left(\pi^n \mathcal{O}_{\widehat{\mathcal{E}_{\Delta}^{\mathrm{nr}}}} + \sum_{\und{d}\in \mathfrak{D}} \und{X}^{\und{d}} \mathcal{O}_{\widehat{\mathcal{E}_{\Delta}^{\mathrm{nr}}}}^+\right) \cap \mathcal{O}_{\mathcal{E}_{\Delta}}=\pi^n \mathcal{O}_{\mathcal{E}_{\Delta}} + \sum_{\und{d}\in \mathfrak{D}} \und{X}^{\und{d}} \mathcal{O}_{\mathcal{E}_{\Delta}}^+.$$ 
			
			\item Pour $n\geq 0$ et si pour tout $\und{d} \in \mathfrak{D}$ on a $\sum \und{d}> 0$, l'égalité dans $\mathcal{O}_{\widehat{\mathcal{E}_{\Delta}^{\mathrm{nr}}}}$ $$\left(\pi^n \mathcal{O}_{\widehat{\mathcal{E}_{\Delta}^{\mathrm{nr}}}} + \sum_{\und{d}\in \mathfrak{D}} \und{X}^{\und{d}} \mathcal{O}_{\widehat{\mathcal{E}_{\Delta}^{\mathrm{nr}}}}^+\right) \cap \mathcal{O}_K=\pi^n\mathcal{O}_K.$$ 
		\end{enumerate}
	\end{lemma}
	\begin{proof}
		\begin{enumerate}[itemsep=0mm]
			\item On se ramène à le prouver pour chaque $F_{\Delta,q}^+$. Pendant la démonstration du Lemme \ref{module_induit_multi_imparf}, nous avions démontré un isomorphisme de $E_{\Delta}^+$-modules $$F_{\alpha}^{\circ}\otimes_{E_{\alpha}^{\circ}} \left( F_{\beta}^{\circ} \otimes_{E_{\beta}^{\circ}} \cdots \left( F_{\delta}^{\circ}\otimes_{E_{\delta}^{\circ}} E_{\Delta}^+\right)\right) \cong F_{\Delta,q}^+.$$ Puisque chaque $F_{\alpha}^{\circ}$ est un $E_{\alpha}^{\circ}$-module libre de type fini pour lequel $\{1\}$ se complète en une base\footnote{Pour ces deux affirmations, considérer que pour une extension finie non ramifiée ou totalement ramifiée, il existe un élément primitif tel que $F^{\circ}=E^{\circ}[x]$. Toute extension finie se décompose comme extension totalement ramifiée d'une extension non ramifiée.}, il existe une $E_{\Delta}$-base de $F_{\Delta,q}$ de la forme $\{1\}\sqcup\mathcal{B}$ telle que $F_{\Delta,q}^+=E_{\Delta}^+ \oplus \bigoplus_{b\in \mathcal{B}} E_{\Delta}^+ b$. Comme chaque $E_{\Delta} b$ est un sous-$E_{\Delta}$-module, un élément de $F_{\Delta,q}$ appartient à $\sum_{\und{d}\in \mathfrak{D}} \und{X}^{\und{d}} F_{\Delta,q}^+$ si et seulement si toutes ses coordonnées appartiennent à $\sum_{\und{d}\in \mathfrak{D}} \und{X}^{\und{d}} E_{\Delta}^+$. Cela conclut.
			
			\item On établit le résultat par récurrence sur $n$. Si $n\!=\!0$, le résultat est vide. Supposons le résultat vérifie pour un $n\geq 0$ et soit $x$ dans l'intersection pour $(n+1)$. Nous savons que $(x \modulo{\pi})$ appartient à $\big(\sum_{\und{d}\in \mathfrak{D}} \und{X}^{\und{d}} E_{\Delta}^{\sep,+}\big) \cap E_{\Delta}$, donc à $\sum_{\und{d}\in \mathfrak{D}} \und{X}^{\und{d}} E_{\Delta}^+$ par le premier point. Ainsi, il existe une famille $(y_{\und{d}})$ d'éléments presque tous nuls de $\mathcal{O}_{\mathcal{E}_{\Delta}}^+$ telle que $x-\sum_{\und{d}\in \mathfrak{D}} \und{X}^{\und{d}} y_{\und{d}}\! \equiv \! 0 \modulo{\pi}$. 
			
			Grâce au premier point du Corollaire \ref{coro_descente_1_imparf}, $E_{\Delta}^{\sep}$ est sans $X_{\Delta}^m$-torsion ; nous savons aussi que $\mathcal{O}_{\widehat{\mathcal{E}_{\Delta}^{\mathrm{nr}}}}$ est sans $\pi$-torsion (voir la condition 1 dans la démonstration du Théorème \ref{equiv_car_zero}). Il en découle que $$X_{\Delta}^n \mathcal{O}_{\widehat{\mathcal{E}_{\Delta}^{\mathrm{nr}}}}^+ \cap \pi \mathcal{O}_{\widehat{\mathcal{E}_{\Delta}^{\mathrm{nr}}}}=\pi X_{\Delta}^n \mathcal{O}_{\widehat{\mathcal{E}_{\Delta}^{\mathrm{nr}}}}^+.$$
			
			Soit $F|E$ une extension finie d'indice de ramification $e$ et de degré d'inertie $f$. On utilise la description de $\hcal{O}_{\hcal{F}_{\Delta}}^+$ à la Remarque \ref{descr_oekand}. Soit $z\in \sum_{\und{d}\in \mathfrak{D}} \und{X}^{\und{d}} \hcal{O}_{\hcal{F}_{\Delta}}^+$ qui est multiple de $\pi$. Fixons une famille $z_{\und{d}}$ d'éléments presque tous nuls de $\hcal{O}_{\hcal{F}_{\Delta}}^+$ tels que $z\!=\!\sum_{\und{d}\in \mathfrak{D}} \und{X'}^{e\und{d}} z_{\und{d}}$. On choisit $\mathfrak{D}_z$ fini contenant l'ensemble des indices où le coefficient est non nul. \'Ecrivons $z_{\und{d}}\!=\!\sum_{\und{i} \in \N^{\Delta}} \und{X}^{\und{i}} a_{\und{i},\und{d}} \und{X'}^{\und{i}}$ avec $a_{\und{i},\und{d}}\!\in\! \left(\otimes_{\alpha \in \Delta, \hcal{O}_K} \hcal{O}_{K_{q^f}}\right)$ avec seulement un nombre fini de termes non multiples par $\pi^k$. L'hypothèse de divisibilité par $\pi$ affirme que $$\forall \und{j}\in \Z^{\Delta}, \,\, \pi \,\, \text{divise} \,\,\sum_{\und{i}+\und{d}=\und{j}} a_{\und{i},\und{d}}.$$ On se donne un ordre lexicographique sur $\Delta$ qui détermine un ordre total sur $\Z^{\Delta}$. Pour chaque couple $(\und{i},\und{d})$ tel que $a_{\und{i},\und{d}}$ est non nul, on peut choisir $\und{j}(\und{i},\und{d}) \in \mathfrak{D}_z$ minimal tel que $\und{X'}^{(e\und{j})} | \und{X'}^{(\und{i}+e\und{d})}$. Alors $$z =\sum_{\und{j}\in \mathfrak{D}_z} \und{X'}^{(e\und{j})} \left(\sum_{(\und{i},\und{d}) \, \, \text{tel que} \,\, \und{j}(\und{i},\und{d})=\und{j}} a_{\und{i},\und{j}} \und{X'}^{\und{i}+ e(\und{d}-\und{j})}\right).$$ En séparant selon $\und{i}+\und{d}$ on se comprend que chaque somme sur $(\und{i},\und{j})$ est divisible par $\pi$. Nous avons donc prouvé que  $$\left(\sum_{\und{d}\in \mathfrak{D}} \und{X}^{\und{d}} \hcal{O}_{\hcal{F}_{\Delta}}^+ \right)\cap \pi \hcal{O}_{\hcal{F}_{\Delta}}^+ = \sum_{\und{d}\in \mathfrak{D}} \und{X}^{\und{d}} \pi \hcal{O}_{\hcal{F}_{\Delta}}^+.$$
			
			On reprend notre hérédité en choissant une extension finie $F|E$ pour laquelle on peut écrire \linebreak$x\!=\!\pi^{n+1} z + \sum_{\und{d}\in \mathfrak{D}} \und{X}^{\und{d}} z_{\und{d}}$ avec $z_{\und{d}}\in \mathcal{O}_{\mathcal{F}_{\Delta}}^+$ presque tous nuls. Nous obtenons que $$\pi \,\, \text{divise}\,\, \sum_{\und{d}\in \mathfrak{D}} \und{X}^{\und{d}}(z_{\und{d}}-y_{\und{d}})$$ donc que cette somme appartient à $\pi \sum_{\und{d}\in \mathfrak{D}} \und{X}^{\und{d}} \mathcal{O}_{\mathcal{F}_{\Delta}}^+$. Ainsi $$\left(x-\sum_{\und{d}\in \mathfrak{D}} \und{X}^{\und{d}} y_{\und{d}}\right) \in \pi\left(\pi^n \mathcal{O}_{\widehat{\mathcal{E}_{\Delta}^{\mathrm{nr}}}} + \sum_{\und{d}\in \mathfrak{D}} \und{X}^{\und{d}} \mathcal{O}_{\widehat{\mathcal{E}_{\Delta}^{\mathrm{nr}}}}^+\right)$$ et l'hypothèse de récurrence conclut.
			
			\item Raisonnement similaire.
		\end{enumerate}
	\end{proof}
	
	\begin{theo}\label{equiv_lt}
		Les foncteurs \begin{align*} \mathbb{D}_{\Delta,\lt}\, : \,\mathrm{Rep}_{\mathcal{O}_K} \cg_{K,\Delta} &\rightarrow \dmod{\Phi_{\Delta,q}\!\times\! \Gamma_{K,\lt,\Delta}}{\Oekand{\Delta}} \\ V &\mapsto \left(\Oehatkand{\Delta} \otimes_{\mathcal{O}_K} V\right)^{\ch_{K,\lt,\Delta}} \end{align*}
		\begin{align*} \mathbb{V}_{\Delta,\lt}\, : \, \cdetaledvproj{\Phi_{\Delta,q}\!\times\! \Gamma_{K,\lt,\Delta}}{\Oekand{\Delta}}{\pi} &\rightarrow \dmod{\cg_{K,\Delta}}{\mathcal{O}_K} \\ D&\mapsto \left( \Oehatkand{\Delta} \otimes_{\Oekand{\Delta}} D\right)^{\Phi_{\Delta,q}}\end{align*} où les topologies en jeu sont respectivement la topologie $\pi$-adique sur $\mathcal{O}_K$ et l'une des trois topologies adique faible, colimite et faible sur $\Oekand{\Delta}$, sont correctement définis, lax monoïdaux et fermés. Leurs images essentielles sont contenues respectivement dans $\cdetaledvproj{\Phi_{\Delta,q}\!\times\! \Gamma_{K,\lt,\Delta}}{\Oekand{\Delta}}{\pi}$ et $\mathrm{Rep}_{\mathcal{O}_K} \cg_{K,\Delta}$. Leurs corestrictions forment une paire de foncteurs quasi-inverses.
	\end{theo}
	\begin{proof}
		La démonstration est similaire au Théorème \ref{equiv_car_zero} et utilise même des résultats démontrés audit théorème. Ici, nous avons crucialement besoin de la variante des $\mathcal{S}$-modules topologiques sur $R$ à $r$-dévissage projectif introduits dans \cite[\S 4]{formalisme_marquis}. Nous appelons topologie choisie l'une des trois topologies adique faible, colimite et faible que nous choisissons pour le reste de la preuve. Nous décomposons alors le foncteur $\mathbb{D}_{\Delta,\lt}$ comme d'habitude comme $\mathrm{Inv}\circ \mathrm{Ex}\circ \mathrm{triv}$ mais nous voudrions à présent qu'il se corestreigne comme la composée suivante :	
		\begin{center}
			\begin{tikzcd}[row sep=small] \cdetaledvproj{\cg_{K,\Delta}}{\mathcal{O}_K}{\pi} \ar[d,"\mathrm{triv}"] \\ \cdetaledvproj{\Phi_{\Delta,q}\!\times\! \cg_{K,\Delta}}{\mathcal{O}_K}{\pi} \ar[d,"\mathrm{Ex}"] \\
				\cdetaledvproj{\Phi_{\Delta,q}\!\times\! \cg_{K,\Delta},\ch_{K,\lt,\Delta}}{\Oehatkand{\Delta}}{\pi} \ar[d,"\mathrm{Inv}"] \\
				\cdetaledvproj{\Phi_{\Delta,q}\!\times\!\Gamma_{K,\lt,\Delta}}{\Oekand{\Delta}}{\pi}
			\end{tikzcd}
		\end{center} où les topologies sont : la topologie $\pi$-adique sur la première ligne, la topologie choisie pour le premier monoïde et $\pi$-adique pour le sous-monoïde $\ch_{K,\lt,\Delta}$ en bas à droite, la topologie choisie en bas à gauche. De même,  le foncteur $\mathbb{V}_{\Delta,\lt}$ se décompose comme $\mathrm{Inv}\circ \mathrm{Ex}\circ \mathrm{triv}$ et nous voudrions qu'il corestreigne comme la composée suivante :
		
		\begin{center}
			\begin{tikzcd}[row sep=small]
				\cdetaledvproj{\Phi_{\Delta,q}\times\Gamma_{K,\lt,\Delta},\Phi_{\Delta,q}}{\Oekand{\Delta}}{\pi} \ar[d,"\mathrm{triv}"] \\ \cdetaledvproj{\Phi_{\Delta,q}\times \cg_{K,\Delta},\Phi_{\Delta,q}}{\Oekand{\Delta}}{\pi} \ar[d,"\mathrm{Ex}"] \\
				\cdetaledvproj{\Phi_{\Delta,q}\times \cg_{K,\Delta}, \Phi_{\Delta,q}}{\Oehatkand{\Delta}}{\pi} \ar[d,"\mathrm{Inv}"] \\ \cdetaledvproj{\cg_{K,\Delta}}{\mathcal{O}_K}{\pi} 
			\end{tikzcd}
		\end{center} où les topologies sont la topologie choisie pour les premiers monoïdes et $\pi$-adiques pour les seconds. On remarquera que, puisque la topologie sur $\Phi_{\Delta,q}$ est discrète et que son action sur chaque anneau est continue pour les deux topologies, nous avons bien coïncidence des deux sous-catégories $$\cdetaledvproj{\Phi_{\Delta,q}\times\Gamma_{K,\lt,\Delta},\Phi_{\Delta,q}}{\Oekand{\Delta}}{\pi} \,\, \text{et} \,\,\cdetaledvproj{\Phi_{\Delta,q}\times\Gamma_{K,\lt,\Delta}}{\Oekand{\Delta}}{\pi}.$$
		
		Pour justifier la définition correcte, nous devons vérifier une liste de conditions pour utiliser les mêmes résultats de \cite{formalisme_marquis} qu'au Théorème \ref{equiv_car_zero}.
		
		\underline{Conditions 1, 2, 5, 6, 7 et 8 :} ce sont les mêmes conditions puisque les anneaux sous-jacents sont ceux du Théorème \ref{equiv_car_zero}.
		
		\underline{Condition 3 :} les inclusions $\mathcal{O}_K \!\subset\! \Oehatkand{\Delta}$ et $\Oekand{\Delta} \!\subset\! \Oehatkand{\Delta}$ sont continues pour les topologies $\pi$-adiques et faibles. Elles sont $(\Phi_{\Delta,q}\!\times\! \cg_{K,\Delta})$-équivariantes. C'est le cas par construction.
		
		\underline{Condition 4 :} les anneaux topologiques $\Oehatkand{\Delta}^{\ch_{K,\lt,\Delta}}$ et $\Oehatkand{\Delta}^{\Phi_{\Delta,q}}$ coïncident avec $\Oekand{\Delta}$ et $\Phi_{\Delta,q}$ pour les topologies $\pi$-adique et choisie. Seule la coïncidence des topologies choisies n'est pas contenue ad verbatim dans la condition analogue du Théorème \ref{equiv_car_zero}. En considérant la description de leur voisinages de zéro (ou de celle donnée pour la topologie colimite en démontrant la Définition/Proposition \ref{defi_oedk}), on réalise que nous avons même démontré une version forte de l'identification de ces topologies au Lemme \ref{topo_faibles}.
		
		\underline{Condition 9 :} identique à la condition 9 du Théorème \ref{equiv_car_zero}. Nous soulignons cependant que les énoncés cohomologiques requis par \cite[Th. 5.26]{formalisme_marquis} ne concernent que la topologie discrète.
		
		\underline{Condition 10 :} les morphismes de comparaison sont exactement ceux dans $\cdetaleproj{\Phi_{\Delta,q}\!\times\! \cg_{E,\Delta}}{E_{\Delta}^{\sep}}$, sauf qu'ils ont le bon goût d'être aussi $\cg_{K,\Delta}$-équivariants. La condition se déduit donc immédiatement de celle au Théorème \ref{equiv_car_zero}.
	\end{proof}
	
	Nous voyons que la démonstration du théorème à partir de l'équivalence pour les corps de caractéristique $p$ revenait à construire les bons anneaux, choisir les bonnes topologies et utiliser le formalisme général.
	
	\begin{rem}
		En utilisant la même stratégie que \cite[Prop. 2.2]{zabradi_equiv} mais en remplaçant l'utilisation de \cite[Prop. 2.1]{zabradi_gln} par la généralisation de Grosse-Klönne au cas Lubin-Tate dans \cite{grosseklonne}, nous obtenons que \linebreak$\cdetaleproj{\Phi_{\Delta,q}\times \Gamma_{K,\lt,\Delta}}{E_{\Delta}}=\cdetale{\Phi_{\Delta,q}\times \Gamma_{K,\lt,\Delta}}{E_{\Delta}}$. Nous en déduisons que les sous-catégories \linebreak$\cdetaledvproj{\Phi_{\Delta,q}\times \Gamma_{K,\lt,\Delta}}{\Oekand{\Delta}}{\pi}$ et $\cdetale{\Phi_{\Delta,q}\times \Gamma_{K,\lt,\Delta}}{\Oekand{\Delta}}$ coïncident. Il est cependant intéressant de se souvenir de la décomposition des modules à $(\pi,\mu)$-dévissage projectif obtenu en \cite[Coro. A.8]{formalisme_marquis}. En utilisant que la continuité d'une $\hcal{O}_K$-représentation de $\cg_{K,\Delta}$ se teste restreint à $\ch_{K,\lt,\Delta}$, \cite{formalisme_marquis} permet de prouver que l'hypothèse de continuité est superflue.
	\end{rem}
	
	\begin{rem}
		Pour obtenir une version Lubin-Tate multivariable de l'équivalence, il était crucial de pouvoir considérer le produit tensoriel sur $\mathcal{O}_K$ et d'obtenir un anneau dont les invariants par $\Phi_{\Delta,q}$ valent $\mathcal{O}_K$. Avec l'anneau non intègre de \cite{zabradi_kedlaya_carter}, il n'est pas clair de savoir comment relever en caractéristique mixte nos anneaux, et même si nous y parvenions par exemple pour une extension non ramifiée, les invariants par $\Phi_{\Delta,p}$ seraient $\zp$ et non $\mathcal{O}_K$. Le problème est encore plus visible modulo $p$ : il faut trouver un sous-monoïde de $M<\Phi_{\Delta,p}$ par lequel quotienter, qui fasse encore marcher le formalisme d'une équivalence de Fontaine et tel que $E_{\Delta,p}^M=\mathbb{F}_q$. Mais de quel $\mathbb{F}_q$ est-il question ?
	\end{rem}

	\vspace{1.5 cm}
	\section{Variantes pour les groupes de Galois plectiques et glectiques}\label{partie_plectique}
	
	Dans \cite{nekoscholl_intro} et \cite{nekoschollI}, J. Nekov\'a\v{r} et T. Scholl formulent différentes conjectures plectiques. Elles prédisent que, pour un corps de nombres $F$, les représentations obtenues dans la cohomologie de la donnée de Shimura associée à la restriction à $\qp$ d'un groupe algébrique sur $F$, qui viennent naturellement avec une action de $\cg_{\qp}$, devraient hériter en réalité d'une action d'un groupe plus gros $\cg_{F,\plec}$. Nous appelons ce dernier \textit{groupe de Galois plectique}. Cette section établit des équivalences de Fontaine pour le groupe de Galois plectique d'un corps local $p$-adique et pour l'un de ses sous-groupes que nous introduisons : le groupe de Galois glectique. Comme pour déduire l'équivalence multivariable Lubin-Tate de l'équivalence pour les corps de caractéristiques $p$, il reste essentiellement à définir un anneau de comparaison correct dans chacun des cas et un monoïde adéquat agissant dessus.
	
	\vspace{0.75 cm}
	\subsection{Préliminaires sur le groupe de Galois plectique}
	
	Soit $\overline{\qp}|K|\qp$ une extension finie. 
	
	\begin{defi}
		Nous définissons le \textit{groupe de Galois plectique} $$\cg_{K,\plec}:= \mathrm{Aut}_K\left(K\otimes_{\qp}\overline{\qp}\right).$$Nous définissons une structure de groupe topologique en prenant pour base de voisinages du neutre les intersections finies de stabilisateurs.
	\end{defi}
	\vspace{0.25cm}
	
	Nous notons $\mathcal{P}=\{\text{plongements} \,\,\, \tau \, : \, K \rightarrow \overline{\qp}\}$, de cardinal $[K:\qp]$, sur lequel $\cg_{\qp}$ agit par post-composition. Pour tout plongement $\tau$, nous notons $K_{\tau}$ son image. Enfin, pour chaque plongement $\tau$, on choisit une extension $\tau_{\ext} \, : \, \overline{\qp} \xrightarrow[]{\sim} \overline{\qp}$. Il existe un isomorphisme d'anneaux canonique $$K\otimes_{\qp}\overline{\qp} \xrightarrow{\sim} \bigoplus_{\tau \in \mathcal{P}} \,\overline{\qp}, \,\,\,\, x\otimes y \mapsto (\tau(x)y)_{\tau\in \mathcal{P}}.$$ La structure de $K$-algèbre est donnée par $\tau$ sur la composante indexée par $\tau$. Considérons le monomorphisme $$\prod \cg_{K_{\tau}} \rightarrow \cg_{K,\plec}, \,\,\, (g_{\tau}) \mapsto \left[ (y_{\tau})\mapsto (g_{\tau}(y_{\tau}))\right].$$ 
	Le choix des $\tau_{\ext}$ permet de tordre ce monomorphisme en $$\cg_{K,\mathcal{P}} \rightarrow \cg_{K,\plec}, \,\,\, (g_{\tau}) \mapsto \left[(y_{\tau})\mapsto (\tau_{\ext} g_{\tau} \tau_{\ext}^{-1})(y_{\tau})\right].$$ Ce choix détermine également un morphisme $$\mathfrak{S}_{\mathcal{P}} \rightarrow \cg_{K,\plec}, \,\,\, \omega \mapsto \left[(y_{\tau}) \mapsto ((\tau_{\ext} (\omega^{-1}(\tau))_{\ext}^{-1})(y_{\omega^{-1}(\tau)}))\right].$$ Si la formule est absconte, nous pouvons la résumer en disant que la copie indexée par $\tau$ de $\overline{\qp}$ est envoyée sur la copie indexée $\omega(\tau)$, de la seule manière possible obtenue à partir des $\tau_{\ext}$ et qui respecte les structures de $K$-algèbres. 
	
	\begin{lemma}
		L'image de $\cg_{K,\mathcal{P}}$ est distinguée dans $\cg_{K,\plec}$ et a pour complément l'image de $\mathfrak{S}_{\mathcal{P}}$. Cela exhibe le groupe topologique $\cg_{K,\plec}$ comme produit en couronne 
		
		$$\cg_{K,\plec} \cong \cg_{K,\mathcal{P}} \rtimes_{\plec} \mathfrak{S}_{\mathcal{P}},$$ où $\plec(\omega)(g_{\tau})=(g_{\omega^{-1}(\tau)})$.
	\end{lemma}
	
	Cette description est la plus concise en termes de notations. Elle repose en revanche drastiquement sur le choix des $\tau_{\ext}$ et l'isomorphisme n'est donc pas canonique. L'image de $\prod \cg_{K_{\tau}}$ est quant à elle canonique, de même que la projection sur $\mathfrak{S}_{\mathcal{P}}$. De plus, agir sur le terme de droite de $K\otimes_{\qp}\overline{\qp}$ produit un plongement canonique de $\cg_{\qp}\hookrightarrow \cg_{K,\plec}$.
	
	\begin{defi}
		Nous appelons \textit{groupe de Galois glectique}, et notons $\cg_{K,\glec}$, le sous-groupe de $\cg_{K,\plec}$ engendré par $\prod \cg_{K_{\tau}}$ et par $\cg_{\qp}$.
	\end{defi}
	
	Appelons $K_{\gaal}$ l'extension composée des $K_{\tau}$. C'est une extension galoisienne de $\qp$ et le noyau de l'action de $\cg_{\qp}$ sur $\mathcal{P}$ est exactement $\cg_{K_{\gaal}}$. Nous en déduisons un monomorphisme de groupes $\mathrm{Per}\, : \, \gal{K_{\gaal}}{\qp} \rightarrow \mathfrak{S}_{\mathcal{P}}$. Dans toute la section \ref{partie_plectique}, on notera aussi $f_{\gaal}$ le degré d'inertie de $K_{\gaal}|K$.
	
	\begin{prop}\label{descrip_gkglec}
	La projection de $\cg_{K,\glec}$ sur $\mathfrak{S}_{\mathcal{P}}$ correspond l'image de $\mathrm{Per}$. Le groupe $\cg_{K,\glec}$ est isomorphe à $\cg_{K,\mathcal{P}}\rtimes_{\plec \, \circ \, \mathrm{Per}} \gal{K_{\gaal}}{\qp}$, canoniquement au choix des $\tau_{\ext}$ près.
		
	En prenant les notations de la Définition \ref{defi_psd_quot}, le groupe $\cg_{K,\glec}$ est canoniquement isomorphe à $$\quot{\left(\prod_{\tau \in \mathcal{P}}\cg_{K_{\tau}}\rtimes_{\glec} \cg_{\qp}\right)}{\tsim \cg_{K_{\gaal}}\tsim}$$ où $\glec(\sigma)(g_{\tau})=(\sigma g_{\sigma^{-1}\tau} \sigma^{-1})$ et où $\cg_{K_{\gaal}}$ est plongé diagonalement dans $\prod \cg_{K_{\tau}}$.
	\end{prop}
	\begin{proof}
		Puisque l'image de $\prod \cg_{K_{\tau}}$ est distinguée dans $\cg_{K,\plec}$, les sous-groupes contenant $\prod \cg_{K_{\tau}}$ correspondent aux sous-groupes du quotient $\sfrac{\cg_{K,\plec}}{\prod \cg_{K_{\tau}}} \cong \mathfrak{S}_{\mathcal{P}}$. Soit $\sigma \in \cg_{\qp}$ et $\tau \in \mathcal{P}$. On considère $\sum_i x_i \otimes y_i$ appartenant au facteur $\tau$. Cela signifie que $$\forall \tau'\neq \tau, \, \sum_i \tau'(x_i)y_i=0.$$ Alors $\sigma(\sum x_i \otimes y_i)=\sum x_i \otimes \sigma(y_i)$ vérifie $$\forall \tau' \neq \sigma \tau, \,\,  \sum \tau'(x_i) \sigma(y_i) = \sigma(\sum (\sigma^{-1}\tau')(x_i) y_i))=0.$$ Donc l'image de $\sigma$ dans $\mathfrak{S}_{\mathcal{P}}$ est donnée par son action sur $\mathcal{P}$.
			
		Nous avons déjà construit deux morphismes de $\prod \cg_{K_{\tau}}$ et $\cg_{\qp}$ vers $\cg_{K,\glec}$. Pour montrer qu'ils induisent un morphisme depuis le quotient du produit semi-direct, nous vérifions les conditions de la Définition \ref{defi_psd_quot}. La première est aisée. Pour la deuxième, on a que $K_{\gaal}|\qp$ galoisienne implique $\cg_{K_{\gaal}}\triangleleft \cg_{\qp}$. Pour la dernière, soit $\sigma \in \cg_{\qp}$, $\chi \in \cg_{K_{\gaal}}$ et $(y_{\tau})_{\tau}\in \oplus \overline{\qp}$ :
	$$
			\left[\glec(\sigma)(\chi)\right](y_{\tau})_{\tau}=(\sigma \chi \sigma^{-1})(y_{\tau})_{\tau} = \left((\sigma \chi \sigma^{-1})(y_{\tau})\right)_{\tau}= \left[(\sigma)(\chi)(\sigma^{-1})\right](y_{\tau})_{\tau}.$$
		
Nous laissons les lecteurs et lectrices vérifier qu'il s'agit d'un isomorphisme.
	\end{proof}

	\vspace{1cm}
	\subsection{\'Equivalence de Fontaine plectique}

	Pour commencer, nous établissons une équivalence pour des représentations du groupe de Galois plectique. Comme pour l'équivalence multivariable, il suffit d'étendre les actions en jeu dans le Théorème \ref{equiv_car_zero}. Dans cette sous-section, toutes les constructions de monoïdes sont automatiquement munies de topologies. Pour un petit bestiaire de constructions et de lemmes sur les monoïdes topologiques, se référer à l'Annexe \ref{annexe_monoides}.
	
	\medskip
	
	Soit $\overline{\qp}|K|\qp$ une extension finie\footnote{Nous fixons ainsi un plongement de $K$ dans $\overline{\qp}$. Il sera parfois utile pour intuiter les formules correctes de se dire que $K$ vit à part dans une autre clôture algébrique, contrairement aux $K_{\tau}$ qui vivent dans $\overline{\qp}$.}. Nous nous fixons une uniformisante $\pi$ de $K$, une loi de Lubin-Tate $\mathrm{f}$ associée à $\pi$ et un système de Lubin-Tate $\pi^{\flat}$.
	
	Nous commençons dans le cadre de la section \ref{partie_equiv_lt} avec $\Delta=\mathcal{P}$, le choix de $(\pi,\mathrm{f},\pi^{\flat})$ pour chaque plongement, le choix d'un $\mathcal{O}_{\widehat{\mathcal{E}^{\mathrm{nr}}_{\tau}}}$ pour chaque plongement et d'une extension aux clôtures séparables de l'isomorphisme $$\mathrm{j}_{\tau}\, : \,\mathbb{F}_q(\!(X_{\tau}^{q^{-\infty}})\!) \rightarrow \widehat{K_{\lt,\mathrm{f}}}^{\flat}, \,\,\, X_{\tau}\mapsto \pi^{\flat}.$$ Ces données déterminent comme à la Proposition \ref{construct_oek} une structure de $(\varphi_q^{\N}\times \cg_K)$-anneau topologique sur $\mathcal{O}_{\widehat{\mathcal{E}^{\mathrm{nr}}_{\tau}}}$ et d'un plongement équivariant $i_{\tau}$ dans $\mathrm{W}_{\mathcal{O}_K}(\mathbb{C}_p^{\flat})$ dont l'image de dépend pas de $\tau$. Pour tout couple $(\tau_1,\tau_2)$, nous notons $i_{\tau_2,\tau_1}\, : \, \mathcal{O}_{\widehat{\mathcal{E}^{\mathrm{nr}}_{\tau_1}}} \rightarrow \mathcal{O}_{\widehat{\mathcal{E}^{\mathrm{nr}}_{\tau_2}}}$ l'isomorphisme de $(\varphi_q^{\N}\times \cg_K)$-anneaux topologiques donné par $i_{\tau_2}^{-1}  i_{\tau_1}$. Pour toute extension galoisienne $E^{\sep}|F|E$, le morphisme $i_{\tau_2,\tau_1}$ envoie $\mathcal{O}_{\mathcal{F}_{\tau_1}}$ sur  $\mathcal{O}_{\mathcal{F}_{\tau_2}}$. De plus, ces isomorphismes sont $\mathcal{O}_K$-linéaires et vérifient \begin{equation*}\label{eqn:eq2}
		\forall \tau_1, \tau_2, \tau_3, \,\,\, i_{\tau_3,\tau_1}=i_{\tau_3,\tau_2}  i_{\tau_2,\tau_1}. \tag{*2} 	\end{equation*}
		
Dans cette sous-section, nous ne travaillons qu'avec la topologie adique faible, mais les deux autres conviendraient également.
	
	\begin{defi}
		Soit $\omega \in \mathfrak{S}_{\mathcal{P}}$. Pour chaque extension finie $E^{\sep}|F|E$, on définit sur $\otimes_{\tau\in \mathcal{P},\, \mathcal{O}_K} \mathcal{O}_{\mathcal{F}_{\tau}}^+$ l'endomorphisme $$\omega \cdot_{\plec} \left(\otimes\, y_{\tau}\right)= \otimes\, i_{\tau,\omega^{-1}(\tau)}(y_{\omega^{-1}(\tau)}).$$ Puisque chaque $i_{\tau_2,\tau_1}$ est continu pour la topologie adique faible, cet endomorphisme est continu pour la topologie $(\pi,\underline{X})$-adique, se complète, passe à la colimite et se complète $\pi$-adiquement en un endomorphisme $\omega \cdot_{\plec}\text{-}\,$ sur $\Oehatkand{\mathcal{P}}$.
	\end{defi}
	
	\begin{defi}
		Nous appelons $\tgplec$ le monoïde topologique défini par $$\tgplec:= \left(\Phi_{\mathcal{P},q}\times \cg_{K,\mathcal{P}}\right) \rtimes_{\plec} \mathfrak{S}_{\mathcal{P}}$$ où le produit semi-direct est donné par $$\forall \omega\in \mathfrak{S}_{\mathcal{P}}, \,\, \, \plec(\omega)\left(\prod \varphi_{\tau,q}^{n_{\tau}}, g_{\tau}\right)= \left(\prod \varphi_{\tau,q}^{n_{\omega^{-1}(\tau)}}, g_{\omega^{-1}(\tau)}\right)$$ avec la topologie usuelle sur $\left(\Phi_{\mathcal{P},q}\times \cg_{K,\mathcal{P}}\right)$ et la topologie discrète sur $\mathfrak{S}_{\mathcal{P}}$.
	\end{defi}
	
	\begin{defiprop}
		La définition précédente fournit une action continue pour la topologie adique faible de $\mathfrak{S}_{\mathcal{P}}$. Combinée à l'action de $(\Phi_{\mathcal{P},q}\times \cg_{K,\mathcal{P}})$ sur $\Oehatkand{\plec}$, nous obtenons une action de $\tgplec$ continue pour la topologie adique faible. 
		
		Définissons $\Oehatkand{\plec}$ le $\tgplec$-anneau topologique obtenu.
	\end{defiprop}
	\begin{proof}
		D'après le deuxième point des Propositions \ref{psd_quotient_desc_1} et \ref{prop_prdsdirect}, il suffit de vérifier des relations sur les actions. C'est le cas sur les produits tensoriels en utilisant à la fois la relation (\ref{eqn:eq2}) et le fait que les $i_{\tau_2,\tau_1}$ sont des morphismes $(\varphi_q^{\N}\times \cg_K)$-équivariants. Les relations sont conservées en complétant et passant à la limite.
	\end{proof}
	
	\begin{prop}
		Le sous-monoïde $\Phi_{\mathcal{P},q}<\tgplec$ est distingué et le quotient $\sfrac{\tgplec}{\Phi_{\mathcal{P},q}}$ s'identifie à $\cg_{K,\plec}$ comme monoïde topologique, canoniquement au vu du choix des $\tau_{\ext}$. De plus, l'inclusion $$\mathcal{O}_K \subseteq \mathcal{O}_{\widehat{\mathcal{E}_{K,\plec}}}^{\Phi_{\mathcal{P},q}}$$ est une égalité.
	\end{prop}
	\begin{proof}
		Pour le caractère distingué, remarquer que les endomorphismes de $(\Phi_{\mathcal{P},q}\times \cg_{K,\mathcal{P}})$ définissant le produit semi-direct avec $\mathfrak{S}_{\mathcal{P}}$ induisent des endomorphismes surjectifs de $\Phi_{\Delta,q}$. L'identification de $\cg_{K,\plec}$ au sous-monoïde ouvert et distingué $\cg_{K,\mathcal{P}}\rtimes_{\plec} \mathfrak{S}_{\mathcal{P}}$, qui est un complément de $\Phi_{\Delta,q}$ conclut pour l'identification du quotient. Quant aux invariants, il s'agit exactement de la condition 4 dans la preuve du Théorème \ref{equiv_car_zero}.
	\end{proof}

	Nous définissons de manière ad hoc l'anneau plectique puis décrivons un peu mieux nos constructions.
	
	\begin{defi}
		Le sous-monoïde $\ch_{K,\lt,\mathcal{P}}<\tgplec$ est distingué. Nous définissons le monoïde topologique $$\tplec:= \sfrac{\tgplec}{\ch_{K,\lt,\mathcal{P}}}.$$ Nous définissons aussi le $\tplec$-anneau topologique $$\Oekand{\plec}:= \Oehatkand{\plec}^{\ch_{K,\lt,\mathcal{P}}}.$$
	\end{defi}
	
	\begin{prop}\label{identif_tgplec}
		Le monoïde topologique $\tplec$ s'identifie canoniquement à $$\left(\Phi_{\mathcal{P},q}\times \Gamma_{K,\lt,\mathcal{P}}\right) \rtimes_{\plec} \mathfrak{S}_{\mathcal{P}}$$ où $$\forall \omega \in \mathfrak{S}_{\mathcal{P}}, \,\,\,\plec(\omega)\left(\prod \varphi_{\tau,q}^{n_{\tau}}, x_{\tau}\right)=\left(\prod \varphi_{\tau,q}^{n_{\omega^{-1}(\tau)}}, x_{\omega^{-1}(\tau)}\right).$$
		
		Le $\left(\Phi_{\mathcal{P},q}\times \Gamma_{K,\lt,\mathcal{P}}\right)$-anneau topologique sous-jacent à $\Oekand{\plec}$ coïncide avec $\Oekand{\mathcal{P}}$ construit à la section \ref{partie_equiv_lt}. En reprenant sa description à la remarque \ref{descr_oekand}, l'action de $\omega \in \mathfrak{S}_{\mathcal{P}}$ est le seul morphisme de $\mathcal{O}_K$-algèbre topologique vérifiant $\forall \tau \in \mathcal{P}, \,\,\, \omega \cdot X_{\tau}= X_{\omega(\tau)}$.
	\end{prop}
	\begin{proof}
		On utilise le troisième point de la Proposition \ref{prop_prdsdirect} pour la description du quotient et la condition 4 dans la démonstration du Théorème \ref{equiv_car_zero} pour l'identification de $\Oekand{\plec}$. Pour écrire l'action, considérons que par définition chaque $X_{\tau}$ est envoyé par $i_{\tau}$ sur $\{\pi^{\flat}\}_{\lt} \in \widetilde{A}_K$. Ainsi, $i_{\tau_2,\tau_1}(X_{\tau_1})=X_{\tau_2}$ et l'écriture (\ref{eqn:eq2}) conclut.
	\end{proof}
	
	Une fois les anneaux construits correctement, nous obtenons gratuitement l'équivalence de Fontaine plectique.
	
	\begin{theo}\label{equiv_plec}
			Les foncteurs $$\mathbb{D}_{\plec,\lt}\, : \,\mathrm{Rep}_{\mathcal{O}_K} \cg_{K,\plec} \rightarrow \dmod{\tplec}{\Oekand{\plec}}, \,\,\, V\mapsto \left(\Oehatkand{\plec} \otimes_{\mathcal{O}_K} V\right)^{\ch_{K,\lt,\mathcal{P}}}$$
	
	$$\mathbb{V}_{\plec,\lt}\, : \, \cdetaledvproj{\tplec}{\Oekand{\plec}}{\pi} \rightarrow \dmod{\cg_{K,\plec}}{\mathcal{O}_K}, \,\,\, D\mapsto \left( \Oehatkand{\plec} \otimes_{\Oekand{\plec}} D\right)^{\Phi_{\Delta,q}},$$ où les topologies en jeu sont respectivement la topologie $\pi$-adique sur $\mathcal{O}_K$ et la topologie faible sur $\Oekand{\plec}$, sont correctement définis, lax monoïdaux et fermés. Leurs images essentielles sont contenues respectivement dans $\cdetaledvproj{\tplec}{\Oekand{\plec}}{\pi}$ et $\mathrm{Rep}_{\mathcal{O}_K} \cg_{K,\plec}$. Leurs corestrictions forment une paire de foncteurs quasi-inverses.
	\end{theo}
	\begin{proof}
		La stratégie est la même qu'au Théorème \ref{equiv_lt} et les conditions à vérifier découlent des constructions ci-dessus et des conditions vérifiées pour le Théorème \ref{equiv_lt}.
	\end{proof}
	
	\begin{rem}
		Avec des stratégies similiaires, on aurait une équivalence entre les fibrés vectoriels équivariants sur $Z_{\lt}$ comme dans \cite{breuilandco_phi} et les $\mathbb{F}_q$-représentations continues de dimension finie de $\cg_{K,\plec}$. Nous escomptons y revenir dans un travail ultérieur.
	\end{rem}

	\vspace{0.75cm}
	\subsection{\'Equivalence de Fontaine glectique semi-linéaire pour des extensions galoisiennes de $\qp$}

	Pour $K|\qp$ galoisienne, une philosophie adjacente se détache en regardant un peu mieux $\widetilde{A}_K$. Nous écrivons $\widetilde{A}_K=\mathcal{O}_K \otimes_{\mathbb{Z}_q} \mathrm{W}(\mathbb{C}_p^{\flat})$ et considérons l'action $\mathcal{O}_K$-semi-linéaire de $\sigma \in\mathrm{W}_{\qp}^+$ par $\sigma\otimes \varphi^{\deg(\sigma)}$. L'image du plongement de $\Oe$ est stable par l'action de $\mathrm{W}_K^+$ mais pas nécessairement par celle de $\mathrm{W}_{\qp}^+$. En choisissant pour chaque plongement $\tau$ une extension $\tau_{\ext}\in \mathrm{W}_{\qp}^+$ de degré minimal, nous pouvons considérer le morphisme $$\mathcal{O}_{\mathcal{E}_{\mathcal{P}}}^+ \rightarrow \widetilde{A}_K, \,\,\, X_{\tau}\mapsto \tau_{\ext} \cdot\{\pi^{\flat}\}_{\lt}.$$ Son image est la plus petite sous-$\mathcal{O}_K$-algèbre fermée de $\widetilde{A}_K$ contenant l'un des $\{(\pi^{\flat})\}_{\lt}$ et stable par $\mathrm{W}_{\qp}^+$. Le morphisme est injectif\footnote{Pour $K=\mathbb{Q}_q$ et $\pi=p$, le résultat de L. Berger \cite[Coro. 3.7]{berger_inj_morph} implique en particulier l'injectivité de notre morphisme. On peut généraliser cette technique à une extension finie quelconque.} ce qui s'interprète comme l'apparition de $|\mathcal{P}|$ variables cachées. Cela suggère également que nous devrions voir apparaître sur notre anneau multivariable $\mathcal{O}_{\mathcal{E}_{\mathcal{P}}}$ une action additionnelle de $\mathrm{W}_{\qp}^+$ et peut-être une équivalence de Fontaine reliant les $(\varphi,\Gamma)$-modules sur cet anneau à une certaine catégorie de représentations. Cette sous-section répond à ces attentes.
	
	Soit $K|\qp$ une sous-extensions finie de $\overline{\qp}$. Nous fixons encore un choix d'extension $\tau_{\ext}$ de chaque plongement à $\overline{\qp}$, que l'on suppose appartenir à $\mathrm{W}_{\qp}^+$ et de degré minimal.
	
		\begin{defiprop}
		Nous avons $$\forall \tau\in \mathcal{P}, \,\forall\sigma \in \cg_{\qp}, \,\, \exists ! \,g_{\sigma,\tau}\in \cg_K,\,\,\,(\sigma\tau)_{\ext} g_{\sigma,\tau}=\sigma \tau_{\ext}.$$
		
		Si $\sigma \in \mathrm{W}_{\qp}^+$, alors $g_{\sigma,\tau} \in \mathrm{W}_K^+$. Son degré absolu est noté $f d_{\sigma,\tau}$ avec $d_{\sigma,\tau}\in \mathbb{N}$ et ne dépend pas du choix des $\tau_{\ext}$ puisqu'ils sont de degré positif minimal.
	\end{defiprop}
	\begin{proof}
		L'existence provient de ce que $\sigma \tau_{\ext}$ coïncide avec $(\sigma \tau)_{\ext}$ sur $K$. Ainsi, l'élément $(\sigma\tau)_{\ext}^{-1} \sigma \tau_{\ext}$ fixe $K$. Si $\sigma \in \mathrm{W}_{\qp}^+$, le degré de $\sigma \tau_{\ext}$ est positif. Or, c'est aussi une extension de $\sigma \tau$, d'où la le choix des extensions tire que $\deg(\sigma \tau_{\ext}) \geq \deg\left((\sigma \tau)_{\ext}\right)$.
	\end{proof}
	
	\begin{rem}
		La définition précédente capture donc l'information des degrés dans $\mathrm{W}_{\qp}^+$ que le choix des $\tau_{\ext}$ oublie. Cela servira entre autre à garantir que l'action sur l'anneau de Fontaine plectique d'une extensions du $p$-Frobenius dans $\cg_{\qp}$ est une racine du $q$-Frobenius.
	\end{rem}
	
	\begin{lemma}
		Les $g_{\sigma,\tau}$ vérifient la relation
		
		$$\forall \sigma_1, \sigma_2 \in \cg_{\qp}, \, \forall \tau \in \mathcal{P}, \,\,\, g_{\sigma_2\sigma_1,\tau}=g_{\sigma_2,\sigma_1 \tau} g_{\sigma_1,\tau}.$$
		
		A fortiori, les $d_{\sigma,\tau}$ vérifient la relation
		
		$$\forall \sigma_1, \sigma_2 \in \cg_{\qp}, \, \forall \tau \in \mathcal{P}, \,\,\, d_{\sigma_2\sigma_1,\tau}=d_{\sigma_2,\sigma_1 \tau} + d_{\sigma_1,\tau}.$$
	\end{lemma}
	\begin{proof}
		On calcule
		\begin{align*}
			(\sigma_2 \sigma_1 \tau)_{\ext}g_{\sigma_2,\sigma_1 \tau} g_{\sigma_1,\tau} &= (\sigma_2(\sigma_1 \tau))_{\ext} g_{\sigma_2,\sigma_1 \tau} g_{\sigma_1,\tau} \\
			&=\sigma_2 (\sigma_1 \tau)_{\ext} g_{\sigma_1,\tau} \\
			&= \sigma_2 \sigma_1 \tau_{\ext}
		\end{align*}
	\end{proof}

	Par la suite, le groupe de Weil $\mathrm{W}_{L}$ d'un corps local $L$ sera muni de la topologie d'union disjointe $\cup_{n\in \Z} \mathrm{Frob}^{\circ n} \cdot \mathrm{I}_L$ et sa version monoïdale $\mathrm{W}_L^+$ de même. Nous appelons \textit{topologie localement profinie} cette topologie.
	\vspace{0.5cm}

	\begin{defiprop}
		Nous appelons $\tgglec$ le monoïde topologique défini en utilisant \ref{defi_psd_quot} par $$\tgglec:= \quot{\left(\left(\Phi_{\mathcal{P},q}\times \cg_{K,\mathcal{P}}\right) \rtimes_{\glec} \sfrac{\mathrm{W}_{\qp}^+}{\mathrm{I}_{K_{\gaal}}}\right)}{\tsim \sfrac{\mathrm{W}_{K_{\gaal}}^+}{\mathrm{I}_{K_{\gaal}}}\tsim}$$ où le produit semi-direct est donné par $$\forall \sigma\in \mathrm{W}_{\qp}^+, \,\, \, \glec(\sigma)\left(\prod \varphi_{\tau,q}^{n_{\tau}}, g_{\tau}\right)= \left(\prod \varphi_{\tau,q}^{n_{\omega^{-1}(\tau)}}, g_{\omega^{-1}(\tau)}\right)$$ et avec le morphisme $\kappa_{\glec}\, : \,\sfrac{\mathrm{W}_{K_{\gaal}}^+}{\mathrm{I}_{K_{\gaal}}}\rightarrow \left(\Phi_{\mathcal{P},q}\times \cg_{K,\mathcal{P}}\right)$ qui envoie le Frobenius arithmétique sur $\varphi_{\mathcal{P},q}^{f_{\gaal}}$, avec la topologie usuelle sur $\left(\Phi_{\mathcal{P},q}\times \cg_{K,\mathcal{P}}\right)$ et la topologie localement profinie sur le groupe de Weil.
	\end{defiprop}
	\begin{proof}
		Pour les mêmes raisons que précédemment, le noyau de $\glec$ contient $\mathrm{I}_{K_{\gaal}}$. Il faut ensuite vérifier les trois conditions de la Définition \ref{defi_psd_quot}. Le Frobenius arithmétique est dans le noyau de $\plec$ et donne donc la même "action de conjugaison" que $\varphi_{\mathcal{P},q}^{f_{\gaal}}$. Puisque $K_{\gaal}|\qp$ est galoisienne $\mathrm{W}_{K_{\gaal}}\triangleleft \mathrm{W}_{\qp}$ d'où nous déduisons que $\sfrac{\mathrm{W}_{K_{\gaal}}^+}{\mathrm{I}_{K_{\gaal}}}\triangleleft \sfrac{\mathrm{W}_{\qp}^+}{\mathrm{I}_{K_{\gaal}}}$. Ce dernier sous-groupe est dans le centre, de même que son image dans $\left(\Phi_{\mathcal{P},q} \times \cg_{K,\mathcal{P}}\right)$. Cela démontre la troisième condition.
	\end{proof}

	\begin{lemma}
		La restriction de la projection sur produit semi-direct vers $\tgglec$ se restreint en des injections depuis $\left(\Phi_{\mathcal{P},q}\times \cg_{K,\mathcal{P}}\right)$ et $\sfrac{\mathrm{W}_{\qp}^+}{\mathrm{I}_{K_{\gaal}}}$. Ce sont des homéomorphismes sur leur image.
	\end{lemma}
	\begin{proof}
		Ajoutons que $\glec(\mathrm{W}_{\qp}^+)$ est constituée d'automorphismes pour appliquer la Proposition \ref{psd_quotient_desc_2}. Cela permet notamment de prouver que les éléments de $\sfrac{\mathrm{W}_{\qp}^+}{\mathrm{I}_{K_{\gaal}}}$ sont dans des classes disjointes. De plus, le sous-groupe $\cg_{K,\mathcal{P}}$ est ouvert dans le produit semi-direct et stable par la relation définissant le quotient. Nous en déduisons que la topologie induite sur $\sfrac{\mathrm{W}_{\qp}^+}{\mathrm{I}_{K_{\gaal}}}$ par l'injection est effectivement discrète. L'autre plongement se démontre de manière analogue.
	\end{proof}
	
	\begin{prop}
		Le sous-monoïde $\Phi_{\mathcal{P},q}<\tgglec$ est distingué et le quotient $\sfrac{\tgglec}{\Phi_{\mathcal{P},q}}$ s'identifie à $\cg_{K,\glec}$ comme monoïde topologique, canoniquement au choix des $\tau_{\ext}$ près.
	\end{prop}
	\begin{proof}
		En utilisant le premier résultat de la Proposition \ref{identités_quotients} grâce au fait que $\mathrm{W}_{\qp}^+$ agit par automorphismes sur $\Phi_{\mathcal{P},q}$, on obtient $\Phi_{\mathcal{P},q} \triangleleft \left[\left(\Phi_{\mathcal{P},q} \times \cg_{K,\mathcal{P}}\right)\rtimes_{\glec} \sfrac{\mathrm{W}_{\qp}^+}{\mathrm{I}_{K_{\gaal}}}\right]$. On utilise le deuxième résultat de la même proposition pour conclure que son image via le plongement dans $\tgglec$ est distinguée. Enfin, le troisième point de ladite proposition permet d'identifier le quotient. Puisque $\kappa_{\glec}$ prend ses valeurs dans $\Phi_{\mathcal{P},q}$ nous obtenons que $$\quot{\tgglec}{\Phi_{\mathcal{P},q}} \cong \cg_{K,\mathcal{P}} \rtimes_{\overline{\glec}} \quot{\mathrm{W}_{\qp}^+}{\mathrm{W}_{K_{\gaal}}^+}$$ où $\overline{\glec}(\sigma)(g_{\tau})=(g_{\sigma^{-1}\tau})$. En identifiant $\sfrac{\mathrm{W}_{\qp}^+}{\mathrm{W}_{K_{\gaal}}^+}$ à $\gal{K_{\gaal}}{\qp}$, nous retombons sur la description de $\cg_{K,\glec}$ donné par le choix des $\tau_{\ext}$ à la Proposition \ref{descrip_gkglec}.
	\end{proof}
	
	\begin{defi}
		Le sous-monoïde $\ch_{K,\lt,\mathcal{P}}<\tgglec$ est distingué pour les mêmes raisons qu'à la proposition précédente. Nous définissons le monoïde topologique $$\tglec:= \sfrac{\tgglec}{\ch_{K,\lt,\mathcal{P}}}.$$ 
	\end{defi}
	
	\begin{prop}		
		Le monoïde topologique $\tglec$ s'identifie canonique à $$\quot{\left(\left(\Phi_{\mathcal{P},q}\times \Gamma_{K,\lt,\mathcal{P}}\right) \rtimes_{\glec} \sfrac{\mathrm{W}_{\qp}^+}{\mathrm{I}_{K_{\gaal}}}\right)}{\tsim \sfrac{\mathrm{W}_{K_{\gaal}}^+}{\mathrm{I}_{K_{\gaal}}}\tsim}$$ où $$\forall \sigma \in \mathrm{W}_{\qp}^+, \,\,\,\glec(\sigma)\left(\prod \varphi_{\tau,q}^{n_{\tau}}, x_{\tau}\right)=\left(\prod \varphi_{\tau,q}^{n_{\sigma^{-1}\tau}}, x_{\omega^{-1}(\tau)}\right)$$ et où l'on identifie le Frobenius arithmétique à $\varphi_{\mathcal{P},q}^{f_{\gaal}}$.
	\end{prop}
	\begin{proof}
		Similaire à celle de la Proposition \ref{identif_tgplec}.
	\end{proof}

	\vspace{0.5 cm}
	Supposons jusqu'à la fin de cette sous-section que $K|\qp$ est galoisienne. Nous identifions canoniquement $\mathcal{P}$ à $\gal{K}{\qp}$ dans cette sous-section. Dans ce cas, nous avons $K_{\gaal}=K$ et $f_{\gaal}=f$.
	
	\begin{defi}
		Le quotient $\sfrac{\cg_{K,\glec}}{\prod_{\gal{K}{\qp}} \cg_{K}}$ s'identifiant à $\gal{K}{\qp}$, l'action de ce dernier sur $\mathcal{O}_K$ le munit d'une structure de $\cg_{K,\glec}$-anneau topologique que l'on note $\mathcal{O}_{K,\mathrm{sl}}$.
		
		Nous notons $\mathrm{sRep}_{\mathcal{O}_K} \cg_{K,\plec}$ la catégorie des représentations\footnote{Elle s'écrit $\cdetaledvproj{\cg_{K,\plec}}{\mathcal{O}_{K,\mathrm{sl}}}{\pi}$ dans le langage de \cite{formalisme_marquis}.} semi-linéaires continues de $\cg_{K,\plec}$ de type fini sur $\mathcal{O}_K$.
	\end{defi}
	
	Nous considérons toujours un choix de $(\pi,\rf,\pi^{\flat})$ pour $K$. Nous commençons à nouveau dans le cadre de la section \ref{partie_equiv_lt} avec $\Delta=\mathcal{P}$, le choix de $(\tau(\pi), \tau(\rf), \tau_{\ext}(\pi^{\flat}))$ pour le plongement $\tau$  et un choix d'extensions de $$\mathrm{j}_{\tau} \, : \, \mathbb{F}_q(\!(X_{\tau}^{q^{-\infty}})\!) \rightarrow \widehat{K_{\lt,\tau(\rf)}}^{\flat}$$ aux clôtures séparables. Comme à la section \ref{partie_equiv_lt}, cela détermine une structure de $(\Phi_{\mathcal{P},q}\times \cg_{K,\mathcal{P}})$-anneau topologique sur $\Oehatkand{\mathcal{P}}$ et une identification de $\cg_{E,\mathcal{P}}$ au sous-groupe $\ch_{K,\lt,\mathcal{P}}=\prod \cg_{K_{\lt,\tau(\rf)}} < \cg_{K,\mathcal{P}}$.
	
	Chaque $\tau_{\ext}$ induit un endomorphisme $\tau$-semi-linéaire de $K$-algèbres de $\mathbb{C}_p$ qui envoie $K_{\lt,\rf}$ sur $K_{\lt,\tau(\rf)}$. En basculant, on obtient  $\tau_{\ext}^{\flat}$ un endomorphisme de $\mathbb{F}_q$-algèbres $\phi^{\deg \tau}$-semi-linéaire\footnote{Ici $\phi$ est le Frobenius absolu et $\deg \tau$ est bien défini modulo $f$, ce qui définit correctement $\phi^{\deg \tau}$ sur $\mathbb{F}_q$.} de $\mathbb{C}_p^{\flat}$ qui envoie $\widehat{K_{\lt,\rf}}^{\flat}$ sur $\widehat{K_{\lt,\tau(\rf)}}^{\flat}$ et $\pi^{\flat}$ sur $\tau_{\ext}(\pi^{\flat})$. Cet endomorphisme est $\varphi_q^{\N}$-équivariant. En revanche, il n'est pas $\cg_K$-équivariant : nous obtenons que \begin{equation*}\label{eqn:relationpasequiv} \forall g\in \cg_K, \,  \forall x \,\,\, \tau_{\ext}^{\flat}(g\cdot x)=(\tau_{\ext} g \tau_{\ext}^{-1}) \cdot(\tau_{\ext}^{\flat}(x)). \tag{*3}\end{equation*} En passant aux vecteurs de Witt, on obtient un diagramme commutatif
	\begin{center}
		\begin{tikzcd}
			\Oe \ar[d,"X\mapsto X_{\tau}"'] \ar[rrr,hook,"X\mapsto \lbrace\pi^{\flat}\rbrace_{\lt,\rf}"] & & & \mathrm{W}_{\mathcal{O}_K}\left(\widehat{K_{\lt,\rf}}^{\flat}\right) \ar[d] \ar[r,hook] & \mathrm{W}_{\mathcal{O}_K}\left(\mathbb{C}_p^{\flat}\right) \ar[d,"\mathrm{W}_{\mathcal{O}_K}(\tau_{\ext}^{\flat})"] \\ 
			\mathcal{O}_{\mathcal{E}_{\tau}} \ar[rrr,hook,"X_{\tau}\mapsto \lbrace\tau_{\ext}(\pi^{\flat})\rbrace_{\lt,\tau(\rf)}"] & & & \mathrm{W}_{\mathcal{O}_K}\left(\widehat{K_{\lt,\tau(\rf)}}^{\flat}\right) \ar[r,hook] & \mathrm{W}_{\mathcal{O}_K}\left(\mathbb{C}_p^{\flat}\right)
		\end{tikzcd}
	\end{center} où toutes les flèches verticales sont des morphismes $\tau$-semi-linéaires de $\mathcal{O}_K$-algèbres. Or, $\Oehat$ est la complétion de la hensélisation stricte de $\Oe$. Le diagramme précédent montre alors que $\mathrm{W}_{\mathcal{O}_K}(\tau_{\ext}^{\flat})$ induit un isorphisme entre les $\varphi_{q}^{\N}$-anneaux topologiques $\Oehat$ et $\mathcal{O}_{\widehat{\mathcal{E}_{\tau}^{\mathrm{nr}}}}$ construits comme en Proposition \ref{prop_defi_oe}, respectivement pour $(\pi,\rf,\pi^{\flat})$ et $(\tau(\pi),\tau(\rf),\tau_{\ext}(\pi^{\flat}))$. Les actions de $\cg_K$ vérifient encore l'équation (\ref{eqn:relationpasequiv}). Ce morphisme de $\mathcal{O}_K$-algèbres est également $\tau$-semi-linéaire. Définissons $\mathrm{j}_{\tau_2,\tau_1} \, : \, \mathcal{O}_{\widehat{\mathcal{E}_{\tau_1}^{\mathrm{nr}}}} \xrightarrow{\sim} \mathcal{O}_{\widehat{\mathcal{E}_{\tau_2}^{\mathrm{nr}}}}$ par $\mathrm{j}_{\tau_2,\tau_1}:=\mathrm{W}_{\mathcal{O}_K}(\tau_{2,\ext}^{\flat})\circ \mathrm{W}_{\mathcal{O}_K}(\tau_{1,\ext}^{\flat})^{-1}$. C'est un isomorphisme de $\mathcal{O}_K$-algèbres topologiques $(\tau_2 \tau_1^{-1})$-semi-linéaire,  $\varphi_q^{\N}$-équivariant et tel que \begin{equation*}\label{eqn:eq4}\forall g\in \cg_K, \,\forall x, \,\,\, \mathrm{j}_{\tau_2,\tau_1}(g\cdot x)=\left(\tau_{2,\ext}\tau_{1,\ext}^{-1} \,g\, \tau_{1,\ext} \tau_{2,\ext}^{-1}\right)\cdot (\mathrm{j}_{\tau_2,\tau_1}(x)). \tag{*4}\end{equation*} Ils vérifient que $$\forall \tau_1,\tau_2,\tau_3, \,\,\, \mathrm{j}_{\tau_3,\tau_1}=\mathrm{j}_{\tau_3,\tau_2}\circ \mathrm{j}_{\tau_2,\tau_1}.$$

	\begin{defiprop}
		Soit $\sigma \in \mathrm{W}_{\qp}^+$. Pour chaque extension finie $E^{\sep}|F|E$, on définit sur $\otimes_{\tau\in \mathcal{P},\, \mathcal{O}_K} \mathcal{O}_{\mathcal{F}_{\tau}}^+$ l'endomorphisme $$\sigma \cdot_{s\glec} \left(\otimes\, y_{\tau}\right)= \otimes\, \varphi_{\tau,q}^{d_{\sigma,\sigma^{-1}\tau}}\left(\mathrm{j}_{\tau,\sigma^{-1}\tau}(y_{\sigma^{-1}\tau})\right).$$ Bien que les $\mathrm{j}_{\tau,\sigma^{-1}\tau}$ ne soient pas $\mathcal{O}_K$-linéaires, ils sont tous $\sigma_{|K}$-semi-linéaires ce qui rend correcte la définition. Ces endomorphismes se complètent $(\pi,\underline{X})$-adiquement, passent à la colimite et se complètent $\pi$-adiquement en une action de $\mathrm{W}_{\qp}^+$ sur $\Oehatkand{\mathcal{P}}$. Son noyau est précisément $\mathrm{I}_K$.
	\end{defiprop}
	\begin{proof}
		Nous montrons en détail, pour cette fois-ci uniquement, qu'il s'agit d'un action sur le produit tensoriel. Soient $\sigma_1,\sigma_2 \in \mathrm{W}_{\qp}^+$ et $\otimes \, y_{\tau}$ appartenant au produit tensoriel. Alors,
		
		\begin{align*}
			(\sigma_2\sigma_1)\cdot_{s\glec}(\otimes \, y_{\tau}) & = \otimes \, \varphi_{\tau,q}^{d_{\sigma_2\sigma_1,\sigma_1^{-1} \sigma_2^{-1}\tau}}(\mathrm{j}_{\tau, \sigma_1^{-1} \sigma_2^{-1}\tau}(y_{\sigma_1^{-1} \sigma_2^{-1}\tau})) \\
			&= \otimes \, \varphi_{\tau,q}^{d_{\sigma_2, \sigma_2^{-1}\tau} +d_{\sigma_1, \sigma_1^{-1} \sigma_2^{-1}\tau}}\left(\left[\mathrm{j}_{\tau, \sigma_2^{-1}\tau} \circ \mathrm{j}_{\sigma_2^{-1}\tau,\sigma_1^{-1} \sigma_2^{-1}\tau} \right](y_{\sigma_1^{-1} \sigma_2^{-1}\tau})\right)) \\
			&=\otimes \, \varphi_{\tau,q}^{d_{\sigma_2, \sigma_2^{-1}\tau}}\left(\mathrm{j}_{\tau, \sigma_2^{-1}\tau} \left(\varphi_{\sigma_2^{-1}\tau,q}^{d_{\sigma_1, \sigma_1^{-1} \sigma_2^{-1}\tau}}\left( \mathrm{j}_{\sigma_2^{-1}\tau,\sigma_1^{-1} \sigma_2^{-1}\tau}\left(y_{\sigma_1^{-1} \sigma_2^{-1}\tau}\right)\right)\right)\right) \\
			&= \sigma_2 \cdot_{s\glec} \left( \otimes \, \varphi_{\tau,q}^{d_{\sigma_1,\sigma_1^{-1}\tau}}\left(\mathrm{j}_{\tau,\sigma_1^{-1}\tau}(y_{\sigma_1^{-1}\tau})\right)\right) \\
			&=\sigma_2\cdot_{s\glec} \left(\sigma_1 \cdot_{s\glec}(\otimes \, y_{\tau})\right) 
		\end{align*} où le passage à la deuxième ligne utilise les relation sur les $d$ et les $\mathrm{j}$ et où celui à la troisième utilise la \linebreak$\varphi_q^{\N}$-équivariance de $\mathrm{j}_{\tau_2,\tau_1}$.

		Prouvons à présent que le noyau sur les produits tensoriels est précisément $\mathrm{I}_{K}$. Soit $\sigma$ dans le noyau. En regardant l'action sur les $X_{\tau}$, on obtient que $\sigma$ agit trivialement sur $\mathcal{P}$, autrement dit que $\sigma \in \mathrm{W}_{K}^+$. Dans ce cas $(\sigma \tau)_{\ext}=\tau_{\ext}$ et on obtient que $g_{\sigma,\tau}=\tau_{\ext}^{-1} \sigma \tau_{\ext}$, en particulier que son degré absolu vaut $\deg \sigma$. Pour que chaque $d_{\sigma,\tau}$ soit nul, il faut donc que $\deg \sigma=0$.
		
		Puisque $\sfrac{\mathrm{W}_{\qp}^+}{\mathrm{I}_K}$ est discret, la continuité de l'action est automatique dès lors que l'on peut compléter.
	\end{proof}

	\begin{defiprop}
	 Considérons l'action de $(\Phi_{\mathcal{P},q} \times \cg_{K,\mathcal{P}})$ sur $\Oehatkand{\mathcal{P}}$ obtenue en tordant l'action définie à la section \ref{partie_equiv_lt} par l'automorphisme $$(\Phi_{\mathcal{P},q} \times \cg_{K,\mathcal{P}}) \xrightarrow{\sim} (\Phi_{\mathcal{P},q} \times \cg_{K,\mathcal{P}}), \,\,\, \left(\prod \varphi_{\tau,q}^{n_{\tau}},g_{\tau}\right) \rightarrow \left(\prod \varphi_{\tau,q}^{n_{\tau}}, \tau_{\ext} g_{\tau} \tau_{\ext}^{-1}\right).$$ Considérons également l'action $\cdot_{s\glec}$ de $\mathrm{W}_{\qp}^+$. Ensemble, elles fournissent une structure de $\tgglec$-anneau topologique pour la topologie faible que l'on note $\Oehatkand{s\glec}$. 
	 
	 Petite remarque, l'action est $\mathcal{O}_K$-semi-linéaire, au sens où $\mathcal{O}_{K,\mathrm{sl}} \rightarrow \Oehatkand{s\glec}$ est $\tgglec$-équivariant.
	\end{defiprop}
	\begin{proof}
		D'après les propriétés universelles des Proposition \ref{prop_prdsdirect} et \ref{psd_quotient_desc_1}, il faut vérifier deux conditions de compatiblité des actions. La compatibilité à $s\glec$ utilise que les $\mathrm{j}_{\tau_2,\tau_1}$ sont $\varphi_q^{\N}$-équivariants et vérifient (\ref{eqn:eq4}). Pour le passage au quotient, nous avons déjà énoncé que le Frobenius arithmétique $\mathrm{Frob}$ dans $\mathrm{W}_{K_{\gaal}}^+$ agit trivialement sur $\mathcal{P}$ et que $\forall \tau, \,\, d_{\mathrm{Frob},\tau}=f$. Ceci dit précisément que $\mathrm{Frob}$ agit comme $\varphi_{\mathcal{P},q}^{f}$.
	\end{proof}
	
	L'action de $\tgglec$ sur $\Oehatkand{s\glec}$ est fidèle mais nous ne l'utiliserons pas. Nous nous intéressons aux invariants par les deux sous-monoïdes des équivalences de Fontaines multivariables.
		
	\begin{prop}
		L'inclusion $\mathcal{O}_{K,\mathrm{sl}} \subset \Oehatkand{s\glec}^{\Phi_{\mathcal{P},q}}$ est une égalité de de $\cg_{K,\glec}$-anneaux topologiques.
	\end{prop}
	\begin{proof}
		Combiner la condition 4 dans la preuve du Théorème \ref{equiv_car_zero} et la semi-$\mathcal{O}_K$-linéarité de l'action.
	\end{proof}

	\begin{defi}
		Nous définissons le $\tglec$-anneau topologique $$\Oekand{s\glec}:= \Oehatkand{s\glec}^{\ch_{K,\lt,\mathcal{P}}}.$$
	\end{defi}
	
	\begin{prop}
		Le $\left(\Phi_{\mathcal{P},q}\times \Gamma_{K,\lt,\mathcal{P}}\right)$-anneau topologique sous-jacent à $\Oekand{\plec}$ coïncide avec $\Oekand{\mathcal{P}}$ construit à la section \ref{partie_equiv_lt}. En reprenant sa description à la remarque \ref{descr_oekand}, l'action continue et semi-linéaire de $\tgglec$ est caractérisée par $$\forall \tau, \,\,\, \varphi_{\tau,q}(X_{\tau})=\tau(\rf)(X_{\tau}),$$
		$$ \forall x=(x_{\tau})\in \Gamma_{K,\lt,\mathcal{P}}, \,\, \forall \tau', \,\,\, x \cdot X_{\tau'}=[x_{\tau'}]_{\lt,\tau'(\rf)}(X_{\tau'})$$ $$\text{et }\,\, \forall \sigma \in \mathrm{W}_{\qp}^+, \,\forall \tau, \,\,\, [\sigma]\cdot X_{\tau}= (\sigma\tau)(\rf)^{d_{\sigma,\tau}}\left(X_{\sigma\tau}\right).$$
	\end{prop}
	\begin{proof}
	Similaire à la Proposition \ref{identif_tgplec} en utilisant que $\mathrm{j}_{\tau_2,\tau_1}(X_{\tau_1})=X_{\tau_2}$.
	\end{proof}
	
	Une fois les anneaux construits correctement, nous obtenons gratuitement l'équivalence de Fontaine glectique semi-linéaire.
	
	\begin{theo}\label{equiv_sglec}
		Les foncteurs $$\mathbb{D}_{s\glec,\lt}\, : \,s\mathrm{Rep}_{\mathcal{O}_K} \cg_{K,\glec} \rightarrow \dmod{\tglec}{\Oekand{s\glec}}, \,\,\, V\mapsto \left(\Oehatkand{s\glec} \otimes_{\mathcal{O}_K} V\right)^{\ch_{K,\lt,\mathcal{P}}}$$
		
		$$\mathbb{V}_{s\glec,\lt}\, : \, \cdetaledvproj{\tglec}{\Oekand{s\glec}}{\pi} \rightarrow \dmod{\cg_{K,\glec}}{\mathcal{O}_K}, \,\,\, D\mapsto \left( \Oehatkand{s\glec} \otimes_{\Oekand{s\glec}} D\right)^{\Phi_{\mathcal{P},q}},$$ où les topologies en jeu sont respectivement la topologie $\pi$-adique sur $\mathcal{O}_K$ et la topologie adique faible sur $\Oekand{s\glec}$, sont correctement définis, lax monoïdaux et fermés. Leurs images essentielles sont contenues respectivement dans $\cdetaledvproj{\tglec}{\Oekand{s\glec}}{\pi}$ et $s\mathrm{Rep}_{\mathcal{O}_K} \cg_{K,\glec}$. Leurs corestrictions forment une paire de foncteurs quasi-inverses.
	\end{theo}
	
	\begin{rem}
		Nous avons ici une action du sous-monoïde $\sfrac{\mathrm{W}_{\qp}^+}{\mathrm{I}_K}$ qui est une extension en générale non scindée de $\gal{K}{\qp}$ par $\varphi_q^{\N}$.
		
		Dans le cas où $K$ n'est pas galoisienne, il existe encore des variables cachées dans $\mathrm{W}_{\mathcal{O}_{K_{\gaal}}}(\mathbb{C}_p^{\flat})$. Nous espérons une équivalence de Fontaine semi-linéaire glectique pour $s\mathrm{Rep}_{\mathcal{O}_{K_{\gaal}}} \cg_{K,\glec}$ en considérant des anneaux du type $$\bigotimes_{\tau \in \mathcal{P}, \, \mathcal{O}_{K_{\gaal}}} \mathcal{O}_{K_{\gaal}}\otimes_{\mathcal{O}_{K_{\tau}}} \mathcal{O}_{\mathcal{E}_{K_{\tau}}}.$$ Il faudrait refaire une série de définitions et de lemmes concernant les invariants, ce que nous laissons de côté dans ce présent article.
	\end{rem}
	
	\vspace{0.75cm}
	
	\subsection{\'Equivalence de Fontaine glectique}

	Dans le cas général d'un corps local $p$-adique, nous pouvons cependant construire une (voire trois) équivalences pour les représentations linéaires de $\cg_{K,\glec}$. Toutes les preuves absentes sont similaires à celles du cas semi-linéaire.

	\begin{defiprop}
		Soit $\sigma \in \mathrm{W}_{\qp}^+$. Pour chaque extension finie $E^{\sep}|F|E$ on définit sur $\otimes_{\tau\in \mathcal{P},\, \mathcal{O}_K} \mathcal{O}_{\mathcal{F}_{\tau}}^+$ l'endomorphisme $$\sigma \cdot_{\glec} \left(\otimes\, y_{\tau}\right)= \otimes\, \varphi_{\tau,q}^{d_{\sigma,\sigma^{-1}\tau}}\left(i_{\tau,\sigma^{-1}\tau}(y_{\sigma^{-1}\tau})\right).$$ Ces endomorphismes se complètent $(\pi,\underline{X})$-adiquement, passent à la colimite et se complètent $\pi$-adiquement en une action de $\mathrm{W}_{\qp}^+$ sur $\Oehatkand{\mathcal{P}}$. Son noyau est précisément $\mathrm{I}_{K_{\gaal}}$.
	\end{defiprop}

	\begin{defiprop}
		L'action de $(\Phi_{\mathcal{P},q}\times \cg_{K,\mathcal{P}})$ sur $\Oehatkand{\mathcal{P}}$  et l'action de $\glec$ de $\sfrac{\mathrm{W}_{\qp}^+}{\mathrm{I}_{K_{\gaal}}}$ sont continues pour la topologie adique faible fournissent une structure de $\tgglec$-anneau topologique que l'on note $\Oehatkand{\glec}$.
	\end{defiprop}

	\begin{prop}
	 L'inclusion $$\mathcal{O}_K \subset \mathcal{O}_{\widehat{\mathcal{E}_{K,\glec}}}^{\Phi_{\mathcal{P},q}}$$ est une égalité.
	\end{prop}

	\begin{defiprop}
		Nous définissons le $\tglec$-anneau topologique $$\Oekand{\glec}:= \Oehatkand{\glec}^{\ch_{K,\lt,\mathcal{P}}}.$$
		
		Le $\left(\Phi_{\mathcal{P},q}\times \Gamma_{K,\lt,\mathcal{P}}\right)$-anneau topologique sous-jacent à $\Oekand{\glec}$ coïncide avec $\Oekand{\mathcal{P}}$ construit à la section \ref{partie_equiv_lt}. En reprenant sa description à la remarque \ref{descr_oekand}, l'action continue et semi-linéaire de $\tgglec$ est caractérisée par $$\forall \tau, \,\,\, \varphi_{\tau,q}(X_{\tau})=\rf(X_{\tau}),$$
		$$\forall \tau', \, \forall x=(x_{\tau})_{\tau}\in \Gamma_{K,\lt,\mathcal{P}}, \,\,\, x \cdot X_{\tau'}=[x_{\tau'}]_{\lt,\rf}(X_{\tau'})$$ $$\text{et }\,\, \forall \sigma \in \mathrm{W}_{\qp}^+, \,\forall \tau, \,\,\, [\sigma]\cdot X_{\tau}= \rf^{d_{\sigma,\tau}}\left(X_{\sigma\tau}\right).$$
	\end{defiprop}

	\begin{theo}\label{equiv_glec}
		Les foncteurs $$\mathbb{D}_{\glec,\lt}\, : \,\mathrm{Rep}_{\mathcal{O}_K} \cg_{K,\glec} \rightarrow \dmod{\tglec}{\Oekand{\glec}}, \,\,\, V\mapsto \left(\Oehatkand{\glec} \otimes_{\mathcal{O}_K} V\right)^{\ch_{K,\lt,\mathcal{P}}}$$
		
		$$\mathbb{V}_{\glec,\lt}\, : \, \cdetaledvproj{\tglec}{\Oekand{\glec}}{\pi} \rightarrow \dmod{\cg_{K,\glec}}{\mathcal{O}_K}, \,\,\, D\mapsto \left( \Oehatkand{\glec} \otimes_{\Oekand{\glec}} D\right)^{\Phi_{\Delta,q}},$$ où les topologies en jeu sont respectivement la topologie $\pi$-adique sur $\mathcal{O}_K$ et la topologie adique faible sur $\Oekand{\glec}$, sont correctement définis, lax monoïdaux et fermés. Leurs images essentielles sont contenues respectivement dans $\cdetaledvproj{\tglec}{\Oekand{\glec}}{\pi}$ et $\mathrm{Rep}_{\mathcal{O}_K} \cg_{K,\glec}$. Leurs corestrictions forment une paire de foncteurs quasi-inverses.
	\end{theo}
	
	Le monoïde $\tglec$ est intéressant car ses deux sous-monoïdes $\Phi_{\mathcal{P},q}$ et $\sfrac{\mathrm{W}_{\qp}^+}{\mathrm{I}_{K_{\gaal}}}$ contiennent $\varphi_{\mathcal{P},q}^{f_{\gaal}\N}$ comme sous-monoïde distingué sans que la suite exacte déduite soit scindée. Dans le cas non ramifié où $K=\mathbb{Q}_q$, on a simplement rajouté à $(\Phi_{\mathcal{P},q}\times \prod \mathbb{Z}_q^{\times})$ une racine $f$-ième de $\varphi_{\mathcal{P},q}$. 
	
	\begin{rem} Toujours dans le cas $K=\mathbb{Q}_q$, nous pouvons faire un lien entre notre $\tglec$-anneau et l'anneau de coefficients $A$ dans \cite{breuilandco_phi}. Choisissons $\rf=T^q+\pi T$ et voyons $(\varphi_p^{\N}\times \mathbb{F}_q^{\times})$ comme sous-monoïde de $\tglec$ en identifiant $\varphi_p^{\N}$ à $\sfrac{\mathrm{W}_{\qp}^+}{\mathrm{I}_{\mathbb{Q}_q}}$ et en plongeant $\mathbb{F}_q^{\times}$ dans $\prod_{0\leq i<f} \mathbb{Z}_q^{\times}$ par les représentants de Teichmüller $x\mapsto ([x])_i$. Alors, le $(\varphi_p^{\N}\times \mathbb{F}_q^{\times})$-anneau $\sfrac{\Oekand{\glec}^+}{\pi\Oekand{\glec}^+}$ s'identifie à un sous-$(\varphi_p^{\N}\times \mathbb{F}_q^{\times})$-anneau de $A$ en envoyant $X_{\mathrm{Frob}^i}$ sur $Y_{f-1_i}^{p^i}$. Cela ouvre des perspectives quant au lien entre représentations galoisiennes et $(\varphi,\Gamma)$-modules glectiques.
	\end{rem}

	\begin{rem}
		Nous esquissons une autre équivalence aux saveurs glectiques dont la preuve suivrait trait pour trait celle qui précède. Nous changeons légèrement l'action du monoïde $\mathrm{W}_{\qp}^+$ sur $\Oehatkand{\mathcal{P}}$ en décrétant que $$\sigma(\otimes \, y_{\tau})=\otimes \left(\varphi_{\tau,q}^{d_{\sigma,\sigma^{-1}\tau}} \circ g_{\sigma,\sigma^{-1}\tau}\circ i_{\tau,\sigma^{-1}\tau}\right)(y_{\sigma^{-1}}).$$ Autrement dit, au lieu de faire tourner les plongements en se souvenant des problèmes de degrés, on se souvient en plus de toute l'information l'élément $g_{\sigma,\tau}$. En identifiant l'action de $\cg_{K,\mathcal{P}}$ à une action de $\prod_{\mathcal{P}} \cg_{K_{\tau}}$ par conjugaison par chacun des $\tau_{\ext}$, les deux actions s'assemblent en action continue du monoïde $$\tggloc:=\quot{\left(\left(\Phi_{\mathcal{P},q}\times \prod_{\mathcal{P}} \cg_{K_{\tau}}\right) \rtimes_{\glec'} \mathrm{W}_{\qp}^+\right)}{\tsim \mathrm{W}_{K_{\gaal}}^+ \tsim}$$  avec $$\glec'(\sigma)\left(\prod \varphi_{\tau,q}^{n_{\tau}},g_{\tau}\right)=\left(\prod \varphi_{\tau,q}^{n_{\sigma^{-1}\tau}}, \sigma \, g_{\sigma^{-1}\tau}\, \sigma^{-1}\right)$$ et l'identification via le morphisme $\kappa_{\glec}'(\sigma)=\left(\varphi_{\mathcal{P},q}^{\sfrac{\deg \sigma}{f}}, \sigma, \ldots, \sigma \right)$. Autrement dit, nous faisons tourner les $\cg_{K_{\tau}}$ par la conjugaison par $\mathrm{W}_{\qp}^+$ et nous identifions ceux de $\mathrm{W}_{K_{\gaal}}^+$ à une puissance convenable du Frobenius $\varphi_{\mathcal{P},q}$ et à la diagonale dans $\prod \cg_{K_{\tau}}$. Nous pouvons à nouveau vérifier que les morphismes $(\Phi_{\mathcal{P},q}\times \prod \cg_{K_{\tau}})$ et $\mathrm{W}_{\qp}^+$ sont des plongements de monoïdes topologiques. Nous notons $\Oehatkand{\glec'}$ le $\tggloc$-anneau topologique obtenu.
		
		Le quotient par $\Phi_{\mathcal{P},q}$ s'identifie à $\sfrac{\left((\prod_{\mathcal{P}}\cg_{K_{\tau}})\rtimes \mathrm{W}_{\qp}^+\right)}{\tsim \mathrm{W}_{K_{\gaal}}^+\tsim}$. Après quotient, $\mathrm{W}_{\qp}^+$ s'injecte encore, mais la topologie induite est désormais celle induite par $\cg_{\qp}$. Grâce à cette remarque, on peut établir que le morphisme $\sfrac{\tggloc}{\Phi_{\mathcal{P},q}}\rightarrow\cg_{K,\glec}$ induit par $\mathrm{W}_{\qp}^+\subset \cg_{\qp}$ est un isomorphisme de monoïdes topologiques.
		
		Le quotient par $\tgloc:=\sfrac{\tggloc}{\ch_{K,\lt,\mathcal{P}}}$ s'identifie à $$\quot{\left(\left(\Phi_{\mathcal{P},q}\times \prod_{\mathcal{P}} \mathcal{O}_{K_{\tau}}^{\times}\right) \rtimes_{\glec'} \quot{\mathrm{W}_{\qp}^+}{\cg_{K_{\mathcal{P}}^{\mathrm{ab}}}}\right)}{\tsim \sfrac{\mathrm{W}_{K_{\gaal}}^+}{\cg_{K_{\mathcal{P}}^{\mathrm{ab}}}}\tsim}$$ où nous avons $$\glec'(\sigma)\left(\prod \varphi_{\tau,q}^{n_{\tau}},x_{\tau}\right)=\left(\prod \varphi_{\tau,q}^{n_{\sigma^{-1}\tau}}, \sigma(x_{\sigma^{-1}\tau})\right)$$ et $$\kappa_{\glec'}(\sigma\cg_{K_{\mathcal{P}}^{\mathrm{ab}}})= (\varphi_{\mathcal{P},q}^{\sfrac{\deg\sigma}{f}}, \mathrm{Art}_{K_{\tau}}^{-1}(\sigma)).$$ Ici, $K_{\mathcal{P}}^{\mathrm{ab}}$ est l'extension composée des $K_{\tau}^{\mathrm{ab}}$ qui contient donc $K^{\mathrm{nr}}$ et correspond au sous-groupe $$\{x \in \mathcal{O}_{K_{\gaal}}^{\times} \, |\, \forall \tau, \,\,\, \mathrm{N}_{K_{\gaal}|K_{\tau}}(x)=1\} \triangleleft \mathcal{O}_{K_{\gaal}}^{\times}$$ par la théorie du corps de classes locale.
		
		En définissant $\Oekand{\glec'}= \Oehatkand{\glec'}^{\ch_{K,\lt,\mathcal{P}}}$, nous obtenons à nouveau une équivalence de catégories $$\mathrm{Rep}_{\mathcal{O}_K} \cg_{K,\glec} \rightleftarrows \cdetaledvproj{\tgloc}{\Oekand{\glec'}}{\pi}.$$ De plus, l'action de $\sigma\in \mathrm{W}_{\qp}^+$ sur $\Oekand{\glec'}$ est l'unique morphisme de $\mathcal{O}_K$-algèbres topologiques tel que $$\forall \tau \in \mathcal{P}, \,\,\, \sigma \cdot X_{\tau}=[ \tau_{\ext}^{-1}\left(\mathrm{Art}_{K_{\tau}}^{-1}(\sigma)\right)]_{\lt,\rf}(X_{\sigma\tau}).$$ Notons que $\tau_{\ext}^{-1}\left(\mathrm{Art}_{K_{\tau}}^{-1}(\sigma)\right)=\mathrm{Art}_K^{-1}(\tau_{\ext}^{-1} \sigma \tau_{\ext})$.
	\end{rem}

	\vspace{1.5cm}
	\appendix
	\section{\'Etude détaillée des anneaux $\widetilde{E}_{\Delta}$, $E_{\Delta,p}^{\sep}$, etc}\label{annexe_algèbre}

	\vspace{0.75cm}
	\subsection{Plongement dans les anneaux de séries de Hahn-Mal'cev multivariables}

	La plupart des propriétés des anneaux $\widetilde{F}_{\Delta,q}$ sont déduites de propriétés des anneaux de séries de Hahn-Mal'cev multivariables que nous introduisons ici. Dans le cas univariable, l'étude des corps maximalement valués dans \cite{kaplansky} démontre que les anneaux d'entiers de notre corps perfectoïde $\widetilde{E}$ et de ses extensions se plongent dans un anneau de séries formelles généralisées $k^{\mathrm{alg}}\llbracket t^{\R} \rrbracket$. Nous introduisons un analogue multivariable, démontrons un plongement similaire et l'étudions de manière fine pour en déduire les propriétés de $\widetilde{F}_{\Delta,q}$.

		\begin{defi}
		Soit $A$ un anneau et $(\Gamma_{\alpha})_{\alpha \in \Delta}$ une famille finie de groupes abéliens totalement ordonnés. Définissons \textit{l'anneau des séries de Hahn-Mal'cev multivariables} associé par $$A\left\llbracket t_{\alpha}^{\Gamma_{\alpha}} \, |\, \alpha \in \Delta \right\rrbracket = \left\{ \sum_{\underline{\gamma} \in \prod_{\alpha \in \Delta} \Gamma_{\alpha,\geq 0}} a_{\underline{\gamma}} t^{\underline{\gamma}} \,\,\Bigg|\,\,\substack{\forall \underline{\delta} \in \prod_{\alpha \in \Delta} \Gamma_{\alpha,\geq 0},\, \forall \alpha \in \Delta, \,\, \text{l'ensemble} \\ \left\{ \gamma_{\alpha}\leq \delta_{\alpha} \, |\, \exists \underline{\gamma'}\in \prod_{\beta\neq \alpha} \Gamma_{\beta,\geq 0,\leq \delta_{\beta}}, \,\,\, a_{(\gamma_{\alpha},\underline{\gamma'})}\neq 0 \right\} \\ \text{est}\,\,\, \text{bien} \,\, \text{ordonn\'e}} \right\}.$$ 
		\end{defi}
		
		\begin{prop}
			Les applications $$A\rightarrow A\left\llbracket t_{\alpha}^{\Gamma_{\alpha}} \, |\, \alpha \in \Delta \right\rrbracket, \,\,\, a\mapsto a t^{\underline{0}},$$
			
			$$\left( \sum a_{\underline{\gamma}} t^{\underline{\gamma}} \right) + \left( \sum b_{\underline{\gamma}} t^{\underline{\gamma}} \right) := \left( \sum \left(a_{\underline{\gamma}}+b_{\underline{\gamma}} \right) t^{\underline{\gamma}} \right)$$
			
			$$\text{et} \,\,\, \left( \sum a_{\underline{\gamma}} t^{\underline{\gamma}} \right) \times \left( \sum b_{\underline{\gamma}} t^{\underline{\gamma}} \right) := \left( \sum_{\underline{\gamma}\in \prod \Gamma_{\alpha,\geq 0}} \left(\sum_{\substack{\underline{\gamma_1}, \underline{\gamma_2}\in \prod \Gamma_{\alpha,\geq 0}\\ \underline{\gamma_1}+\underline{\gamma_2}=\underline{\gamma}}} a_{\underline{\gamma_1}} b_{\underline{\gamma_2}}\right) t^{\underline{\gamma}} \right)$$ munissent $A\left\llbracket t_{\alpha}^{\Gamma_{\alpha}} \, |\, \alpha \in \Delta \right\rrbracket$ d'une structure de $A$-algèbre. Elle est séparée et complète par rapport à la famille d'idéaux $(t_{\alpha}^{\gamma_{\alpha}} \,|\, \alpha \in \Delta)$ pour $\underline{\gamma}\in \prod_{\alpha \in \Delta} \Gamma_{\alpha,\geq 0}$. Elle est intègre (resp. réduite) dès que $A$ est intègre (resp. réduit).
		\end{prop}
		\begin{proof}
	Laissée aux lecteurs et lectrices.
		\end{proof}

	\begin{rem}
		\begin{enumerate}[itemsep=0mm]
		\item De manière informelle, il s'agit de l'anneau des séries multivariables telle que pour tout multi-indice $\underline{\gamma}$, l'ensemble d'indices de la réduction modulo $(t_{\alpha}^{\gamma_{\alpha}}\, |\, \alpha \in \Delta)$ de notre série a toutes ses projections bien ordonnées. 
		
		\item Cet anneau est en général légèrement plus gros que la complétion de l'anneau $$\bigotimes_{A, \, \alpha \in \Delta} A\llbracket t_{\alpha}^{\Gamma_{\alpha}} \rrbracket$$ par rapport aux idéaux $(t_{\alpha}^{\gamma_{\alpha}}\, |\, \alpha \in \Delta)$. Par exemple, nous avons $$\sum_{n\geq 1} (t_1 t_2)^{1-\frac{1}{n}} \in A \llbracket t_1^{\mathbb{Q}}, \, t_2^{\mathbb{Q}}\rrbracket$$ sans qu'il n'appartienne à la complétion. Il est toutefois moins aisé de décrire la complétion, a fortiori de la manipuler. 
		
		\item Si nous imposions uniquement que les projections des indices apparaissant soient bien ordonnées, nous aurions un anneau non-complet par rapport aux idéauw $(t_{\alpha}^{\gamma_{\alpha}}\, |\, \alpha \in \Delta)$. Il serait inadapté pour plonger $\widetilde{E}_{\Delta}^+$.
		
		\item La définition précédente coïncide avec les séries de Hahn-Mal'cev (voir \cite{hahnmalcev}) lorsque $|\Delta|=1$.
		\end{enumerate}
	\end{rem}
	
	Nous nous concentrons sur le cas qui nous servira où tous les groupes $\Gamma_{\alpha}$ sont égaux à un même sous-groupe $\Gamma$ de $\R$. 
	
	\begin{defi}
			Soit $A$ un anneau, $\Delta$ un ensemble fini et $\Gamma$ un sous-groupe de $\R$. Pour $\gamma \in \Gamma$ et $\underline{c} \in \mathbb{N}_{\geq 1}^{\Delta}$, définissons l'idéal $J_{\gamma,\Delta,\underline{c}}$ de $A\left\llbracket t_{\alpha}^{\Gamma} \, |\, \alpha \in \Delta \right\rrbracket$ par $$J_{\gamma,\Delta,\underline{c}}:=\left( t^{\underline{\delta}} \, \,\bigg|\,\, \underline{\delta}\in \Gamma_{\geq 0}^{\Delta}, \,\,\, \sum_{\alpha \in \Delta} c_{\alpha} \delta_{\alpha} \geq \gamma\right).$$ 
	
	Lorsque $\underline{c}=(1)_{\alpha \in \Delta}$, nous notons $J_{\gamma,\Delta}$ l'idéal obtenu.
	\end{defi}
	
	\begin{defiprop}\label{expression_val}
		Dans le cadre de la définition précédente, soit $\underline{c} \in \N_{\geq 1}^{\Delta}$. Alors, $$\forall x=\sum_{\underline{\gamma} \in \Gamma_{\geq 0}^{\Delta}} a_{\underline{\gamma}}t^{\underline{\gamma}} \in A\llbracket t_{\alpha}^{\Gamma} \, |\, \alpha \in \Delta \rrbracket, \,\,\,\left\{ \sum c_{\alpha} \gamma_{\alpha} \, |\, \underline{\gamma} \,\, \text{tel que} \,\, a_{\underline{\gamma}}\neq 0\right\} \,\, \text{est bien ordonné}$$ et la formule $$|x|_{\Delta,\underline{c}}:=e^{-\max \left\{ \gamma \, |\, x\in J_{\gamma,\Delta,\underline{c}}\right\}}$$ définit une norme sous-multiplicative. Lorsque $A$ est intègre, la norme $|\cdot|_{\Delta,\underline{c}}$ est multiplicative.
	\end{defiprop}
	\begin{proof}
	Dans cet article, nous laissons ces vérifications aux lecteurs et lectrices.
	\end{proof}
	
	\begin{coro}\label{norme_alpha}
		Dans le cadre précédent, soit $\beta\in \Delta$. Si $A$ est intègre, la norme $t_{\beta}$-adique est multiplicative sur $A\llbracket t_{\alpha}^{\Gamma}\, |\, \alpha \in \Delta\rrbracket$.
	\end{coro}
	\begin{proof}
		Considérer les $|\cdot|_{\Delta,\underline{c_n}}$ pour $c_{n,\beta}=n$ et $\forall \alpha \neq \beta, \,\, c_{n,\alpha}=1$ puis faire diverger $n$.
	\end{proof}
	
	\begin{lemma}\label{sep_series_formelles}
		Soit $k$ un corps et $A$ une $k$-algèbre noethérienne. Soit $\Delta$ un ensemble fini et $\Gamma$ un groupe abélien totalement ordonné. L'application naturelle $$\iota_0 \, : \,\bigotimes_{\alpha \in \Delta,\, k} A\llbracket t_{\alpha}^{\Gamma} \rrbracket \rightarrow \bigg(\bigotimes_{\alpha \in \Delta, \, k} A\bigg)\llbracket t_{\alpha}^{\Gamma} \, |\, \alpha \in \Delta \rrbracket$$ est injective. 
		
		Pour toute famille finie de multi-indices $(\underline{\gamma_i})_{1\leq i \leq m}$, nous avons un égalité d'idéaux $$\iota_0^{-1}\left((t^{\underline{\gamma_i}} \, |\, 1\leq i \leq m)\right)=(t^{\underline{\gamma_i}}\, |\, 1\leq i \leq m).$$
	\end{lemma}
	\begin{proof}
		\underline{Énoncé intermédiaire :} soit $\alpha \in \Delta$. Nous commençons par montrer que le morphisme
		$$\bigg(\bigotimes_{\beta\in \Delta \backslash \{\alpha\}, \, k} A \bigg) \llbracket t_{\beta}^{\Gamma} \, |\, \beta \in \Delta\backslash \{\alpha\}\rrbracket \otimes_k A \llbracket t_{\alpha}^{\Gamma} \rrbracket \rightarrow \bigg(\bigotimes_{\beta \in \Delta, \, k} A\bigg)\llbracket t_{\beta}^{\Gamma} \, |\, \beta \in \Delta \rrbracket$$ est injectif. Considérons $\sum_{1\leq i \leq n} f_i\otimes g_i$ dans le noyau. On écrit $g_i=\sum_{\gamma_{\alpha} \in \Gamma_{\geq 0}} a_{i,\gamma_{\alpha}} t_{\alpha}^{\gamma_{\alpha}}$. En identifiant les termes en $t_{\alpha}^{\gamma_{\alpha}}$ au but, nous obtenons $$\forall \gamma_{\alpha}, \,\,\, \sum_{1\leq i \leq n} f_i a_{i,\gamma_{\alpha}}=0 \,\,\, \text{dans}\,\, \bigg(\bigotimes_{\beta \in \Delta,\, k} A\bigg) \llbracket t_{\beta}^{\Gamma} \, |\, \beta \in \Delta \backslash \{\alpha\} \rrbracket.$$  Par noethérianité de $A$, nous choisissons une famille de vecteurs $(e_k)_{1\leq k\leq r}$ qui engendre le noyau  $$A^n\rightarrow \bigg(\bigotimes_{\beta\in \Delta  \, k} A \bigg) \llbracket t_{\beta}^{\Gamma_{\beta}} \, |\, \beta\in \Delta \backslash \{\alpha\} \rrbracket, \,\,\,\, (a_i) \mapsto \sum_i f_i a_i.$$ Nous écrivons ensuite $(a_{i,\gamma_{\alpha}})_{1\leq i \leq n}=\sum_k b_{k,\gamma_{\alpha}}\, e_k$ avec $\forall k, \gamma_{\alpha}, \, b_{k,\gamma_{\alpha}}\in A$. Le support de chaque famille $(b_{k,\gamma_{\alpha}})_{\gamma_{\alpha} \in \Gamma_{\geq 0}}$ peut être choisi contenu dans l'union de tous les supports des $(a_{i,\gamma_{\alpha}})_{\gamma_{\alpha} \in \Gamma_{\geq 0}}$ : ce support vérifie les conditions les séries de Hahn-Mal'cev ce qui donne du sens aux calculs suivants :
		
		\begin{align*}
			\sum_{1\leq i\leq n} f_i \otimes g_i &= \sum_{1\leq i \leq n} f_i \otimes \left(\sum_{\gamma_{\alpha}\in \Gamma_{\alpha,\geq 0}} a_{i,\gamma_{\alpha}} t_{\alpha}^{\gamma_{\alpha}}\right)\\
			&=\sum_{1\leq i \leq n} f_i \otimes \left(\sum_{\gamma_{\alpha} \in \Gamma_{\alpha,\geq 0}} \left(\sum_{1\leq k\leq r} b_{k,\gamma_{\alpha}} e_{k,i}\right) t_{\alpha}^{\gamma_{\alpha}}\right) \\
			&=\sum_{1\leq i\leq n,\, 1\leq k \leq r}  f_i \otimes \left(\sum_{\gamma_{\alpha}\in \Gamma_{\geq 0}} b_{k,\gamma_{\alpha}} t_{\alpha}^{\gamma_{\alpha}} \right) e_{k,i} \\
			&= \sum_{1\leq k \leq r} \left(\sum_{\gamma_{\alpha}\in \Gamma_{\geq 0}} b_{k,\gamma_{\alpha}} t_{\alpha}^{\gamma_{\alpha}} \right) \left(\sum_{1\leq i\leq n} f_i \otimes e_{k,i}\right)	
		\end{align*} et chaque $\sum_i f_i\otimes e_{k,i}$ est dans le noyau de $$\bigg(\bigotimes_{\beta\in \Delta \backslash \{\alpha\},\, k} A\bigg) \llbracket t_{\beta}^{\Gamma_{\beta}} \, |\, \beta\in \Delta \backslash \{\alpha\} \rrbracket \otimes_k A \rightarrow \bigg(\bigotimes_{\beta\in \Delta,\, k} A\bigg) \llbracket t_{\beta}^{\Gamma_{\beta}} \, |\, \beta\in \Delta \backslash \{\alpha\}\rrbracket.$$ Prouver que cette dernière application est injective suffit à conclure. À nouveau, on identifie les coefficients. De plus, le noyau de l'application $$\bigg(\bigotimes_{\beta\in \Delta \backslash \{\alpha\},\, k} A\bigg)^n \rightarrow \left(\bigotimes_{\beta\in \Delta,\, k} A\right), \,\,\,\, (b_i) \mapsto \sum_i e_i \otimes b_i$$ est de type fini sur $\left(\otimes_{\beta\neq\alpha, \,k} A\right)$ puisqu'il s'agit du changement de base de $$k^n \rightarrow A, \,\,\,\, (x_i)\mapsto \sum_i x_i e_i$$ dont le noyau est de type fini sur $k$. On se ramène de cette manière à prouver que $$\left(\bigotimes_{\beta \in \Delta \backslash \{\alpha\}, \, k} A\right) \otimes_k A \rightarrow \left(\bigotimes_{\beta \in \Delta} A\right)$$ est injective, ce qui est tautologique.
		
		\underline{Démonstration de l'injectivité :} la flèche dont nous voulons montrer l'injectivité est une composée de changements de base au-dessus de $k$ du morphisme ci-dessus.
		
		\underline{\'Etude des idéaux :} Soit $(\underline{\gamma_i})_{1\leq i \leq m}$ une famille finie de multi-indices. Appelons $J$ l'idéal engendré par les $t^{\underline{\gamma_i}}$ au but et $J^{\otimes}$ celui engendré à la source. Pour chaque $\alpha\in \Delta$, nous posons $\sigma_{\alpha}\in \mathfrak{S}_m$ telle que $$\gamma_{\sigma_{\alpha}(1),\alpha} \leq \gamma_{\sigma_{\alpha}(2),\alpha} \leq \ldots \leq \gamma_{\sigma_{\alpha}(m),\alpha}.$$ Supposons que $x=\sum_{1\leq j \leq n} \otimes_{\alpha\in \Delta} f_{j,\alpha}\in J$. Nous découpons chaque $f_{j,\alpha}$ en $$f_{j,\alpha}=\sum_{k=1}^m t_{\alpha}^{\gamma_{\sigma_{\alpha}(k),\alpha}} f_{j,\alpha,k}$$ où les monômes en $t_{\alpha}$ de $f_{j,\alpha,k}$ sont de degré strictement inférieur à $(\gamma_{\sigma_{\alpha}(k+1),\alpha}-\gamma_{\sigma_{\alpha}(k),\alpha})$. Il est alors possible d'écrire 
		
		$$x=\sum_{1\leq j \leq n} \sum_{\underline{k}\in \llbracket 1,m\rrbracket^{\Delta}}  \otimes_{\alpha\in \Delta} t_{\alpha}^{\gamma_{\sigma_{\alpha}(k_{\alpha}),\alpha}} f_{j,\alpha,k_{\alpha}}.$$ Appelons $\Lambda$ l'ensemble des uplets $\underline{i}$ tels que $$\exists i, \,\, \forall \alpha, \,\, \gamma_{\sigma_{\alpha}(k_{\alpha}),\alpha}\geq \gamma_{i,\alpha}.$$ Autrement dit, nous sélectionnons les tranches qui appartiendront automatiquement à l'idéal $J^{\otimes}$. Pour chaque uplet $\underline{k}\in \Lambda$, choissons $i_{\underline{k}}$ qui souligne son appartenance à $\Lambda$. Il est alors possible d'écrire 		
		$$x_{\geq} :=\sum_{1\leq j \leq n} \sum_{\underline{k}\in \Lambda}  \otimes_{\alpha\in \Delta} t_{\alpha}^{\gamma_{\sigma_{\alpha}(k_{\alpha}),\alpha}} f_{j,\alpha,k_{\alpha}} = \sum_{1\leq j \leq n} \sum_{1\leq i\leq m} t^{\underline{\gamma_i}} \left(\sum_{\substack{\underline{k} \in \Lambda \\ \text{tel que} \,\, i_{\underline{k}}=i}} \otimes_{\alpha\in \Delta} t_{\alpha}^{\gamma_{\sigma_{\alpha}(k_{\alpha}),\alpha}-\gamma_{i,\alpha}} f_{j,\alpha,k_{\alpha}}\right) \in J^{\otimes}.$$ Par conséquent, l'image par $\iota_0$ de $$x_{<}:=\sum_{1\leq j \leq n} \sum_{\underline{k}\notin \Lambda} \otimes_{\alpha\in \Delta} t_{\alpha}^{\gamma_{\sigma_{\alpha}(k_{\alpha}),\alpha}} f_{j,\alpha,k_{\alpha}}$$ doit également appartenir à $J$. Or, $$\forall \underline{k} \notin \Lambda, \,\,\, \forall i, \,\, \exists \alpha, \,\, \gamma_{\sigma_{\alpha}(k_{\alpha}),\alpha}<\gamma_{i,\alpha}$$ ce qui implique immédiatemment par définition des $\sigma_{\alpha}$ $$\forall \underline{k}\notin \Lambda, \,\, \forall i, \,\, \exists \alpha, \,\, \gamma_{\sigma_{\alpha}(k_{\alpha}+1),\alpha}\leq \gamma_{i,\alpha}$$ Si un monôme $t^{\underline{\gamma}}$ apparaît dans $\iota_0(\sum_j \otimes t_{\alpha}^{\gamma_{\sigma_{\alpha}(k_{\alpha}),\alpha}} f_{j,\alpha,k_{\alpha}})$ nous avons par construction des $f_{j,\alpha,k}$ que $$\forall \alpha, \,\, \gamma_{\alpha}<\gamma_{\sigma_{\alpha}(k_{\alpha}+1),\alpha}.$$ Il en découle que $\not \exists i, \,\, \forall \alpha, \,\,\gamma_{\alpha} \geq \gamma_{i,\alpha}$. Dans l'anneau des séries de Hahn-Mal'cev multivariables, l'appartenance à $J$ se teste sur les monômes qui apparaissent. Il en découle que $\iota_0(x_{<})=0$. Puisque nous avons déjà démontré que $\iota_0$ est injective, il en découle que $x_{<}=0$ soit $x=x_{\geq} \in J^{\otimes}$.
	\end{proof}
	\vspace{0.5 cm}

	Nous reprenons à présent les notations de \S \ref{section_wddelta}. Nous fixons $\widetilde{E}$ un corps perfectoïde de caractéristique $p$, tel que la clôture algébrique de $\mathbb{F}_p$ dans $\widetilde{E}$ est munie d'un isomorphisme avec $\mathbb{F}_q$. Nous fixons également un ensemble fini $\Delta$. Nous étudions tout d'abord l'anneau $\widetilde{E}_{\Delta}$ (et par conséquent tous les $\widetilde{F}_{\Delta}$).
	
		\begin{defi}
		Pour tout $r\in \mathbb{N}[q^{-1}]$, nous définissons l'idéal $I_{r,\Delta}$ de $\widetilde{E}_{\Delta}^+$ par $$
			I_{r,\Delta} :=\left(\underline{\varpi}^{\underline{r}} \, |\, \underline{r}\in (\mathbb{N}[q^{-1}])^{\Delta}, \,\,\, \sum r_{\alpha}=r\right) = \bigcup_{k>>k_r} \varphi_{\Delta,q}^{-k}\left((\underline{\varpi})^{q^k r}\right)$$ où $k_r$ est tel que $q^k r \in \mathbb{N}$.
		
		On définit de même la famille d'idéaux $$I_{r^-,\Delta}:=\bigcap_{\substack{r'\in \N[q^{-1}]\\ \text{tel que} \,\, r'<r}} I_{r',\Delta}.$$
		\end{defi}
	
		\begin{prop}\label{inj_1}
		Il existe une injection d'anneaux $$\hat{\iota}\, : \, \widetilde{E}^+_{\Delta} \hookrightarrow \left(\bigotimes_{\alpha \in \Delta, \, \overline{\mathbb{F}_q}} k^{\mathrm{alg}}\right) \llbracket t_{\alpha}^{\mathbb{R}} \, |\, \alpha \in \Delta\rrbracket.$$
		
		\noindent De plus, nous pouvons choisir cette injection de telle sorte que $$\forall r\in \mathbb{N}[q^{-1}], \,\,\, \hat{\iota}^{-1}(J_{r,\Delta})=I_{r^-,\Delta}.$$ 
	\end{prop}
	\begin{proof}
		Grâce à \cite[V \S 17, Proposition 9]{bourbaki_alg}, l'extension $\widetilde{E}|\mathbb{F}_q$ est régulière. Il en découle que l'anneau $\widetilde{E}\otimes_{\mathbb{F}_q} \overline{\mathbb{F}_q}$ est intègre. Sa complétion $\varpi$-adique est encore un corps perfectoïde, de pseudo-uniformisante $\varpi$ et de corps résiduel muni d'un plongement de $\overline{\mathbb{F}_q}$ ; nous le notons $\widetilde{E}'$. Il est lui-même plongé dans un corps perfectoïde\footnote{Considérer par exemple la complétion de sa clôture algébrique.} $\widetilde{F}$ de pseudo-uniformisante $\varpi$ et de corps résiduel $k^{\mathrm{alg}}$. Nous obtenons grâce à \cite[Coro. du Théorème 8]{kaplansky} que $\widetilde{F}$ est analytiquement isomorphe à un sous-corps de $k^{\mathrm{alg}}(\!(t^{\mathbb{R}})\!)$, avec $\varpi$ envoyé sur $t$ : notre corps est perfectoïde de corps résiduel algébriquement clos et de groupe de valeurs contenu dans $\mathbb{R}$ ce qui permet de vérifier sans problème les hypothèses du théorème de Kaplansky. De plus, en examinant attentivement la preuve de \cite[Lem. 13]{kaplansky}, nous pouvons garantir que chaque $\varpi^{q^{-n}}$ est envoyé sur $t^{q^{-n}}$. Il existe ainsi une suite d'injections d'anneaux $$\widetilde{E}\otimes_{\mathbb{F}_q} \overline{\mathbb{F}_q} \hookrightarrow \widetilde{E}^{'+}\hookrightarrow \widetilde{F}^+ \hookrightarrow k^{\mathrm{alg}}\llbracket t^{\mathbb{R}}\rrbracket$$ qui envoie $\varpi^{q^{-n}}$ sur $t^{q^{-n}}$. Nous obtenons une suite d'injections
		\begin{align*}
			\bigotimes_{\alpha \in \Delta, \, \mathbb{F}_q} \widetilde{E}_{\alpha}^+  & \hookrightarrow \left(\bigotimes_{\alpha \in \Delta, \, \mathbb{F}_q} \widetilde{E}_{\alpha}^+\right) \otimes_{\mathbb{F}_q} \overline{\mathbb{F}_q} \cong \bigotimes_{\alpha \in \Delta, \, \overline{\mathbb{F}_q}} \left(\widetilde{E}_{\alpha}^+ \otimes_{\mathbb{F}_q} \overline{\mathbb{F}_q} \right)
			\hookrightarrow \bigotimes_{\alpha \in \Delta, \,\overline{\mathbb{F}_q}} \widetilde{E}^{'+}_{\alpha} \\
			&\hookrightarrow \bigotimes_{\alpha \in \Delta, \, \overline{\mathbb{F}_q}} k^{\mathrm{alg}} \llbracket t_{\alpha}^{\mathbb{R}} \rrbracket \\
			&\hookrightarrow \left(\bigotimes_{\alpha \in \Delta, \, \overline{\mathbb{F}_q}} k^{\mathrm{alg}}\right) \llbracket t_{\alpha}^{\mathbb{R}} \, |\, \alpha \in \Delta \rrbracket
		\end{align*} où la dernière provient du lemme \ref{sep_series_formelles}. Appelons $\iota$ la composée. 
		
		Commençons par considérer une famille finie $(\underline{k_i})_{1\leq i \leq m}$ de multi-indices à valeurs dans $\mathbb{N}[q^{-1}]$ et $I^{\otimes}$ (resp. $I'^{\otimes}$, resp. $J^{\otimes}$, resp. $J$) l'idéal engendré par la famille $(\varpi^{\underline{k_i}})_{1\leq i \leq m}$ (resp. $(\varpi^{\underline{k_i}})_{1\leq i \leq m}$, resp. $(t^{\underline{k_i}})_{1\leq i \leq m}$, resp. $(t^{\underline{k_i}})_{1\leq i \leq m}$) à la source du morphisme $\iota$ (resp. chaque étape de sa décomposition ci-dessus par ordre croissant). Nous démontrons que $\iota^{-1}(J)=I^{\otimes}$.	
		% Dans ce qui suit, les idéaux vont se multiplier, ainsi que leurs notations. Par convention, nous utiliserons toujours les undices $(r,\Delta)$ ou $(r,\Delta,s)$ pour indiquer par quels éléménets nos idéaux sont engendrés, en se calquant sur $J_{r,\Delta}$ et $J_{r,\Delta,s}$. Nous ajouterons une puissance $\otimes$ lorsque l'idéal vivra dans le produit tensoriel "non complété". Et nous utiliserons plus la lettre $I$ ou $I'$ lorsque nous considérons les idéaux associé au corps perfectoïde $\widetilde{E}$ ou $\widetilde{E}'$.
		
		%Nous allons démontrer que pour tout $r,s\in \mathbb{N}[q^{-1}]$, l'image inverse de l'idéal $J_{r,\Delta}$ (resp. l'idéal $J_{r,\Delta,s}$) est l'idéal $I^{\otimes}_{r,\Delta}$ (resp. l'idéal $I^{\otimes}_{r,\Delta,s}$).
	
		%Il se trouve que cet idéal est une union croissante indexée par $n$, par exemple, pour $n$ tel que $r$ et $s$ soient dans $q^{-n}\mathbb{N}$, en posant $$J^{\otimes}_{r,n,\Delta}=\left( t^{\underline{k}} \, |\, \underline{k} \in q^{-n} \mathbb{N}^{\Delta},\,\,\, \sum_{\alpha \in \Delta} k_{\alpha}=r\right),$$ nous obtenons effectivement que $$J^{\otimes}_{r,\Delta}=\bigcup_{n \gg 0} J^{\otimes}_{r,n,\Delta}.$$ et respectivement $J^{\otimes}_{r,\Delta,s}=\bigcup_{n\gg 0} \left[ J^{\otimes}_{r,n,\Delta,s}\right]$ Il nous suffit donc de démontrer que l'énoncé en rajoutant cet indice $n$. Tous les arguments qui suit sont identiques dans les deux cas, pour omettre des indices, nous nous contentons de traiter le cas de $J^{\otimes}_{r,\Delta}$.

			Nous savons déjà grâce au Lemme \ref{sep_series_formelles} réciproque de $J$ par le dernier morphisme est $J^{\otimes}$. Pour passer de $J^{\otimes}$ à $I^{\otimes}$, fixons $n$ tel que $\forall i,\alpha, \,\,k_{i,\alpha} \in q^{-n}\mathbb{N}$. Pour démontrer que l'image réciproque de $J^{\otimes}$ est $I^{\otimes}$, nous commençons par une remarque. Prenons une famille $(e_i)_{i\in \mathcal{I}}$ dans $\widetilde{E}^+$ telle que la famille des réductions est une $\mathbb{F}_q$-base  de $\sfrac{\widetilde{E}^+}{\varpi^{q^{-n}}}$. Puisque $\widetilde{E}^+$ est $\varpi^{q^{-n}}$-adiquement séparé et complet, tout élément s'écrit de manière unique comme $$\sum_{k\geq 0,\, i \in \mathcal{I}} a_{k,i} e_i \varpi^{\sfrac{k}{q^n}}$$ pour une famille presque nulle $((a_{k,i})_{k\geq 0})_{i\in\mathcal{I}}$  d'éléments de $\mathbb{F}_q$. Autrement dit, nous avons l'égalité algébrique $$\widetilde{E}^+=\bigoplus_{i\in \mathcal{I}} \mathbb{F}_q\llbracket \varpi^{q^{-n}}\rrbracket e_i.$$
		Pour tout multi-indice $\underline{i} \in \mathcal{I}^{\Delta}$, nous appelons $e_{\underline{i}}=\otimes e_{i_{\alpha}}$ de tel sorte que $$\bigotimes_{\alpha \in \Delta,\,\, \mathbb{F}_q} \widetilde{E}_{\alpha}^+= \bigoplus_{\underline{i} \in \mathcal{I}^{\Delta}} \left(\bigotimes_{\alpha \in \Delta,\,\, \mathbb{F}_q} \mathbb{F}_q \llbracket \varpi_{\alpha}^{q^{-n}}\rrbracket\right) e_{\underline{i}}.$$ Appelons $$A_n := \bigotimes_{\alpha \in \Delta, \, \mathbb{F}_q} \mathbb{F}_q\llbracket \varpi_{\alpha}^{q^{-n}} \rrbracket, \,\, A'_n := \bigotimes_{\alpha \in \Delta, \, \overline{\mathbb{F}_q}} \overline{\mathbb{F}_q}\llbracket \varpi_{\alpha}^{q^{-n}} \rrbracket \,\, \text{et} \,\, B_n:=\bigotimes_{\alpha \in \Delta, \, \overline{\mathbb{F}_q}} k^{\mathrm{alg}}\llbracket t_{\alpha}^{q^{-n}} \rrbracket.$$  
		Puisque $\sfrac{\widetilde{E}^+}{\varpi^{q^{-n}}}\otimes_{\mathbb{F}_q} \overline{\mathbb{F}_q} \cong \sfrac{\widetilde{E}'^+}{\varpi^{q^{-n}}} \hookrightarrow \sfrac{k^{\mathrm{alg}}\llbracket t^{\R}\rrbracket}{t^{q^{-n}}}$, nous pouvons étendre la famille $(e_i)_{i\in \mathcal{I}}$ en une famille indexée par $\mathcal{J}\supset \mathcal{I}$ telle que $(f_j)_{j\in\mathcal{J}}$ soit une $\overline{\mathbb{F}_q}$-base de $\sfrac{k^{\mathrm{alg}}\llbracket t^{\R}\rrbracket}{t^{q^{-n}}}$ et que l'image de $e_i$ coïncide avec $f_i$. La famille $(e_i)_{i\in \mathcal{I}}$ s'identifie aussi à une $\overline{\mathbb{F}_q}$-base de $\sfrac{\widetilde{E}'^+}{\varpi^{q^{-n}}}$. Nous obtenons donc trois décompositions comme $A_n$-module (resp. $A'_n$- ou $B_n$-module) $$\bigotimes_{\alpha \in \Delta,\,\, \mathbb{F}_q} \widetilde{E}_{\alpha}^+= \bigoplus_{\underline{i} \in \mathcal{I}^{\Delta}} A_n e_{\underline{i}},\,\, \bigotimes_{\alpha \in \Delta,\,\, \overline{\mathbb{F}_q}} \widetilde{E}_{\alpha}'^+= \bigoplus_{\underline{i} \in \mathcal{I}^{\Delta}} A'_n e_{\underline{i}} \,\, \text{et} \,\, \bigotimes_{\alpha \in \Delta,\,\, \overline{\mathbb{F}_q}} k^{\mathrm{alg}}\llbracket t_{\alpha}^{\R}\rrbracket = \bigoplus_{\underline{j} \in \mathcal{J}^{\Delta}} B_n f_{\underline{j}}.$$ Les idéaux $I^{\otimes}$, $I'^{\otimes}$ et $J$ sont engendrés par des éléments qui appartiennent respectivement à $A_n$, $A'_n$ et $B_n$. Nous appelons $I^{\otimes,\mathrm{res}}$ l'idéal engendré dans $A_n$, $I'^{\otimes,\mathrm{res}}$ dans $A'_n$ et $J^{\otimes,\mathrm{res}}$ dans $B_n$. Un élément appartient à $I^{\otimes}$ (resp. $I'^{\otimes}$, resp. $J^{\otimes}$) si et seulement si\footnote{Pour remontrer, faire attention que seulement un nombre fini de coordonnées sont non nulles et l'idéal considéré est de type fini.} toutes ses coordonnées appartiennent à $I^{\otimes,\mathrm{res}}$ (resp. $I'^{\otimes,\mathrm{res}}$, resp. $J^{\otimes,\mathrm{res}}$). 
		
		Nous allons exécuter des manipulations lourdes en notations, qui reviennent au fond à considérer les coordonnées dans une décomposition sur la famille des $e_{\underline{i}} \varpi^{\sfrac{\underline{k}}{q^n}}$. Toutefois, en se contentant de décomposer ainsi, on perdrait la trace de l'appartenance des coordonnées au produit tensoriel des séries formelles à l'intérieur des séries multivariables. Cette justification sur le caractère hideux de ce qui suit terminée, prenons le morphisme $$\bigotimes_{\alpha \in \Delta,\,\, \overline{\mathbb{F}_q}} \widetilde{E}'^+_{\alpha} \hookrightarrow \bigotimes_{\alpha \in \Delta,\,\, \overline{\mathbb{F}_q}} k^{\mathrm{alg}}\llbracket t_{\alpha}^{\mathbb{R}}\rrbracket .$$ Nous obtenons donc un diagramme commutatif comme suit 
		
		\begin{center}
			\begin{tikzcd}
				\bigotimes_{\alpha \in \Delta,\,\, \overline{\mathbb{F}_q}} \widetilde{E}'^+_{\alpha} \ar[r,"="] \ar[d,hook] &   \bigoplus_{\underline{i} \in \mathcal{I}'^{\Delta}} \left(\bigotimes_{\alpha \in \Delta,\,\, \overline{\mathbb{F}_q}} \overline{\mathbb{F}_q} \llbracket \varpi_{\alpha}^{q^{-n}}\rrbracket\right) e_{\underline{i}}  \ar[d,"\varpi^{q^{-n}}\mapsto t^{q^{-n}} \,\, \mathrm{et} \,\, \mathcal{I}'\subset \mathcal{J}"]\\
				
				\bigotimes_{\alpha \in \Delta,\,\, \overline{\mathbb{F}_q}} k^{\mathrm{alg}}\llbracket t_{\alpha}^{\mathbb{R}}\rrbracket \ar[r,"="]  & \bigoplus_{\underline{j} \in \mathcal{J}^{\Delta}} \left(\bigotimes_{\alpha \in \Delta,\,\, \overline{\mathbb{F}_q}} \overline{\mathbb{F}_q} \llbracket t_{\alpha}^{q^{-n}}\rrbracket\right) f_{\underline{j}}
			\end{tikzcd}
		\end{center} L'idéal $J^{\otimes}$ est exactement les éléments de coordonnées dans $J^{\otimes,\mathrm{res}}$, qui s'identifie à $I'^{\otimes,\mathrm{res}}$ via $\varpi\mapsto t$. L'image réciproque de $J^{\otimes}$ est consitutée des éléments de coordonnées dans $I'^{\otimes,\mathrm{res}}$, i.e. les éléments de $I'^{\otimes}$.
		
		Regardons le morphisme $$\bigotimes_{\alpha \in \Delta,\,\, \mathbb{F}_q} \widetilde{E}^+_{\alpha} \hookrightarrow \bigotimes_{\alpha \in \Delta,\,\, \overline{\mathbb{F}_q}} \widetilde{E}'^+_{\alpha}$$ Puisque $\widetilde{E}'^+$ est le complété de $\widetilde{E}^+ \otimes_{\mathbb{F}_q} \overline{\mathbb{F}_q}$, il se trouve qu'une famille $(e_i)_{i\in \mathcal{I}}$ pour $\widetilde{E}^+$ et $\mathbb{F}_q$ fournit également une famille pour $\widetilde{E}'^+$ et $\overline{\mathbb{F}_q}$. Nous obtenons ainsi un diagramme commutatif
		
		\begin{center}
			\begin{tikzcd}
				\bigotimes_{\alpha \in \Delta,\,\, \mathbb{F}_q} \widetilde{E}^+_{\alpha} \ar[r,"="] \ar[d,hook] &   \bigoplus_{\underline{i} \in \mathcal{I}^{\Delta}} \left(\bigotimes_{\alpha \in \Delta,\,\, \mathbb{F}_q} \mathbb{F}_q \llbracket \varpi_{\alpha}^{q^{-n}}\rrbracket\right) e_{\underline{i}}  \ar[d,"\otimes \overline{\mathbb{F}_q}"]\\
				
				\bigotimes_{\alpha \in \Delta,\,\, \overline{\mathbb{F}_q}} \widetilde{E}'^+_{\alpha} \ar[r,"="]  & \bigoplus_{\underline{j} \in \mathcal{I}^{\Delta}} \left(\bigotimes_{\alpha \in \Delta,\,\, \overline{\mathbb{F}_q}} \overline{\mathbb{F}_q} \llbracket t_{\alpha}^{q^{-n}}\rrbracket\right) e_{\underline{i}}
			\end{tikzcd}
		\end{center} Comme les appartenances à $I^{\otimes}$ et $I'^{\otimes}$ se voient coordonnée par coordonnée, il suffit de démontrer que l'image réciproque de $I'^{\otimes,\mathrm{res}}$ par la tensorisation par $\overline{\mathbb{F}_q}$ est $I^{\otimes,\mathrm{res}}$. Prenons $\{1\} \sqcup \mathcal{B}$ une $\mathbb{F}_q$-base de $\overline{\mathbb{F}_q}$. Il est possible d'écrire $$\bigotimes_{\alpha \in \Delta,\,\, \overline{\mathbb{F}_q}} \overline{\mathbb{F}_q} \llbracket t_{\alpha}^{q^{-n}}\rrbracket = \left(\bigotimes_{\alpha \in \Delta,\,\, \mathbb{F}_q} \mathbb{F}_q \llbracket t_{\alpha}^{q^{-n}}\rrbracket\right) \oplus \left[\bigoplus_{b\in \mathcal{B}} \left(\bigotimes_{\alpha \in \Delta,\,\, \mathbb{F}_q} \mathbb{F}_q \llbracket t_{\alpha}^{q^{-n}}\rrbracket\right) b\right].$$ À nouveau, l'appartenance à $I'^{\otimes,\mathrm{res}}$ se lit sur les coordonnées : il faut qu'elles appartiennent à $I^{\otimes,\mathrm{res}}$. Cela conclut.

		\medskip
		
		Nous avons démontré par circonvolutions et tourbillons de notations que $\iota^{-1}(J)=I^{\otimes}$ pour toute famille finie de multi-indices. En prenant l'union, on obtient le résultat pour des familles quelconques de multi-indices.
		
		\medskip
		En particulier, pour tout entier $k\geq 1$, l'image réciproque de l'idéal $(\underline{t})^k$ est l'idéal $(\underline{\varpi})^k$. L'anneau $\widetilde{E}^+_{\Delta}$ est la complétion de $\otimes_{\mathbb{F}_q} \widetilde{E}^+_{\alpha}$ par rapport aux idéaux $(\underline{\varpi})^k$ pour $k\geq 1$. Dire que l'image réciproque de $(\underline{t})^k$ est$(\underline{\varpi})^k$ implique exactement que $\iota$ se complète en un morphisme d'anneaux $\hat{\iota}$ $$\widetilde{E}^+_{\Delta}=\lim \limits_{\substack{\longleftarrow \\ k}} \quot{\left(\bigotimes_{\alpha \in \Delta, \,\,\mathbb{F}_q} \widetilde{E}^+_{\alpha}\right)}{(\underline{\varpi})^k} \hookrightarrow  \lim \limits_{\substack{\longleftarrow \\ k}} \quot{\left(\bigotimes_{\alpha \in \Delta, \, \overline{\mathbb{F}_q}} k^{\mathrm{alg}}\right) \llbracket t_{\alpha}^{\mathbb{R}} \, |\, \alpha \in \Delta \rrbracket}{(\underline{t})^k} =\left(\bigotimes_{\alpha \in \Delta, \, \overline{\mathbb{F}_q}} k^{\mathrm{alg}}\right) \llbracket t_{\alpha}^{\mathbb{R}} \, |\, \alpha \in \Delta \rrbracket$$

		Revenons aux énoncés sur les idéaux. Nous commençons par raffiner notre propriété sur les images inverses en remplaçant $\iota$ par $\hat{\iota}$. Considérons une famille de multi-indice $(\underline{k_i})_{i\in \mathfrak{I}}$ telle que $$\exists n, \, \forall \alpha \in \Delta, \, \exists k_{\alpha}\leq n, \,\, (0,\ldots, 0, k_{\alpha}, 0,\ldots, 0) \in \{\underline{k_i}\, |\, i \in \mathfrak{I}\}.$$ Cette condition sur les indices équivaut à ce que l'idéal $(t^{\underline{k_i}} \, |\, i\in \mathfrak{I})$ est ouvert pour la topologie $(\underline{t})$-adique dans les séries de Hahn-Mal'cev multivariables. Prouver que $\hat{\iota}^{-1}((t^{\underline{k_i}} \, |\, i\in \mathfrak{I}))=(\varpi^{\underline{k_i}} \, |\, i\in \mathfrak{I})$ dans $\widetilde{E}^+_{\Delta}$ se restreint alors à le prouver pour $\otimes_{\alpha,\Delta,\, \mathbb{F}_q} \widetilde{E}^+_{\alpha}$.

		Cette remarque va nous permettre de conclure quant à l'image réciproque de $J_{r,\Delta}$. A première vue, nous voudrions que $\hat{\iota}^{-1}(J_{r,\Delta})=I_{r,\Delta}$ mais $J_{r,\Delta}$ n'est pas engendré par $\{t^{\underline{s}}\, |\, \underline{s}\in \mathbb{N}[q^{-1}]^{\Delta} \,\, \text{tel que} \,\, \sum s_{\alpha}=r\}$. En réalité, $$J_{r,\Delta}=\bigcap_{\substack{r'\in \mathbb{N}[q^{-1}] \\ \mathrm{tel} \,\, \mathrm{que} \,\, r'<r}} \left( t^{\underline{s}} \, \big|\, \underline{s} \in \mathbb{N}[q^{-1}]^{\Delta} \,\, \mathrm{tel} \,\, \mathrm{que} \,\, \sum s_{\alpha}=r'\right).$$ Le calcul des images réciproques des idéaux ouverts passent à l'intersection en $\hat{\iota}^{-1}(J_{r,\Delta})=I_{r^-,\Delta}$.
	\end{proof}

	\begin{coro}\label{edplus_intègre}
		L'anneau $\widetilde{E}_{\Delta}^+$ est intègre.
	\end{coro}
	\begin{proof}
		Puisque $\overline{\mathbb{F}_q}$ est parfait et algébriquement clos, nous savons grâce à \cite[V \S 5, Prop. 9]{bourbaki_alg} que toutes ses extensions sont régulières, en particulier $\left(\otimes_{\alpha\in \Delta,\,\overline{\mathbb{F}_q}} k^{\mathrm{alg}}\right)$ est intègre. Nous déduisons que l'anneau de séries de Hahn-Mal'cev multivariable est intègre, puis que $\widetilde{E}_{\Delta}^+$ est intègre par injectivité de $\hat{\iota}$.
	\end{proof}
	
	\begin{coro}\label{norm_mult}
	\begin{enumerate}[itemsep=0mm]
			\item La formule $$|x|_{\widetilde{E},\Delta}=e^{-\sup \left\{r\in \mathbb{N}[q^{-1}] \, |\,\, x\in I_{r^-,\Delta}\right\}}$$ fournit une norme multiplicative sur $\widetilde{E}_{\Delta}^+$.
			
			\item Nous avons
			
			$$\forall \underline{k} \in (\mathbb{N}[q^{-1}])^{\Delta}, \,\,\,\hat{\iota}^{-1}\left((t^{\underline{k}})\right)=(\varpi^{\underline{k}}).$$
		\end{enumerate}
	\end{coro}
	\begin{proof}
		\begin{enumerate}[itemsep=0mm]
			\item Les propriétés de $|\cdot|_{\widetilde{E},\Delta}$ se déduisent alors de celles de $|\cdot|_{\Delta,\underline{1}}$ à la Proposition \ref{expression_val} et des égalités $\hat{\iota}^{-1}(J_{r,\Delta})=I_{r^-,\Delta}$ à la Proposition \ref{inj_1}.

			\item Il faut démontrer que $\hat{\iota}^{-1}((t^{\underline{k}}))\subseteq(\varpi^{\underline{k}})$. Soit $x$ dans l'image réciproque. Nous avons $$\forall n,\,\,\,\hat{\iota}(x)\in (t^{\underline{k}},J_{n,\Delta}).$$ Ce dernier idéal étant $(\underline{\varpi})$-adiquement ouvert, nous avons démontré en prouvant la Proposition \ref{inj_1} que $\forall n, \,\,\, x \in (\varpi^{\underline{k}},I_{n-,\Delta})$. \'Ecrivons $x=\varpi^{\underline{k}} x'_n+y_n$ avec $y_n\in I_{n-,\Delta}$. Alors, $x=\lim \varpi^{\underline{k}} x'_n$ pour la norme $|\cdot|_{\widetilde{E},\Delta}$. Cette dernière étant multiplication, la suite $x'_n$ est également de Cauchy pour la norme $|\cdot|_{\widetilde{E},\Delta}$. Puisque $$I_{(n+1)|\Delta|-,\Delta} \subset I_{n|\Delta|,\Delta} \subset (\underline{\varpi^n}) \subset (\underline{\varpi})^n,$$ la suite $(x'_n)$ est également $(\underline{\varpi})$-adiquement de Cauchy et possède une limite $x'\in \widetilde{E}_{\Delta}^+$ qui vérifie que $x=\varpi^{\underline{k}} x'$.
		\end{enumerate}
	\end{proof}

	\vspace{0.75cm}
	\subsection{Relations de coinduction perfectoïdes}
	
	Pour comprendre les anneaux $\widetilde{F}_{\Delta,q}$ nous les exprimons comme des coinduites des $\widetilde{F}_{\Delta}$, que nous venons d'analyser. L'intuition nous vient du lemme suivant sur les corps finis.
	
		\begin{lemma}\label{identif_coindu_fq}
		Définissons une structure de $\Phi_{\Delta,q',r}^{\mathrm{gp}}$-anneau sur $\mathbb{F}_{q'}$ en en faisant agir le générateur $\varphi_{\Delta,r}$ du quotient $\sfrac{\Phi_{\Delta,q',r}^{\gp}}{\Phi_{\Delta,q'}^{\gp}}\cong \sfrac{\varphi_{\Delta,r}^{\Z}}{\varphi_{\Delta,q'}^{\Z}}$ comme le $r$-Frobenius. Il existe un isomorphisme de $\Phi_{\Delta,q,r}^{\mathrm{gp}}$-anneaux 
		
		$$\bigotimes_{\alpha \in \Delta,\,\, \mathbb{F}_q} \mathbb{F}_{q'} \cong \coindu{\Phi_{\Delta,q',r}^{\mathrm{gp}}}{\Phi_{\Delta,q,r}^{\mathrm{gp}}}{\mathbb{F}_{q'}}.$$
	\end{lemma}
	\begin{proof}
		Pour tout $\psi\in \Phi_{\Delta,q,r}^{\mathrm{gp}}$, nous définissons $$\psi_{\mathrm{sp\acute{e}}} \, : \, \bigotimes_{\alpha \in \Delta,\,\,\mathbb{F}_q} \mathbb{F}_{q'} \rightarrow \mathbb{F}_{q'}, \,\,\, \otimes\, x_{\alpha} \mapsto \prod \psi_{\alpha}(x_{\alpha})$$ et nous considérons \begin{equation*}\label{eq:times} \bigotimes_{\alpha \in \Delta, \,\,\mathbb{F}_q} \mathbb{F}_{q'} \rightarrow \coindu{\Phi_{\Delta,q',r}^{\mathrm{gp}}}{\Phi_{\Delta,q,r}^{\mathrm{gp}}}{\mathbb{F}_{q'}}, \,\,\, x\mapsto [\psi \mapsto \psi_{\spe}(x)] \tag{*1} \end{equation*} pour lequel nous laissons aux lecteurs et lectrices le soin de vérifier qu'il s'agit d'un isomorphisme de \linebreak$\Phi_{\Delta,q,r}^{\mathrm{gp}}$-anneaux.
	\end{proof}
	
	Gardons en tête les notations de la preuve précédente pour établir un analogue pour nos anneaux multivariables. Nous nous plaçons dans le même contexte qu'à la sous-section précédente.
	
	\begin{prop}\label{identif_coindu_perf}
		Rappelons que $\widetilde{F}_{\Delta}^+$ est muni d'une structure de $\Phi_{\Delta,q',r}^{\mathrm{gp}}$-anneau et $\widetilde{F}_{\Delta,q}^+$ d'une structure de $\Phi_{\Delta,q,r}^{\mathrm{gp}}$-anneau. Il existe un isomorphisme canonique de $\Phi_{\Delta,q,r}^{\mathrm{gp}}$-anneaux
		
		$$\widetilde{F}_{\Delta,q}^+ \cong \coindu{\Phi_{\Delta,q',r}^{\mathrm{gp}}}{\Phi_{\Delta,q,r}^{\mathrm{gp}}}{\widetilde{F}_{\Delta}^+}.$$
		C'est un homéomorphisme pour les topologies discrètes (resp. adiques) sur $\widetilde{F}_{\Delta,q}^+$ et $\widetilde{F}_{\Delta}^+$, et la topologie produit sur la coinduite. Le même résultat est vrai pour $\widetilde{F}_{\Delta,q}$ et $\widetilde{F}_{\Delta}$.
	\end{prop}
	\begin{proof}
		Nous appelons encore $\psi_{\spe}$ le morphisme de $\widetilde{F}_{\Delta,q}$ dans $\widetilde{F}_{\Delta}$ donné par la même formule que dans la démonstration du lemme \ref{identif_coindu_fq}.
		
		Choisissons un système $\mathcal{R}$ de représentants du quotient fini $\sfrac{\Phi_{\Delta,q,r}^{\mathrm{gp}}}{\Phi_{\Delta,q',r}^{\mathrm{gp}}}$. Nous avons un isomorphisme d'anneaux $$\coindu{\Phi_{\Delta,q',r}^{\mathrm{gp}}}{\Phi_{\Delta,q,r}^{\mathrm{gp}}}{\mathbb{F}_{q'}} \xrightarrow{\sim} \prod_{\psi\in \mathcal{R}} \mathbb{F}_q',\,\,\,\, f \mapsto (f(\psi))_{\psi\in \mathcal{R}}.$$ %EN fixant $\beta\in \Delta$, il est possible de prendre pour $\mathcal{R}$ le système $\{ \prod_{\alpha \in \Delta \backslash \{\beta\}} \varphi_{\alpha,q}^{n_{\alpha}} \, |\, (n_{\alpha})\in \llbracket 0, f-1\rrbracket^{\Delta \backslash \{\beta\}}\}$ où $q'=q^f$. La coinduite est donc un produit fini.
		
		Pour n'importe quel $\beta \in \Delta$, nous pouvons établir la suite d'isomorphismes d'anneaux suivante :
		
		\begin{center}
			\begin{align*}
				\left(\bigotimes_{\alpha\in \Delta \backslash \{\beta\}, \mathbb{F}_q} \mathbb{F}_{q'} \right) \otimes_{\mathbb{F}_q} \widetilde{F}_{\beta}^+ & \cong \left(\bigotimes_{\alpha\in \Delta , \mathbb{F}_q} \mathbb{F}_{q'} \right) \otimes_{\mathbb{F}_{q',\beta}} \widetilde{F}_{\beta}^+ \\
				&\cong \left( \prod_{\psi \in \mathcal{R}} \mathbb{F}_{q'}\right) \otimes_{\mathbb{F}_{q',\beta}} \widetilde{F}_{\beta}^+ \\
				& \cong \prod_{\psi\in \mathcal{R}} \left( \mathbb{F}_{q'} \otimes_{\psi_{\beta},\mathbb{F}_{q'}} \widetilde{F}_{\beta}^+\right) \\
				& \cong \prod_{\psi\in \mathcal{R}} \widetilde{F}_{\beta}^+ \\
				& \cong \coindu{\Phi_{\Delta,q',r}^{\mathrm{gp}}}{\Phi_{\Delta,q,r}^{\mathrm{gp}}}{\widetilde{F}_{\beta}^+}
			\end{align*}
		\end{center} où $\mathbb{F}_{q',\beta}$ est une notation pour indiquer que le produit tensoriel est vu comme $\mathbb{F}_{q'}$-algèbre via son facteur $\beta$ et où $\psi_{\beta}$ est la composante de $\psi$ sur $\varphi_{\beta,p}^\N$. Le passage à la troisième ligne correspond simplement à expliciter cette structure d'algèbre sur chaque facteur du produit, et à faire commuter le produit fini au produit tensoriel. Le passage à la quatrième ligne correspond à appliquer au facteur $\psi$ l'isomorphisme $x\otimes y \mapsto x \psi_{\beta}(y)$. Nous vérifions que cet isomorphisme est un isomorphisme de $\Phi_{\Delta,q,r}^{\mathrm{gp}}$-anneau en suivant l'image de $$\chi\bigg(\left(\otimes_{\alpha \neq \beta} x_{\alpha} \right) \otimes y_{\beta}\bigg)=\otimes_{\alpha \neq \beta} \chi_{\alpha}(x_{\alpha})\otimes\chi_{\beta}(y_{\beta})$$ le long des isomorphismes et en vérifiant qu'elle vaut $$\left[\psi \mapsto \prod_{\alpha \neq \beta} (\psi\chi)_{\alpha}(x_{\alpha})\times  (\psi \chi)_{\beta}(y_{\beta})\right].$$
		
		Une itération de ces arguments fournit un isomorphisme de $\Phi_{\Delta,q,r}^{\mathrm{gp}}$-anneaux : 
		
		\begin{equation*}\label{eqn:grosequation}\bigotimes_{\alpha \in \Delta,\mathbb{F}_q} \widetilde{F}_{\alpha}^+ \xrightarrow{\sim} \bigcoindu{\Phi_{\Delta,q',r}^{\mathrm{gp}}}{\Phi_{\Delta,q,r}^{\mathrm{gp}}}{\bigotimes_{\alpha \in \Delta, \mathbb{F}_{q'}} \widetilde{F}_{\alpha}^+}, \,\,\,\, \otimes \, y_{\alpha} \mapsto \left[\psi \mapsto \psi_{\mathrm{sp\acute{e}}}(\otimes y_{\alpha})= \otimes  \psi_{\alpha}(y_{\alpha})\right].\tag{**}\end{equation*} Attention, les arguments de $\psi_{\mathrm{sp\acute{e}}}$ sont des tenseurs sur $\mathbb{F}_q$ mais les valeurs de sortie sont des tenseurs sur $\mathbb{F}_{q'}$.

		Grâce à l'expression explicite du morphisme, chaque $\varpi_{\alpha}$ est envoyé dans la coinduite sur $(\psi_{\alpha}(\varpi_{\alpha}))_{\psi\in\mathcal{R}}$. Les puissances de l'idéal $\left((\psi_{\alpha}(\varpi_{\alpha}))_{\psi\in\mathcal{R}} \,\big|\, \alpha \in \Delta\right)$ et de l'idéal $\left((\varpi_{\alpha})_{\psi\in \mathcal{R}}\,\big|\, \alpha \in \Delta\right)$ étant cofinales les unes dans les autres, l'isomorphisme est un homéomorphisme pour la topologie $(\underline{\varpi})$-adique à la source et et le produit des topologies $(\underline{\varpi})$-adiques au but. L'isomorphisme se complète donc en un isomorphisme de $\Phi_{\Delta,q,r}^{\mathrm{gp}}$-anneaux $$\widetilde{F}_{\Delta,q}^+\cong \coindu{\Phi_{\Delta,q',r}^{\mathrm{gp}}}{\Phi_{\Delta,q,r}^{\mathrm{gp}}}{\widetilde{F}_{\Delta}^+}.$$ C'est évidemment un homéomorphisme pour les topologies discrètes partout.
		
		Par intégrité de $\widetilde{F}_{\Delta}^+$, un élément $x \in \widetilde{F}_{\Delta}^+$ appartient à $\varpi_{\Delta}\widetilde{F}_{\Delta,q}^+$ si et seulement si chaque $\psi_{\mathrm{sp\acute{e}}}(x)$ appartient à $\psi_{\mathrm{sp\acute{e}}}(\varpi_{\Delta}) \widetilde{F}_{\Delta}^+$. En écrivant la coinduite comme un produit fini et en remarquant que tous les $(\psi_{\spe}(\varpi_{\Delta}))$ ont même radical, c'est un homéomorphisme pour la topologie $\varpi_{\Delta}$-adique et le produit des topologies $\varpi_{\Delta}$-adiques. 
			
		En inversant $\varpi_{\Delta}$ et en remarquant que $(\varpi_{\Delta})$ et $(\psi_{\mathrm{sp\acute{e}}}(\varpi_{\Delta}))$ ont même radical dans $\widetilde{F}_{\Delta}$, on obtient l'isomorphisme de $\Phi_{\Delta,q,r}^{\mathrm{gp}}$-anneaux continus pour les topologies adiques.
	\end{proof}
	
	\begin{coro}\label{sanstors_perf}
	L'anneau $\widetilde{F}_{\Delta,q}^+$ est réduit et sans $\widetilde{E}_{\Delta}^+$-torsion. 
	\end{coro}
	\begin{proof}
		La coinduite est un anneau de fonctions à valeurs dans $\widetilde{F}^+_{\Delta}$. Nous avons prouvé que ce dernier est intègre au Corollaire \ref{edplus_intègre} donc la coinduite est réduite.
		
		Pour la torsion, utiliser que le plongement de $\widetilde{E}_{\Delta}^+$ dans la coinduite s'écrit $\prod \psi$, que chaque $\psi$ est injectif et que $\widetilde{F}_{\Delta}^+$ est sans $\widetilde{E}_{\Delta}^+$-torsion.
	\end{proof}
	
	\begin{coro}\label{inv_phi_carp_perf}
		L'inclusion $\mathbb{F}_r\subseteq \widetilde{F}_{\Delta,q}^{\Phi_{\Delta,q,r}^{\gp}}$ est une égalité.
	\end{coro}
	\begin{proof}
		Grâce à l'identification à une coinduite, on se ramène à prouver que l'inclusion $\mathbb{F}_r\subseteq \widetilde{F}_{\Delta}^{\Phi_{\Delta,q',r}^{\gp}}$ est une égalité.
		
		Soit $y$ un élément invariant. Nous fixons $n$ tel que $\varpi_{\Delta}^n y \in \widetilde{F}_{\Delta}^+$. En appliquant $\varphi_{\Delta,r}$, on trouve $$(\varpi_{\Delta}^n y)^r=\varpi_{\Delta}^{rn} \varphi_{\Delta,r}(y)=\varpi_{\Delta}^{(r-1)n} (\varpi_{\Delta}^n y).$$ En appliquant $|\cdot|_{\widetilde{F},\Delta}$ dont nous avons prouvé au Corollaire \ref{expression_val} qu'elle est multiplicative, nous trouvons \linebreak$|\varpi_{\Delta}^n y|_{\widetilde{F},\Delta}=0$ ou $|\varpi_{\Delta}^n y|_{\widetilde{F},\Delta}=|\varpi_{\Delta}^n|_{\widetilde{F},\Delta}$. En utilisant la séparation de la norme, sa multiplicativité et le fait que $\varpi_{\Delta}^n$ n'est pas diviseur de zéro, on en déduit que $y=0$ ou $y\in \widetilde{F}_{\Delta}^+$ de norme $1$. Il existe donc $y_0 \in \mathbb{F}_{q'}$ tel que $|y-y_0|_{\widetilde{F},\Delta}<1$. En appliquant le même raisonnement à cette différence, on trouve $y=y_0$.
	\end{proof}
	
	\begin{coro}\label{cestperfectoid}
		Si l'on munit $\widetilde{F}_{\Delta,q}$ de topologie adique, la paire $(\widetilde{F}_{\Delta,q},\widetilde{F}_{\Delta,q}^+)$ est une paire de Huber perfectoïde.
	\end{coro}
	\begin{proof}
	La définition de la topologie adique implique que $\widetilde{F}_{\Delta,q}^+$ est un anneau de définition. Il est de plus parfait de pseudo-uniformisante $\varpi_{\Delta}$. Nous savons alors que $\widetilde{F}_{\Delta,q}$ est un anneau de Tate perfectoïde. Pour conclure, il reste à démontrer que $\widetilde{F}_{\Delta,q}^+$ est intégralement clos. Nous montrons même que $\widetilde{F}_{\Delta,q}^+=\widetilde{F}_{\Delta,q}^{\circ}$.
		
		Avec l'expression de $\widetilde{F}_{\Delta,q}$ comme coinduite, la topologie adique est produit des topologies sur $\widetilde{F}_{\Delta}$ ; son anneau de définition comme anneau de Huber est la limite des anneaux des définitions terme à terme ; son idéal de définition canonique est la limite de différents idéaux de définitions pour chaque terme. Ainsi, ses éléments bornés sont exactement la limite des éléments bornés : on se restreint à démontrer le résultat pour $\widetilde{F}_{\Delta}$. Soit $y \in \widetilde{F}_{\Delta}^{\circ}$ que l'on écrit $\sfrac{z}{\varpi_{\Delta}^k}$ avec $z\in \widetilde{F}^+_{\Delta}$. Par hypothèse, fixons $l\geq 1$ tel que $$\forall n\geq 0, \,\,\, \varpi_{\Delta}^l y^{q^n} \in \widetilde{F}^+_{\Delta}.$$ Puisque $\widetilde{F}^+_{\Delta}$ est parfait, cela implique que $$\forall n\geq 0, \,\,\, z \in \varpi_{\Delta}^{k-\sfrac{l}{q^n}} \widetilde{F}^+_{\Delta}.$$ Dans l'anneau des séries de Hahn-Mal'cev multivarialbes, nous en déduisons $$\forall n\geq 0, \,\,\, \hat{\iota}(z)\in (t_{\Delta}^{k-\sfrac{l}{q^n}})$$ puis $\hat{\iota}(z)\in (t_{\Delta}^{k})$. Grâce au Corollaire \ref{norm_mult}, ceci implique que $z\in \varpi_{\Delta}^k\widetilde{F}_{\Delta}^+$, i.e. que $y \in \widetilde{F}^+_{\Delta}$.
	\end{proof}

	\vspace{0.75cm}
	
	\subsection{\'Etude des anneaux imparfaits}

	Pour obtenir une version imparfaite de l'équivalence de Carter-Kedlaya-Z\'abr\'adi, nous établissons des relations de coinductions sur ces anneaux imparfaits. Nous avons besoin de la notion de sous-monoïde d'indice subtil fini introduite dans \cite[\S 3.3]{formalisme_marquis}. 
	
	\begin{defi}
		Soit $\mathcal{T}$ un monoïde et $\mathcal{S}<\mathcal{T}$ un sous-monoïde. L'ensemble des classes à gauche $\{t\mathcal{S}\, |\, t\in \mathcal{T}\}$ est muni d'un ordre par l'inclusion. Soit $\mathcal{R}_{\min}\subset \mathcal{T}$ tels que les $(t\mathcal{S})_{t\in \mathcal{R}_{\min}}$ sont distincts et parcourent toutes les classes maximales pour l'inclusion. Définissons $\mathcal{L}(\mathcal{R}_{\min}):=\{(s_1,s_2,t_1,t_2) \in \mathcal{S}^2 \times \mathcal{R}_{\min}^2 \, |\, s_1 t_1=s_2 t_2\}$ et munissons-le d'un ordre en fixant que $$ \forall s\in  \mathcal{S}, \, \forall (s_1,s_2,t_1,t_2)\in \mathcal{L}(\mathcal{R}_{\min}), \,\,\,(s s_1,s s_2,t_1,t_2)\leq (s_1,s_2,t_1,t_2).$$  Remarquons que l'ensemble ordonné $\mathcal{L}(\mathcal{R}_{\min})$ ne dépend pas à isomorphisme près du choix de représentants. Nous l'appelons $\mathcal{L}$ par abus de notation.
		
		Le sous-monoïde $\mathcal{S}$ est dit \textit{d'indice subtil fini} si les classes à gauche maximales sont en nombre fini et cofinales parmi les classes à gauche et si les quadruplets maximaux de $\mathcal{L}$ sont en nombre finis et cofinaux.
	\end{defi}

	Avec les notations précédentes, si l'on suppose simplement que les classes à gauches maximales sont cofinales, nous obtenons déjà $$\coindu{\mathcal{S}}{\mathcal{T}}{X}\cong\left\{ (x_t)\in \prod_{t\in \mathcal{R}_{\min}} X \, \bigg|\, \forall (s_1,s_2,t_1,t_2) \in \mathcal{L}(\mathcal{R}_{\min}), \,\,\, \varphi_{s_1}\left(x_{t_1}\right)=\varphi_{s_2}\left(x_{t_2}\right)\right\}.$$ La condition peut se retreindre à une famille cofinale de $\mathcal{L}(\mathcal{R}_{\min})$. L'indice subtil fini garantit  que la coinduction s'exprime comme une limite finie.

	\begin{lemma}\label{indice_subtil_fini_Phi}
	Le sous-monoïde $\Phi_{\Delta,q,p}<\Phi_{\Delta,p}$ est d'indice subtil fini.
	\end{lemma}
	\begin{proof}
		En posant $q=p^f$, on réécrit cette inclusion de monoïdes comme $(f\N)^{\Delta} + (1,\cdots,1)\N \subset \N^{\Delta}$. Toute classe à gauche $(n_i)+(f\N)^{\Delta} +  (1,\ldots,1) \N$ où $(n_i) \in \N^{\Delta}$ est contenue dans l'une des classes minimales pour l'inclusion $$\left\{(k_i) + (f\N)^{\Delta}+  (1,\ldots,1) \N \,\,\big|\,\, \forall i, \,\, 0\leq k_i <f \,\, \mathrm{et} \,\,\exists i, \,\, k_i=0\right\}.$$ Pour finir la démonstration, nous allons démontrer que pour toute paire $((n_i), (m_i))$ de $\Delta$-uplets, les relations entre classes à gauche associées ont une famille finie et cofinale de relations minimales. Puisque tout élément de $\N^{\Delta}$ est régulier, une relation est $(n_i)+(k_i)=(m_i)+(l_i)$ est entièrement déterminée par $(k_i)$. 
		
		Considérons $$R=\left\{ (r_i,r)\in \N^{\Delta} \times \N \,\,\big|\,\, (n_i)+(fr_i)+r(1,\ldots,1) \in (m_i)+(f\N)^{\Delta}+  (1,\ldots,1) \N\right\}.$$ L'ensemble $\N^{\Delta}\times \N$ muni de l'ordre partiel produit est un bel ordre (utiliser le lemme de Dickson pour comprendre les antichaînes). Ainsi $R$ possède une famille finale et finie. Les relations données par ces éléments minimaux forment un système fini cofinal des relations entre $(n_i)+(d\N)^{\Delta}+  (1,\ldots,1) \N$ et $(m_i)+(d\N)^{\Delta}+  (1,\ldots,1) \N$.
	\end{proof}

	\begin{prop}\label{identif_coindu_imperf}
		\begin{enumerate}[itemsep=0mm]
			\item Pour $F|E$ finie galoisienne, il existe un isomorphisme de $\left(\cg_{E,\Delta}\times\Phi_{\Delta,p}\right)$-anneaux topologiques discrets $$F_{\Delta,p}\cong \coindu{\Phi_{\Delta,q,p}}{\Phi_{\Delta,p}}{F_{\Delta,q}},$$ où la coinduite est munie de la topologie limite.
			
			\item Il existe un isomorphisme de $(\cg_{E,\Delta}\times \Phi_{\Delta,p})$-anneaux topologiques discrets $$E_{\Delta,p}^{\mathrm{sep}}\cong \coindu{\Phi_{\Delta,q,p}}{\Phi_{\Delta,p}}{E_{\Delta}^{\mathrm{sep}}},$$ où la coinduite est munie de la topologie limite.
		\end{enumerate}
	\end{prop}
	\begin{proof}
		\begin{enumerate}[itemsep=0mm]
			\item La démonstration suit la même stratégie que la Proposition \ref{identif_coindu_perf}, avec plusieurs changements notables. Tout d'abord, le morphisme similaire à celui donné dans la preuve en (\ref{eqn:grosequation}) est invariant par $\cg_{E,\Delta}$ puisque l'action commute à chaque $\psi_{\mathrm{sp\acute{e}}}$. Cette invariance est ensuite conservée par complétion et localisation. Il faut également prendre en compte que nos actions ne sont plus que des actions de monoïdes. Fixons $\mathcal{R}_{\min}$ un système fini de représentants des classes à gauche minimales pour l'inclusion, qui sont finales pour l'inclusion, pour le sous-monoïde $\Phi_{\Delta,q,p}<\Phi_{\Delta,p}$. Fixons $\mathcal{L}_{\min}\subset \mathcal{L}(\mathcal{R}_{\min})$ un système fini et final des éléments minimaux. Ceci est possible grâce au Lemme \ref{indice_subtil_fini_Phi}. Notons $(\mathcal{R}\nearrow \mathcal{L})_{\min}$ la petite catégorie ayant pour objets $\mathcal{R}_{\min} \sqcup \{ \psi_1 \tau_1\, |\, (\psi_1,\tau_1,\psi_2,\tau_2) \in \mathcal{L}_{\min}\}$ et pour flèches les identités et les $\tau_1\rightarrow \psi_1\tau_1$ pour toute relation dans $\mathcal{L}_{\min}$. Pour tout $\Phi_{\Delta,q,p}$-anneau $A$, nous avons alors un isomorphisme
			\begin{equation*}\label{eq:coinduecritureun}\coindu{\Phi_{\Delta,q,p}}{\Phi_{\Delta,p}}{A} \xlongrightarrow{\sim} \lim \limits_{\psi \in (\mathcal{R}\nearrow \mathcal{L})_{\min}} A, \,\,\, f\mapsto [\psi \mapsto f(\psi)] \tag{*2}\end{equation*} où le diagramme dont on prend la limite associe l'endomorphisme $\psi$ de $A$ au morphisme $\left[\tau \rightarrow \psi \tau\right]$ dans $(\mathcal{R}\nearrow \mathcal{L})_{\min}$. Cette description ne suffit pas : en effet le passage à la quatrième ligne dans la première suite d'isomorphismes en Proposition \ref{identif_coindu_perf} utilise que $$\mathbb{F}_{q'}\otimes_{\psi_{\beta},\mathbb{F}_{q'}} \widetilde{F}_{\beta}^+ \rightarrow \widetilde{F}_{\beta}^+, \,\,\, x\otimes y\mapsto x\psi_{\beta}(y)$$ est un isomorphisme. Pour $F_{\beta}^+$, il sera seulement injectif. Bonne nouvelle, les projections sur chaque facteur dans la limite de (\ref{eq:coinduecritureun}) ne sont pas non plus surjectives. Pour réparer l'argument, il suffit de démontrer que $$\lim \limits_{\psi \in (\mathcal{R}\nearrow \mathcal{L})_{\min}} F_{\beta}^+ = \lim \limits_{\psi \in (\mathcal{R}\nearrow \mathcal{L})_{\min}} \psi_{\beta}(F_{\beta}^+).$$ Démontrons que si la fonction $f$ appartient à $\coindu{\Phi_{\Delta,q,p}}{\Phi_{\Delta,p}}{F_{\beta}^+}$ et $\psi \in \Phi_{\Delta,p}$, alors $f(\psi)\in \psi_{\beta}(F_{\beta}^+)$. Supposons que $\psi_{\beta}=\varphi_{\beta,p}^n$ et fixons $m,k$ tels que $\varphi_{\beta,p}^{n+m}=\varphi_{\beta,q}^k$.Alors, on a $\varphi_{\Delta,p}^m \psi= \varphi_{\Delta,q}^k \psi_1$, puis $$\varphi_{\beta,p}^m(f(\psi))=\varphi_{\Delta,p}^m(f(\psi)) =f( \varphi_{\Delta,p}^m \psi)=f(\varphi_{\Delta,q}^k \psi_1)=\varphi_{\beta,q}^k (f(\psi_1))=\varphi_{\beta,p}^m(\varphi_{\beta,p}^n(f(\psi_1))).$$ L'injectivité de $\varphi_{\beta,p}^m$ conclut.

			%% pour l'injectivité de $\varphi_{\beta,r}$, on a caché les étapes d'itération qui sont plus délicates
			
			\item Puisque les colimites filtrantes sont tamisées, elles commutent en particulier naturellement à la coinduction. Il s'agit donc simplement de prendre la colimite des isomorphismes précédents.
		\end{enumerate}
	\end{proof}
	
	\vspace{0.75cm}
	
	\subsection{\'Etude des anneaux de caractéristique mixte}
	
	Nous reprenons les notations de \S \ref{section_equiv_car_zero_imparf}. Nous définissons $\mathcal{O}_{\mathcal{F}_{\Delta}}^+$ de manière identique à $\mathcal{O}_{\mathcal{F}_{\Delta}}^+$ en remplaçant $E$ par $F$ et $L$ par $L'=K\mathbb{Q}_{q'}$.

\begin{prop}\label{inj_dans_Witt}
	Il existe une injection d'anneau $$\hat{j}\, : \, \mathcal{O}_{\mathcal{E}_{\Delta}}^+ \hookrightarrow \mathrm{W}_L\left(\left(\bigotimes_{\alpha \in \Delta, \, \overline{\mathbb{F}_q}} k^{\mathrm{alg}}\right) \llbracket t_{\alpha}^{\mathbb{R}} \, |\, \alpha \in \Delta\rrbracket\right)$$ telle que chaque $\hat{j}(X_{\alpha})$ est un relevé de $t_{\alpha}$ et que $$\forall n\geq 1, \,\,\, \hat{j}^{-1}((\pi,\underline{[t]})^n)=(\pi,\underline{X})^n.$$
	\end{prop}
	\begin{proof}
		L'anneau $\mathcal{O}_{\mathcal{E}_{\Delta}}^+$ est une $\mathcal{O}_L$-algèbre $\pi$-adiquement séparée et complète, d'anneau résiduel $E_{\Delta}^+$ en $\pi$ et munie du relèvement $\phi_{\Delta,q}$ du $q$-Frobenius. Considérons la composée de l'injection $E_{\Delta}^+\hookrightarrow \widetilde{E}_{\Delta}^+$ du Lemme \ref{lien_anneaux_e} et de celle $\hat{\iota}$ en Proposition \ref{inj_1}. Elle envoie chaque $X_{\alpha}$ sur $t_{\alpha}$. La propriété universelle des vecteurs de Witt ramifiés construit alors le morphisme de $\mathcal{O}_L$-algèbres $\hat{j}$ annoncé. Il est injectif puisque la source est séparée et qu'il est injectif modulo $\pi$.
		
		Intéressons-nous aux images réciproques des idéaux. Le Lemme \ref{lien_anneaux_e} et la Proposition \ref{inj_1} obtiennent que pour tout idéal ouvert $\hat{\iota}((t^{\underline{d_i}} \, |\, i\in \mathcal{I}))\cap E_{\Delta}^+=(X^{\underline{d_i}})$. Soit à présent $(k_i)_{1\leq i \leq n}$ une famille finie d'entiers et $(\underline{d_i})_{1\leq i \leq n}$ de multi-indices. Les Teichmüller $[t]^{\underline{d_i}}]$ et les $\hat{j}(X)^{\underline{d_i}}$ diffèrent par des éléments inversibles. Soit \linebreak$x\in \hat{j}^{-1}\left((\pi^{k_i}[t^{\underline{d_i}}])\right)$. Soit $k=\min k_i$ et $I=\{1\leq i \leq n \, |\, k_i=k\}$. Nous savons que $x=\pi^k y \modulo{\pi^{k+1}}$ où $(y \modulo{\pi} \in \hat{j}^{-1}((t^{\underline{d_i}}\,|\, i\in I))$. Avec l'énoncé juste prouvé, cela implique que $x=\pi^k z_0+\pi^{k+1} z_1 $ où $z_0\in (\pi^{k_i} X^{\underline{d_i}}\, |\, i\in I)$ et $z_1 \in \hat{j}^{-1}((\pi^{k_i}X^{\underline{d_i}}\,|\, i\notin I))$. On conclut en répétant l'opération.
	\end{proof}
	
	\begin{prop}
	Soit $F|E$ finie galoisienne. Il existe un isomorphisme de $(\Phi_{\Delta,q,r}\times \cg_{E,\Delta})$-anneaux topologiques pour la topologie $\pi$-adique
	
	$$\mathcal{O}_{\mathcal{F}_{\Delta,q}}\cong \coindu{\Phi_{\Delta,q',r}}{\Phi_{\Delta,q,r}}{\mathcal{O}_{\mathcal{F}_{\Delta}}}.$$
	\end{prop}
	\begin{proof}
		Comme au Lemme \ref{indice_subtil_fini_Phi}, on prouve que $\Phi_{\Delta,q',r}<\Phi_{\Delta,q,r}$ est d'indice subtil fini. Nous choisissons $\mathcal{R}_{\min}$ un système de représentants des classes minimales pour l'inclusion et $\mathcal{L}_{\min}$ les relations minimales déduites. Nous considérons le morphisme que nous obtenons comme aux Propositions \ref{identif_coindu_perf} et \ref{identif_coindu_imperf} $$\mathcal{O}_{\mathcal{F}_{\Delta,q}}^+ \cong \coindu{\Phi_{\Delta,q',q}}{\Phi_{\Delta,q}}{\mathcal{O}_{\mathcal{F}_{\Delta}}^+}= \lim \limits_{(\mathcal{R}\nearrow \mathcal{L})_{\min}} \mathcal{O}_{\mathcal{F}_{\Delta}}^+.$$ De même qu'auxdites propositions, nous démontrons que c'est un isomorphisme modulo $\pi$. La limite étant finie, les deux côtés sont $\pi$-adiquement séparés et complets ce qui conclut que le morphisme est un isomorphisme. Remarquons que le choix de $L$ et $L'$ fournit bien des structures de $\Phi_{\Delta,q,r}$-anneau (resp. $\Phi_{\Delta,q',r}$-anneau), qui sont coinduits l'un de l'autre puisque la réduction modulo $\pi$ s'identifie à la relation de coinduction du Lemme \ref{identif_coindu_fq}.
	\end{proof}

	\begin{coro}
	L'anneau $\mathcal{O}_{\mathcal{E}_{\Delta}}$ est intègre et l'anneau $\mathcal{O}_{\widehat{\mathcal{E}_{\Delta}^{\mathrm{ur}}}}$ est sans $\mathcal{O}_{\mathcal{E}_{\Delta}}$-torsion.
	\end{coro}
	\begin{proof}
	L'intégrité est une conséquence de la Proposition \ref{inj_dans_Witt}. Le caractère sans $\mathcal{O}_{\mathcal{E}_{\Delta}}$-torsion de chaque $\mathcal{O}_{\mathcal{F}_{\Delta,q}}$ se déduit de la coinduction comme au Corollaire \ref{sanstors_perf}. Nous voulons passer à la complétion $\Oedhat$ de leur colimite. Soit $x \in \mathcal{O}_{\mathcal{E}_{\Delta}}$. \'Ecrivons $x=\pi^ny$ avec $(y\modulo{\pi})\neq 0$. Puisque chaque $F_{\Delta,q}$ est sans $(y\modulo{\pi})$-\linebreak-torsion, on en déduit que la multiplication par $x$ est un homéomorphisme sur son image pour la topologie \linebreak$\pi$-adique sur $\colim \limits_{F\in \mathcal{G}\mathrm{al}_{E}} \mathcal{O}_{\mathcal{F}_{\Delta}}$. La complétion $\pi$-adique est alors encore une injection, ce qui conclut.
	\end{proof}

	\vspace{1,5cm}
	\section{Monoïdes topologiques}\label{annexe_monoides}
	
	Dans cette annexe, nous donnons une suite d'énoncés qui nous aident à manipuler les monoïdes topologiques. Nous nous épargnons de rédiger la plupart des preuves ; il s'agit surtout de fixer ce qui est vrai.

	\begin{defi}
		Rappelons qu'un monoïde topologique est un objet en monoïdes dans la catégorie des espaces topologiques. Nous appelons $\mathrm{MndTop}$ la catégorie des monoïdes topologiques.
	\end{defi}
	
	\begin{prop}
		La catégorie des monoïdes topologiques admet toutes les limites et le foncteur d'oubli vers les espaces topologiques commute naturellement aux limites.
	\end{prop}

	\begin{defi}
		Soit $M$ un monoïde et $\mathcal{R}$ une relation d'équivalence sur $M$ telle que \begin{equation*}\label{eqn:quotient}\forall m,m',n,n' \in M, \,\,\, m\mathcal{R}m' \,\, \text{ et }\,\, n\mathcal{R}n' \, \implies mm'\mathcal{R}nn'. \tag{Q1}\end{equation*} La loi sur $M$ passe au quotient en une loi de monoïde sur $\sfrac{M}{\mathcal{R}}$ ayant pour élément neutre la classe du neutre. Si $M$ était un monoïde topologique, la topologie quotient sur $\sfrac{M}{\mathcal{R}}$ en fait encore un monoïde topologique.
		
		Pour tout sous-ensemble $\mathcal{Q}\subset M\times M$, il existe une relation d'équivalence sur $M$ contenant $\mathcal{Q}$ minimale pour l'inclusion et vérifiant (\ref{eqn:quotient}). On appelle $\sfrac{M}{\mathcal{Q}}$ le quotient par cette dernière relation.
	\end{defi}
	
	\begin{defi}
		Soit $M$ un monoïde topologique et $X$ un espace topologique. Une \textit{action continue de $M$ sur $X$} est morphisme de monoïdes $M\rightarrow \mathrm{Hom}_{\mathrm{Ens}}(X,X)$ tel que l'application déduite $M\times X \rightarrow X$ est continue.
		
		Lorsque $X$ est muni de structure algébriques additionnelles, par exemple lorsque c'est un groupe topologique, on défini une action continue de $M$ sur $X$ de manière identique, en imposant que l'image soit contenue dans les morphismes de groupes.
	\end{defi}
	
	\begin{prop}\label{psd_quotient_desc_1}
		Soit $M$ un monoïde topologique et $\mathcal{Q}\subset M\times M$. 
		
		\begin{enumerate}[itemsep=0mm]
			\item Le quotient $\sfrac{M}{\mathcal{Q}}$ représente le foncteur $$\mathrm{MndTop} \rightarrow \mathrm{Ens}, \,\,\, S \mapsto \left\{ a\, :  \, M \rightarrow S \,\, \text{ continu tel que}\, \forall (m,n) \in \mathcal{Q}, \,\, a(m)=a(n)\right\}.$$
			
			\item Les foncteurs
			$$\mathrm{Top}\rightarrow \mathrm{Ens}, \,\,\, X\mapsto \left\{\text{Actions continues de } \,\quot{M}{\mathcal{Q}}  \, \text{ sur } \, X \right\}$$
			et 
				$$\mathrm{Top}\rightarrow \mathrm{Ens}, \,\,\, X\mapsto \left\{\substack{\text{Actions continues de } \, M \, \text{ sur } \, X \\ \text{telles que } \, \forall (m,n) \in \mathcal{Q}, \, x\in X, \,\,\, m\cdot x=n\cdot x}\right\}$$ sont isomorphes. Le même résultat est vrai avec les catégories $\mathrm{MndTop}$, $\mathrm{AnnTop}$, etc comme catégories sources.
		
			\item La relation d'équivalence par laquelle on quotiente est la clôture réflexive, symétrique et transitive de $\{(amb,anb) \, |\, a,b\in M \text{  et  } (m,n)\in \mathcal{Q}\}$.
			\end{enumerate}
	\end{prop}
	
	\begin{ex}
		Un sous-monoïde $N<M$ est appelé distingué si $\forall m, \,\, mN=Nm$. On notera alors $N\triangleleft M$. Dans ce cas, la relation $m_1\mathcal{R}m_2 \leftrightarrow m_1 N=m_2N$ est une relation d'équivalence qui vérifie (\ref{eqn:quotient}). On note $\sfrac{M}{N}$ le quotient. Il est en bijection avec les classes à gauche et représente le foncteur $$\mathrm{MndTop} \rightarrow \mathrm{Ens}, \,\,\, S \mapsto \left\{ a\, :  \, M \rightarrow S \,\, \text{ continu tel que}\, N\subseteq \mathrm{Ker}(a)\right\}.$$
	\end{ex}
	
	\begin{rem}
		Attention, cette notion de sous-monoïde distingué est à prendre avec des pincettes puisqu'un monoïde n'est pas nécessairement distingué dans lui-même : considérer par exemple le monoïde des matrices carrées de taille $d$. La propriété universelle du quotient admet toujours un représentant, mais le monoïde peut être bien plus petit que l'ensemble des classes à gauche. De la même manière, le noyau d'un morphisme de monoïdes n'est pas toujours distingué.
	\end{rem}

	Lorsque $X$ est localement compact une action continue est exactement un morphisme de monoïdes vers $\mathrm{Hom}_{\mathrm{Top}}(X,X)$ muni de la topologie compacte ouverte. En revanche, rien ne garantit en général ni l'adjonction, ni même que cette topologie soit une topologie de monoïde. La seule topologie de monoïde sur $\mathrm{Hom}_{\mathrm{Top}}(X,X)$ serait la topologie "ouverte-ouverte" qui donne une notion d'action continue plus forte.
	
	\begin{defi}
		Soient $M$ et $N$ deux monoïdes topologiques et $\lambda\, : \, N \rightarrow \mathrm{End}_{\mathrm{Mnd}}(M)$ un morphisme de monoïdes tel que l'action $N\times M \rightarrow M$ déduite est continue. L'ensemble $M\times N$ muni de la loi $$\forall (m_1,m_2,n_1,n_2)\in M^2\times N^2, \,\,\, (m_1,n_1)\cdot_{\lambda}(m_2,n_2)=(m_1 \lambda(n_1)(m_2),n_1 n_2)$$ est un monoïde topologique que l'on note $M\rtimes_{\lambda} N$.
	\end{defi}
	
	\begin{prop}\label{prop_prdsdirect}
		Soient $M$ et $N$ deux monoïdes topologiques et $\lambda\, : \, N \rightarrow \mathrm{End}_{\mathrm{Mnd}}(M)$ un morphisme de monoïdes tel que l'action de $N$ sur $M$ déduite est continue.
		
		\begin{enumerate}[itemsep=0mm]
			\item Le monoïde topologique $M\rtimes_{\lambda} N$ représente le foncteur $$\mathrm{MndTop} \rightarrow \mathrm{Ens}, \,\,\, S \mapsto \left\{ \substack{(b\, : \, M \rightarrow S, \, c\, : \, N \rightarrow S) \,\, \text{ continus}\\ \text{tels que } \,\, \forall (n,m) \in N \times M, \,\,\, c(n) b(m)=b\left(\lambda(n)(m)\right) c(n)}\right\}.$$
			
			\item Les deux foncteurs $$\mathrm{Top} \rightarrow \mathrm{Ens}, \,\,\, X \mapsto \{ \text{Actions continues de }\, M\rtimes_{\lambda} N\, \text{ sur }\, X \}$$ et $$\mathrm{Top} \rightarrow \mathrm{Ens}, \,\,\, X \mapsto \left\{\substack{\text{Paire d'actions continues de }\,  M \,\text{ et }\, N\, \text{ sur }\, X\\ \text{telles que } \, \forall (m,n,x) \in M\times N\times X, \,\,\, n\cdot_N (m\cdot_M x)=(\lambda(n)(m))\cdot_M(n\cdot_N x)}\right\}$$ sont isomorphes. Le même résultat est vrai avec les catégories $\mathrm{MndTop}$, $\mathrm{AnnTop}$, etc comme catégories sources.
			
		\end{enumerate}
	\end{prop}
	
	\begin{defi}\label{defi_psd_quot}
		Soit $M,N$ deux monoïdes topologiques et $\lambda$ une action continue de $M$ sur le monoïde $N$. Soit également $I<N$ et $\kappa\, : \, I \rightarrow M$ un morphisme de monoïde. Nous supposons de plus que $I\triangleleft N$, que $$\forall (i,m)\in I\times M, \,\,\, \kappa(i)m=\lambda(i)(m)\kappa(i)$$ et que $$\forall (i,j,n)\in I^2 \times N \,\, \text{ tels que }\,\, ni=jn, \,\,\lambda(n)(\kappa(i))=\kappa(j).$$ Nous définissons $$\quot{\left(M\rtimes_{\lambda} N\right)}{\tsim I \tsim}:=  \text{quotient de} \, (M\rtimes_{\lambda} N) \, \text{par l'ensemble de couples} \,\{((\kappa(i),1_N),(1_M,i))\, |\, i\in I\}.$$
	\end{defi}
	
	\begin{rem}
		Les trois conditions ne sont pas nécessaire à la définition mais servent à capturer une famille de quotients raisonnables à décrire. Le première condition semble raisonnable quitte à remplacer $I$ par le sous-groupe distingué engendré. Le deuxième condition impose que $i$ et $\lambda(i)$ ait la même "action par conjugaison" sur $M$, la troisième que les actions par conjugaison de $N$ sur $I$ et sur $\kappa(I)$ soient cohérentes.
	\end{rem}
	
	\begin{prop}\label{psd_quotient_desc_2}
		Conservons le cadre de la Définition \ref{defi_psd_quot} et supposons que $\lambda(I)$ est formé automorphismes. La relation d'équivalence par laquelle nous quotientons est exactement la clôture symétrique de $$\forall (m,n,i)\in M\times N \times I, \,\,\, (\kappa(i),1_N)(m,n)=(\kappa(i)m,n)\mathcal{R} (\lambda(i)(m),in)=(1_M,i)(m,n).$$

	\end{prop}
	\begin{proof}
		La Proposition \ref{psd_quotient_desc_1} dit déjà que la relation d'équivalence par laquelle nous quotientons est la clôture symétrique et transitive de \begin{align*} & \forall (m,m',n,n',i)\in M^2\times N^2\times I, \\ & (m\lambda(n)(\kappa(i)m'),nn')=(m,n)(\kappa(i),1_N)(m,m') \mathcal{Q} (m,n)(1_M,i)(m',n')=(m\kappa(ni)(m'),nin').\end{align*} Fixons $(m,m',n,n',i)$ comme ci-dessus. Pour $j$ tel que $ni=jn$ et $n''=nn'$ et $m''=\lambda(n'')^{-1}(m)\lambda(n)(m')$, nous obtenons que $$(m,n)(\kappa(i),1_N)(m,m')=(\kappa(j),1_N)(m'',n'') \text{  et  } (m,n)(1_M,i)(m',n')=(1_M,j)(m'',n'').$$ Il nous reste simplement à démontrer que sa clôture symétrique est transitive. Pour cela, considérons une suite $$(\kappa(i)m,n)\mathcal{Q}(\lambda(i)(m),in) = (\kappa(j)m',n')\mathcal{Q} (\lambda(j)(m'),jn').$$ De l'égalité centrale nous déduisons que $\lambda(i)(m)=\kappa(j)m'$ et $in=n'$. Le terme de droite se réécrit donc $$(\lambda(j)(m'),jn')=\left(\lambda(ji)\left(\lambda(i)^{-1}(m')\right),jin\right)$$ En posant $m''=\lambda(i)^{-1}(m')$, on obtient que 
		\begin{align*}
			\kappa(ji)m''&= \kappa(j) \kappa(i)\lambda(i)^{-1}(m') \\
			&= \kappa(j) m'\kappa(i) \\
			&= \lambda(i)(m) \kappa(i) \\
			&= \kappa(i) m
		\end{align*} où le passage à la troisième ligne utilise l'égalité ci-dessus et où les passages à la deuxième ligne à la quatrième utilise les hypothèses de définition de $\sfrac{\left(M\rtimes_{\lambda} N\right)}{\sim I\sim}$. Les deux extrêmes de la ligne appartiennent donc à la relation comme $(\kappa(ji)m'',n)\mathcal{Q}(\lambda(ji)(m''),jin)$.
	\end{proof}

	Donnons à présent quelques propriétés sur les sous-groupes distingués et les quotients.
	
	\begin{prop}\label{identités_quotients}
		\begin{enumerate}[itemsep=0mm]
			\item Soit $N\triangleleft M$ deux monoïdes et $\mathcal{R}$ une relation d'équivalence vérifiant (\ref{eqn:quotient}). Alors, l'image de $N$ est distinguée dans $\sfrac{M}{\mathcal{R}}$.
			
			\item Soit $M_0 < M$. Pour que $\left(M_0\times \{1_N\}\right) \triangleleft \left(M\rtimes_{\lambda} N\right)$, il faut et suffit que $M_0 \triangleleft M$ et que pour tout $n\in N$, l'endomorphisme $\lambda(n)$ se restreigne-corestreigne en un endomorphisme surjectif de $M_0$. Dans ce cas, il existe un isomorphisme naturel de monoïdes topologiques $$\quot{\left(M\rtimes_{\lambda} N\right)}{\left(M_0 \times \{1_N\}\right)} \xrightarrow{\sim} \left(\quot{M}{M_0}\right) \rtimes_{\lambda} N.$$
			
			\item Gardons les notations du point précédent avec des monoïdes topologiques, supposons que $\lambda$ induit une action continue $N\times M \rightarrow M$ et ajoutons la donnée d'un sous-monoïde $I<N$ et d'un morphisme continu $\kappa\, : \, I\rightarrow M$ qui vérifie les hypothèses de la Définition \ref{defi_psd_quot} et tels que $\kappa^{-1}(M_0)\triangleleft N$. Alors les morphismes $\lambda$ et $\kappa$ passent au quotient en $\overline{\lambda}\, : \, \sfrac{N}{\kappa^{-1}(M_0)} \rightarrow \mathrm{End}_{\mathrm{Mnd}}(\sfrac{M}{M_0})$ qui fournit une action continue et $\overline{\kappa}\, : \, \sfrac{I}{\kappa^{-1}(M_0)} \rightarrow \sfrac{M}{M_0}$ qui vérifient encore les hypothèses de la Définition \ref{defi_psd_quot}. En appelant $M_1$ l'image de $M_0$ dans $\sfrac{\left(M\rtimes_{\lambda} N\right)}{\sim I \sim}$, elle y est distinguée et nous avons une identification naturelle $$\quot{\left(\quot{\left(M\rtimes_{\lambda} N\right)}{\sim I \sim} \right)}{M_1} \xrightarrow{\sim} \quot{\left(\quot{M}{M_0}\rtimes_{\overline{\lambda}} \quot{N}{\kappa^{-1}(M_0)}\right)}{\sim \sfrac{I}{\kappa^{-1}(M_0)}\sim}.$$
		\end{enumerate}
	\end{prop}
	\begin{proof}
		Les deux premiers énoncés se démontrent à la main en écrivant des égalités entre classes. Pour le troisième énoncé, utilisez les deux premiers pour démontrer les distinctions et construisez l'isomorphisme par propriété universelle.
	\end{proof}

	%\begin{prop}\label{presqisolim}
	%	Soit $A$ un anneau et $\varpi$ un élément ayant toutes ses racines $p^n$-ièmes. Soit $(f_n) \, : \,(B_n)_{n\geq 0} \rightarrow (C_n)_{n\geq 0}$ un morphisme de systèmes projectifs de $A$-modules indexés par $\N$ tels que tous les morphismes $f_n$ sont des presque-$(A,\varpi)$-isomorphismes. Le morphisme $$f\, :\, \lim_n B_n \rightarrow \lim_n C_n$$ est un presque-$(A,\varpi)$-isomorphisme.
	%\end{prop}
	%\begin{proof}
	%	Appelons $(\gamma_n)$ les morphismes de transition de $(C_n)$. La presque-nullité du noyau ne pose pas de problème. Pour le conoyau, soit $(c_n) \in \lim_n C_n$ et $k\geq 0$. Puisque chaque $f_n$ est un presque-$(A,\varpi)$-isomorphisme, on choisit une famille $(b_n)$ telle que $$\forall n, \,\,\, f_n(b_n)=\varpi^{\frac{1}{p^{k+1}}}c_n.$$ Il en découle que $$\forall n, \,\,\, \gamma_{n+1}(b_{n+1})-b_n \in \mathrm{Ker}(f_n).$$ Puisque $f_n$ est un presque-$(A,\varpi)$-isomorphisme $(\varpi^{\frac{1}{p^k}(1-\frac{1}{p})} b_n) \in \lim_n B_n$ et $$f\left((\varpi^{\frac{1}{p^k}(1-\frac{1}{p})} b_n)\right)=\varpi^{\frac{1}{p^k}} (c_n).$$
	%\end{proof}

	\vspace{1.5cm}	
	\printbibliography

\end{document}